\documentclass{article}
\usepackage[utf8]{inputenc}
\usepackage{graphicx}
\usepackage{amsmath,amsfonts,amsthm,amssymb}
\usepackage{hyperref}
\usepackage{xcolor}
\usepackage{comment}
\usepackage[margin=4cm]{geometry}
\usepackage{cleveref}
\usepackage{float}
\usepackage{tikz,pgfplots}
\usepackage{caption}
\usepackage{subcaption}
\usepackage{enumitem}
\usepackage{nicefrac}

\theoremstyle{plain}
\newtheorem{theorem}{Theorem}[section]
\newtheorem{conjecture}[theorem]{Conjecture}
\newtheorem{heuristic}[theorem]{Heuristic}
\newtheorem{proposition}[theorem]{Proposition}
\newtheorem{lemma}[theorem]{Lemma}

\theoremstyle{definition}
\newtheorem{example}[theorem]{Example}
\newtheorem{definition}[theorem]{Definition}

\theoremstyle{remark}
\newtheorem{remark}[theorem]{Remark}

\definecolor{ourPurple}{RGB}{128,0,128}

\newcommand{\R}{\mathbb{R}}
\newcommand{\C}{\mathbb{C}}
\newcommand{\Z}{\mathbb{Z}}

\newcommand{\PP}{\mathbb{P}}
\renewcommand{\P}{\mathbb{P}}

\newcommand{\Sing}{\mathrm{Sing}\,}

\newcommand{\ol}{\overline}

 \DeclareMathOperator{\Tr}{Tr}
    \DeclareMathOperator{\Res}{Res}

\title{Adjoints and Canonical Forms of Polypols}
\author{Kathl\'en Kohn, Ragni Piene, Kristian Ranestad,  Felix Rydell, \\ Boris Shapiro, Rainer Sinn,  Miruna-Stefana Sorea, Simon Telen}

\begin{document}

\maketitle

\begin{abstract}
Polypols are natural generalizations of polytopes, with boundaries  given by nonlinear algebraic hypersurfaces. We describe polypols in the plane and in 3-space that admit a unique adjoint hypersurface and  study them from an algebro-geometric perspective. We relate planar polypols to positive geometries introduced originally in particle physics, and identify the adjoint curve of a planar polypol with  the numerator of the canonical differential form associated with the positive geometry. We settle several cases of a conjecture by Wachspress claiming that the adjoint curve of a regular planar polypol does not intersect its interior. In particular, we provide a complete characterization of the real topology of the adjoint curve for arbitrary convex polygons. Finally, we determine all types of planar polypols such that the rational map sending a polypol to its adjoint is finite, and explore connections of our topic with algebraic statistics.
\end{abstract}

\tableofcontents

\section{Introduction}\label{Introduction}
Polytopes are very familiar geometric objects, with boundaries given by linear equations. Their beautiful and important properties have been extensively studied from different perspectives and have numerous applications. The present paper studies several more general classes of real domains/shapes  with non-linear algebraic boundaries, known as \textit{polypols}, \textit{polycons}, \textit{polypoldra}, \textit{positive geometries}, etc., in the existing literature. They share some of their properties with polytopes, but, in general, are quite different from the latter.  They find applications 
for example in finite element methods, quantum physics, and algebraic statistics. Unlike the case of polytopes,  mathematical properties of such objects are still relatively unexplored, and a rigorous theoretical framework is currently largely missing. This paper summarizes the recent efforts of our 
 reading group,  and its  main goal is to establish  polypols  as a separate topic of study, from the perspective of complex and real algebraic geometry. 
 We focus on polypols in the plane and in  3-space.  
 Motivated by applications, we mainly study the existence, uniqueness, and real topology of \textit{adjoint hypersurfaces} associated with polypols, as well as formulas for their \textit{canonical differential forms}. But obviously, we have barely even scratched the surface of a large terra incognita. 

\subsection{Previous studies}

In the 1970's, E.~Wachspress  introduced  \emph{polypols} as  bounded semialgebraic subsets of $\R^n$ that generalize polytopes~\cite{MR0426460}, \cite{Wachspress16}. 
He aimed to generalize barycentric coordinates from simplices to arbitrary polytopes and further to polypols.
Wachspress's work mainly focused on planar polypols with rational boundary curves.
To define barycentric coordinates on such a \emph{rational polypol} $P$ in $\R^2$, he introduced the \emph{adjoint curve} $A_P$ as the minimal degree curve that passes through both the singular points of the boundary curves and their intersection points that are ``outside'' of $P$; see Figure~\ref{fig:polypolsAdjoints} for examples. As he mentions in the introduction to his early book  \cite{MR0426460},   a more ambitious goal of his study was to extend the finite element method (which at that time and ever since  has been  extremely popular in numerical methods)  by using  both arbitrary polytopes and polypols as basic approximating elements for multi-dimensional domains.     

\begin{figure}[htb]
    \centering
    \includegraphics[height=2cm]{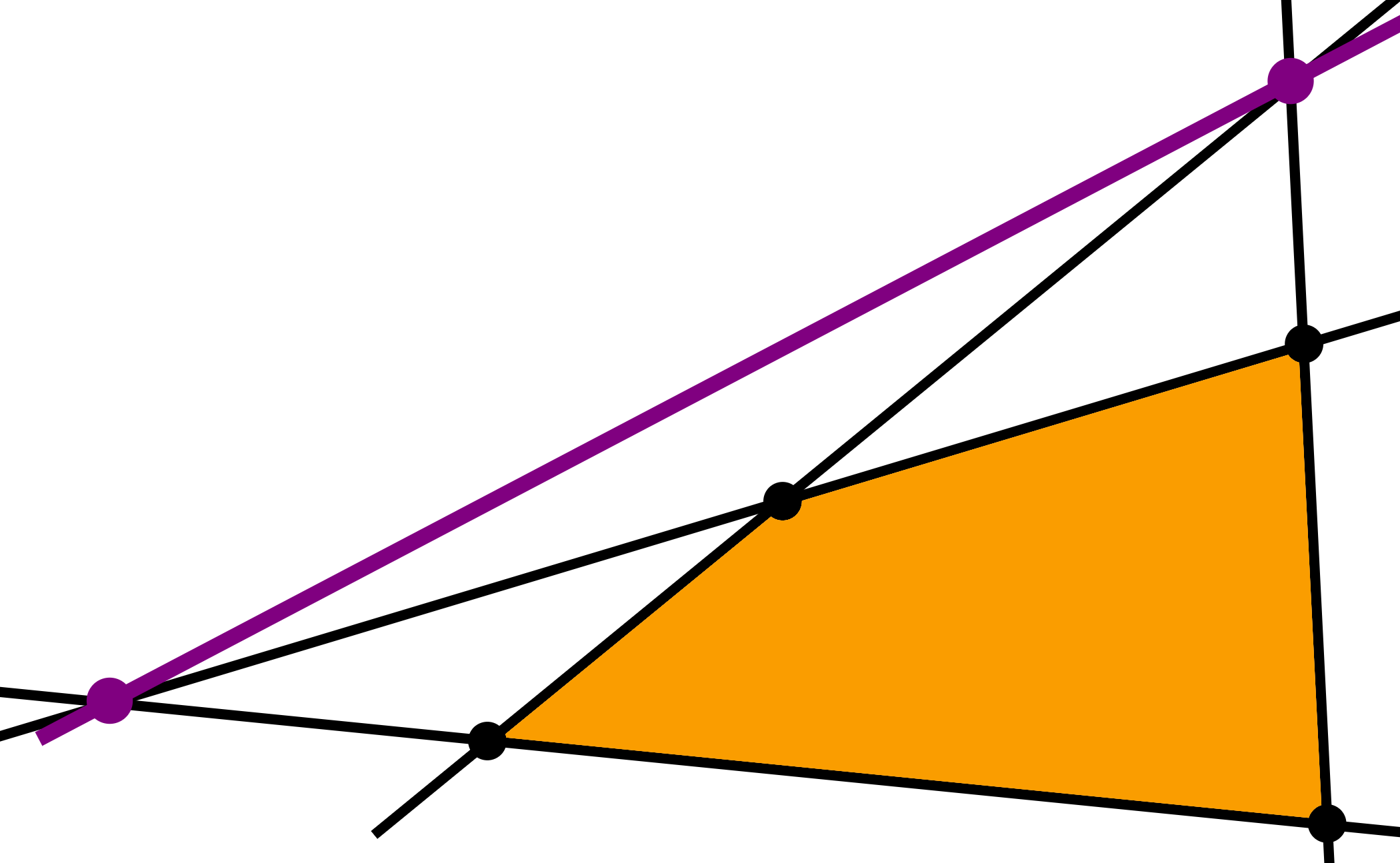} \hfill
    \includegraphics[height=3cm,width=4cm]{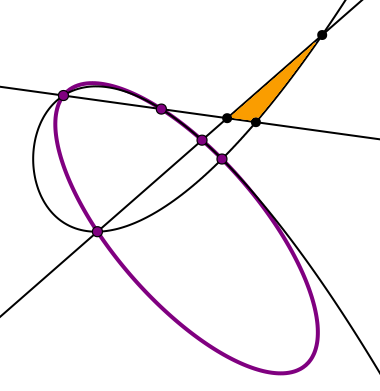} \hfill
    \includegraphics[height=3cm,width=4cm]{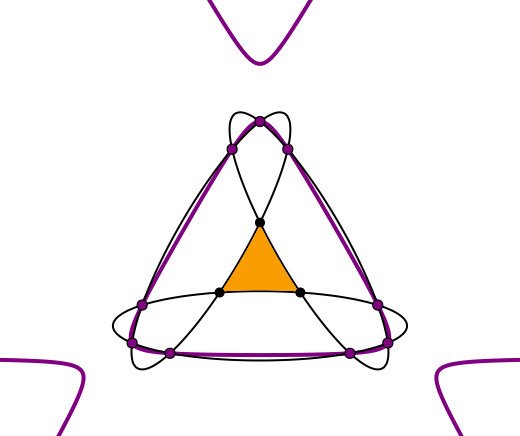}
    \caption{Three polypols (shaded in orange) and their adjoints (purple). The boundary curves and the vertices of the polypols are shown in black. The singular points and ``outside'' intersection points of the boundary curves are purple.}
    \label{fig:polypolsAdjoints}
\end{figure}

Several years ago, N.~Arkani-Hamed, Yu.~Bai, and T.~Lam~\cite{arkani2017positive}  introduced \emph{positive geometries} in their studies of scattering amplitudes in particle physics. These are given by a real semi-algebraic subset $P$ of an $n$-dimensional complex variety together with a unique rational \emph{canonical $n$-form} $\Omega(P)$ that is recursively defined via the boundary components of $P$. Important motivating examples that led to the introduction of positive geometries include both the \emph{amplituhedron} \cite{arkani2015positive} and the \emph{ABHY associahedron} \cite{arkani2018scattering}.
The \emph{push-forward formula} expresses their canonical form as a global residue over the solutions to a system of rational function equations \cite[Sec.~7]{arkani2021stringy}. Recently, these equations were identified as the \emph{likelihood equations} for \emph{positive statistical models}, relating positive geometries to algebraic statistics \cite[Sec.~6]{sturmfels2021likelihood}.

Usual polytopes are the easiest examples of both polypols and positive geometries.
The adjoint hypersurface $A_P$ and canonical form $\Omega(P)$ of a polytope $P$ are known to be unique;
see~\cite{kohn2019projective} resp.~\cite{arkani2017positive}.
In particular, unpublished lecture notes by C.~Gaetz show that the rational canonical form of a polytope $P$ has its poles along the boundary $\partial P$ and its zeros along the adjoint hypersurface $A_P$. 

Observe that in complex algebraic geometry, there is a classical notion of adjoints of a hypersurface.  Namely, the holomorphic $(n-1)$-forms on a nonsingular hypersurface $X$ in $\P^n$ are all residues of rational $n$-forms on $\P^n$ with simple poles along $X$ and zeros along an adjoint hypersurface $A_X$ to $X$; see the adjunction formula in \cite[p.~147]{GH}. When $X$ has degree $d$, the adjoint hypersurfaces are all hypersurfaces of degree $d-n-1$.  When $X$ is singular, then the holomorphic $(n-1)$-forms on a resolution of singularities $X'\to X$ are obtained from the rational $n$-forms  on $\P^n$ with zeroes along adjoint hypersurfaces that pass through  singularities of $X$. 
The adjoints we discuss in the present paper are slightly different; they pass through some, but not  all, singularities of the boundary hypersurface.

\subsection{Main results and outline}

We focus on the basic definitions and general properties of rational polypols. 
In Section \ref{Unique Adjoints and Canonical Forms}, we introduce \textit{quasi-regular} rational polypols in $\R^2$ and show that they are examples of planar positive geometries.  
We formally define and establish the uniqueness of both the adjoint curve $A_P$ and the canonical form $\Omega(P)$ of such a polypol $P\subset \R^2$.
Moreover, we provide an explicit formula for $\Omega(P)$ in terms of defining equations of $A_P$ and the boundary curves of the polypol $P$ (Theorem \ref{thm:rationalPolypolsArePosGeometries}):
As in the polytopal case, the rational canonical form $\Omega(P)$ has its poles along the boundary $\partial P$ and its zeros along the adjoint curve $A_P$.

E.~Wachspress used the adjoint curve $A_P$ of a \emph{regular} rational polypol $P$ in $\R^2$ to define barycentric coordinates on $P$.
These coordinates should be rational functions that are positive on the interior of $P$ and have poles on the adjoint curve $A_P$. For several simple polypols, such as convex polygons, the rational functions suggested by E.~Wachspress indeed enjoy   these properties, but for his coordinates to make sense for regular rational polypols $P$ in general, he conjectured that the adjoint curve $A_P$  does not pass through the interior of $P$ (Conjecture \ref{conj:Wachs}). 
This intriguing claim is  at present widely open. 
In Section~\ref{Outside Adjoints}, we define regular polypols and discuss Wachspress's conjecture. In particular, for the case of convex polygons, we show that the adjoint curve is hyperbolic and provide an explicit description of its nested ovals (Theorem \ref{thm:hyperbolicadjoint}). Further, for polypols defined by three ellipses, which is the first unsolved case of Wachspress's conjecture, we identify all possible topological types of triples of ellipses and prove Wachspress's  conjecture for 33 out of all 44 of them (Theorem \ref{thm:wachspressellipses}). 
These 33 cases include all types of maximal real intersection, i.e., where the ellipses meet pairwise in four real points.
The proof, which is deferred to Appendix \ref{sec:appendix}, requires careful cataloging of these 44 configurations.

 In Section \ref{Finite Adjoint Correspondences}, we study the map that associates to a given planar polypol $P$ its  adjoint curve $A_P$. 
We identify all types of planar polypols for which this map is finite (Theorem \ref{th: finite adjoint}).
Using numerical monodromy computations, we calculate the degree of the adjoint map in the case of heptagons, answering a question posted in~\cite{kohn2018moment}:
our computations indicate that a general quartic curve is the adjoint of 864 distinct heptagons (Proposition \ref{prop: hept}).

In Section \ref{Statistics and Pushforward}, we recall the connection between positive geometries and positive models in algebraic statistics. An essential role is played by the push-forward formula for canonical forms, see Heuristic \ref{heur:pushfwd}. For the toric models studied in \cite{amendola2019maximum}, this gives an explicit rational expression for a global residue over all the critical points of the log-likelihood function (Proposition \ref{prop:toricamplitude}). In Section \ref{subsec:CFlikelihood}, we propose to use this as a numerical \emph{trace test} to verify the completeness of a set of numerically obtained approximate critical points. Moreover, Example \ref{ex:toricsquare3} illustrates how the explicit formula from Theorem \ref{thm:rationalPolypolsArePosGeometries} can be used to design a family of trace tests for any 2-dimensional toric model. 
 
 In Section \ref{3D-Polypols}, we discuss three-dimensional polypols.
We provide general criteria for a complex algebraic surface to be the algebraic boundary of a polypol $P\subset \R^3$, and give a number of examples where all boundary components are quadric surfaces such that these criteria are sufficient to have a unique adjoint surface $A_P$.

\section{Adjoints and Canonical Forms of Planar Polypols}
\label{Unique Adjoints and Canonical Forms}

In Section~\ref{ssec:adjoints}, we start by defining rational polypols in the complex projective plane and showing that they have unique adjoint curves.
In Section~\ref{ssec:positiveGeometries}, we define positive geometries and their canonical forms and compare planar positive geometries with rational polypols. 
In Section~\ref{ssec:canonicalForm}, we show that \emph{quasi-regular} rational polypols are always positive geometries and 
we provide an explicit formula for the canonical form of such a polypol in terms of its adjoint and boundary curves.
In Section~\ref{ssec:additivity}, we see that unions and differences of quasi-regular rational polypols provide examples of positive geometries with a unique canonical form, but no unique adjoint curve. For this, we introduce pseudo-positive geometries, following~\cite{arkani2017positive}, and  show the additivity of their canonical forms in dimensions one and two.

\subsection{The adjoint curve of a rational polypol}
\label{ssec:adjoints}
Let $C \subset \P^2$ be a complex plane curve with $k \geq 2$ irreducible components $C_1,\dots,C_k$. 
Assume there are $k$ points $v_{12}\in C_1\cap C_2$, \ldots, $v_{k1}\in C_k\cap C_1$ such that $v_{ij}$ is nonsingular on $C_i$ and $C_j$, and that $C_i$ and $C_j$ intersect transversally at $v_{ij}$.
Then we say that the irreducible curves $C_i$ and the points $v_{ij}$ form a \emph{polypol}~$P$.
The set of points $V(P):=\{v_{ij}\}$ is called the \emph{vertices} of $P$, and the complement $R(P):=\Sing C \setminus V(P)$ of the vertices in the singular locus of $C$ is called the set of \emph{residual points} of $C$.
We say that the polypol $P$ is \emph{rational} if the curves $C_1, \ldots, C_k$ are rational.
All examples in Figure~\ref{fig:polypolsAdjoints} are rational polypols such that all vertices (black) and all residual points (purple) are real.  

We consider the \emph{partial normalization} $\nu:Z\to C$ of $C$, obtained by desingularizing all the residual points $R(P)\subset C$. One can construct $Z$ as the strict transform of $C$ under the sequence of blow-ups  $X\to \P^2$ of all points in, and infinitely near, $R(P)$. 
Let $\omega_Z$ (resp. $\omega_C$) denote the dualizing sheaf of $Z$ (resp. $C$).
It follows from \cite{piene1978ideals}  that the trace map $\Tr_\nu:\nu_*\omega_Z\to \omega_C$ induces an isomorphism
\[\nu_*\omega_Z \overset{\sim}\rightarrow \mathcal C_\nu \omega_C,\] where the sheaf of ideals
$\mathcal C_\nu$ is the conductor of $\nu$.
This result goes back to D.~Gorenstein and M.~Rosenlicht  \cite{rosenlicht1952}, in
the case when $\nu: Z\to C$ is the normalization map.

\begin{definition}
Set $d_i:=\deg C_i$ and $d:=\sum_{i=1}^k d_i$. 
An \emph{adjoint curve} $A_P$ of the polypol~$P$ is a curve defined by a polynomial $\alpha_P \in H^0( \P^2, \mathcal O_{\P^2}(d-3))$ which maps to a non-zero element in $H^0(C,\mathcal C_\nu\omega_C)\subseteq H^0(C,\omega_C)$ under the restriction map 
\[\rho: H^0(\P^2,\Omega^2_{\P^2}\otimes \mathcal O_{\P^2}(C))=H^0(\P^2,\mathcal O_{\P^2}(d-3))\to H^0(C,\omega_C),\]
obtained from the short exact sequence
$0\to \Omega^2_{\P^2} \to \Omega^2_{\P^2}\otimes \mathcal O_{\P^2}(C)\to \omega_C \to 0$.
\end{definition}

Note that an element in $H^0( \P^2, \mathcal O_{\P^2}(d-3))$ which maps to an element in $H^0(C,\mathcal C_\nu\omega_C)$ defines a curve of degree $d-3$ which contains the 0-dimensional subscheme of $C$ defined by $\mathcal C_\nu \subset \mathcal O_C$. 
In particular, if all residual points are nodes on $C$, then 
$\mathcal C_\nu$ is just the product of the maximal ideals 
$\mathfrak m_{C,p} \subset \mathcal O_{C,p}$ for $p\in R(P)$, and therefore an adjoint curve is a curve passing through all the residual points of $P$.

\begin{proposition}
\label{prop:uniqueAdjoint}
A rational polypol has precisely one adjoint curve. Moreover, the adjoint curve does not contain any of the boundary curves and does not pass through any of the vertices.
\end{proposition}

\begin{proof}
Since $H^0( \P^2,\Omega^2_{\P^2})=H^1(\P^2,\Omega^2_{\P^2})=0$, the restriction map $\rho$ is an isomorphism. Hence, to show the first statement, it suffices to show that $h^0(C,\mathcal C_\nu\omega_C)=1$, since then $\dim \rho^{-1}H^0(C,\mathcal C_\nu\omega_C)=1$. We have $Z=\cup_{i=1}^k \widetilde C_i$, where the $\widetilde C_i\to C_i$ are the normalization maps, and where the only singularities of $Z$ are the points above $V(P)$. By the arithmetic genus formula for the reducible curve $Z$ \cite[Thm.~3, p.~190]{hironaka1957on}, we have
 \[\textstyle g_a(Z)=\sum_{i=1}^k g_a(\widetilde C_i) +\#V(P) -(k-1)=0+k-k+1=1.\]
So $Z$ has arithmetic genus 1; hence we have $h^1(Z,\mathcal O_Z)=1$ and, by duality, $h^0(Z,\omega_Z)=1$. Since 
$H^0(C,\mathcal C_\nu\omega_C)=H^0(C,\nu_*\omega_Z)$ and 
$H^0(C,\nu_*\omega_Z)=H^0(Z,\omega_Z)$, we conclude that
$h^0(C,\mathcal C_\nu\omega_C)=1$.

To show the second statement, we observe that since the arithmetic genus of the curve $Z$ is 1, $h^0(Z,\omega_Z)=1$ and any non-zero section $\sigma$ of $\omega_Z$ is nowhere vanishing. Since $H^0(Z,\omega_Z)=H^0(C,\mathcal C_\nu\omega_C) \subset H^0(C,\omega_C)$, the lift $\rho^*\sigma \in H^0(\P^2,\mathcal{O}_{\P^2}(d-3))$ defines the adjoint curve. Since $\sigma$ was nowhere vanishing on $Z$, $\rho^*\sigma$ cannot vanish on any boundary curve $C_i$. Similarly, if $\rho^*\sigma$ had a zero at a vertex, then $\sigma$ would vanish at the point of $Z$ mapping to that vertex.
\qedhere
\end{proof}

\begin{example}\label{ex:adjoint}
Here is a more direct proof of the \emph{existence} of an adjoint curve, under 
the assumption that all the residual points are nodes. Namely, assume 
that the boundary curves $C_i$ are rational nodal curves that intersect transversally. 
Then an adjoint $A_P$ is a curve of degree $d-3$ that passes through the residual points $R(P)$.

There are two types of residual points:
the singularities (nodes) of each boundary curve $C_i$,
and the intersection points of the boundary curves that are not one of the $k$ vertices of the polypol~$P$.
Since $C_i$ is rational, it has $\binom{d_i-1}2$ nodes. The total number of intersection points of the curves $C_i$ is $\sum_{1\le i < j\le k} d_id_j$.  
Hence, the number of residual points of $P$  is
\begin{align*}
\textstyle   \# R(P) =  \sum_{i=1}^k\binom{d_i-1}2+\sum_{1\le i < j\le k} d_id_j-k =  
 \frac{1}2 \sum_{i=1}^k(d_i^2-3d_i)+\sum_{1\le i < j\le k} d_id_j.
\end{align*}

This is the same as the dimension of the space of curves of degree $d-3$:
\begin{align*}
  \textstyle h^0(\P^2,\mathcal O_{\P^2}(d-3))-1 
=\binom{d-1}2 -1 = \frac{1}{2}(d^2 - 3d) =  \frac{1}2 \bigl((\sum_{i=1}^k d_i)^2-3\sum_{i=1}^k d_i\bigr) \\
\textstyle =\frac{1}2\bigl(\sum_{i=1}^k d_i^2+2\sum_{1\le i < j\le k} d_id_j\bigr)-\frac{3}2 \sum_{i=1}^k d_i 
\textstyle = \frac{1}2 \sum_{i=1}^k(d_i^2-3d_i)+\sum_{1\le i < j\le k} d_id_j.
\end{align*}
 
Since the condition to pass through a point is  linear, this shows that there exists at least one adjoint curve. 
\end{example}

\begin{remark}
E.~Wachspress gave a more explicit and intuitive, but less formal, construction of the adjoint in comparison to the above proof of Proposition~\ref{prop:uniqueAdjoint}.
For rational polypols with boundary curves that intersect non-transversally or have more complicated singularities than nodes, he required  the adjoint curve to have appropriate multiplicities at the resulting residual points~\cite{Wachspress16}.
\end{remark}

\begin{remark}
\label{rem:singleBoundary}
A more general notion than polypols, as we defined them above, include semi-algebraic subsets $P$ of the plane with an irreducible boundary curve $C$. 
If $C$ is a rational nodal curve and $P$ has exactly one of the nodes on its Euclidean boundary, then
there is a unique adjoint curve of degree $\deg(C)-3$ passing through the remaining nodes.
For example, for the ampersand curve in Figure \ref{fig:ampersand}, if we let $P$ be one of the two regions with exactly one of the three nodes on its boundary, then the adjoint is the line passing through the two other nodes.
As in our discussion above, there is also a unique adjoint curve if the rational curve $C$ has more complicated singularities. 
\begin{figure}[hbt]
    \centering
    \includegraphics[scale=2,trim={0.15cm 0.14cm 0.14cm 0.15cm},clip]{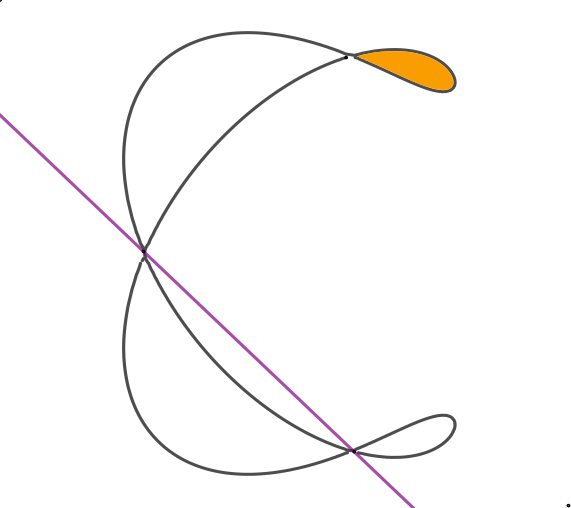}
    \caption{The ampersand curve given by $(y^2- x^2)(x - 1)(2x - 3) - 4(x^2 + y^2 - 2x)^2 = 0$ bounds a semi-algebraic set with a unique adjoint line.
    }
    \label{fig:ampersand}
\end{figure}
\end{remark}

\begin{remark}
\label{rem:triangleadjoint}
One might hope that the adjoint curve is nonsingular as long as the boundary curves intersect transversally and have only nodes as singularities.  However, this is \emph{not} true, as the following example demonstrates:

A rational polypol with boundary consisting of three conics has a cubic adjoint curve passing through the nine singular points that are not vertices of the polypol.  Figure \ref{fig:conicTriangle} shows an example with vertices $A,B,C$, where the adjoint curve is the union of three lines.  
To think of the polypol as a real semi-algebraic set, one needs to specify on each conic which segment is part of the boundary. With a choice of two segments for each conic, there are 8 different real polypols with vertices $A,B,C$ and boundary on the three conics. In particular, this example shows that cubics passing through $9$-tuples of intersection points of $3$ conics behave differently than cubics passing through $9$-tuples of general points. 
\begin{figure}[hbt] 
    \centering
    \includegraphics[width=0.5\textwidth,height=6cm,trim={0.15cm 0.14cm 0.14cm 0.15cm},clip]{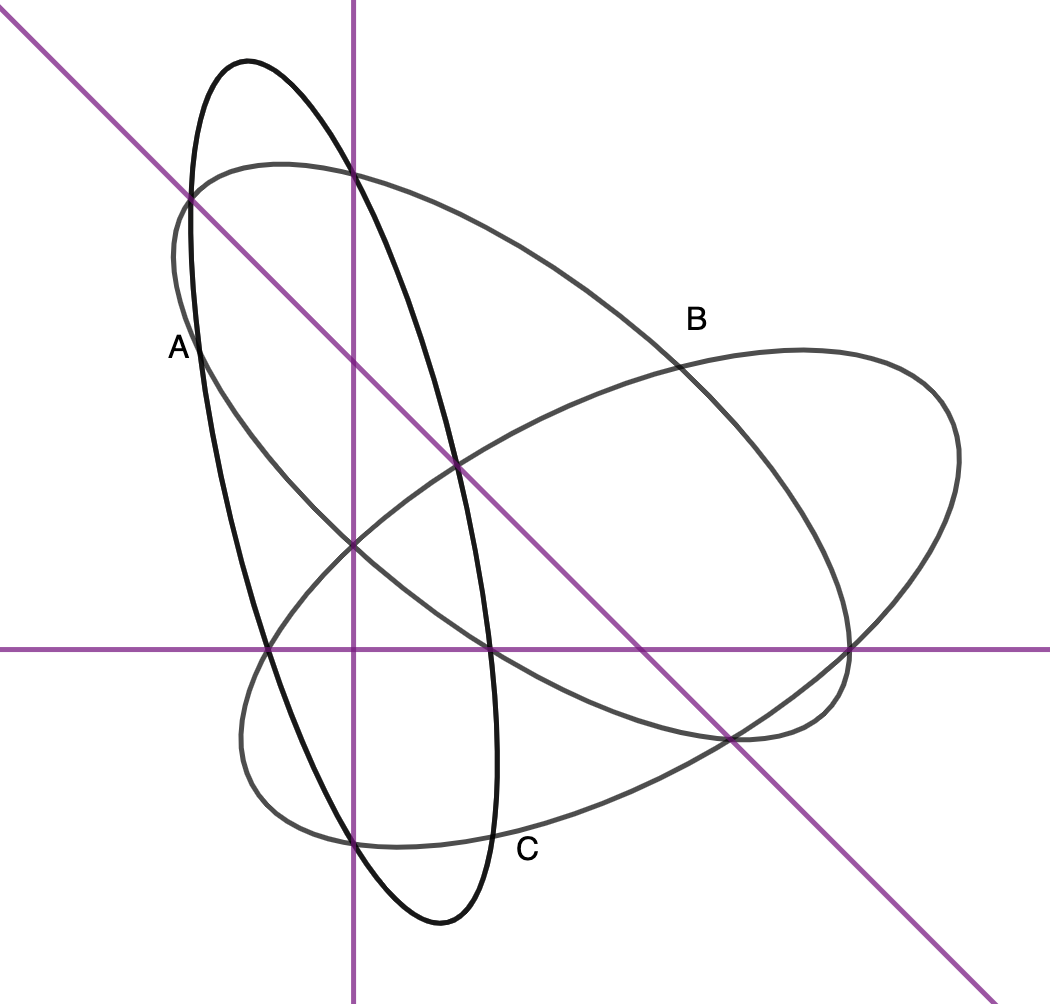}
    \caption{Polypol defined by three conics and vertices $A$, $B$, $C$. 
    The adjoint curve is the union of three lines.
    }
    \label{fig:conicTriangle}
\end{figure}
\end{remark}

\subsection{Positive geometries and their canonical forms}
\label{ssec:positiveGeometries}

We begin by defining positive geometries and their canonical forms, following~\cite{arkani2017positive}.
Let $X$ be a projective, complex, irreducible, $n$-dimensional variety, $n \geq 0$, defined over $\R$. 
Moreover, let $X_{\geq 0} \subset X(\R)$ be a non-empty closed semi-algebraic subset of the real part of $X$ such that
its Euclidean interior $X_{>0}$ is an open oriented $n$-dimensional manifold whose Euclidean closure is again $X_{\geq 0}$.
We also assume that $X_{>0}$ is contained in the nonsingular part of $X$, but we do not assume that $X$ is normal.
We write $\partial X_{\geq 0} := X_{\geq 0} \setminus X_{>0}$ for the Euclidean boundary of $X_{\geq 0}$
and denote by $\partial X$ the Zariski closure of $\partial X_{\geq 0}$ in $X$.
Let $C_1, C_2, \ldots, C_k$ be the irreducible components of $\partial X$.
For each component $C_i$, we denote by $C_{i,\geq 0}$
the Euclidean closure of the interior of $C_i \cap X_{\geq 0}$ in the subspace topology on $C_i(\R)$.

\begin{definition}
\label{def:posGeometry}
$(X,X_{\geq 0})$ is a \emph{positive geometry} if there is a unique non-zero rational $n$-form $\Omega(X,X_{\geq 0})$, called its \emph{canonical form}, satisfying the following recursive axioms:
\begin{enumerate}
    \item[(a)] If $n=0$, then $X = X_{\geq 0}$ is a point and we define $\Omega(X,X_{\geq 0}) := \pm 1$, depending on the orientation of the point.
    \item[(b)] If $n>0$, we require for every $i = 1, 2, \ldots, k$ that the boundary component $(C_i, C_{i,\geq 0})$ is a positive geometry whose canonical form is the residue of $\Omega(X,X_{\geq 0})$ along $C_i$:
    \begin{align*}
        \mathrm{Res}_{C_i} \Omega(X,X_{\geq 0}) = \Omega(C_i,C_{i,\geq 0}),
    \end{align*}
    and that $\Omega(X,X_{\geq 0})$ is holomorphic on $X\setminus \cup_{i=1}^k C_i$.
\end{enumerate}
\end{definition}

In particular, for a positive geometry $(X,X_{\geq 0})$, the variety $X$ cannot have non-zero holomorphic $n$-forms, because otherwise $\Omega(X,X_{\geq 0})$ could not be unique.
Conversely, the uniqueness of the canonical form follows immediately if $X$ has no non-zero holomorphic differential forms: if $\Omega$ and $\Omega'$ are canonical forms, then their difference $\Omega-\Omega'$ is a holomorphic form on $X$, and hence it must vanish identically. 

\begin{example}
\label{ex:posGeomOneDim}
Let $n=1$ and let $(X,X_{\geq 0})$ be a one-dimensional positive geometry.
Then $X$ must be a rational curve as these are the only projective complex curves $X$ without non-zero holomorphic $1$-forms. 
Moreover, $X_{\geq 0}$ can neither be the empty set nor the whole $X(\R)$, as otherwise $\partial X = \emptyset$, but we do not allow the empty set to be a positive geometry.
Hence, $X_{\geq 0}$ must be a finite disjoint union of closed segments in $X(\R)$, where each open segment in $X_{> 0}$ is nonsingular.

Conversely, any finite disjoint union of closed, real segments in a rational curve $X$, such that the open segments are nonsingular,  is a positive geometry. Any such segment can be rationally and birationally parameterized by a closed interval $[a,b] := \{ (1:t) \in \P^1(\R) \mid a \leq t \leq b \}$. Its canonical form can be identified with the
canonical form of the interval $[a,b] \subset \mathbb P^1(\R)$, which is equal to
\begin{align}
\label{eq:canonicalFormInterval}
    \Omega(\P^1, [a,b]) = \frac{1}{t-a} dt - \frac{1}{t-b} dt = \frac{b-a}{(t-a)(b-t)} \, dt,
\end{align}
where $t$ is the coordinate on the affine chart $\{ (1:t) \} \subset \P^1(\R)$. 
This holds because we have $\Res_{a} \Omega(\P^1, [a,b])=\frac{b-a}{b-t}|_{t=a}=1$ and 
$\Res_{b} \Omega(\P^1, [a,b])  =-\frac{b-a}{t-a}|_{t=b}=-1$. Here the interval $[a,b]$ is oriented along the increasing $t$-direction.
The canonical form of a disjoint union of closed intervals is the sum of the individual canonical forms (see Lemma~\ref{add1}).
\end{example}

We now 
move to dimension two and study
positive geometries $(X,X_{\geq 0})$ in the plane $X = \P^2$.
By Example~\ref{ex:posGeomOneDim}, each boundary component $(C_i, C_{i, \geq 0})$ is a union of closed (real) segments on a rational curve $C_i$ such that each open segment (in $C_{i,> 0}$) is nonsingular.
We call the endpoints of the closed segments the \emph{vertices} of the positive geometry.

\begin{definition}
\label{def:quasi-regular}
A \emph{real polypol} $P$ is a polypol with real boundary curves $C_i$, real vertices $v_{i-1,i}\in C_{i-1}\cap C_i$, and a given choice of segments 
connecting $v_{i-1,i}$ to $v_{i,i+1}$ in $C_i(\R)$, called the \emph{sides} of the polypol, and a closed semi-algebraic set $P_{\ge 0}$ 
such that each connected component of its interior is simply connected 
and its boundary is exactly the union of the sides of the polypol. 
A \emph{quasi-regular polypol} is a real polypol whose sides are contained in the nonsingular locus of $C_i$.
\end{definition}

\begin{example}
There are eight real polypols in \Cref{fig:conicTriangle} as discussed in \Cref{rem:triangleadjoint}, all of which are quasi-regular.
\end{example}

\begin{proposition}
\label{prop:posGeometriesAreRationalPolypols}
If $(\P^2, X_{\geq 0})$ is a positive geometry 
with $k \geq 2$ 
boundary components $(C_i, C_{i, \geq 0})$, $C_i \neq C_j$ for $i \neq j$,
and $k$ vertices $v_{12}\in C_1\cap C_2$, \ldots, $v_{k1}\in C_k\cap C_1$ such that $C_i$ and $C_j$ intersect transversally at $v_{ij}$,
then the curves $C_i$ and vertices $v_{ij}$ define a quasi-regular rational polypol with sides $C_{i, \geq 0}$.
\end{proposition}

\begin{proof}
The assumption that $X_{> 0}$ is orientable implies in $\P^2(\R)$ that it is a union of disks, i.e., simply connected sets.
Since the boundary components $(C_i, C_{i, \geq 0})$ are positive geometries themselves, each boundary curve $C_i$ is rational by Example~\ref{ex:posGeomOneDim}, and the open segments between consecutive vertices are nonsingular.  Since the curves $C_i$ are pairwise distinct, the curve $\cup_i C_i$ and the vertices $v_{ij}$ form a quasi-regular rational polypol with sides $C_{i, \geq 0}$. 
\end{proof}

\begin{remark}
\label{rem:specialPosGeometries}
There are more positive geometries $(\P^2, X_{\geq 0})$ than  quasi-regular rational polypols.
Such other positive geometries can be as follows: 
\begin{enumerate}
    \item there is a single rational boundary curve $C$ as in Remark~\ref{rem:singleBoundary}, or
    \item some vertices are either singular on their boundary curves (see Figure \ref{fig:rmk2-11} for an example) or they are non-transversal intersections of the boundary curves, or
   \item  at least one rational boundary curve $C_i$ contributes to several parts of the boundary, i.e.,         $C_{i,> 0}$ is a union of at least two disjoint nonsingular intervals, as in Example~\ref{ex:conicQuadrangle}. 
\end{enumerate}
A positive geometry may also be the union of these, see \Cref{ssec:additivity}.
\begin{figure}[h!]
        \centering
        \input{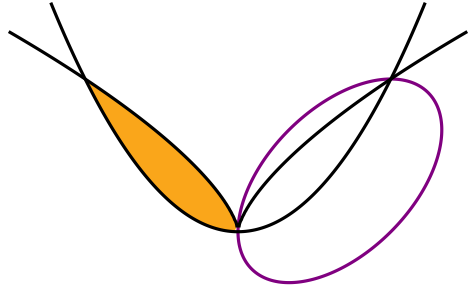}
        \caption{A positive geometry 
        bounded by a parabola and a cubic with a cusp which is a vertex.
        The canonical form is equal to $\frac{x^2+xy+y^2+x}{(y-x^2)(x^2-y^3)}\,\,dx\wedge dy$. Note that the adjoint curve is tangent to the cusp.}
        \label{fig:rmk2-11}
    \end{figure}
\end{remark}

\begin{example}
\label{ex:conicQuadrangle}
Let us consider two real irreducible conics $C_1, C_2$ in the plane that intersect in four real points $v_1, \ldots, v_4$, and let $X_{\geq 0}$ be the simply connected semi-algebraic subset of $\P^2(\R)$ that has all four points $v_1, \ldots, v_4$ on its Euclidean boundary; see Figure~\ref{fig:conicQuadrangle}.
We will see in Example~\ref{ex:conicQuadrangle2} that $(\P^2, X_{\geq 0})$ is a positive geometry with canonical form 
\begin{align}
\label{eq:conicQuadrangle}
    \frac{\alpha_Q}{f_1 f_2} \, dx \wedge dy,
\end{align}
where $f_i$ is a defining equation of the conic $C_i$ and $\alpha_Q$ is a defining equation of the adjoint line $A_Q$ of the quadrangle $Q$ that is the convex hull of $v_1, \ldots, v_4$.
\begin{figure}[hbt]
    \centering
    \includegraphics[width=0.7\textwidth]{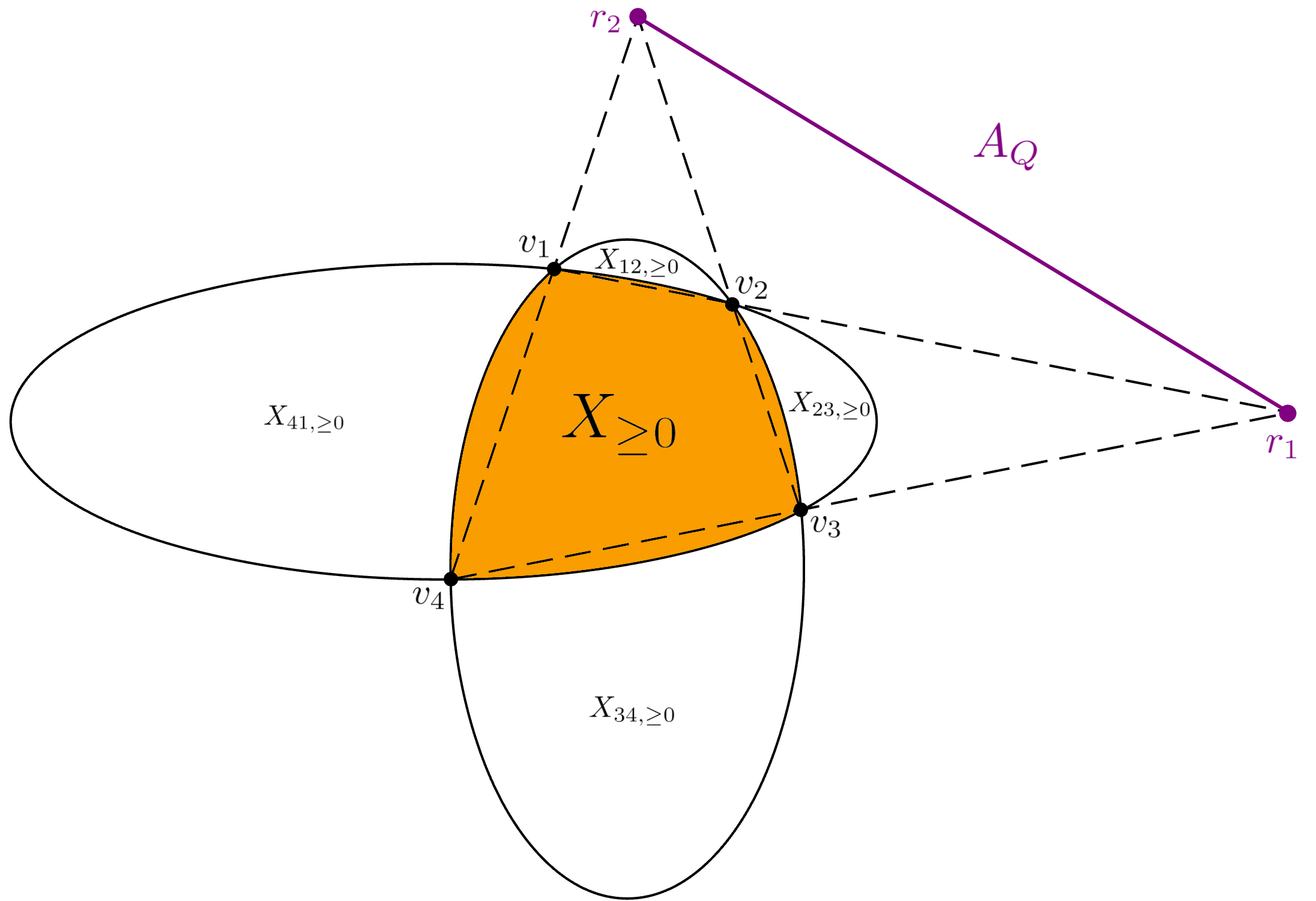}
    \caption{Positive geometry with 2 boundary curves and 4 vertices.}
    \label{fig:conicQuadrangle}
\end{figure}
Notice that $C_1\cup C_2$ together with the four vertices $v_i$ do not form a rational polypol since each curve contains more than two vertices.  Since $C_1\cup C_2$ has no other singularities than the vertices, there is no unique adjoint curve in the sense that there is no unique line passing through the empty set.
The numerator of the unique rational canonical $2$-form, however, is an adjoint whose unicity follows from the residue condition at the vertices.  In this particular example, it may also be characterized as the unique adjoint of the convex hull of the four vertices.
\end{example}

The numerator of the canonical form of a positive geometry is in fact always an adjoint curve in the sense that it passes through the singular points on the boundary curve that are not vertices (but it is not necessarily the only such curve):
\begin{proposition} 
\label{prop:canonicalFormNumeratorIsAdjoint}
Let $(\P^2,X_{\geq 0})$ be a positive geometry in the plane, with boundary curve $\partial X\subset \P^2$. Assume all singularities of $\partial X$ are nodes. Let $\Omega(\P^2, X_{\geq 0})$ be its rational canonical form.  It has poles along $\partial X$ and zeros along some curve $A$.
Then $A$ is an adjoint curve to $\partial X$, i.e., $\deg A = \deg \partial X-3$,
and $A$ passes through the singular points of $\partial X$
that are not vertices of the positive geometry.
\end{proposition}

We defer the proof to the end of Section \ref{ssec:canonicalForm}.
\subsection{The canonical form of a quasi-regular rational polypol}
\label{ssec:canonicalForm}

As a converse to Proposition~\ref{prop:posGeometriesAreRationalPolypols}, we show now that every quasi-regular  rational polypol is a positive geometry, and we provide an explicit formula for its canonical form. 
Let $P$ be a quasi-regular rational polypol, defined by real rational curves $C=\cup_{i=1}^k C_i$ and real vertices $V(P)=\{v_{12}, \dots, v_{k1}\}$, and with sides being chosen
real segments of $C_i(\mathbb R)$ connecting $v_{i-1,i}$ to $v_{i,i+1}$ for each $i$ and bounding a closed semi-algebraic set $P_{\ge 0}$.

To work in an affine chart, we fix a line $L_\infty$ that does not contain any of the vertices $v_{i-1,i}$ and intersects $C$ transversally.
Let  $(x,y)$ be affine coordinates on $\C^2=\mathbb P^2\setminus L_\infty$. 
We denote by $f_i\in \mathbb R[x,y]$ a defining polynomial for $C_i\cap \C^2$ and by $\alpha_P\in \C[x,y]$ a defining polynomial for the adjoint $A_P \cap \C^2$ of the polypol  $P$. 

Consider the meromorphic differential form
\[
\Omega(P) := 
\frac{\alpha_P}{f_1\cdots f_k}  \, dx\wedge dy.
\]
Note that $\Omega(P)$ is uniquely defined (up to multiplication by a non-zero constant).
Write $f_{i \bullet} = \partial f_i/ \partial \bullet$.
We may assume the coordinates are chosen such that $f_{iy}$ is nonzero. Then
we have $dx\wedge df_i=dx\wedge (f_{ix}dx+f_{iy}dy)=f_{iy}dx\wedge dy$, so that
\[\Omega(P) = \frac{\alpha_P}{f_1\cdots \hat f_i \cdots f_k f_{iy}}\, dx\wedge \frac{df_i}{f_i}.\]
Hence, its \emph{residue}  along $C_i$ is equal to the restriction  of
\[\frac{\alpha_P }{f_1\cdots \hat f_i \cdots f_k f_{iy}} \,dx\]
to $C_i$. 
Moreover, we have 
\[df_{i+1}\wedge df_i =-(f_{ix}dx+f_{iy}dy)\wedge (f_{i+1,x}dx+f_{i+1,y}dy)=-J_{f_i,f_{i+1}}~dx\wedge dy,\]
where $J_{f_i,f_{i+1}} =f_{ix}f_{i+1,y}-f_{i+1,x}f_{iy}$ is the Jacobian of $f_i,f_{i+1}$.
Hence, the residue of $\Res_{C_i}\Omega(P)$ at the vertex $v_{i,i+1}$ is equal to the constant
\[c_{i,i+1}:=\frac{-\alpha_P(v_{i,i+1})}{f_1(v_{i,i+1})\cdots \hat f_i \hat f_{i+1}\cdots f_k(v_{i,i+1}) J_{f_i,f_{i+1}}(v_{i,i+1})}.\]
To fix the multiplicative scaling of $\Omega(P)$, let us now replace $\alpha_P$ by $- c_{1,2}^{-1}\,\alpha_P$, so that the residue at $v_{1,2}$, considered as the last vertex of the first side,  becomes $-1$.

\begin{theorem}
\label{thm:rationalPolypolsArePosGeometries}
Let $P$ be a quasi-regular  rational polypol.
Then $(\P^2, P_{\geq 0})$ is a positive geometry, and its canonical form 
is the meromorphic 2-form 
$\Omega(P)$ defined above: its numerator is an equation of the adjoint curve, its denominator is an equation of the curve $C$, and its residues at the vertices of $P$ are $\pm 1$.
\end{theorem}

\begin{proof}
We must show that for each $i$,  $\Res_{C_i} \Omega(P)$ is the canonical form on the side $C_i$ (with vertices $v_{i-1,i}$ and $v_{i,i+1}$). 

The adjoint curve $A_P$ is a curve of degree $d-3$ that passes through all singular points of $C$ other than the vertices $V(P)$; more precisely, the adjoint curve contains the fat points, defined by the conductor ideal at each singular point of $C$, except the vertices of $P$. Let $\P^1\to C_i\subset C\subset \P^2$ be a (birational) parameterization of  $C_i$:
Using affine coordinates, 
we write $x(t)=\frac{r(t)}{h(t)}$ and $y(t)=\frac{s(t)}{h(t)}$, where $r,s,h$ are polynomials of degree $d_i$. We can then write the restriction of $\Res_{C_i}\Omega(P)$ to $C_i$ as
\[\Omega_i(t):=\frac{\alpha_P(x(t),y(t)) \, x'(t)}{f_1(x(t),y(t))\cdots \hat f_i \cdots f_k(x(t),y(t)) \, f_{iy}(x(t),y(t))}\, dt.\]
In the numerator, $h(t)^{-1}$ appears to the power $d-3$ in $\alpha_P$ and to the power $2$ in $x'(t)$. In the denominator, it appears to the power $d-d_i+d_i-1$. So the $h(t)$'s cancel, and we are left with polynomials in $t$  both in the numerator and in the denominator.

Let $a_i,b_i\in \R$ be such that $v_{i-1,i}=(x(a_i),y(a_i))$ and $v_{i,i+1}=(x(b_i),y(b_i))$. The canonical form of the positive geometry $(\P^1,[a_i,b_i])$ is given in~\eqref{eq:canonicalFormInterval}.
So we must show that $\Omega_i(t)=\Omega(\P^1,[a_i,b_i])$.

We observe that the numerator of $x'(t)=\frac{h(t)r'(t)-h'(t)r(t)}{h(t)^2}$ is a polynomial of degree $2d_i-2$, since the highest terms of the two products cancel. It follows that, after cancellation of the powers of $h(t)$, the numerator of $\Omega_i(t)$ is a polynomial in $t$ of degree $d_i(d-3)+2d_i-2=d_i(d-1)-2$, whereas the denominator has degree $d_i(d-d_i)+d_i(d_i-1)=d_i(d-1)$. Moreover, the denominator contains $(t-a_i)(b_i-t)$ as a factor since $t=a_i,b_i$ are roots of $f_{i-1}(x(t),y(t))$ and $f_{i+1}(x(t),y(t))$, respectively.
Hence, we can write
\[\Omega_i(t)=\frac{F(t)}{G(t)(t-a_i)(b_i-t)}\,dt,\]
where $F$ and $G$ both have degree $d_i(d-1)-2$.
We need to show that $F$ and $G$ have the same roots, with the same multiplicities.

Indeed, assume that $p=(x(c),y(c))\in C_i$ is a singular point of $C$. Let $C_{i_1},\dots, C_{i_r}$ be the other components of $C$ that contain $p$. The Jacobian ideal of $C$ at $p$ is  \[ \left \langle (f_i \cdot g )_x, (f_i \cdot g)_y \right \rangle =\left \langle f_{ix} \cdot g +f_i \cdot g_x, f_{iy}\cdot g +f_i \cdot g_y\right \rangle ,\] 
where $g: = \prod_{j=1}^r f_{i_j}$. Restricted to $C_i$ this ideal is generated by $f_{ix} \cdot g $ and $ f_{iy} \cdot g$. We may assume (by switching $x$ and $y$ if necessary) that 
$ f_{iy}\cdot g$ generates the pullback of the Jacobian ideal in the local ring $\C[t]_{(t-c)}$. We know that $\alpha_P(x(t),y(t))$  pulls back to generate the conductor ideal in $\C[t]_{(t-c)}$ and $x'(t)$ to the ramification ideal of the parameterization. Hence the product is the pullback of the Jacobian ideal of the singularity \cite[p.~261]{piene1978polar}, which is the same as the pullback of the ideal generated by $f_{iy} \cdot g $. Hence the zeros and poles of $\Omega_i(t)$ cancel above this point.

Hence the roots of $F$ and $G$ are the same, with the same multiplicities, and so $\frac{F(t)}{G(t)}$ is a constant. Therefore we can write
$\Omega_i(t)=\gamma_i \, \Omega(\P^1,[a_i,b_i])$, for some constant $\gamma_i$.
Since $\Res_{b_i}\Omega(\P^1,[a_i,b_i])=-1$, we have that $\Res_{b_i} \Omega_i(t) = -\gamma_i$.
Moreover, because of $\Res_{a_i}\Omega(\P^1,[a_i,b_i])=1$, we have $\Res_{a_i} \Omega_i(t) =\gamma_i$.
We now claim that the constant $\gamma_i$ is the same for all $i = 1, \ldots, k$.
This can be seen as follows:

The vertex $v_{i,i+1}$ can be regarded as the last vertex of the $i$th side of $P$ and as the first vertex of the $(i+1)$th side of $P$. The value of $\Res_{v_{i,i+1}} \Omega(P)=\Res_{b_i} \Omega_i(t)$  given above corresponds to first taking the residue along $C_{i}$ and then the residue at the last vertex $v_{i,i+1}$ of the $i$th side. If we instead first take the residue along $C_{i+1}$ and then the residue at the first vertex $v_{i,i+1}$  of the $(i+1)$th side (i.e., we switch the roles of $C_i$ and $C_{i+1}$), we get 
$-\Res_{b_i} \Omega_i(t)=\Res_{a_{i+1}} \Omega_{i+1}(t) = \gamma_{i+1}$, hence $\gamma_{i+1}=-(-\gamma_i)=\gamma_i$. 
Continuing along all sides of $P$, we see that the constant $\gamma_i$ is the same for all $i = 1, \ldots, k$.

As we scaled $\alpha_P$ so that $\Res_{b_{1}} \Omega_1(t)=-1$,
we get $\gamma_i=1$ and thus $\Omega_{i}(t) = \Omega(\P^1,[a_i,b_i])$ for all $i = 1, \ldots, k$. 
This shows that $\Omega(P)$ is the canonical form of $P$. 
The uniqueness of $\Omega(P)$ follows from the fact that $\P^2$ does not admit non-zero holomorphic 2-forms.
\end{proof}

\begin{proof}[Proof of Proposition \ref{prop:canonicalFormNumeratorIsAdjoint}] We follow the argument suggested in \cite [Section 7.1]{arkani2017positive}. 
Let $q\in \partial X$ be a singular point that is not a vertex and assume  that  $C\subset \partial X$ is an irreducible component containing $q$. 
Then the residue 
$\Res_C\Omega(\P^2, X_{\geq 0})$ has poles only at the vertices on $C$, so 
 $\Omega(\P^2, X_{\geq 0})$ can have a pole of order at most one at $q$.
But locally,  
\[\textstyle \Omega(\P^2, X_{\geq 0})=\frac {f}{g} dx\wedge dy,\] 
with coprime polynomials $f$ and $g$ such that $C\subset \{g=0\}$ and $\{g=0\}$ has multiplicity at least two at $q$, so $A=\{f=0\}$ must have multiplicity at least one for the rational function $\frac {f}{g}$ to have a pole of order at most one at $q$.

To show that  $\deg A = \deg \partial X -3$
we reverse the argument in the proof of Theorem \ref{thm:rationalPolypolsArePosGeometries}.  Assume that $A=\{f=0\}$, $a=\deg A$ and $\partial X=\{g_1g_2=0\}$, where $\deg g_1=b, \deg g_2=c$ and $C = \{g_2=0\}$ is an irreducible rational curve parameterized by $x(t)=r(t)/h(t),y(t)=s(t)/h(t)$ where $r(t),s(t),h(t)$ are polynomials of degree $c$ in $t$ with no common zeros.  The points $\{h(t)=0\}$ are mapped to the line at infinity in the affine plane. We may choose a parameterization such that $\Res_C\Omega(\P^2, X_{\geq 0})$ has no poles or zeros in these points. 
Substituting the parameterization as in the proof of Theorem \ref{thm:rationalPolypolsArePosGeometries}, $h$ appears with exponent $b+c-1$ in the numerator and exponent $a+2$ in the denominator of the $1$-form
$\Res_C\Omega(\P^2, X_{\geq 0})$.  Since this form has no zeros or poles on $\{h(t)=0\}$, the factor $h$ must cancel, in particular $a+2=b+c-1$, i.e., $a=b+c-3$ or $\deg A = \deg \partial X -3$.
\end{proof}

\subsection{Additivity of the canonical form}
\label{ssec:additivity}
The set of positive geometries has a certain additivity property that allows construction of new ones. This additivity is discussed in general terms in \cite[Section 3]{arkani2015positive}, and the key observation is additivity of canonical forms when two positive geometries are disjoint or share part of their boundary.
Wachspress's arguments  \cite[Thm.~5.1, p.~98]{MR0426460} show that the canonical form of rational polypols with boundary curves of degree $\le 2$ are additive. The argument uses M.~ Noether's fundamental theorem 
and can be extended  to the case of boundary curves of arbitrary degree. 
Here, we use only the uniqueness of the canonical form to show  additivity in more general cases in dimensions one and two.

First, let us show additivity in dimension one.
Let $X$ be a rational curve with points $v_1,v_2,v_3,v_4\in X(\R)$.
Assume that the segments $X_{1,\geq 0}$ of $X(\R)$ between $v_1$ and $v_2$ and $X_{2,\geq 0}$ between $v_3$ and $v_4$ are nonsingular. 
As in Example~\ref{ex:posGeomOneDim}, these segments give positive geometries $(X, X_{1,\geq 0})$ and $(X, X_{2,\geq 0})$.
If $\varphi: \P^1(\R) \dashrightarrow X(\R)$ is a rational, birational parameterization such that $\varphi(a)=v_1, \varphi(b)=v_2, \varphi(c)=v_3,\varphi(d)=v_4$ for $a< b\leq c< d$, then the canonical forms $\Omega(X, X_{1,\geq 0})$ and $\Omega(X, X_{2,\geq 0})$ are as in \eqref{eq:canonicalFormInterval}. 

\begin{lemma}\label{add1}
Let  $X_{\geq 0}=X_{1,\geq 0}\cup X_{2,\geq 0}$ be the union of the two segments on $X$. If either $v_2\neq v_3$, or $v_2=v_3$ and the open segment from $v_1$ to $v_4$ is nonsingular,  then $(X, X_{\geq 0})$ is a positive geometry with canonical form
\[\Omega(X,X_{\geq 0})=\Omega(X,X_{1,\geq 0})+\Omega(X,X_{2,\geq 0}).\]
\end{lemma}

\begin{proof}
If $v_2\neq v_3$, then  $X_{\geq 0}$ is the union of two disjoint segments on $X$.
The sum of the canonical forms 
\[ \Omega(X,X_{1,\geq 0})+\Omega(X,X_{2,\geq 0})
=\frac{b-a}{(t-a)(b-t)}\, dt+\frac{d-c}{(t-c)(d-t)}\ dt \]
has poles with residues $\pm 1$ at the endpoints of these segments.

If $v_2=v_3$ and the open segment from $v_1$ to $v_4$ is nonsingular, then $b=c$ and $X_{\geq 0}$ is a nonsingular segment on $X$.
The sum of the canonical forms is given by
\[ \Omega(X,X_{1,\geq 0})+\Omega(X,X_{2,\geq 0})
=\frac{b-a}{(t-a)(b-t)}\, dt+\frac{d-c}{(t-c)(d-t)}\, dt
=\frac{d-a}{(t-a)(d-t)}\, dt. \]
It has poles only at the endpoints of $X_{\geq 0}$ with residues that coincide with the residues of $\Omega(X,X_{1,\geq 0})$ and $\Omega(X,X_{2,\geq 0})$, respectively. 
In both cases, the assertion follows.
\end{proof}

Now, to show additivity in the plane, we introduce pseudo-positive geometries, as in~\cite{arkani2017positive}.
Let $(X, X_{\geq 0})$ be a pair as described in the beginning of Section~\ref{ssec:positiveGeometries}, except that we allow $X_{\geq 0}$ to be empty. 
The main difference between positive and pseudo-positive geometries is that we now allow zero canonical forms:
\begin{definition}
$(X, X_{\geq 0})$ is a \emph{pseudo-positive geometry} if there is a unique rational $n$-form $\Omega(X, X_{\geq 0})$, called its \emph{canonical form}, such that
$\Omega(X, X_{\geq 0}) = 0$ if $X_{\geq 0}=\emptyset$
and that otherwise satisfies the recursive axioms Definition~\ref{def:posGeometry} (a) and (b), 
except that we only require in (b) that the boundary components are pseudo-positive geometries. 
\end{definition}

 For positive geometries in the plane,  the uniqueness of the canonical form together with Lemma \ref{add1} allow an additivity
 property 
 that we now describe.  By Theorem \ref{thm:rationalPolypolsArePosGeometries}, quasi-regular
 rational polypols are positive geometries, so it follows that 
by taking unions and differences, we can use quasi-regular rational polypols as building blocks to construct new positive or pseudo-positive geometries. Two illustrative examples are given in Examples \ref{ex:twotriangles} and \ref{ex:conicQuadrangle2}. 
 
 Let $(\P^2,X_{\geq 0})$ and $(\P^2,Y_{\geq 0})$ be positive geometries in the plane with a fixed orientation. 
 We write the Euclidean boundary of $X_{\ge 0}$ as 
 $$\textstyle \partial X_{\geq 0}=\bigcup_{(i,j)\in I_X}C_{i,j},$$ where the $C_{i,j}$ are disjoint 
 segments on the rational boundary curve $C_i$ and 
 $I_X$ is the index set such that for each $i$, the set $\{C_{i,j} \mid (i,j)\in I_X\}$ consists of the segments of $C_i$ that are in $ C_{i,\geq 0}$.  
 Analogously, we write $\partial Y_{\geq 0}=\bigcup_{(i,j)\in I_Y}D_{i,j}$.
 Moreover, we define $I_X(Y):=\{(i,j)\, |\, \exists \, (l,m)\, \text{ s.t. } \, C_{i,j}\subseteq D_{l,m}\}$, and similarly for $I_Y(X).$
 By definition of positive geometries, the interiors of all segments $C_{i,j}$ and $D_{i,j}$ are nonsingular. 
 The canonical 1-form 
 $$\Omega(C_i, C_{i,\geq 0})=\Res_{C_i}\Omega(X, X_{\geq 0})$$
 is the sum of the canonical forms of the disjoint segments $C_{i,j}$, 
 with poles at the two vertices of each $C_{i,j}$. 

 \begin{proposition}\label{add/subtract}
  Let $(\P^2,X_{\geq 0})$ and $(\P^2,Y_{\geq 0})$ be positive geometries. 
  We consider the following three cases:
  \begin{enumerate}\label{cases}
\item[{\rm (1)}] $X_{\geq 0}\cap Y_{\geq 0}=\emptyset$,
\item[{\rm (2)}] $X_{\geq 0}\cap Y_{\geq 0}= \partial X_{\geq 0}\cap \partial Y_{\geq 0}=\bigcup_{(i,j)\in I_X(Y)}C_{i,j}=\bigcup_{(l,m)\in I_Y(X)}D_{l,m}$, 
\item[{\rm (3)}] $X_{\geq 0}\subset Y_{\geq 0}$ and $\partial X_{\geq 0}\cap \partial Y_{\geq 0}=\bigcup_{(i,j)\in I_X(Y)}C_{i,j}.$  Additionally, if  $C_{i,j}$ and $D_{l,m}$ are boundary segments of $X_{\geq 0}$ and $Y_{\geq 0}$, respectively, on the same boundary curve $C_i=D_l$, then either $C_{i,j}\subseteq D_{l,m}$ or $C_{i,j}\cap D_{l,m}=\emptyset$.
\end{enumerate}
  We have the following:
\begin{enumerate}
    \item[{\rm (1)}] $(\P^2,Z_{\geq 0})$ with $Z_{\geq 0}:=X_{\geq 0}\cup Y_{\geq 0}$   
    is a positive geometry with canonical form \linebreak[4] $\Omega(\P^2,Z_{\geq 0})=\Omega(\P^2,X_{\geq 0})+\Omega(\P^2,Y_{\geq 0})$.
    \item[{\rm (2)}] $(\P^2,Z_{\geq 0})$ with $Z_{\geq 0}:=X_{\geq 0}\cup Y_{\geq 0}$ 
    is a pseudo-positive geometry with canonical form $\Omega(\P^2,Z_{\geq 0})=\Omega(\P^2,X_{\geq 0})+\Omega(\P^2,Y_{\geq 0})$.
     \item[{\rm (3)}] $(\P^2,Z_{\geq 0})$ with $Z_{\geq 0}:=
     \overline{Y_{\geq 0}\setminus X_{\geq 0}}$ 
    is a pseudo-positive geometry with canonical form $\Omega(\P^2,Z_{\geq 0})=\Omega(\P^2,Y_{\geq 0})-\Omega(\P^2,X_{\geq 0})$.
\end{enumerate}
 \end{proposition}
 \begin{proof}
 In case (1) it suffices to note that $\Omega(\P^2,X_{\geq 0})+\Omega(\P^2,Y_{\geq 0})$ has poles along $\partial X\cup \partial Y$ whose residues along segments $C_{i,j}$ or $D_{i,j}$ coincides with those of $\Omega(\P^2,X_{\geq 0}) $ and $\Omega(\P^2,Y_{\geq 0})$, respectively.
 
 In case (2) we note that $Z_{\geq 0}$ is semialgebraic, that its boundary $\partial Z_{\ge 0} = (\partial X_{\ge 0} \cup \partial Y_{\ge 0}) \setminus (\partial X_{\ge 0} \cap \partial Y_{\ge 0})$ is the connected union of rational curve segments, and that the rational form $\Omega(\P^2,X_{\geq 0})+\Omega(\P^2,Y_{\geq 0})$ has poles along each of these boundary segments.  Furthermore, its residue along each of these segments coincides with the residues of $\Omega(\P^2,X_{\geq 0})$ or $\Omega(\P^2,Y_{\geq 0})$ depending on whether the segment is in $\partial X_{\ge 0}$ or $\partial Y_{\ge 0}$.  
 
 For each segment $C_{i,j}=D_{l,m}\subset \partial X_{\geq 0}\cap \partial Y_{\geq 0}$, the induced orientations from $X_{>0}$ and $Y_{>0}$ have different signs.
 Therefore, the residue of $\Omega(\P^2,X_{\geq 0})$ along $C_{i,j}$ and the residue of $\Omega(\P^2,Y_{\geq 0})$ along $D_{l,m}$ are 1-forms with residues that sum to zero at the vertices of the segment.  So the sum of 1-forms has no poles on the segment.  By Lemma \ref{add1}, if we consider the rational curve $C_i=D_l$, then the residue $\Omega(C_i,C_{i,\geq 0})$ of $\Omega(\P^2,X_{\geq 0})$ along $C_i$ is a sum of canonical 1-forms with poles at endpoints of the segments $C_{i,j}$:
 $$\textstyle \Omega(C_i,C_{i,\geq 0})=\sum_j\Omega(C_{i},C_{i,j}).$$
 Similarly, $\Omega(D_l,D_{l,\geq 0})=\sum_m\Omega(D_{l},D_{l,m})$. 
 The sum of these forms has no poles at the vertices of common segments $C_{i,j}, \; (i,j)\in I_X(Y)$, but poles at the vertices of the remaining segments:
 \[\textstyle \Omega(C_i,C_{i,\geq 0})+\Omega(D_l,D_{l,\geq 0})=\sum_{(i,j)\in I_X \setminus I_X(Y) }\Omega(C_{i},C_{i,j})+\sum_{(l,m)\in I_Y \setminus I_Y(X)}\Omega(D_{l},D_{l,m}),\]
 where the two summations are over segments in $C_i$ and $D_l$, respectively, that are not in the intersection $\partial X_{\geq 0}\cap \partial Y_{\geq 0}.$  
 If there are no such segments, then $(C_i,C_{i,\geq 0})$ and $(D_l,D_{l,\geq 0})$ coincide as positive geometries with opposite signs for their canonical form.
 In this case the 2-form $\Omega(\P^2,X_{\geq 0})+\Omega(\P^2,Y_{\geq 0})$  has no poles along $C_i$. 
We conclude that the 2-form $\Omega(\P^2,X_{\geq 0})+\Omega(\P^2,Y_{\geq 0})$ has poles along the union $\partial X\cup \partial Y$, except along the curves $C_i=D_l$ for which  $C_{i,\ge 0}=D_{l,\ge 0}.$
It has residues along each boundary segment coinciding with the residue of $\Omega(\P^2,X_{\geq 0})$ or $\Omega(\P^2,Y_{\geq 0})$, when the segment is on the boundary $\partial X_{\geq 0}$ or $\partial Y_{\geq 0}$, respectively.
In conclusion, $(\P^2, Z_{\geq 0})$ with canonical form $\Omega(\P^2,Z_{\geq 0})=\Omega(\P^2,X_{\geq 0})+\Omega(\P^2,Y_{\geq 0})$,
is a pseudo-positive geometry.
 
Case (3) is similar to case (2), with set difference instead of union between $X_{\geq 0}$ and $Y_{\geq 0}$, but the underlying geometry needs more attention. 
The segments of $\partial Z_{\geq 0}$ are of three kinds. The two first kinds are the segments of  $\partial Y_{\geq 0}$ that do not contain segments of  $\partial X_{\geq 0}$ and the segments of  $\partial X_{\geq 0}$ that are not contained in segments of  $\partial Y_{\geq 0}$.  
The third kind are segments in $D_{l,m}\setminus C_{i,j}$, whenever $C_{i,j} \subsetneq D_{l,m}$.    
Each endpoint of a segment of the third kind is an endpoint of $D_{l,m}$  or an endpoint of a segment $C_{i,j}$ contained in $D_{l,m}$ but not both,  by the additional condition on segments $C_{i,j}$ and $D_{l,m}$.  Of course, the vertices of $Z_{\geq 0}$ are simply the set of endpoints of these three kinds of segments.  

 Clearly $\Omega(\P^2,Y_{\geq 0})-\Omega(\P^2,X_{\geq 0})$ has poles along any curve in $\partial Y\setminus \partial X$ and $\partial X\setminus \partial Y$.  
 Note that since $X_{\geq 0} \subset Y_{\geq 0}$, their orientations coincide on $X_{> 0}$. In particular, along any segment in the common boundary $\partial Y\cap \partial X$, the orientation coincides. 
 The difference between the canonical forms $\Omega(\P^2,Y_{\geq 0})-\Omega(\P^2,X_{\geq 0})$ will therefore have vanishing residue at vertices that are vertices of both $X_{\geq 0}$ and $Y_{\geq 0}$, similar to the additive case (2). 
 For segments of the third kind, only one of the canonical forms has a non-zero residue at each vertex, so the difference has residue $\pm 1$.  In particular, $\Omega(\P^2,Y_{\geq 0})-\Omega(\P^2,X_{\geq 0})$ would have poles along a common boundary curve only if it contains segments of the third kind.
    In conclusion, $(\P^2, Z_{\geq 0})$ with canonical form $\Omega(\P^2,Z_{\geq 0})=\Omega(\P^2,Y_{\geq 0})-\Omega(\P^2,X_{\geq 0})$,
is a pseudo-positive geometry.
 \end{proof}
 
In the next examples, we use the additivity property described in the proposition to find positive geometries in the plane whose boundary curve, unlike the rational polypol case, does not have a unique adjoint through its residual points.  
The numerator of the canonical form  
is still an adjoint to the boundary curve in this sense (see \Cref{prop:canonicalFormNumeratorIsAdjoint}), but when there is more than one adjoint, the numerator cannot be characterized only by the condition of being a curve passing through the residual points. 
Non-unique adjoints appear when there are fewer boundary components than vertices, or when the boundary forms several cycles and not one as in the rational polypol case.

\begin{example}
\label{ex:twotriangles}
Let  $(\P^2,X_{\geq 0})$ and $(\P^2,Y_{\geq 0})$ be positive geometries, where $X_{\geq 0}$ and $Y_{\geq 0}$ are disjoint triangles in the plane.  Then $(\P^2,Z_{\geq 0})$, where $Z_{\geq 0}=X_{\geq 0}\cup Y_{\geq 0}$, is a positive geometry with canonical form
\begin{align*}
\Omega(\P^2,Z_{\geq 0}) = 
     \Omega(\P^2,X_{\geq 0})+\Omega(\P^2,Y_{\geq 0}) = \frac{f_1+f_2}{f_1 f_2} \, dx \wedge dy,
\end{align*}
where $f_1$ and $f_2$ are cubic forms that define the triangles $\partial X$ and $\partial Y$.  The cubic curve $\{f_1+f_2=0\}$ lies in the pencil of curves generated by the two triangles.
The sextic curve $\partial X \cup \partial Y$ has $9 = 3 \cdot 3$ residual points, namely the pairwise intersections of a line on $\partial X$ with a line on $\partial Y$. 
Hence, all cubic curves in the pencil generated by the two triangles are adjoints to the sextic curve.
\end{example}

\begin{example}
\label{ex:conicQuadrangle2}
We now return to Example~\ref{ex:conicQuadrangle} and show that the canonical form of $(\P^2, X_{\geq 0})$ in Figure~\ref{fig:conicQuadrangle} is as claimed in~\eqref{eq:conicQuadrangle}.
As in Figure~\ref{fig:conicQuadrangle}, we denote by $X_{ij, \geq 0}$ the (simply connected region of the) polypol with vertices $v_i$ and $v_j$. 
Its adjoint is the line $L_{lm}$ spanned by $v_l$ and $v_m$ where $\{ i,j,l,m\} = \{1,2,3,4\}$.
By Theorem~\ref{thm:rationalPolypolsArePosGeometries}, $(\P^2, X_{ij,\geq 0})$ is a positive geometry with canonical form 
$\tfrac{\alpha_{lm}}{f_1f_2} \, dx \wedge dy$,
where $f_i$ is a defining equation of the conic $C_i$
and $\alpha_{lm}$ is a defining equation of the adjoint line $L_{lm}$.

Consider the rational polypol $Y_{ij,\geq 0}=X_{\geq0}\cup X_{ij,\geq 0}$ with vertices $v_l$ and $v_m$. Its unique adjoint is the line $L_{ij}$.  
By Theorem~\ref{thm:rationalPolypolsArePosGeometries}, $(\P^2, Y_{ij,\geq 0})$ is a positive geometry with canonical form $\tfrac{\alpha_{ij}}{f_1f_2} \, dx \wedge dy$.
By Proposition \ref{add/subtract}(3),
\begin{align*}
\Omega(\P^2, X_{\geq 0}) = 
     \Omega(\P^2, Y_{12,\geq 0})- \Omega(\P^2, X_{12,\geq 0})
     =\frac{\alpha_{12}-\alpha_{34}}{f_1 f_2} \, dx \wedge dy,
\end{align*}
so the numerator defines a line in the pencil generated by $L_{12}$ and $L_{34}$, i.e., a line that passes through their intersection point $r_1$; see Figure~\ref{fig:conicQuadrangle}.
Similarly, 
$$\Omega(\P^2, X_{\geq 0}) = \frac{\alpha_{14}-\alpha_{23}}{f_1 f_2} \, dx \wedge dy,$$
so the numerator defines a line that is also in the pencil generated by $L_{14}$ and $L_{23}$, i.e., that passes through their intersection point $r_2$. 
Hence, the unique line in both pencils is the adjoint line to the quadrangle that is the convex hull of the four vertices.

Note that in this case, there are no residual points. 
Therefore, any line is an adjoint curve, passing through the empty set of residual points.  An ad hoc argument characterizes the numerator of the canonical form as a particular adjoint curve.

\end{example}

\section{Real Topology of Adjoint Curves}
\label{Outside Adjoints}

In finite element computation, one seeks basis functions for the elements which achieve a certain degree of approximation within each element while maintaining global continuity. This was a main motivation for E.~Wachspress to prove that the adjoint curve of a real rational polypol $P$ is a common denominator for a rational basis on $P$, which can be generated to achieve any specified degree of approximation within $P$ while maintaining global continuity \cite{MR0426460}. 

At present, except for the case of polygonal/polytopal elements, the latter rational elements have limited applicability. This is primarily due to the complexity of integrations required for generation of finite element equations, see for instance \cite[page 33]{MR0426460}. Additionally, their practical value depends upon the validity of \emph{Wachspress's conjecture} (see \Cref{conj:Wachs}), which is only settled in very few cases. It was originally stated for \em polycons\em, i.e., polypols with lines and conics for boundary curves.

In Subsection \ref{subsec:regpolypols}, we present his extension of the conjecture to the case of polypols, see \cite{Wachspress1980}. 
The statement is known to hold in the case of convex polygons. We give a complete characterization of the real topology of the adjoint in this case in Subsection \ref{subs:polygons}. In particular, we prove that the adjoint is strictly hyperbolic and show that an analogous statement fails to hold in higher dimensions. Finally, in Subsection \ref{subsec:threeconics}, we consider polycons defined by three ellipses, which is the first unsolved case of Wachspress’s conjecture. We prove the conjecture for 33 out of 44 topological classes of configurations, including all cases of maximal real intersection, that is to say all cases where the ellipses meet pairwise in four real points.

\subsection{Regular polypols and Wachspress's conjecture} \label{subsec:regpolypols}

Let $P$ be a quasi-regular  rational polypol defined by real curves $C_i$, real vertices $v_{ij}$, and given sides. Recall that the sides of $P$ are 
segments of $C_i(\R)$ going from $v_{i-1,i}$ to $v_{i,i+1}$ that bound a closed semi-algebraic region $P_{\ge 0}$ in $\P^2(\R)$. 

\begin{definition}\label{def:regularpolypol}
We say that a quasi-regular polypol $P$ is \emph{regular} if 
all points on the sides of $P$ except the vertices are nonsingular on $C=\bigcup_{i=1}^k C_i$ and $C$ does not intersect the interior of $P_{\ge 0}$.
\end{definition}

This definition generalizes Wachspress's notion of well-set polycons from \cite[p.~9]{MR0426460} and captures his notion of a regular algebraic element in \cite{Wachspress1980}. It implies that a regular polypol has no residual points contained in $P_{\ge 0}$.
Moreover, the union of the sides is homeomorphic to the circle $S^1$, so the complement has a simply connected component which is the interior of the set $P_{\ge 0}$.

\begin{example}
    A polygon is a regular polypol if and only if it is convex. 
\end{example}

\begin{example}
All three polypols in Figure~\ref{fig:polypolsAdjoints} are regular.
Of the eight quasi-regular polypols in \Cref{fig:conicTriangle} (discussed in \Cref{rem:triangleadjoint}), none are regular.
However, there are four choices of triples of vertices and sides in Figure~\ref{fig:conicTriangle} that define a regular polypol, 
namely the four ``triangles'' (one adjacent to each $A$, $B$, $C$, and one in the middle).
\end{example}

\begin{conjecture}[{see \cite[p.~153-154]{MR0426460}}, \cite{Wachspress1980}]
\label{conj:Wachs}
The adjoint curve of a regular rational polypol $P$ does not intersect the interior of $P_{\ge 0}$. 
\end{conjecture}

The regularity assumption is needed in this conjecture, which we can already see in the case of linear boundary. Indeed, a non-convex polygon is a quasi-regular polypol but its adjoint can intersect the interior, see Figure~\ref{fig:AdjointInterior}.

\begin{figure}[hbt]
    \centering
    \includegraphics[width=0.5\textwidth]{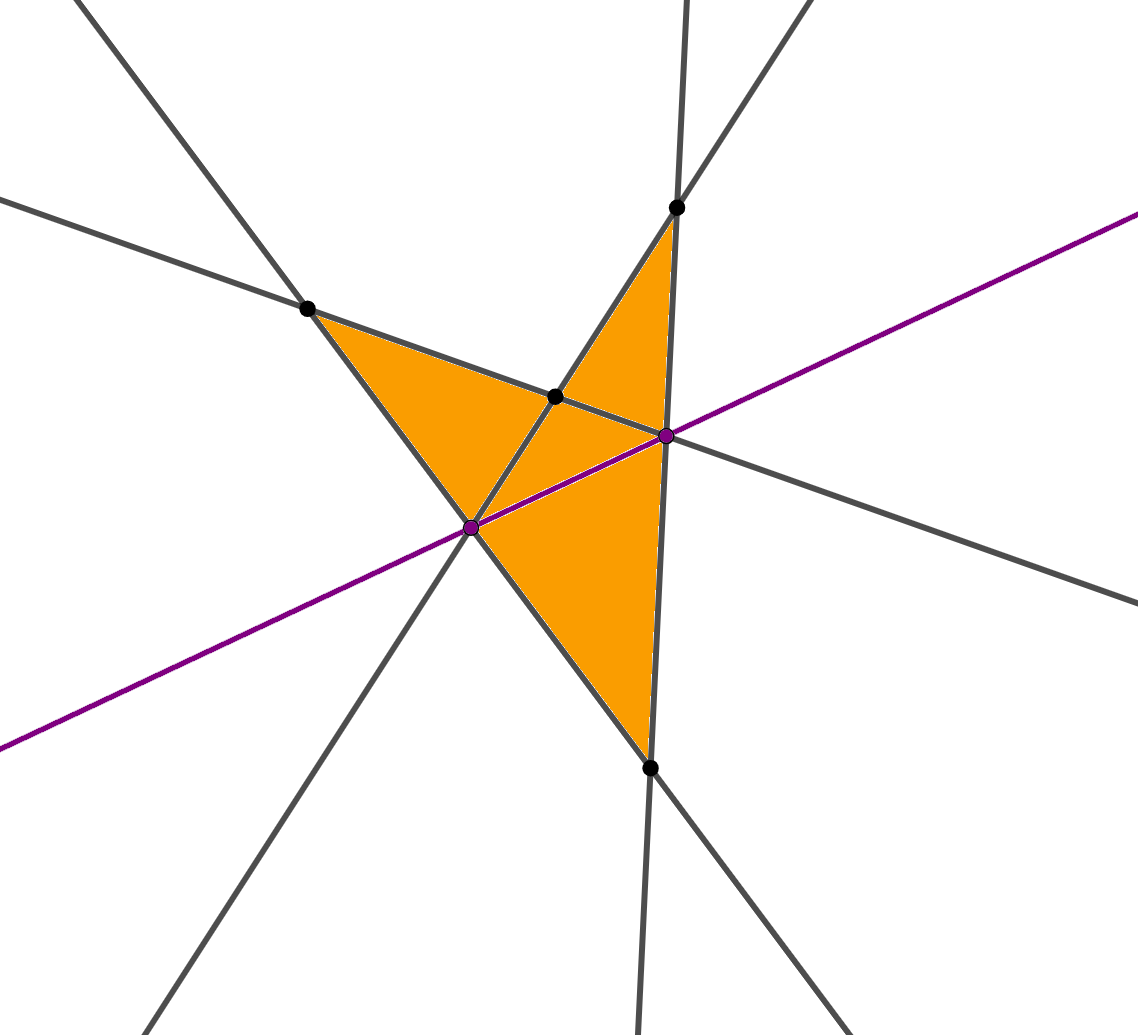}
    \caption{A non-convex polygon for which the adjoint curve intersects the interior.}
    \label{fig:AdjointInterior}
\end{figure}

Wachspress himself claims that it is easy to show that  the adjoint curve cannot intersect the sides of a regular rational polypol $P$, see \cite[page 396]{Wachspress1980}. However he presents no supporting arguments. For the sake of completeness, we prove this statement below.

Let $p\in C$ be a singular point on a plane curve $C$. Recall that the $\delta$-invariant (or genus discrepancy) of $C$ at $p$ is defined as $\delta_p:=\dim {\mathcal O}'/{\mathcal O}$, where $\mathcal O:={\mathcal O}_{C,p}$ is the local ring of $C$ at $p$ and ${\mathcal O}'$ its normalization. The difference between the arithmetic genus and geometric genus of $C$ is equal to $\sum_{p\in \Sing C}\delta_p$.

\begin{lemma}\label{lemma:transversalResidual}
Let $P$ be a rational polypol defined by boundary curves $C_1,\ldots, C_k$ that intersect transversely. Then the adjoint curve $A_P$ intersects $C_i$ only at the residual points, with intersection multiplicity equal to $2\delta_p$ at each singular point $p\in C_i$ and with intersection multiplicity one at each of the remaining residual points. In particular, if $P$ is regular, then the adjoint curve does not contain any points on the sides of $P$.
\end{lemma}

\begin{proof}
Let $d_i=\deg C_i$ and $d=d_1+\cdots+d_k$. There are $d_i(d-d_i)-2$ residual points on $C_i$ that are intersection points with the other boundary curves. The adjoint curve $A_P$ has to intersect $C_i$ at these points with multiplicity at least one. 
Note that by Proposition \ref{prop:uniqueAdjoint}, no $C_i$ is a component of $A_P$.
Since the genus of $C_i$ is zero, we have $\sum_{p\in \Sing C_i}\delta_p=\binom{d_i-1}{2}$. By \cite[Cor.~4.7.3]{casas2000}, the intersection multiplicity of $C_i$ and $A_P$ at $p \in \Sing C_i$ is equal to $2\delta_p$. 
Therefore, the total intersection number of $A_P$ and $C_i$ is at least 
\begin{align*} \label{eq:adjoint_meets_curve}
\textstyle    d_i(d-d_i)-2+2\binom{d_i-1}{2}=d_i(d-3) = d_i \cdot \deg A_P.
\end{align*}  
By B\'ezout's theorem, there cannot be any further intersection points.
\end{proof}

\begin{remark}\label{rem:genResidual}
The lemma holds without the assumption that the boundary curves intersect transversally: The adjoint curve $A_P$ intersects the total boundary curve $C$ only at the residual points, with multiplicity $2 \delta_p$ at each such point $p$. Indeed, by \cite[Thm.~2, p.190]{hironaka1957on}, 
\[\textstyle g_a(C)=\sum g(C_i)+\sum_{p\in R(P)} \delta_p+k-(k-1)=\sum_{p\in R(P)}\delta_p+1,\]
hence $2\sum_{p\in R(P)}\delta_p=2g_a(C)-2=d(d-3).$ 
By \cite[Cor.~4.7.3]{casas2000}, the total intersection number of $A_P$ and $C$ is at least $2\sum_{p\in R(P)}\delta_p$, so by B\'ezout's theorem, there cannot be any further intersection points.
\end{remark}

Lemma \ref{lemma:transversalResidual} is insufficient to show Wachspress's conjecture (i.e., that the adjoint is outside of a regular rational polypol) since the adjoint curve might have an oval (or a singular oval) contained strictly inside $P_{\ge0}$. However, it is enough for polypols of total degree at most $5$, as E.~Wachspress observed already, see \cite[Section 5.3]{Wachspress16}.

\begin{proposition}\label{wachspress5}
Wachspress's conjecture holds for polypols of total degree at most $5$.
\end{proposition}

\begin{proof}
If the total degree of the polypol $P$ is $4$, the degree of the adjoint is $1$, so it is a real line in the projective plane. In particular, its real locus is $\P^1(\R)$. 
However, the region $P_{\ge 0}$ cannot contain a real line that does not intersect the bounding sides and therefore the adjoint cannot pass through the polypol by \Cref{lemma:transversalResidual}.

If the total degree of the polypol $P$ is $5$, the degree of the adjoint is $2$ and hence its real locus is always connected. It can either be an acnode or it is connected of dimension $1$.
Using \Cref{lemma:transversalResidual}, the conjecture follows by showing that there is always a real residual point, which is outside of $P_{\ge 0}$ by regularity, so that there is a real point of the adjoint outside of the polypol.

The possible degrees of the bounding curves are the following: $(1,4)$, $(2,3)$, $(1,1,3)$, $(1,2,2)$, $(1,1,1,2)$, and $(1,1,1,1,1)$. In the first case, the rational real quartic curve has at least one real singularity, which is a residual point. The second and third case are the same because the real rational cubic also has a real singularity. In the next two cases with a conic, one of the lines intersects the conic in a vertex, so a real point, which means that the other intersection point is a real residual point. The last case is a convex pentagon with five real residual points.
\end{proof}

\subsection{Convex polygons}\label{subs:polygons}
In the case of convex polygons, which are precisely regular polypols with lines as boundary curves, Wachspress proved that the adjoint polynomial does not vanish within the polygon; see \cite[p.~96, 147 and 154]{MR0426460}.

Let $P_{\geq 0}\subset \R^2$ be a convex $k$-gon for $k\in\mathbb{N}, k\geq 4$. We label its edges $E_1,\ldots,E_k$ cyclically around the boundary of $P_{\ge 0}$ and write  $C_i:=\ol{E_i}$ for their Zariski closure in $\P^2$. Our main result in this section is a complete description of the real topology of the adjoint curve of a convex polygon. 

\begin{theorem}\label{thm:hyperbolicadjoint}
The adjoint curve $A_P$ of a convex $k$-gon $P_{\ge 0}$ is hyperbolic with respect to every point $e\in P_{\ge 0}$.
Moreover, it is  strictly hyperbolic, i.e., it does not have any real singularities.

More precisely, $A_P$ has $\lfloor\frac {k-3}{2}\rfloor$ disjoint nested ovals. If the total degree $k$ is even, there is additionally a pseudoline contained in the region in the complement of the ovals that is not simply connected. 
In this case, the residual intersection point of $C_i$ and $C_{i+k/2}$ lies on the pseudoline component (where the index should be read modulo $k$). In general, for $k$ even or odd, the residual intersection point of $C_i$ and $C_{i+1+m}$ for a positive integer $m < \frac{k}{2}-1$ lies on the $m$-th oval counting from the inside (that is to say from $P_{\ge 0}$ outwards).
\end{theorem}

Here, a \emph{pseudoline} of  a smooth real algebraic curve $X\subset \P^2$ is a connected component $S$ of $X(\R)$ such that $\P^2(\R)\setminus S$ is connected. 
An \emph{oval} of $X$ is a connected component $S$ of $X(\R)$ such that $\P^2(\R)\setminus S$ has two connected components. In that case, one connected component of $\P^2(\R)\setminus S$ is contractible while the other is a M\"obius strip. 
An oval $S_1$ of $X$ is \emph{nested} in another oval $S_2$ if $S_1$ is contained in the contractible connected component of $\P^2(\R)\setminus S_2$. 

The adjoint curves of a convex heptagon and octagon are illustrated in \Cref{fig:7gon}.
Note that in particular, Theorem \ref{thm:hyperbolicadjoint} implies Conjecture \ref{conj:Wachs} for polygons. 
In the sequel we give some auxiliary lemmas that we need for proving Theorem \ref{thm:hyperbolicadjoint}.
The following result is common knowledge in the literature of hyperplane arrangements. 
\begin{lemma}\label{lem:regions}
 Let $P_{\ge 0}\subset \R^2$ be a convex $k$-gon for $k\geq 4$. Then the set $\P^2(\R)\setminus \bigcup_{i=1}^k C_i$ has exactly the following connected components:
\begin{enumerate}
    \item the interior of the polygon $P_{\ge 0}$;
    \item for every $i=1,\ldots,k$, a connected component bounded by the three lines $C_{i-1}$, $C_i$, $C_{i+1}$; note that this component is the interior of a triangle bounded by the edge~$E_i$; 
    \item for every pair of vertices $v_{i-1,i}$ and $v_{j-1,j}$ with cyclic distance at least two, there is a connected component bounded by the four lines $C_{i-1}$, $C_i$, $C_{j-1}$, and $C_j$.
\end{enumerate}
\end{lemma}

\begin{proof}
We first note that on every line $C_i$, the intersection points with the other lines appear in the order $v_{i1},v_{i2},\ldots,v_{ik}$, where we leave out $v_{ii}$ in that list.
Since each line $C_i$ contains the edge $[v_{i-1,i},v_{i,i+1}]$, this claim follows from the fact that  $P_{\ge 0}$ is convex. 

For example, in the case of the heptagon (see Figure \ref{fig:7gon}) on the line $C_2$, we have the consecutive points: $v_{21},v_{23},v_{24},v_{25},v_{26},v_{27}.$ The figure also illustrates the $k = 8$ case.

A union of connected components of $\P^2(\R)\setminus \bigcup_{i=1}^k C_i$ is bounded by intervals on the lines $C_i$ adjacent to the region between intersection points $v_{i,j}$ and $v_{i,l}$ such that we have a cyclic ordering of these intersection points $v_{i_1,i_2},v_{i_2,i_3},\ldots,v_{i_{l-1},i_l}, v_{i_l,i_1}$. This is a connected component of the complement of the lines $C_i$ if there are no intersection points in the interval between two intersection points $v_{ij}$ and $v_{il}$ on line $C_i$, which is to say that $l=j+1$ or $l=j-1$. The cyclic ordering follows from the assumption that the $k$-gon $P_{\ge 0}$ is convex.
The statement of the lemma is now a combinatorial case analysis.
\end{proof}

\begin{figure}[h]
    \centering
    \includegraphics[scale=0.13,trim={0.15cm 0.2cm 0.2cm 0.15cm},clip]{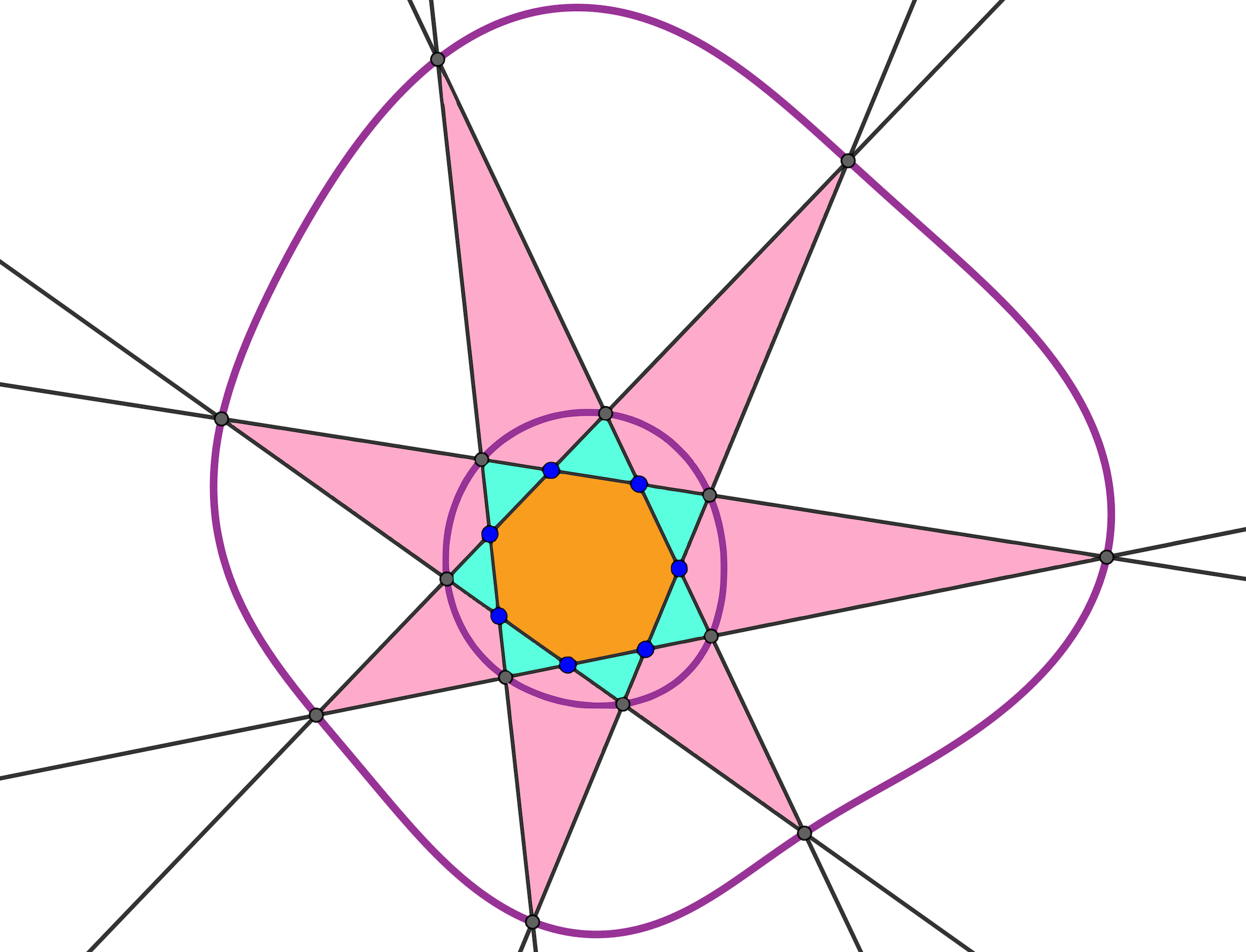}
    \includegraphics[scale=0.13,trim={0 0.15cm 0.15cm 0},clip]{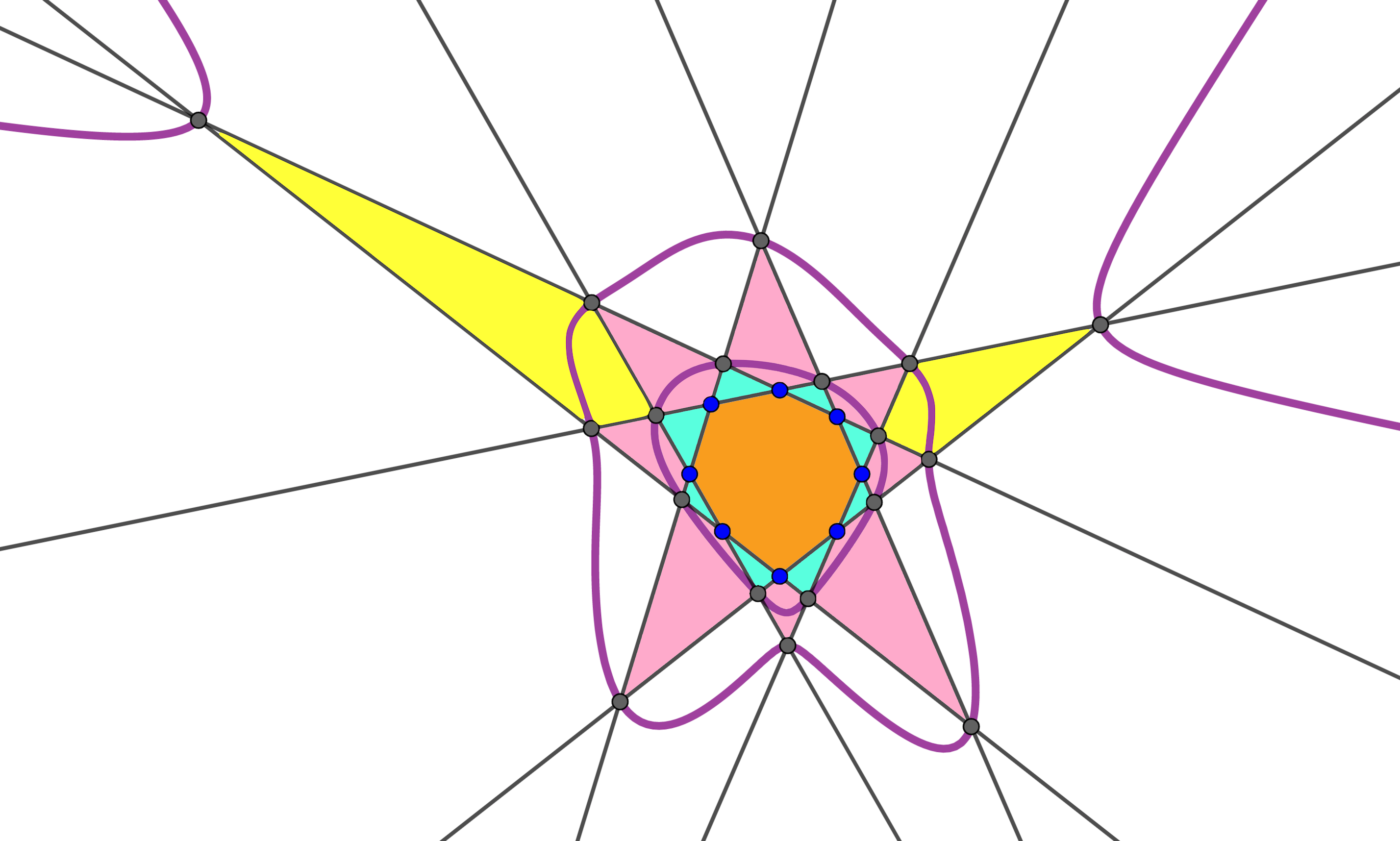}
    \caption{The adjoint of a heptagon and an octagon. See Lemma \ref{lem:regions}.}
    \label{fig:7gon}
\end{figure}

\begin{lemma}\label{lemma:countingOvals}
Let $P_{\ge 0}\subset \R^2$ be a convex $k$-gon for $k\geq 4$. The real locus $A_P(\R)\subset \P^2(\R)$ of the adjoint curve $A_P$ has at least  $\lfloor\frac {k-3}{2}\rfloor$ disjoint nested ovals.
\end{lemma}

\begin{proof}
The main idea of the proof is that we can determine the sign of the adjoint polynomial on every interval on every line $C_i$. For this, fix a homogeneous polynomial $\alpha_P$ of degree $k-3$ defining the adjoint curve $A_P$.

Throughout the proof one should keep in mind that there is a case distinction: if $k$ is even, then the adjoint has odd degree and the position of the line at infinity with respect to the polygon matters for sign considerations. In particular, the adjoint polynomial changes sign along the line at infinity. For $k$ odd, the adjoint has even degree and a well defined sign on $\P^2(\R)$ so that the position of the line at infinity relative to $P_{\ge 0}$ does not matter. Nevertheless, without loss of generality we assume that the polygon is in general position: Namely, no two lines $C_i$ and $C_j$ intersect at infinity and the line  at infinity does not intersect the interior of the polygon $P_{\ge 0}$.

 \emph{Step 1. Existence of the ovals.}
 Let us fix a line $L := C_i$ that is the Zariski closure of an edge of $P_{\ge 0}$. Then there are lines $C_{i-1}$ and $C_{i+1}$ whose intersection points with $L$ lie in the boundary of $P_{\ge 0}$. The intersections with the other $k-3$ lines $C_j$ ($j\notin \{i-1,i,i+1\}$) lie on $A_P$. By Lemma \ref{lemma:transversalResidual}, the adjoint curve meets each boundary line  only at the residual points. Thus, the line $L$ is not entirely contained in $A_P$ and the $k-3$ intersection points of $L$ with the lines $C_j$ are the roots of the restriction of $\alpha_P$ to $L$. So they are all simple roots and the sign of $\alpha_{P}|_L$ changes in each of these 
 points, and only at these points.

This argument determines the sign of $\alpha_P$ (up to a global change of signs) on all the $k$ lines $C_1,\ldots,C_k$. By changing the sign of $\alpha_P$, if necessary, we can assume that $\alpha_p$ is positive on every edge of $P_{\ge 0}$.

We now look at the regions in the complement of the line arrangement $C_1,\ldots,C_k$ described in \Cref{lem:regions} and the sign of $\alpha_P$ on the bounding edges. 
 
By our choice of sign of $\alpha_P$, it is positive on all edges bounding regions of type $1$ and $2$ in \Cref{lem:regions}. We look at regions of type 3 that have four bounding line segments. For these regions, the equation $\alpha_P$ is positive on two of these line segments and negative on the other two. Therefore $\alpha_P$ has a connected component inside this region which passes through two intersection points $C_{j_1}\cap C_{j_2}$, respectively $C_{l_1}\cap C_{l_2}$ in the boundary of this region. 
If the line at infinity intersects the region, we have to modify this argument slightly: Even in this case, there are two residual points where the adjoint curve locally around this point enters the bounded region. So we also conclude in this case that each of the regions of type 3 intersects at least one connected component of the real locus of the adjoint curve that divides it into two $2$-dimensional parts.

 \emph{Step 2. Computing the number of ovals.}
 Let us now count how many ovals one can construct as shown in Step 1 above.
 Consider a region of type 3 in \Cref{lem:regions} which is a $4$-gon with edges from $v_{i-1,j-1}$ to $v_{i-1,j}$ on $C_{i-1}$, from $v_{i-1,j}$ to $v_{i,j}$ on $C_j$, from $v_{i,j}$ to $v_{i,j-1}$ on $C_i$, and from $v_{i,j-1}$ to $v_{i-1,j-1}$ on $C_{j-1}$. Since the adjoint polynomial has the same sign on the intervals on the curves $C_{i-1}$ and $C_j$ and a different sign on the other two intervals on $C_{j-1}$ and $C_i$, it has to change sign inside the region. Therefore, there is at least one branch of the real locus inside this region. 
 By moving the indices cyclically, these regions arrange in a circle around the polygon (one example shown in pink in \Cref{fig:7gon}). 
 In particular, two regions of type 3 determined by the vertices $v_{i-1,i}$ and $v_{j-1,j}$, resp.~$v_{k-1,k}$ and $v_{l-1,l}$ belong to the same circle of regions if the cyclic distances between the indices $i$ and $j$ and the indices $k$ and $l$ are equal. 
 Overall, we have $\lfloor\frac {k-3}{2}\rfloor$ such circles of regions for vertices of cyclic distance less than $\frac{k}{2}$.
 
 If we take a line $L$ through the interior of the polygon that does not contain a residual point, it must therefore pass through at least $2 \lfloor\frac {k-3}{2}\rfloor$ such regions, two for each circle. By the sign argument above, it has at least one real intersection point with the adjoint in each region. Such a line intersects the adjoint curve therefore in $k-3$ real points if $k-3$ is even, hence all intersection points with the adjoint are real. Otherwise, we have $k-4$ real intersection points and the only missing one must therefore also be real. 
 
 A posteriori we conclude, that there is precisely one branch of the adjoint curve in every region of type $3$ determined by vertices $v_{i-1,i}$ and $v_{j-1,j}$ of cyclic distance less than $\frac{k}{2}$. This branch contains the residual points $v_{i-1,j-1}$ and $v_{i,j}$. Moreover, these branches, by shifting the indices cyclically, glue to a real oval of the adjoint.
\end{proof}

\begin{example}
In the case of the octagon ($k=8$) from Figure \ref{fig:oddDegAdjoint}, the sign sequence on any line $C_i$ is, up to cyclic permutation, $(+,-,+,+,+,-,+,-)$. Whereas for the heptagon ($k=7$) we obtain the  sign sequence $(+,-,+,+,+,-,+)$.
\end{example}

\begin{figure}
    \centering
    \includegraphics[scale=0.14,trim={0.2cm 0.15cm 0.25cm 0},clip]{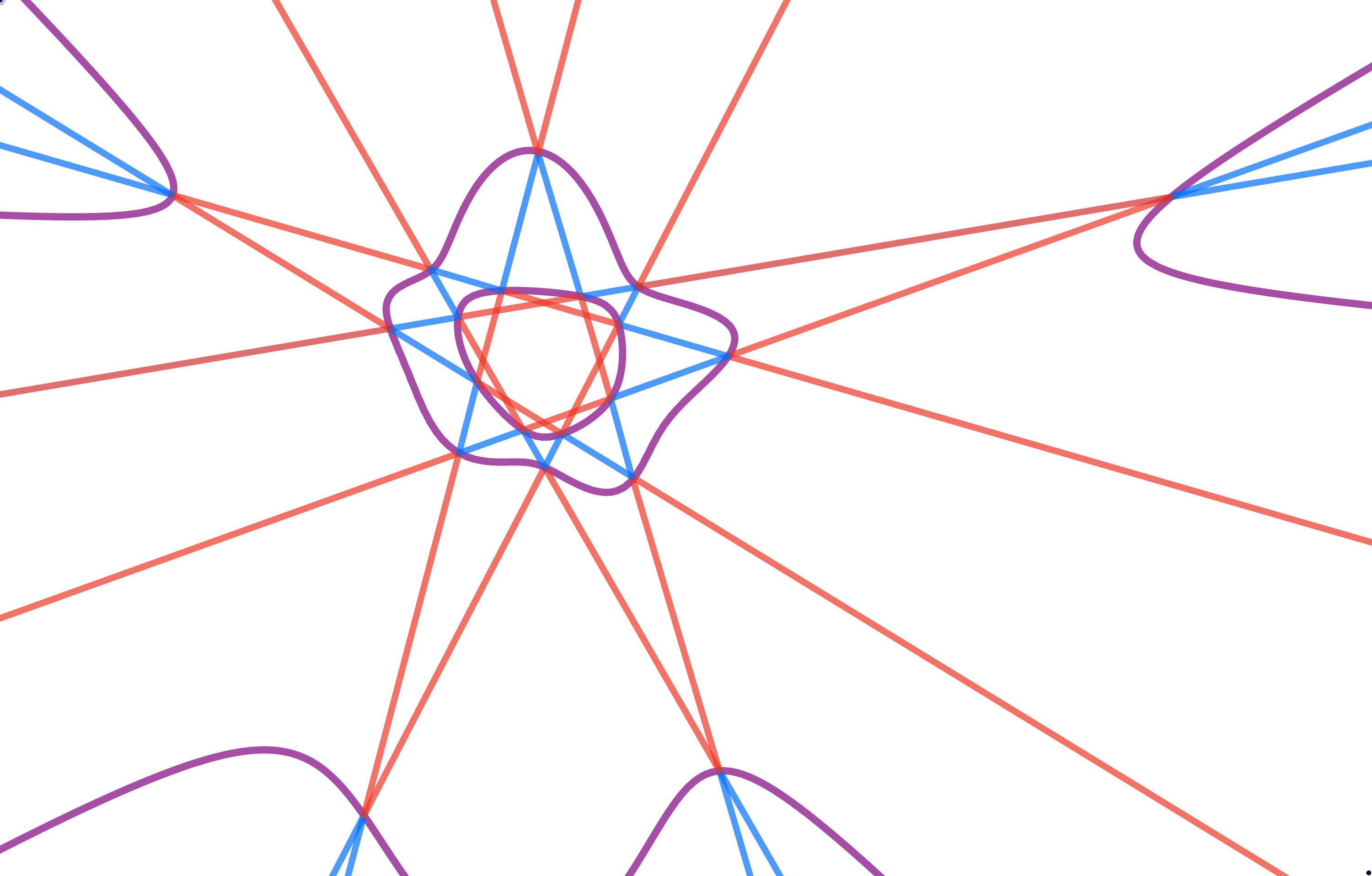}
    \caption{The adjoint polynomial of an octagon changes sign on every line $C_i,$ $i=1,\ldots,k$, when $C_i$ crosses the line at infinity. On the red line segments the adjoint polynomial is positive, whereas on the blue line segments it is negative.}
    \label{fig:oddDegAdjoint}
\end{figure}

We are now ready to prove Theorem \ref{thm:hyperbolicadjoint}.

\begin{proof}[{Proof of \Cref{thm:hyperbolicadjoint}}]
By Lemma \ref{lemma:countingOvals} , the adjoint curve has at least $\lfloor\frac {k-3}{2}\rfloor$ disjoint nested ovals and $k-3$ is the degree of the adjoint curve $A_P$. By \cite[Cor.~1.3.C.]{virointroduction}, $A_P$ has no other ovals besides the ones we counted above. Therefore (see for instance \cite[Thm.~5.2]{helton2007linear}), the adjoint $A_P$ is hyperbolic with respect to every point in the interior of the region in the complement of the adjoint curve in $\P^2(\R)$ containing the polygon $P_{\ge 0}$. 
The fact that the curve is strictly hyperbolic now follows from Bézout's Theorem by contradiction: If it had a real singularity, a line spanned by the singularity and an interior point of the polygon would intersect the adjoint in more than its degree many points.
\end{proof}

\begin{remark}\label{polytopes} The hyperbolicity of the adjoint of convex polygons does not generalize to higher-dimensional polytopes, as Example \ref{example:polytopeNotHyp} shows. The adjoint hypersurface of a polytope can be defined via an analogous vanishing condition as discussed in Section \ref{ssec:adjoints}; see \cite{kohn2019projective}.
Let $P_{\ge 0}\subset\P^n(\R)$ be a polytope with $k$ facets. Its \emph{residual arrangement} consists of all linear spaces that are intersections of hyperplanes spanned by facets of $P_{\ge 0}$ but that do not contain any face of $P_{\ge 0}$. If the hyperplane arrangement spanned by the facets of $P_{\ge 0}$ is simple (i.e., through any point in $\P^n(\R)$ pass at most $n$ hyperplanes), the \emph{adjoint hypersurface} of $P_{\ge 0}$ is the unique hypersurface of degree $k-n-1$ that passes through its residual arrangement. 
\end{remark}
\begin{example}\label{example:polytopeNotHyp}
Figure \ref{fig:counterexPolytope} shows a three-dimensional convex polytope whose adjoint surface is not hyperbolic. The combinatorial type of the polytope is a cube with a vertex cut off. However, its vertices are perturbed such that the plane arrangement spanned by the seven facets of the polytope is simple. Its residual arrangement consists of six lines and one isolated point. The adjoint is the unique surface of degree three passing through the six lines and the point. In this example, it is a smooth surface (see~\cite[Exm.~4.12]{kohn2019projective}).
Since hyperbolic cubic surfaces contain exactly three real lines, the adjoint surface of the polytope cannot be hyperbolic; see \cite[Chapter 5]{silhol} -- the hyperbolic type is case (5.4.5) on page 134 in Silhol's book.
\end{example}

\begin{figure}[htb]
    \centering
    \includegraphics[width=0.49\textwidth]{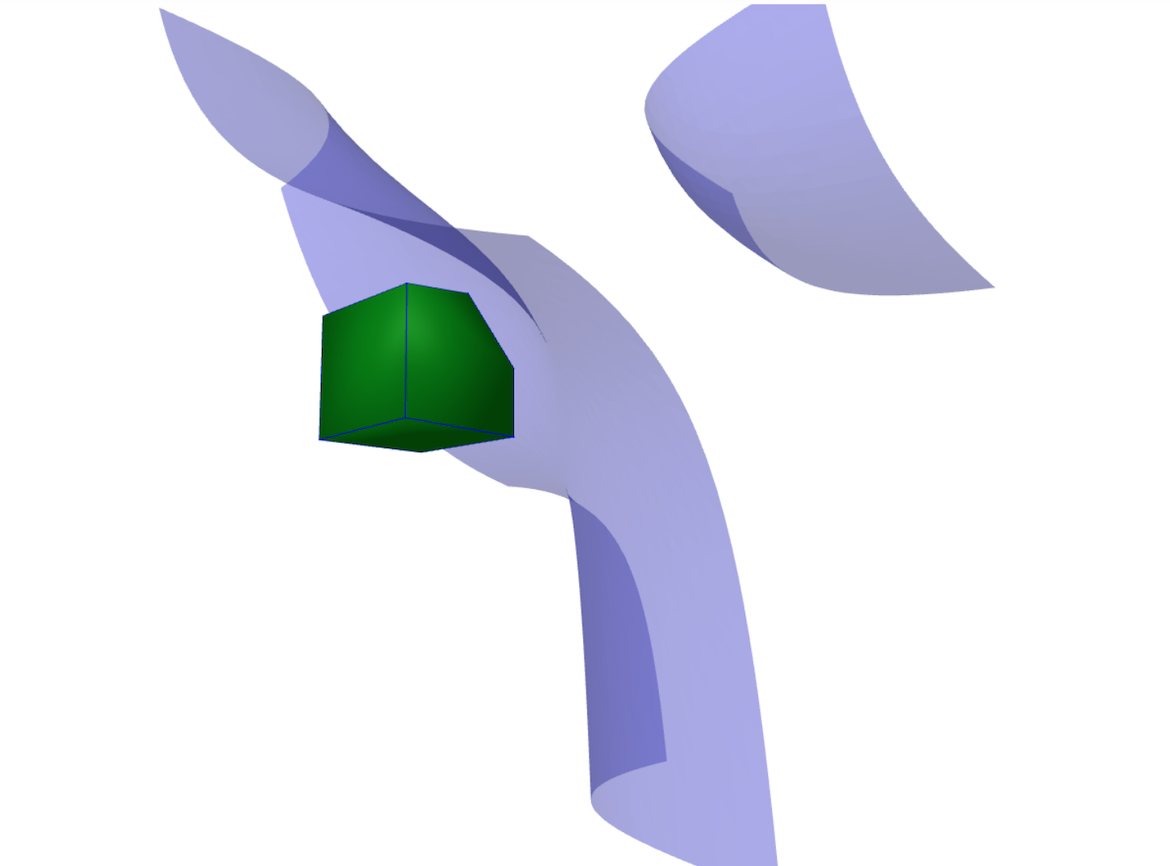}
    \includegraphics[width=0.49\textwidth]{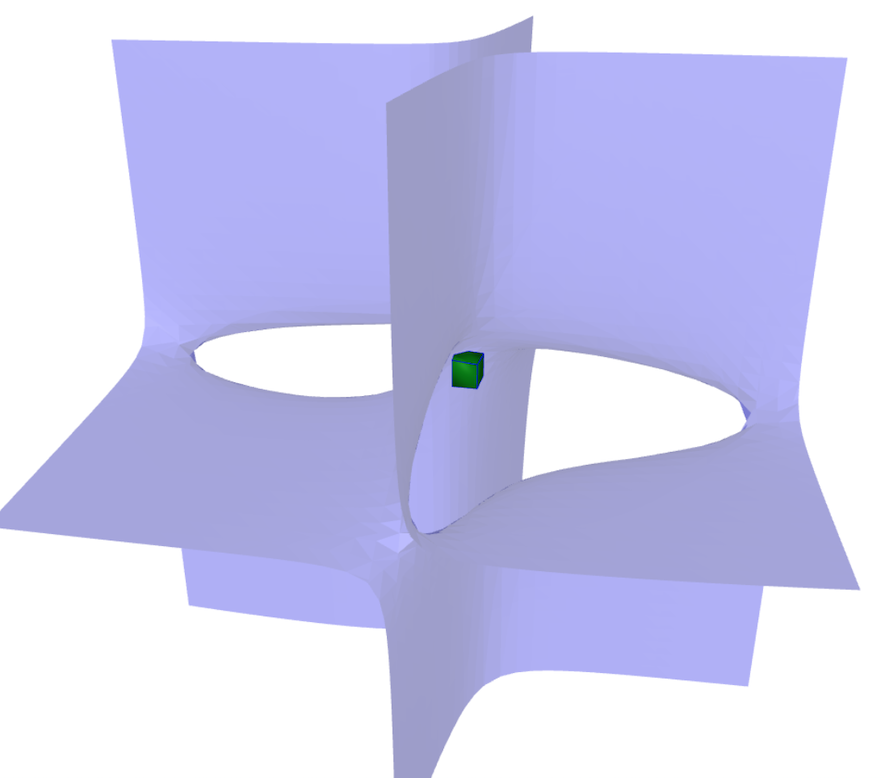}
    \caption{Non-hyperbolic adjoint surface of a convex polytope.}
    \label{fig:counterexPolytope}
\end{figure}

\subsection{Three conics} \label{subsec:threeconics}
The first non-trivial case of Conjecture~\ref{conj:Wachs} when a polypol is bounded by curves of degree greater than one is the case of three real conics; see \Cref{wachspress5}. 
Polycons of total degree six have recently been discussed by his  grandson J.~Wachspress in \cite{wachspress2020denominator}, but his arguments are non-conclusive. Below we attempt to settle the same case. Firstly, we classify all regular polycons bounded by three ellipses that meet transversally. 
Secondly, for almost all such polycons (including the cases of completely real intersections, the $M$-case), we show that the adjoint curve lies strictly outside of the polycon. Thirdly, for each of the  remaining cases, we find a representing triple of ellipses (with a polycon) and  compute the equation of its adjoint. In every example we have constructed, Wachspress's conjecture turns out to hold, but a formal argument which settles these cases in complete generality is currently still missing. 
Additionally, we observe that, in contrast to the case of polygons discussed in \Cref{subs:polygons}, there are configurations of three ellipses and a polycon for which the adjoint curve is not hyperbolic; see \Cref{rem:nonHyperbolic}.

\begin{theorem}\label{thm:wachspressellipses}
There are exactly $44$ topologically non-equivalent configurations of three ellipses in $\mathbb{R}^2$ such that each pair of ellipses intersects each other transversally and in at least two real points, and all three of them do not intersect at a common real point. In $33$ of these configurations, the adjoint curve of any regular polycon $P$ in the configuration does not intersect the interior of $P_{\ge 0}$.
\end{theorem}

For the proof of Theorem \ref{thm:wachspressellipses}, we refer the reader to the Appendix in Section \ref{sec:appendix}, where the $44$ admissible configurations of three ellipses are presented in a catalog; see Subsections \ref{sec:makeCatalog}--\ref{sec:catalog}.
Moreover, for $28$ of these configurations (and all polycons existing in these configurations) we prove Wachspress's conjecture by showing that the adjoint curve has to be hyperbolic with  the oval lying strictly outside of the polycon; see \Cref{prop:polyconHyperbolic}.
Thus, we are left with $16$ problematic configurations. Finally, in \Cref{prop:problematicPolycons} we give a more intricate argument for $5$ of the problematic configurations. We also compute the adjoint curve for example instances of the problematic polycons in the remaining $11$ configurations, see \Cref{fig:problematicAdjoint}.

\section{Finite Adjoint Maps}
\label{Finite Adjoint Correspondences}

The adjoint of a heptagon is a quartic curve, which has 14 degrees of freedom. Moreover, the adjoint is parametrized by the 14 vertex coordinates of the heptagon. It has been observed in \cite{kohn2018moment} that the \emph{adjoint map}, taking the vertices to the adjoint, is finite for heptagons and the authors posed the question to determine its degree. In this section, we conjecture that a general quartic curve is the adjoint of 864 heptagons and give numerical evidence. Moreover, we introduce an analogous adjoint map for arbitrary rational polypols and identify all cases in which it is generically finite. 

We fix positive integers $d_1,\ldots,d_k$ for $k \geq 2$ to denote the degrees of the boundary curves of a polypol. Let ${\cal R}_{d_i}$ be the space of complex rational plane curves of degree $d_i$. We consider the space ${\cal Y}^\circ_{d_1, \ldots, d_k} \subset {\cal R}_{d_1} \times \cdots \times {\cal R}_{d_k} \times (\P^2)^k$ given by 
$${\cal Y}^\circ_{d_1, \ldots, d_k} = \left \{ (C_1,\ldots,C_k,v_{12},\ldots,v_{k1}) \;\middle\vert\; 
\begin{array}{l}
C_i \in {\cal R}_{d_i} \text{ is nodal},  
    \text{$C_i$ and $C_j$ for $i \neq j$ }\\ \text{intersect transversally},  
    v_{ij} \in C_i\cap C_{j}\\ \text{is a nonsingular point on $C_i$ and $C_j$}
\end{array}
 \right \},$$
 and its Zariski closure ${\cal Y}_{d_1, \ldots, d_k} = \overline{{\cal Y}^\circ_{d_1, \ldots, d_k}}$ in ${\cal R}_{d_1} \times \cdots \times {\cal R}_{d_k} \times (\P^2)^k$. 
 \begin{lemma} \label{lem:dimY}
 The variety ${\cal Y}_{d_1, \ldots, d_k}$ has dimension $3d - k$, where $d = d_1 + \cdots + d_k$.
 \end{lemma}
 \begin{proof}
 The image of the projection ${\cal Y}^\circ_{d_1, \ldots, d_k} \rightarrow {\cal R}_{d_1} \times \cdots \times {\cal R}_{d_k}$ is dense and has finite fibers (of cardinality $d_1^2d_2^2\cdots d_k^2$). 
  Therefore 
 \[ \dim {\cal Y}_{d_1, \ldots, d_k} = \dim {\cal Y}^\circ_{d_1, \ldots, d_k} = \dim \left ( {\cal R}_{d_1} \times \cdots \times {\cal R}_{d_k} \right ) = 3d-k. \qedhere \] 
 \end{proof}

The \em adjoint map \em is the rational map 
\[\alpha_{d_1, \ldots, d_k}: {\cal Y}_{d_1, \ldots, d_k} \dashrightarrow \P \, (\C[x,y,z]_{d-3}),\]
 that takes a rational polypol to a defining equation of its adjoint curve. 
 
 \begin{theorem}\label{th: finite adjoint} The adjoint map $\alpha_{d_1, \ldots, d_k}$ is dominant and generically finite precisely in the following cases, up to cyclic permutations, and up to reversing the order:
  \begin{align} \label{eq:finiteadjmap}
  \begin{matrix}
   (d_1, \ldots, d_k)=  & (2,2,2,2), & (1,2,2,3), & (1,2,3,2), & (1,1,3,3), \\
& (1,1,2,4), & (1,2,1,4), &  (1,1,1,5), & (1,1,1,1,1,1,1).
  \end{matrix}
\end{align}
 \end{theorem}

To prove Theorem \ref{th: finite adjoint}, we need the following lemma. 

\begin{lemma} \label{lem:Yirred}
For all types listed in \eqref{eq:finiteadjmap}, as well as $(d_1, \ldots, d_4) = (1,3,1,3)$, the variety ${\cal Y}_{d_1, \ldots, d_k}$ is irreducible. 
\end{lemma}
\begin{proof}
To prove this, we give explicit rational parameterizations of ${\cal Y}^\circ_{d_1, \ldots, d_k}$ for all types in \eqref{eq:finiteadjmap} and for $(d_1, \ldots, d_4) = (1,3,1,3)$. In each of the cases with $k = 4$, the parameterization is given by $(\P^2)^{10} \dashrightarrow {\cal Y}^\circ_{d_1,d_2,d_3,d_4}$, taking 6 residual points $r_1, \ldots, r_6$ and the four vertices $v_{12},v_{23},v_{34},v_{41}$ to a polypol.
This is illustrated in Figure \ref{fig:8types}.
\begin{figure}
    \centering
    \includegraphics[width=\textwidth]{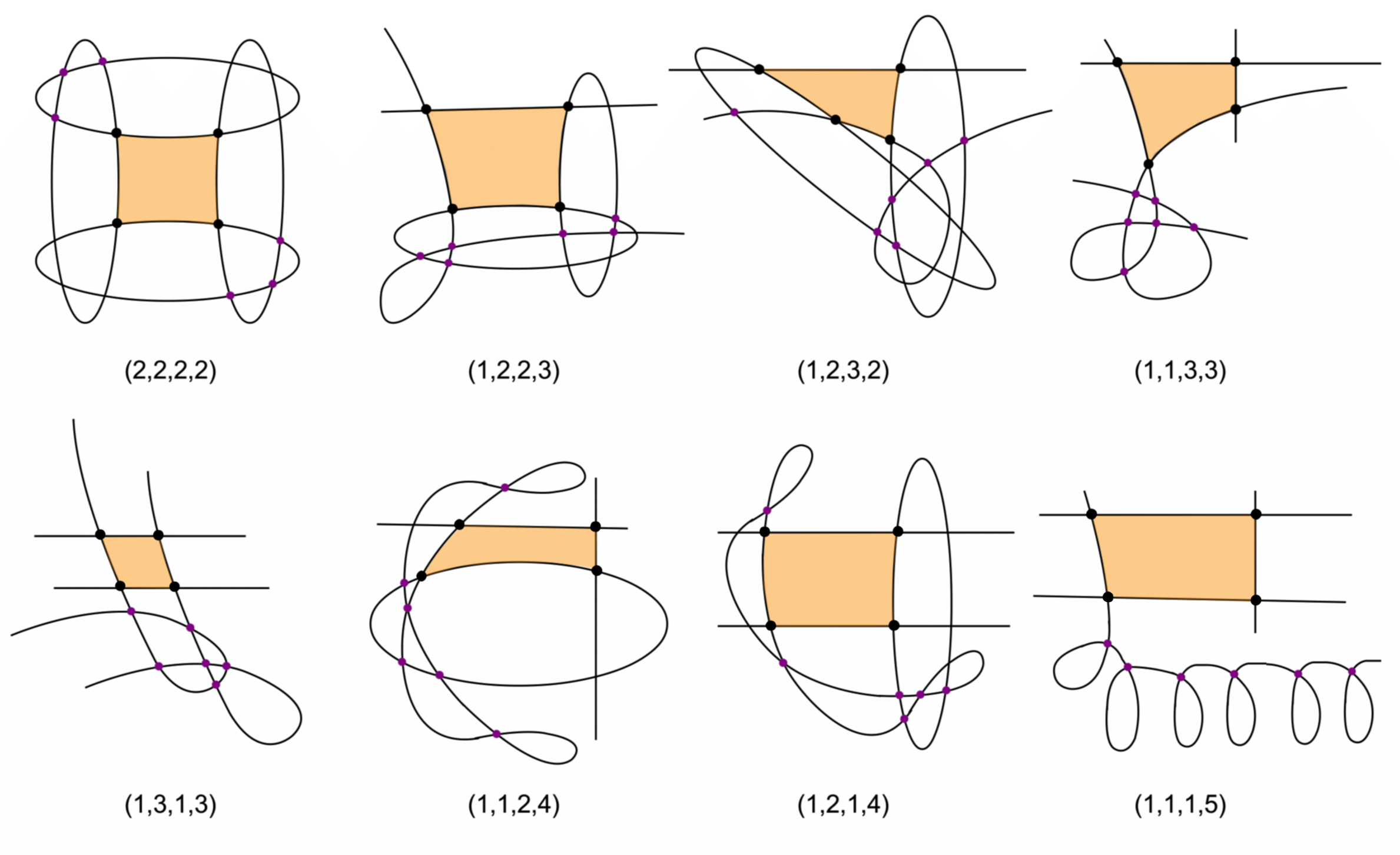}
    \caption{Four vertices and six residual points parameterizing ${\cal Y}^\circ_{d_1,d_2,d_3,d_4}$ for all  types in Lemma \ref{lem:Yirred} with $k = 4$.}
    \label{fig:8types}
\end{figure}

For instance, for type $(2,2,2,2)$ we get four conics in the following way. Let $C_1$, respectively $C_4$, be the unique conic passing through $\{ r_1,r_2,r_3 \}$ and $ \{v_{12},v_{41} \}$, respectively $\{v_{34},v_{41} \}$. Similarly, let $C_2$, respectively $C_3$, be the conic uniquely determined by $\{r_4,r_5,r_6\}$ and $\{v_{12},v_{23}\}$, respectively $\{v_{23},v_{34}\}$. The parameterization of all other $k = 4$ cases can be deduced from Figure \ref{fig:8types}.

In the case $(d_1, \ldots, d_7) = (1,1,1,1,1,1,1)$, the boundary curves are straight lines and given 7 points in the plane we find the equations for the 7 lines defining this polygon. This gives a rational parameterization $(\P^2)^{7} \dashrightarrow {\cal Y}^\circ_{1,1,1,1,1,1,1}$.
\end{proof}

 \begin{proof}[Proof of Theorem \ref{th: finite adjoint}]
For $\alpha_{d_1, \ldots, d_k}$ to be dominant and generically finite, we must have $\dim {\cal Y}_{d_1, \ldots, d_k} = \dim \P \, (\C[x,y,z]_{d-3}) = \binom{d-1}{2}-1$. Using Lemma \ref{lem:dimY}, we obtain 
\begin{align}
    \label{eq:finiteDominant}
 \textstyle   3d-k = \binom{d-1}{2}-1.
\end{align}
To find all integer solutions of this equation where $d \geq k$, we write \eqref{eq:finiteDominant} as $d^2 - 9d +2k = 0$. Hence, the solutions are
\begin{align*}
\textstyle    d = \frac{9}{2} \pm \frac{1}{2} \sqrt{81-8k}.
\end{align*}
We have positive integer solutions if and only if $k \in \{ 4,7,9,10 \}$. Considering the constraint $d \geq k$, the only possible solutions are $(k,d) \in \{ (4,8), (7,7) \}$. It follows that, based on this dimension count, the 9 possible candidates for $(d_1, \ldots, d_k)$ such that the adjoint map can be dominant and generically finite are given in \Cref{lem:Yirred}. Below, in Proposition \ref{lem:(1,3,1,3)}, we will show that generic finiteness fails for $(d_1, \ldots, d_k) = (1,3,1,3)$.

We now show that for the types listed in \eqref{eq:finiteadjmap}, the adjoint map is dominant and has generically finite fibers. Since both domain and codomain of the adjoint map are irreducible by Lemma \ref{lem:Yirred} and of the same dimension, by \cite[Ch.~1, \S 8, Thm.~2 and Cor.~1]{mumford1999red}, it suffices to show that there is an isolated point in $\alpha_{d_1, \ldots, d_k}^{-1}(\alpha_P)$ for some adjoint $\alpha_P \in \P \, (\C[x,y,z]_{d-3})$. We do this for the cases in \eqref{eq:finiteadjmap} via \emph{certified numerical computation}. We sketch these computations and work out an example for $(d_1, \ldots, d_4) = (2,2,2,2)$ below (Example \ref{ex:2222}). In what follows, we write shortly ${\cal Y}^\circ, {\cal Y}$ for ${\cal Y}^\circ_{d_1, \ldots, d_k}, {\cal Y}_{d_1, \ldots, d_k}$.

The variety ${\cal Y}$ is embedded in the space ${\cal R}_{d_1} \times \cdots \times {\cal R}_{d_k} \times (\P^2)^k$ of tuples of boundary curves and vertices. In order to write down explicit polynomial equations, we work in a larger ambient space which also keeps track of residual points and the adjoint. 
The residual points come in groups of size $d_i d_j-1$ for $(i,j)\in\mathcal{I}:=\{(1,2),\ldots,(k-1,k),(1,k)\}$ and $d_i d_j$ for $(i,j)\in\mathcal{S}\setminus\mathcal{I}$ where $\mathcal{S}:=\{(i,j)\mid 1 \leq i < j \leq k\}$. 
Hence, a point $P \in {\cal Y}^\circ$ corresponds to a rational polypol with ordered tuple of residual points
\[ \textstyle R(P) \in ~W := \prod_{(i,j)\in\mathcal{S}\setminus\mathcal{I}} {\rm (\P^2)}^{d_id_j} \times \prod_{(i,j)\in\mathcal{I}} {\rm (\P^2)}^{d_id_j - 1} \]
and with a unique adjoint $\alpha_P \in \P \, (\C[x,y,z]_{d-3})$. 
 
We consider the subset of
\begin{equation} \label{eq:bigspace}
    \left [ {\cal R}_{d_1} \times \cdots \times {\cal R}_{d_k} \times (\P^2)^k \right ] \times W \times \P \, (\C[x,y,z]_{d-3}) 
\end{equation}
consisting of all points $(P, R(P), \alpha_p)$ such that 
\begin{enumerate}[label=(\roman*)]
    \item $P \in {\cal Y}^\circ$,
    \item $\alpha_P(p) = 0$ for any projection $p$ of $R(P)$ onto a factor of $W$,
    \item $p\in C_i$ for $p \in \{v_{i-1,i},v_{i,i+1}\}$ and for $p$ equal to any of the relevant residual points.
\end{enumerate}  

The Zariski closure of 
this set in \eqref{eq:bigspace} is denoted by $\hat{{\cal Y}}$. It is clear that generic finiteness of the projection $\hat{{\cal Y}} \rightarrow \P \, (\C[x,y,z]_{d-3})$ is equivalent to generic finiteness of $\alpha_{d_1, \ldots, d_k}$. With this setup, $\hat{{\cal Y}}$ is contained in the variety given by the incidences (ii) and (iii). 
For all $(d_1, \ldots, d_k)$ from \eqref{eq:finiteadjmap}, this gives a system of polynomial equations $F(P,R(P);\alpha_P) = 0$ with as many equations as unknowns (the coordinates of $P, R(P)$), parameterized by the coefficients of the adjoint $\alpha_P$; see Example \ref{ex:2222} for an illustration. The solution set for a general $\alpha_P$ contains the fiber of $\hat{{\cal Y}} \rightarrow \P \, (\C[x,y,z]_{d-3})$ over $\alpha_P$.  
We proceed as follows. 
\begin{enumerate}
    \item Fix a polypol $P \in {{\cal Y}^\circ}$.
    \item Compute its residual points $R(P)$ and its adjoint $\alpha_P$.
    \item Verify that $(P,R(P))$ is a regular, isolated solution of $F(P,R(P);\alpha_P) = 0$. Here, by a regular, isolated solution we mean an isolated solution of multiplicity one. Equivalently, these are solutions at which the Jacobian matrix of the square polynomial system has full rank. 
    \item Conclude that $\hat{{\cal Y}} \rightarrow \P \, (\C[x,y,z]_{d-3})$ is dominant and has generically finite fibers.
\end{enumerate}
Step 2 is performed using the software \texttt{HomotopyContinuation.jl} (v2.3.1) \cite{breiding2018homotopycontinuation} and \texttt{EigenvalueSolver.jl} \cite{bender2021yet}. Although this only leads to numerical approximations of $R(P)$ and $\alpha_P$, using \emph{certification} in Step 3 leads to a rigorous proof that we found a regular, isolated solution. For this, we used the certification method proposed in \cite{breiding2020certifying} and implemented in \texttt{HomotopyContinuation.jl}. The certificates obtained from this computation will be made available at \url{https://mathrepo.mis.mpg.de}.
 \end{proof}
 
 \begin{example}[The case of $(2,2,2,2)$-polypols] \label{ex:2222}
To illustrate how we compute fibers of $\hat{{\cal Y}} \rightarrow \P \, (\C[x,y,z]_{d-3})$ in the proof of Theorem \ref{th: finite adjoint}, we work out the case $(d_1, \ldots, d_4) = (2,2,2,2)$ explicitly. We have $k = 4$ and $d = 8$. To describe the variety $\hat{\cal Y}$, we consider the incidence variety given by all points $(C_1; C_2; C_3; C_4; v_{12},v_{23},v_{34},v_{41}; r_1, \ldots, r_{20}; \alpha_P)$ in ${\cal R}_2 \times {\cal R}_2 \times {\cal R}_2 \times {\cal R}_2 \times (\P^2)^4 \times (\P^2)^{20} \times \P \, (\C[x,y,z]_{5})$ satisfying 
\begin{align} \label{eq:equations2222}
\begin{matrix}
    \alpha_P(r_i) = 0, \quad i = 1, \ldots, 20, \\
    f_1(v_{14}) = f_1(v_{12}) = 0, &
    f_1(r_i) = 0, & i = 1,2,3,7,8,9,13,14,15,16, \\
    f_2(v_{12}) = f_2(v_{23}) = 0, &
    f_2(r_i) = 0, & i = 4,5,6,7,8,9,17,18,19,20, \\
    f_3(v_{23}) = f_3(v_{34}) = 0, &
    f_3(r_i) = 0, & i = 4,5,6,10,11,12,13,14,15,16, \\
    f_4(v_{34}) = f_4(v_{41}) = 0, &
    f_4(r_i) = 0, & i = 1,2,3,10,11,12,17,18,19,20,
\end{matrix}
\end{align}
where $f_i$ is a defining equation of $C_i$.
The 20 residual points $(r_1, \ldots, r_{20}) \in (\P^2)^{20}$ are subdivided into 
\begin{align} \label{eq:respointsgroups2222}
\begin{matrix}
    \{r_1, r_2, r_3 \} \subset C_1 \cap C_4, &
    \{r_4, r_5, r_6 \} \subset C_2 \cap C_3, \\ 
    \{r_7, r_8, r_9 \} \subset C_1 \cap C_2, &
    \{r_{10},r_{11},r_{12} \} \subset C_3 \cap C_4, \\
    \{r_{13},r_{14},r_{15},r_{16} \} \subset C_1 \cap C_3, &
    \{r_{17},r_{18},r_{19},r_{20} \} \subset C_2 \cap C_4.
\end{matrix}
\end{align}
For instance, for the $(2,2,2,2)$-polypol in Figure \ref{fig:8types}, the two groups of residual points in $C_1 \cap C_4$ and $C_2 \cap C_3$ are marked by purple dots. The points $r_7, \ldots, r_{12}$ are the non-marked intersection points, and $r_{13}, \ldots, r_{20}$ are complex. 
The 68 equations \eqref{eq:equations2222} define a variety in the 88-dimensional space ${\cal R}_2 \times {\cal R}_2 \times {\cal R}_2 \times {\cal R}_2 \times (\P^2)^4 \times (\P^2)^{20} \times \P \, (\C[x,y,z]_{5})$. 

This variety contains the 20-dimensional $\hat{\cal Y}$. We view \eqref{eq:equations2222} as a system of 68 polynomial equations on the 68-dimensional space ${\cal R}_2 \times {\cal R}_2 \times {\cal R}_2 \times {\cal R}_2 \times (\P^2)^4 \times (\P^2)^{20}$, parameterized by the coefficients of $\alpha_P$. The fiber of $\hat{{\cal Y}} \rightarrow \P \, (\C[x,y,z]_{5})$ over a generic quintic plane curve $\alpha_P$ consists  of all solutions to these equations. 

Theorem \ref{th: finite adjoint} implies that $\hat{{\cal Y}} \rightarrow \P \, (\C[x,y,z]_{5})$ is a branched covering, which means that once we found a solution to \eqref{eq:equations2222}, we can use it as a seed for a \emph{monodromy} computation \cite{duff2019solving}. This allows us to compute different polypols with the same adjoint. An example obtained using the \texttt{monodromy\_solve} command in \texttt{HomotopyContinuation.jl} is given in Figure \ref{fig:2222adj}. 
\begin{figure}[htb]
    \centering
    \includegraphics[scale=0.35]{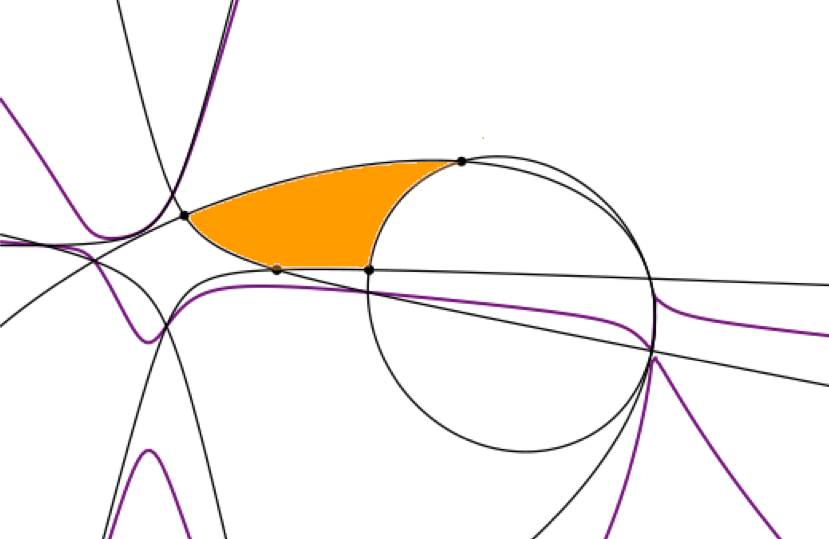}~
    \includegraphics[scale=0.35]{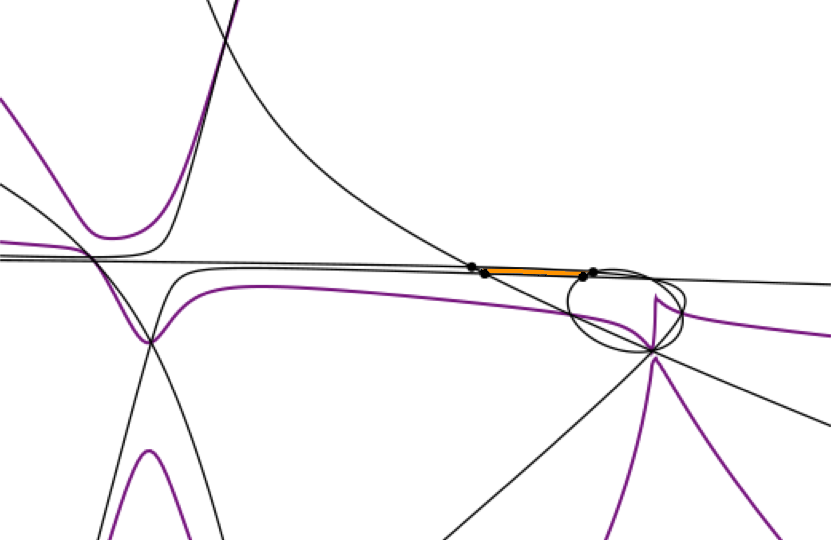}~
    \caption{Two regular (2,2,2,2)-polypols that share the same adjoint. (The adjoint is nonsingular, even though it here looks singular.) }
    \label{fig:2222adj}
\end{figure}
 \end{example}
 
 Theorem \ref{th: finite adjoint} covers all types of polypols in Lemma \ref{lem:Yirred}, except $(1,3,1,3)$. The following result implies that the set of quintic plane curves that are the adjoint curve of a $(1,3,1,3)$-polypol is a lower-dimensional subvariety of $\PP (\C[x,y,z]_5)$. 
 \begin{proposition} \label{lem:(1,3,1,3)}
 The adjoint map 
$\alpha_{1,3,1,3}: {\cal Y}_{1,3,1,3} \dashrightarrow \P \, (\C[x,y,z]_{5})$
that takes a polypol of type $(1,3,1,3)$ to its adjoint curve is not dominant.
\end{proposition}
\begin{proof} 
By Lemma \ref{lem:Yirred}, the source and target are irreducible of the same dimension, so the map is dominant only if the general quintic curve is an adjoint for finitely many polypols.  Thus, if the map is dominant, then the general adjoint curve would be nonsingular.

So let $A$ be a nonsingular quintic plane curve, adjoint to a rational polypol of type (1,3,1,3).
We claim that $A$ is the adjoint curve of a positive-dimensional family of rational polypols of type (1,3,1,3).

Let $C_1$ and $C_3$ be the lines and $C_2$ and $C_4$ the irreducible nodal cubic curves of the polypol for which $A$ is the adjoint.
Let $r_1$ be the intersection point of the two lines, $r_2$ and $r_4$ the nodes of $C_2$ and $C_4$ respectively, and let $C_1\cap C_2=\{r_3, r_5, v_{12}\}, C_2\cap C_3=\{r_6, r_7, v_{23}\}, C_1\cap C_4=\{r_8, r_9, v_{14}\}, C_3\cap C_4=\{r_{10}, r_{11}, v_{34}\}$
and 
$C_2\cap C_4=\{r_{12},...,r_{20}\}$.
Then the quintic curve $A$ passes through the points $r_1,...,r_{20}.$  

On $A$ the points $r_{12},...,r_{20}$ form a divisor of degree $9$.  It is nonspecial, since a canonical divisor of $A$ is the intersection with a conic and the $9$ points $r_{12},...,r_{20}$ are not contained in a conic section.  Therefore, by Riemann--Roch, it moves in a complete linear system $L$ of projective dimension $3$.  
Now $C_2$ belongs to the linear system $|C_2|$ of cubic curves with a node at $r_2$ and passing through the four points $r_3,r_5, r_6,r_7$ which has dimension at least $2$.
Similarly $C_4$ belongs to the linear system $|C_4|$ of cubic curves with a node at $r_4$ and passing through the four points $r_8,r_9, r_{10},r_{11}$, also of dimension at least $2$.
These two linear systems of curves restrict to linear systems of divisors $L_2$ and $L_4$ of degree $15$ on $A$ that have dimension at least $2$.
Both $L_2$ and $L_4$ have a fixed part of degree $6$ and a moving part of degree $9$.  The fixed part of $L_2$ is $r_2$ with multiplicity $2$ and $r_3,r_5, r_6,r_7$.  Similarly
$r_4$ and $r_8,r_9, r_{10},r_{11}$ is the fixed part of $L_4$.
In both cases, $r_{12},...,r_{20}$ belong to the moving part. 
So the moving part $M_2$ of $L_2$ and the moving part $M_4$ of $L_4$ are subsystems of $L$.  For dimension reasons they must meet at least in a pencil.  Each divisor $r'_{12},...,r'_{20}$ in that pencil is the intersection of nodal cubic curves $C_2'$ and $C_4'$ that together with $C_1$ and $C_3$ define a rational polypol whose adjoint curve is $A$.
This concludes the proof of the claim, and thereby the proposition.
 \end{proof}
The proof of Proposition \ref{lem:(1,3,1,3)} does not exclude that the (closure of the) image of the adjoint map $\alpha_{1,3,1,3}$ is the discriminant locus of $\PP \, ( \C[x,y,z]_5)$. The following example shows that almost all quintic curves adjoint to a $(1,3,1,3)$-polypol are nonsingular. 
\begin{example}
The image of the $(1,3,1,3)$-polypol with vertices $v_{12} = (1:-1:1), v_{23} = (1:1:1), v_{34} = (-1:1:1), v_{14} = (-1:-1:1)$ and boundary curves given by
\[ \begin{matrix}
    f_1 = y+z, &
f_2 =  -x^3-x^2y+xy^2-y^3+x^2z+2xyz+3y^2z+8xz^2-12z^3, \\
f_3 =  y-z, & 
f_4 =  x^3-x^2y-xy^2-y^3+x^2z-2xyz+3y^2z-8xz^2-12z^3 
\end{matrix} \]
under the adjoint map $\alpha_{1,3,1,3}$ is the nonsingular quintic curve given by 
\[
3x^2y^3+5y^5+x^4z-3x^2y^2z-14y^4z-5x^2yz^2+y^3z^2-16x^2z^3+48y^2z^3+12yz^4+48z^5 = 0.
\]

\end{example}
In the light of Theorem \ref{th: finite adjoint}, the next natural question is to compute the degree of the adjoint map in the cases \eqref{eq:finiteadjmap}. 
 
Using
\texttt{monodromy\_solve} in \texttt{HomotopyContinuation.jl}, we found the following answer for heptagons. 

\begin{proposition}\label{prop: hept} A generic quartic curve is the adjoint of at least 864 complex 7-gons.
\end{proposition}

\begin{proof} There is a 1-1 correspondence between heptagons and their dual heptagons. We consider the rational map taking a heptagon to the adjoint of its dual, which has the same degree as the adjoint map $\alpha_{1,1,1,1,1,1,1}$. In general, the adjoint of the dual of a polytope $P\subseteq \mathbb{R}^n$ is given by Warren's formula  \cite{warren1996barycentric,kohn2019projective}:
$$\alpha(t)=\sum_{\sigma \in \tau(P)}\mathrm{vol}(\sigma)\prod_{v\in V(P)\setminus V(\sigma)} \ell_v(t), $$
where $\tau(P)$ is a triangulation of $P$ that uses only the vertices of $P$, $V(P)$ denotes the set of vertices of $P$, $V(\sigma)$ is the set of vertices of the simplex $\sigma$, and $\ell_v(t)=1-v_1t_1-\cdots-v_nt_n$. Therefore, we map the 7 vertices of a heptagon to the adjoint of its dual using this formula, giving us a polynomial map. We get the same heptagon if we act on the vertices with the dihedral group $D_7$. Using the command \texttt{monodromy\_solve} with a random seed, we find $12096=14\cdot 864$ different points in the fiber. Certification guarantees that all of these are distinct, regular, isolated solutions. Since $|D_7|=14$, the statement follows. The certificates will be available at \url{https://mathrepo.mis.mpg.de}.
\end{proof}

Our numerical computations consistently gave the same answer for different random seeds in the monodromy computations. This supports the following conjecture. 

\begin{conjecture} \label{conj:heptagon-adjoints}A generic quartic curve is the adjoint of precisely 864 complex 7-gons.
\end{conjecture}

As an illustration of the monodromy computations for the heptagon case, Figure \ref{fig:heptadj} shows two different heptagons with the same adjoint curve. 

\begin{figure}[htb]
    \centering
    \includegraphics[scale=0.25]{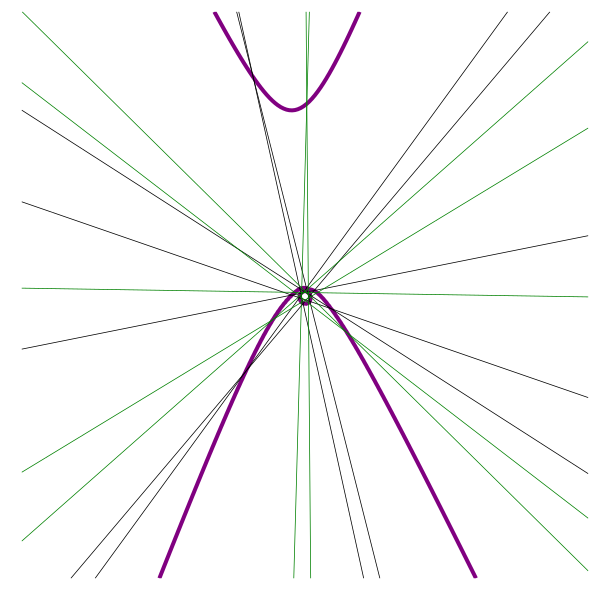}~
    \includegraphics[scale=0.25]{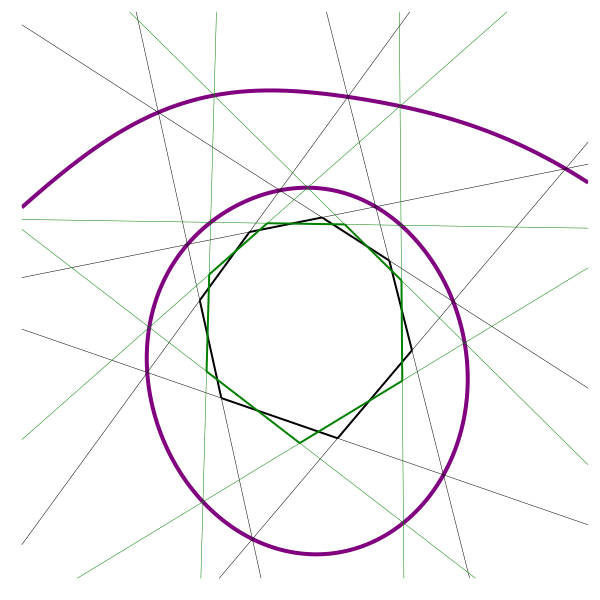}~
    \caption{Two heptagons (green and black) and their shared quartic adjoint (purple), zoomed out and zoomed in.}
    \label{fig:heptadj}
\end{figure}

\begin{remark} 
For the cases in \eqref{eq:finiteadjmap} with $k = 4$, it is much more challenging to compute the degree of the adjoint map. One difficulty comes from the large symmetry groups acting on the fibers in the polynomial systems constructed as in \eqref{eq:equations2222}. An orbit in the $(2,2,2,2)$ case consists of $|S_3|^4 \cdot |S_4|^2 = 6^4 \cdot 24^2 = 746496$ points, since the residual points come in 4 groups of size 3 and 2 groups of size 4, see \eqref{eq:respointsgroups2222}. The command \texttt{monodromy\_solve} can take these group actions into account via the option \texttt{group\_action}. Using this we found 14095 points in a general fiber of $\alpha_{2,2,2,2}$, each representing an orbit of size $746496$. More solutions can be found, but we interrupted the computation after 72 hours. 
\end{remark}

\section{Statistics and Push-forward}
\label{Statistics and Pushforward}

As mentioned in the introduction, the study of positive geometries and their canonical forms was originally motivated by their connection to scattering amplitudes in particle physics \cite{arkani2017positive}. Prominent geometries in this context are the amplituhedron \cite{arkani2015positive}, whose canonical form describes $\textup{N} = 4$ SYM amplitudes, and the ABHY associahedron \cite{arkani2018scattering}, which plays the same role in bi-adjoint scalar $\phi^3$ theories. In the latter setting, the amplitude can alternatively be computed as a global residue over the solutions to the \emph{scattering equations} \cite{cachazo2014scattering}. The connection between these two descriptions is given by the \emph{push-forward formula} (Heuristic \ref{heur:pushfwd}). 

In \cite{sturmfels2021likelihood}, it was observed that the scattering equations are the likelihood equations for the moduli space ${\cal M}_{0,n}$ of genus 0 curves with $n$ marked points, interpreted as a statistical model. It was also established that many other statistical models (called \emph{positive} models, see Definition \ref{def:posmodels}) admit a natural definition for a \emph{statistical amplitude}. This amplitude is the rational function defining the canonical form $\Omega(X_{\geq 0})$ of a positive geometry $X_{\geq 0}$, naturally associated to a positive model ${\cal X}$. It may be obtained, via the push-forward formula, as a sum over the critical points of the log-likelihood function.

In the next subsection, we recall the basics of likelihood estimation in algebraic statistics. For a more detailed exposition, see \cite{huh2014likelihood}. In  \Cref{subsec:pushfwd} we recall the definition of the push-forward of a differential form and state a heuristic from \cite{arkani2017positive}. Finally, in \Cref{subsec:CFlikelihood}, we make the connection between the push-forward formula and likelihood estimation. We propose to use canonical forms as a \emph{trace test} to detect whether a set of approximate solutions to the likelihood equations is complete. Example \ref{ex:toricsquare3} presents a regular polypol with nonlinear boundary whose canonical form is given by a global residue over the solutions to the likelihood equations of a toric statistical model. 

\subsection{Likelihood equations}

In algebraic statistics, a \emph{statistical model} ${\cal X}$ is a subvariety of $\P^n$, whose intersection with the probability simplex $\PP^n_{>0} = \{(p_0:\cdots:p_n) \in \P^n(\R) ~|~ p_i \neq 0,  \textup{ sign}(p_0) = \cdots = \textup{sign}(p_n) \}$ is non-empty. The points ${\cal X} \cap \PP^n_{>0}$ represent feasible probability distributions for a discrete random variable with $n+1$ states. In our context, a statistical model ${\cal X}$ of dimension $d$ will arise via a parametrization $(p_0(x_1, \ldots, x_d): \cdots: p_n(x_1,\ldots, x_d))$, where $p_i(x_1,\ldots, x_d)$ are functions that sum to one. 

Suppose that, in an experiment, state $i$ is observed a total number of $u_i$ times. We would like to find a probability distribution $(p_0:\cdots:p_n) \in {\cal X} \cap \PP^n_{>0}$ that `best explains' these experimental data $(u_0, \ldots, u_n)$. An approach to this problem coming from \emph{likelihood inference} is to consider the probability distribution that maximizes the \emph{log-likelihood function} $L = u_0 \log p_0 + \cdots + u_n \log p_n$ on ${\cal X} \cap \PP^n_{>0}$. In practice, this can be done by solving the \emph{likelihood equations}
\begin{equation} \label{eq:likelihood}
    \frac{\partial L(x_1,\ldots, x_d)}{\partial x_1} = \cdots = \frac{\partial L(x_1,\ldots, x_d)}{\partial x_d} = 0,
\end{equation}
where $L(x_1,\ldots, x_d) = u_0 \log (p_0(x_1,\ldots, x_d)) + \cdots + u_n \log (p_n(x_1,\ldots, x_d))$. The number of complex solutions to \eqref{eq:likelihood}, i.e., the number of critical points of $L$ on ${\cal X}$, is an invariant called the \emph{maximum-likelihood degree} (ML degree) of ${\cal X}$ \cite{catanese2006maximum}. It measures the complexity of the maximum likelihood estimation problem for the model ${\cal X}$. 

For the purposes of this paper, we will consider a particular class of parametrized, statistical models, introduced in \cite{sturmfels2021likelihood}.
The definition uses the notion of \emph{positive rational functions}, which are fractions of two polynomials with real, positive coefficients. 
\begin{definition}[positive models] \label{def:posmodels}
A model ${\cal X}$ is called \emph{positive} if it is parametrized by $(p_0:\cdots: p_n): \C^d \dashrightarrow \P^n$ and $p_i(x_1, \ldots, x_d)$ are positive rational functions that sum to 1. 
\end{definition}
Positive models form a rich class of statistical models. For examples, see \cite[Sec.~6]{sturmfels2021likelihood}. The likelihood estimation problem for positive models leads in a natural way to a \emph{morphism of positive geometries}. To illustrate how this works, we will focus on the subclass of \emph{toric models}. For any Laurent polynomial $q = \sum_{i=0}^n c_i x^{a_i}$ (here $x^{a_i}$ is short for $x_1^{a_{i,1}}x_2^{a_{i,2}} \cdots x_d^{a_{i,d}}$, $a_i \in \Z^d$) with positive coefficients $c_i > 0$, we obtain a positive statistical model ${\cal X}$ as the closure of the image of the map $\C^d \dashrightarrow \P^n$ given by
\[ x \mapsto \left( \frac{c_0 x^{a_0}}{q(x)} : \cdots : \frac{c_n x^{a_n}}{q(x)} \right ). \]
The log-likelihood function is  
\begin{align*}
  \textstyle  L(x) & \textstyle = \sum_{i=0}^n u_i \log (c_i x^{a_i}) - \left( \sum_{i=0}^n u_i \right) \log (q(x)) \\
     & \textstyle = \sum_{j = 1}^d \left(\sum_{i=0}^n u_i a_{i,j} \right) \log x_j - \left( \sum_{i=0}^n u_i \right) \log (q(x)) + \sum_{i=0}^n u_i \log (c_i),
\end{align*}
and the likelihood equations are 
\begin{equation} \label{eq:likelihoodtoric}
  \textstyle  v_j = \left( \sum_{i=0}^n u_i \right)~ x_j \frac{\partial q(x)/ \partial x_j}{q(x)}, \quad j = 1, \ldots, d,
\end{equation}
with $v_j = \sum_{i=0}^n u_i a_{i,j}$. Let $X_{\geq 0}$ be the Newton polytope of $q(x)$, dilated by the sample size $\sum_{i=0}^n u_i$. The equations \eqref{eq:likelihoodtoric} give a rational map $\C^d \dashrightarrow \C^d$, sending $(x_1, \ldots, x_d)$ to $(v_1, \ldots, v_d)$, which extends to $\phi: \P^d \dashrightarrow \P^d$. The restriction of $\phi$ to $\R_{>0}^d$ is a diffeomorphism $\R_{>0}^d \rightarrow X_{>0}$ known as the \emph{toric moment map}.

\begin{definition}[Morphism of pseudo-positive geometries] \label{def:morphism}
A \emph{morphism} between two \linebreak[4] pseudo-positive geometries $(X,X_{\geq 0})$, $(Y,Y_{\geq 0})$ is a rational map $\phi: X \dashrightarrow Y$ which restricts to an orientation preserving diffeomorphism $\phi_{|X_{>0}}: X_{>0} \rightarrow Y_{>0}$.
\end{definition}

In our toric example, $\phi$ is a morphism between the positive geometries $(\P^d, \P^d_{\geq 0})$ and $(\P^d, X_{\geq 0})$, where $\P^d_{\geq 0}$ is the Euclidean closure of $\P^d_{> 0}$ in $\P^d(\R)$. In the rest of Section \ref{Statistics and Pushforward}, we will see that $\phi$ is the key ingredient for relating the solutions to the likelihood equations \eqref{eq:likelihoodtoric} to the canonical form of the polytope $X_{\geq 0}$. 

\begin{example} \label{ex:toricsquare1}
Consider the 2-dimensional toric model ${\cal X}$ given by $q(x) = c_0 + c_1 x + c_2 y + c_3 xy$. For general coefficients $c_i$, the ML degree of ${\cal X}$ is 2 \cite{amendola2019maximum}. The associated positive geometry is the square $X_{\geq 0} = [0,u_0+u_1+u_2+u_3]^2 \subset \P^2(\R)$. Its interior $X_{>0}$ is diffeomorphic to $\R_{>0}^2$ via the rational map 
\begin{equation} \label{eq:phitoricsquare}
 \textstyle   \phi: (x,y) \mapsto \left ( \left( \sum_{i=0}^3 u_i \right)x \frac{c_1 + c_3 y}{c_0 + c_1 x + c_2 y + c_3 xy},~ \left( \sum_{i=0}^3 u_i \right) y \frac{c_2 + c_3 x}{c_0 + c_1 x + c_2 y + c_3 xy} \right )
\end{equation} 
obtained from the likelihood equations as in \eqref{eq:likelihoodtoric}. In this case, $v_1$ and $v_2$ are given by $v_1 = u_1 + u_3$ and $v_2 = u_2 + u_3$ respectively. 
\end{example}

\subsection{Push-forward} \label{subsec:pushfwd}

Let $X$ and $Y$ be nonsingular projective varieties of the same dimension $d$.
Suppose $(X,X_{\ge 0})$ and $(Y,Y_{\ge 0})$ are pseudo-positive geometries, and that $\phi:X\dashrightarrow Y$ is a rational map that induces a morphism $\Phi: (X,X_{\ge 0}) \to (Y,Y_{\ge 0})$ of pseudo-positive geometries. Let $D\subset X$ and $C\subset Y$ denote the boundary divisors of $(X,X_{\ge 0})$ and $(Y,Y_{\ge 0})$. Each irreducible component of $D$ is mapped either onto an irreducible component of $C$, or to a smaller-dimensional subvariety of a component of $C$, and each component of $C$ is the image of a component of $D$, so $\phi(D)=C$. The trace map 
$\Tr_\phi: \phi_*\Omega_X^d \to \Omega_Y^d$
extends to the pushforward map (see \cite[Sec.~4]{arkani2017positive} and \cite[II (b), p.~352]{griffiths1976variations})
\[\Phi_*: \phi_*\mathcal M_X^d \to \mathcal M_Y^d,\]
where $\mathcal M_X$ and $\mathcal M_Y$ are the sheaves of meromorphic differentials.

If $X = Y = \P^d$, $\phi: \C^d \rightarrow \C^d$ is locally given by $d$ rational functions $g_1, \ldots, g_d$, then the push-forward of $\omega = h ~ d x_1 \wedge \cdots \wedge d x_d$ at a general point $v\in Y$ is
$$ \Phi_*(\omega)(v) = \left ( \sum_{\phi(x) = v} \frac{h(x)}{J_{g_1, \ldots, g_d} (x)} \right ) dv_1 \wedge \cdots \wedge dv_d,$$
where the sum is over all preimages $x$ of $v$ under $\phi$ and $J_{g_1, \ldots, g_d} (x) = \det \left ( \frac{\partial g_j}{\partial x_i} \right )_{i,j}$ is the Jacobian of $\phi$. 

\begin{example} \label{ex:linesegment}
Let $X = \P^1$ and $X_{\geq 0} = [0, a] \subset \R = \{(1:t)\, |\,\, t\in \R\} \subset X$. The map $\phi: \P^1 \rightarrow \P^1$ given by $\phi(x_0:x_1) = (x_0^2: x_1^2)$ is a morphism of the positive geometries $(\P^1,[0,a])$ and $(Y, Y_{\geq 0 }) = (\P^1, [0, a^2])$. Locally, on $\C = \{x_0 \neq 0\} \subset X$ it is given by $\phi(x) = x^2$. The canonical form of $X_{\geq 0}$ is $\Omega(X_{\geq 0}) = \nicefrac{a}{x(a-x)}\, dx $. Note that $\phi^{-1}(v) = \{ \sqrt{v}, - \sqrt{v} \}$. The push-forward is given by 
\begin{align*}
\Phi_*(\Omega(X_{\geq 0}))(v) &= \left ( \sum_{\phi(x) = v} \frac{a}{x(a-x)} \frac{1}{2x} \right)  dv\\
&= \left (  \frac{a}{\sqrt{v} (a - \sqrt{v})} \frac{1}{2 \sqrt{v}} +  \frac{a}{-\sqrt{v} (a + \sqrt{v})} \frac{-1}{2 \sqrt{v}} \right ) dv \\
&= \frac{a}{2v} \left( \frac{1}{a - \sqrt{v}} + \frac{1}{a + \sqrt{v}} \right ) dv \\
&= \frac{a^2}{v(a^2-v)} dv = \Omega(Y_{\geq 0})(v). \qedhere
\end{align*}
\end{example}

The following is phrased as a heuristic in \cite[Heuristic 4.1, p.~11]{arkani2017positive}.
\begin{heuristic} \label{heur:pushfwd}
The map $\Phi_*$ takes the canonical form of $(X,X_{\ge 0})$ to the canonical form of $(Y,Y_{\ge 0})$.
\end{heuristic}

In dimension one, the intuition behind Heuristic \ref{heur:pushfwd} comes from the global residue theorem. We generalize \cite[Ex.~7.9 and 7.10]{arkani2017positive} to prove a stronger version of Heuristic \ref{heur:pushfwd} in this case. 
\begin{proposition} \label{prop:dim1}
Consider a real interval $[a, b] \subset \mathbb{R}$ with $a < b$ and let $\phi \in \mathbb{R}[x]$ be any non-constant polynomial. If $\phi([a,b]) = [\phi(a),\phi(b)]$, then $\Omega(\phi([a,b])) = \Phi_*(\Omega([a,b]))$. 
\end{proposition}
Note that $\phi$ need not induce a morphism of $[a,b]$ and $[\phi(a),\phi(b)]$ as positive geometries.
\begin{proof}
Define the rational function 
\[g = \frac{b-a}{(x-a)(b-x)(\phi(x)-y)} = \frac{b-a}{h}, \]
where $y$ is a complex parameter. By the global residue theorem, we have for almost all $y \in \mathbb{C}$ that
\[ {\rm res}_a (g) + {\rm res}_b (g) + \sum_{x \text{ s.t. } \phi(x) = y } {\rm res}_x(g) = \frac{b-a}{h'(a)} + \frac{b-a}{h'(b)} + \sum_{x \text{ s.t. } \phi(x) = y }  \frac{b-a}{h'(x)} = 0. \]
Since $h'(x) = -(x-a)(\phi(x)-y) + (b-x)(\phi(x)-y) + (x-a)(b-x)\phi'(x)$, we obtain
\[  \frac{b-a}{(b-a)(\phi(b)-y)} -  \frac{b-a}{(b-a)(\phi(a)-y)} = \sum_{x \text{ s.t. } \phi(x) = y }  \frac{b-a}{(x-a)(b-x)\phi'(x)}.  \]
This holds for all but finitely many $y \in \mathbb{C}$. Therefore, this is an equality of rational functions equivalent to
\[ \Omega([\phi(a),\phi(b)]) = \frac{\phi(b)-\phi(a)}{(y-\phi(a))(\phi(b)-y)} dy = \Phi_*(\Omega([a,b])). \qedhere \]
\end{proof}

\begin{remark}[Push-forward by birational maps]
If $\phi$ is an isomorphism, the heuristic clearly works. If
$\phi$ is birational, the map $\Phi_*$ is an isomorphism \cite[II (b), p.~352]{griffiths1976variations}, but even so, the heuristic needs to be proven. An example of an explicit calculation is the ``teardrop'' example of the nodal cubic (see \cite[5.3.1]{arkani2017positive}). In that case, one computes 
\[ \Phi_* \left (\frac{2a}{(a-t)(t+a)u(1-u)}dt\wedge du \right )=\frac{2a}{y^2-x^2(x+a^2)}dx\wedge dy.\]
Note that three of the boundary components of the positive geometry in the source collapse to the same point in the image.
The adjoint curve of the source polypol (a quadrangle) is the line at infinity.

More generally, if $d=2$ and $\phi : X\to Y$ is birational, then, since $\Phi_*$ is an isomorphism, $\Phi_*$ commutes with the wedge product. Assume further that for each irreducible boundary curve $D_i$ of $(Y,Y_{\ge 0})$ there is a unique irreducible boundary curve $C_i$ of $(X,X_{\ge 0})$ such that $\phi(C_i)=D_i$. Then Heuristic \ref{heur:pushfwd} works. Indeed, write $\Omega(X,X_{\geq 0})=\omega$ and let $x$ be a local equation for $C_i$. Then $\omega=\omega'\wedge \frac{dx}{x}+ \dots$ where $\dots$ is holomorphic near $C_i$, and $\Res_{C_i} \omega =\omega'$. If $y$ is a local equation for $D_i$, then $x=\phi^*y$, and we get
$\Phi_* \omega=\Phi_*\omega' \wedge \Phi_* \frac{d\phi^*y}{\phi^* y}+\dots = \Phi_*\omega' \wedge \frac{dy}{y}+\dots.$ Hence
$\Res_{D_i}\Phi_*\omega=(\Phi|_{C_i})_* \Res_{C_i} \omega=(\Phi|_{C_i})_*\Omega(C_i,C_{i,\geq 0})$. Here we wrote $(\Phi|_{C_i})_*$ for the push-forward map induced by the restriction $\phi|_{C_i}$.
Since the heuristic works in dimension 1 (Proposition \ref{prop:dim1}), we have  $(\Phi|_{C_i})_*\Omega(C_i,C_{i,\geq 0})=\Omega(D_i,D_{i,\geq 0})$. Hence we have shown that for all boundary components $D_i$, we have
\[\Res_{D_i}\Phi_* \Omega(X,X_{\geq 0})= \Omega(D_i,D_{i,\geq 0}),\]
hence $\Phi_* \Omega(X,X_{\geq 0})=\Omega(Y,Y_{\geq 0}).$

\end{remark}

\subsection{Canonical forms and trace tests for likelihood problems} \label{subsec:CFlikelihood}

We return to the likelihood estimation problem for a toric model ${\cal X}$. Recall that the likelihood equations \eqref{eq:likelihoodtoric} give a rational map $\P^d \dashrightarrow \P^d$ which induces a morphism of the positive geometries $(\P^d, \P^d_{\geq 0})$ and $(\P^d, X_{\geq 0})$, where $X_{\geq 0} \subset \R^d$ is the $(u_0 + \cdots + u_n)$-dilation of the Newton polytope of $q(x)$. For ease of notation, we will write \eqref{eq:likelihoodtoric} as $v_j = g_j(x)$ in what follows, i.e., $\phi = (g_1, \ldots, g_d)$.

\begin{proposition} \label{prop:toricamplitude}
Let ${\cal X}$ be a toric model of dimension $d$ with associated positive geometry $(\P^d, X_{\geq 0})$. Let $\Omega(X_{\geq 0}) = h(v_1, \ldots, v_d) ~ dv_1 \wedge \cdots \wedge dv_d$. We have the following identity of rational functions: 
\[ h(v_1, \ldots, v_d) = \sum_{v_j = g_j(x) \, \forall j} \frac{1}{x_1 \cdots x_d} \frac{1}{J_{g_1, \ldots, g_d}(x)}, \]
where the sum is over all pre-images of $(v_1, \ldots, v_d)$ under $\phi$.
\end{proposition}
\begin{proof}
Heuristic \ref{heur:pushfwd} works in the case of the toric moment map \cite[Thm.~7.12]{arkani2017positive}. The statement follows directly from applying the push-forward formula, taking into account that $\Omega(\PP^d_{\geq 0}) = (x_1 \cdots x_d)^{-1} dx_1 \wedge \cdots \wedge dx_d$.
\end{proof}
We point out that this formulation is equivalent to \cite[Thm.~16]{sturmfels2021likelihood}, which uses the toric Hessian of the log-likelihood function. The rational function $h$ evaluated at $v_j = \sum_{i=1}^n u_i a_{i,j}$, is the \emph{amplitude} of ${\cal X}$ \cite[Sec.~6]{sturmfels2021likelihood}. By Proposition \ref{prop:toricamplitude}, and \eqref{eq:likelihoodtoric}, it is the global residue of $(x_1 \cdots x_d)^{-1}$ over the critical points of the log-likelihood function. 

Numerical approximations for the solutions to the likelihood equations can be obtained using methods from numerical nonlinear algebra. As illustrated in Section \ref{Finite Adjoint Correspondences}, these solutions can be \emph{certified} to prove a lower bound on the ML degree of a model ${\cal X}$. Showing that a set of approximate critical points contains an approximation for \emph{all} solutions is more difficult. For witness set computations, \emph{trace tests} have been developed to give strong numerical evidence for completeness \cite{leykin2018trace}. Here, we propose to use Proposition \ref{prop:toricamplitude} as a specialized trace test for likelihood problems coming from toric models. Trace tests for other positive models can be obtained in a similar way. 

Suppose we have computed a set of $\ell$ approximate solutions to the likelihood equations \eqref{eq:likelihoodtoric} for a given set of positive integer data $u_i$. By Proposition \ref{prop:toricamplitude} we can evaluate the amplitude of ${\cal X}$ at $v_j = \sum_{i=1}^n u_i a_{i,j}$ in two different ways. The first is to compute the canonical form of the dilated Newton polytope of $q(x)$ and plug in $v_j = \sum_{i=1}^n u_i a_{i,j}$ in this rational function. The second is to compute the sum of $(x_1 \cdots x_d J_{g_1, \ldots, g_d})^{-1}$ evaluated at our $\ell$ critical points. These numbers should coincide, and if they do, this test gives strong numerical evidence that the ML degree of ${\cal X}$ is $\ell$.

\begin{example} \label{ex:toricsquare2}
We illustrate Proposition \ref{prop:toricamplitude} for the model ${\cal X}$ in Example \ref{ex:toricsquare1}. Consider $(c_0,c_1,c_2,c_3) = (15,2,8,23)$ and suppose  $(u_0,u_1,u_2,u_3) = (1,2,5,2)$ are the collected data in an experiment. Numerical approximations of the critical points of $L(x,y)$ are $\textup{Crit}(L) = \{ (0.42266, 2.08633), (-4.11469,-0.18234) \}$. These were obtained using the Julia package \texttt{HomotopyContinuation.jl} as outlined in \cite[Sec.~3]{sturmfels2021likelihood}.
The canonical form of the square $X_{\geq 0}$ is 
$$ \Omega(X_{\geq 0}) = \frac{100}{v_1 v_2 (10-v_1)(10-v_2)} ~ dv_1 \wedge dv_2 = h(v_1, v_2) ~ dv_1 \wedge dv_2,$$
where the numerator 100 of $h$ is the square of the sample size, i.e., $(u_0+u_1+u_2+u_3)^2$. Evaluating the amplitude $h$ at $(v_1,v_2) = (u_1+u_3, u_2+u_3) = (4,7)$ gives $25/126$. Evaluating the sum 
\[ \sum_{(x,y) \in \textup{Crit}(L)} \frac{1}{xy} \frac{1}{J_{g_1,g_2}(x,y)}, \]
where $g_1, g_2$ are the rational functions from \eqref{eq:phitoricsquare}, at our approximate critical points gives $0.1984126984126981$. This agrees with 25/126 up to 15 significant decimal digits, which is about the unit round-off in double precision floating point arithmetic. 
\end{example}

We conclude this section with an example that illustrates how to design alternative trace tests using canonical forms of polypols with non-linear boundaries in the plane. 

\begin{example} \label{ex:toricsquare3}
Consider again the model from Examples \ref{ex:toricsquare1} and \ref{ex:toricsquare2}. In Example \ref{ex:toricsquare2}, we wrote the rational function giving the canonical form of the square $X_{\geq 0}$ as a sum over the critical points of the likelihood function $L(x,y)$. Here we make the observation that a similar construction can be applied to other polypols, obtained as the image of the map $\phi$ in \eqref{eq:phitoricsquare} restricted to subgeometries of $\PP^2_{\geq 0}$. For instance, the image under $\phi$ of the standard simplex $T_{\ge 0}$ with vertices $(0,0),(1,0),(0,1)$ is the polypol $P_{\ge 0}$ with boundary curve $C = \{f_1f_2f_3 = 0 \}$ where $f_1 = v_1, f_2 = v_2$ and 
\begin{align*} f_3 &= s^{2}{v}_{2}{c}_{0}{c}_{1}^{2}-s\,{v}_{1}{v}_{2}{c}_{0}{c}_{1}^{2}+
     s^{2}{v}_{1}{c}_{0}{c}_{1}{c}_{2}-s\,{v}_{1}^{2}{c}_{0}{c}_{1}{c}_{2}+s^{2}{v}_{2}{c}_{0}{c}_{1}{c}_{2}-s\,{v}_{2}^{2}{c}_{0}{c}_{1}{c}_{2}-s^{3}{c}_{1}^{2}{c}_{2} \\
     &+2\,s^{2}{v}_{1}{c}_{1}^{2}{c}_{2}-s\,{v}_{1}^{2}{c}_{1}^{2}{c}_{2}+s^{2}{v}_{2}{c}_{1}^{2}{c}_{2}-s\,{v}_{1}{v}_{2}{c}_{1}^{2}{c}_{2}+s^{2}{v}_{1}{c}_{0}{c}_{2}^{2}
     -s\,{v}_{1}{v}_{2}{c}_{0}{c}_{2}^{2} \\ &-s^{3}{c}_{1}{c}_{2}^{2}+s^{2}{v}_{1}{c}_{1}{c}_{2}^{2}+2\,s^{2}{v}_{2}{c}_{1}{c}_{2}^{2}-s\,{v}_{1}{v}_{2}{c}_{1}{c}_{2}^{2} -s\,{v}_{2}^{2}{c}_{1}{c}_{2}^{2}+s\,{v}_{1}^{2}{c}_{0}^{2}{c}_{3} \\ & -2\,s\,{v}_{1}{v}_{2}{c}_{0}^{2}{c}_{3}+s\,{v}_{2}^{2}{c}_{0}^{2}{c}_{3}-s^{2}{v}_{1}{c}_{0}{c}_{1}{c}_{3} 
     +s\,{v}_{1}^{2}{c}_{0}{c}_{1}{c}_{3}+s^{2}{v}_{2}{c}_{0}{c}_{1}{c}_{3}-s\,{v}_{1}{v}_{2}{c}_{0}{c}_{1}{c}_{3} \\ &+s^{2}{v}_{1}{c}_{0}{c}_{2}{c}_{3}-s^{2}{v}_{2}{c}_{0}{c}_{2}{c}_{3}  
     -s\,{v}_{1}{v}_{2}{c}_{0}{c}_{2}{c}_{3}+s\,{v}_{2}^{2}{c}_{0}{c}_{2}{c}_{3}-s^{3}{c}_{1}{c}_{2}{c}_{3}+s^{2}{v}_{1}{c}_{1}{c}_{2}{c}_{3} \\ &+s^{2}{v}_{2}{c}_{1}{c}_{2}{c}_{3}  -s\,{v}_{1}{v}_{2}{c}_{1}{c}_{2}{c}_{3},
     \end{align*}
     where $s = u_0 + u_1 + u_2 + u_3$.
     With the numerical data from Example \ref{ex:toricsquare2}, this evaluates to $f_3 = -528000 + 383000v_1 - 111400v_2 - 153480v_1v_2 + 55930v_1^2 + 75670v_2^2$.
     Using Theorem \ref{thm:rationalPolypolsArePosGeometries}, we find that the canonical form of $P_{\ge 0}$ is $\Omega(P_{\ge 0}) = h_P(v_1, v_2) ~ dv_1 \wedge dv_2$ with 
     \[ h_P = \frac{528000 + 65800 v_1 + 263200 v_2}{v_1 v_2 (-528000 + 383000v_1 - 111400v_2 - 153480v_1v_2 + 55930v_1^2 + 75670v_2^2)}  . \]
     The numerator defines the adjoint line $A_P$. This is illustrated in Figure \ref{fig:toricexample}.
     Evaluating $h_P$ at $(v_1,v_2) = (u_1+u_3, u_2+u_3) = (4,7)$ gives $65840/370629$. Using the push-forward formula for $\Phi_*(\Omega(T_{\ge 0}))$ with $\Omega(T_{\ge 0}) = \frac{1}{xy(x+y-1)} dx \wedge dy$, we get the alternative expression
     \[ \sum_{(x,y) \in \textup{Crit}(L)} \frac{1}{xy(x+y-1)} \frac{1}{J_{g_1,g_2}(x,y)} \]
     for $h_P$. Plugging in our numerically obtained critical points gives the number $0.17764\ldots$, which has relative error
     \[ \frac{|0.17764395122885715 - 65840/370629|}{65840/370629} \approx 1.4 \cdot 10^{-15}. \]
     
     \begin{figure}
         \centering
         \input{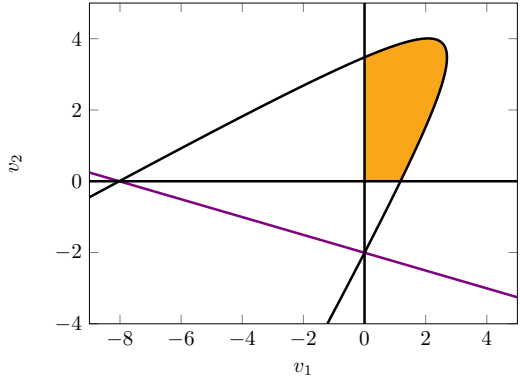}
         \caption{The polypol from Example \ref{ex:toricsquare3} (shaded in orange), its boundary curves (in black) and its adjoint line  (in purple). }
         \label{fig:toricexample}
     \end{figure}
\end{example}

\section{3D-Polypols}
\label{3D-Polypols}

Wachspress introduced three-dimensional polypols in $\R^3$, called \emph{polypoldrons} \cite[Chapter 7]{MR0426460}.
They are semialgebraic sets whose boundary $2$-dimensional facets are semialgebraic subsets of boundary surfaces, 
enclosing a simply connected set. The boundary facets intersect on boundary edges that are segments of space curves. The edges meet at the vertices of the semialgebraic set.

    A polypoldron is \emph{simple} if the number of facets (F), edges (E) and vertices (V) satisfy the Euler equation $$F-E+V=2,$$
    and \emph{well-set} if in addition, each vertex is a triple point, the edges and the boundary facets are nonsingular, and the boundary surfaces have no points in the interior of the polypoldron.
    
    A well-set polypoldron has adjoint surfaces. Such surfaces pass by the triple points that are not vertices and the curves of double points that do not contain edges.
    To make the adjoint surface unique, these restrictions are in general not enough, so Wachspress adds the condition to pass by a number of additional points on the boundary surface to assure a unique adjoint to the polypoldron \cite[ 7.5]{MR0426460}. 
    
    In this section, we concentrate on polypols with quadric boundary surfaces and an adjoint that is unique, without adding conditions beyond passing by certain singularities of the boundary surface.
 
The case of polytopes, i.e. with linear boundary components, was made precise and generalized to any dimension in \cite{kohn2019projective}; see Remark~\ref{polytopes}.  We will discuss and give infinite families of polypols with quadric boundary surfaces  in $\R^3$. 
 
 \subsection{Rational quadric 3D-polypols}
 
 We start with a definition that differs slightly  from  the one of polypoldrons.
It is a generalization of real polypols in the plane; see Definition \ref{def:quasi-regular}.   
 \begin{definition} A \emph{polypol} $P$ in $\R^3$ is a connected compact semialgebraic set such that:
 \begin{enumerate}
     \item The Zariski closure of its Euclidean boundary $\partial P$ is a surface which is a finite union of irreducible \emph{boundary surfaces}: $S=S_1\cup \ldots\cup S_k$.
 
\item Each boundary component $S_{i,\geq 0}:=\overline {(S_i\cap \partial P)^o}$, the Euclidean closure of the interior of $S_i\cap \partial P$,  is a connected semialgebraic set such that
the Zariski closure of its Euclidean boundary $\partial S_{i,\geq 0}$ is a finite union of irreducible \emph{edge curves}
$C_{i,j_1}\cup \ldots \cup C_{i,j_{d_i}}$,
where $ C_{i,{j_l}}\subset S_i\cap S_{j_l}.$

\item Each \emph{edge} $C_{(i,j),\geq 0}:=\overline{(C_{i,j}\cap \partial S_{i,\geq 0})^o}$ is a nonsingular segment in  $C_{i,j}$, whose  endpoints are  two \emph{vertices} that are nonsingular on $C_{i,j}$:
$$v_{i,j,l}\in C_{i,j}\cap C_{i,l}\quad {\rm and}\quad  v_{i,j,l'}\in C_{i,j}\cap C_{i,l'},$$

\item  On each boundary component $S_{i,\geq 0}$, the edges
$C_{(i,j_1),\geq 0}\cup \cdots \cup C_{(i,j_{d_i}),\geq 0}$ form a unique cycle of length $d_i$, with
vertices $v_{i,j_1,j_2},\ldots,v_{i,j_{d_i},j_1}$.
  \end{enumerate}

We denote by $S(P)$ the set of boundary surfaces, by $C(P)$  the set of edge curves, and by $V(P)$ the set of vertices of the polypol $P$. 
  We say that $P$ is {\em rational} if all the boundary surfaces $S_i$ and the edge curves $C\in C(P)$ are rational. Moreover, we say that $P$ is {\em simple} if the singular locus of $S$ is reduced, with no points of multiplicity more than three, and every triple point is ordinary.
  In the special case when all boundary surfaces $S_i$ are nonsingular quadric surfaces, and all edge curves are nonsingular conic sections, we call $P$  a {\em quadric polypol}.
 \end{definition}
 
 We consider two special cases of quadric polypols in $\R^3 $.  
 First, we define {\em polyhedral polypols}.
 For a convex simple polyhedron $H$ 
 with $k$ facets  labelled by $\{1,\dots,k\}$, we write $I_H\subset \{(i,j)|1\leq i<j\leq k\}$ for its edge set and  $J_H\subset \{(l,m,n)|1\leq l<m<n\leq k\}$ for its vertex set. 
 \begin{definition}
 Let $P\subset \R^3$ be a rational polypol with  boundary surface $S=S_1\cup \cdots\cup S_k$. Then $P$ is  {\em polyhedral} if there is a convex
 simple polyhedron $H$ with $k$ facets such that 
  $$C(P)=\{C_{i,j}\subset S_i\cap S_j|(i,j)\in I_H\}
 \quad \text{and} \quad
V(P)=\{ v_{l,m,n}\subset S_l\cap S_m\cap S_n|(l,m,n)\in J_H\}.$$
  \end{definition}

Let $P$ be a quadric polypol, i.e. all boundary surfaces are nonsingular quadric surfaces and all edge curves are nonsingular conic sections. Complex projective quadric surfaces are isomorphic to $\P^1\times \P^1$, so any curve on a nonsingular quadric has a bidegree. The nonsingular intersection of two quadric surfaces is an elliptic curve of bidegree $(2,2)$. An edge curve, i.e. a nonsingular conic section, $C_{i,j}\subset S_i\cap S_j$ is a component of bidegree $(1,1)$ of the singular intersection between two boundary surfaces $S_i$ and $S_j$.
On each boundary surface $S_i$, the edge curves form a cycle and  their union  defines an elliptic curve.
More precisely, after blowing up the intersections between edge curves other than the vertices, the strict transform of the edge curves forms a cycle of nonsingular rational curves.  By the adjunction formula \cite[Ch.~V, Prop.~1.5 and Ex.~1.3]{H}, such a cycle has  arithmetic genus~$1$. 

We will be concerned with adjoint surfaces to a quadric polypol.  Similarly to the polytope case, these are best understood in the complex projective setting, so from here on we shall consider the boundary surfaces and the edge curves as complex projective varieties in $\P^3.$

The singular locus of the boundary surface $S=\bigcup_iS_i\subset\P^3 $ of a quadric polypol $P$ is the union  $\bigcup_{ij}(S_i\cap S_j)$.
We define the \emph{residual points} $R(P)$ of $P$ to be the union of the curves in $\bigcup_{ij}(S_i\cap S_j)$ that are not edge curves and the triple points, where three boundary surfaces meet, that are not vertices of the polypol.
  
\begin{theorem}\label{quadjoint} Let $P$ be a simple quadric polyhedral polypol with boundary surface  $S=S_1\cup \cdots\cup S_k$.  If  
one of the $S_i$ 
has exactly three edge curves,
then there is  a unique  surface $A_P$, called the  \emph{adjoint surface} of $P$, of degree $2k-4$ that contains $R(P)$.
\end{theorem}
We postpone the proof to Section \ref{subsection:quadricpolyhedral}.

In the construction and analysis of quadric polypols, we observe three possible configurations a triple of edge conics may form.
\begin{lemma}\label{triplesofconics}
 Let $C_1,C_2,C_3$ be nonsingular conics in pairwise distinct planes $P_1,P_2,P_3$, respectively, and assume any two of the conics intersect in two distinct points.  Then the triple of conics forms one of three possible configurations:
    \begin{enumerate}
    \item (First kind). The intersection $P_1\cap P_2\cap P_3$ is a point $q$ that does not lie on any of the three conics.  
    Then the union of the three conics lies in a unique quadric surface.
    \item (Second kind). The intersection $P_1\cap P_2\cap P_3$ is a point $q$ that lies on all three conics.
    The union of this kind of three conics lies in a quadric surface if and only if their tangent lines at the common point $q$ are coplanar.
    \item (Third kind). The intersection $P_1\cap P_2\cap P_3$ is a line $L$, and the three pairs of intersection points coincide. The union of this kind of conics lie in a quadric surface if and only if their tangent lines at both intersection points are coplanar.
\end{enumerate}
\end{lemma}
\begin{proof}  The three possibilities of intersections are clear from the assumptions. To see when the union $C_1\cup C_2\cup C_3$ lies in a quadric surface, we observe that any two of the conics lie in a pencil of quadric surfaces.  In the first case, the general quadric in the pencil intersects the third conic only in its four intersection points with the first two conics. Therefore, one quadric in the pencil contains also the third.  In the two other cases, the same argument applies, when the tangent lines at the common intersection point $q$ are coplanar.  If they are not and the three tangent lines span $\P^3$, then $q$ would be a singular point on the quadric surface.   But a quadric singular at $q$ cannot contain three nonsingular conics through $q$ that lie in distinct planes.    
\end{proof}

\begin{remark}\label{rm:existence}
We show the existence of simple  polyhedral quadric polypols with a boundary component containing three edge curves.  We do this inductively, 
and  start by showing the existence of a tetrahedral quadric polypol.
Start with a 
triple of nonsingular conics $C_{1,2},C_{2,3},C_{1,3}$ of the second kind, with tangent lines at the common intersection point $v_4$ spanning $\P^3$. 
Each pair $C_{1,2}\cup C_{1,3}$, $C_{1,2}\cup C_{2,3}$ and $C_{2,3}\cup C_{1,3}$ lies in a pencil of quadric surfaces, of which the general one is nonsingular. So we 
let $S_1,S_2,S_3$ be nonsingular 
quadric surfaces that contain $C_{1,2}\cup C_{1,3}$, $C_{1,2}\cup C_{2,3}$ and $C_{2,3}\cup C_{1,3}$, respectively. 
Choose vertices $v_1,v_2,v_3$ on $C_{2,3}$, $C_{1,3}$ and $C_{1,2}$, respectively, each lying in only one of the conics. 

We have
$S_1\cap S_2=C_{1,2}\cup D_{1,2}$, where $D_{1,2}$ is another conic section.  Similarly,
$S_i\cap S_j=C_{i,j}\cup D_{i,j}$ for any $i\neq j$.  By the choice of quadrics $S_i$, we may assume that the conics $D_{i,j}$ are nonsingular.
The intersection $S_1\cap S_2\cap S_3$ is eight points. Each conic $C_{i,j}$ contains four of these, while the triple of conics, being of the second kind,  has four intersection points, one triple and three double. But then exactly one of the eight points does not lie on any of the conics $C_{i,j}$, hence must lie on all three conics $D_{i,j}$.
Therefore the triple of conics $D_{1,2},D_{1,3},D_{2,3}$ is of the second kind, just as $C_{1,2}, C_{1,3}, C_{2,3}$.

To find $S_4$ and conics $C_{1,4}, C_{2,4},C_{3,4}$ for a tetrahedral polypol, 
we consider the three quadric cones $Q_1, Q_2,Q_3$ with vertices at $v_1,v_2,v_3$  
containing the conics $D_{2,3}, D_{1,3}, D_{1,2}$, respectively.
The intersection $Q_1\cap Q_2\cap Q_3$
consists of eight points and includes the point $D_{2,3}\cap D_{1,3}\cap D_{1,2}$.  We choose one of the other seven points, $w$, which we may assume is not on any of the surfaces $S_1,S_2,S_3$.  Consider the three lines $l_1,l_2,l_3$ spanned by $w$ and the three vertices $v_1,v_2,v_3$, and  let $w_i=l_i\cap D_{j,k}$ for each triple of pairwise distinct indices $i,j,k$. Then  $v_i,w_i,v_j,w_j$ span a plane $\Pi_{i,j}$.  
Moreover, $v_i,w_i,v_j,w_j$ lie on $S_k$, in particular on a conic section $C_{k,4}=\Pi_{i,j}\cap S_k$. 
Thus, we get a triple of conics $C_{1,4}, C_{2,4}, C_{3,4}$ of the first kind.  Their union has arithmetic genus $4$ and lies in a unique quadric surface that we denote by $S_4$.
For a general choice of vertices $v_1,v_2,v_3$, the surface $S_4$ and the conic curves $C_{k,4}$ may be shown to be nonsingular.
Note that residual to the curves $C_{k,4}$ in $S_k\cap S_4$, we find conics $D_{k,4}$, and that the triple $D_{1,4},D_{2,4},D_{3,4}$  is of the third kind.

By this construction, the surface $S_1\cup\ldots\cup S_4$ together with edge curves $C_{i,j}$ and vertices $v_1,\dots,v_4$ form a tetrahedral quadric polypol. 
Figure \ref{fig:kuleflater} shows a tetrahedral polypol of spheres.  Since any real intersection of spheres is a point or a circle, any four spheres that have a real curve of intersection between any two form a tetrahedral polypol.   The four spheres all intersect along a common imaginary circle at infinity, so it is not a simple polypol.

\begin{figure}[ht]
\centering
\includegraphics[width=7.2cm]{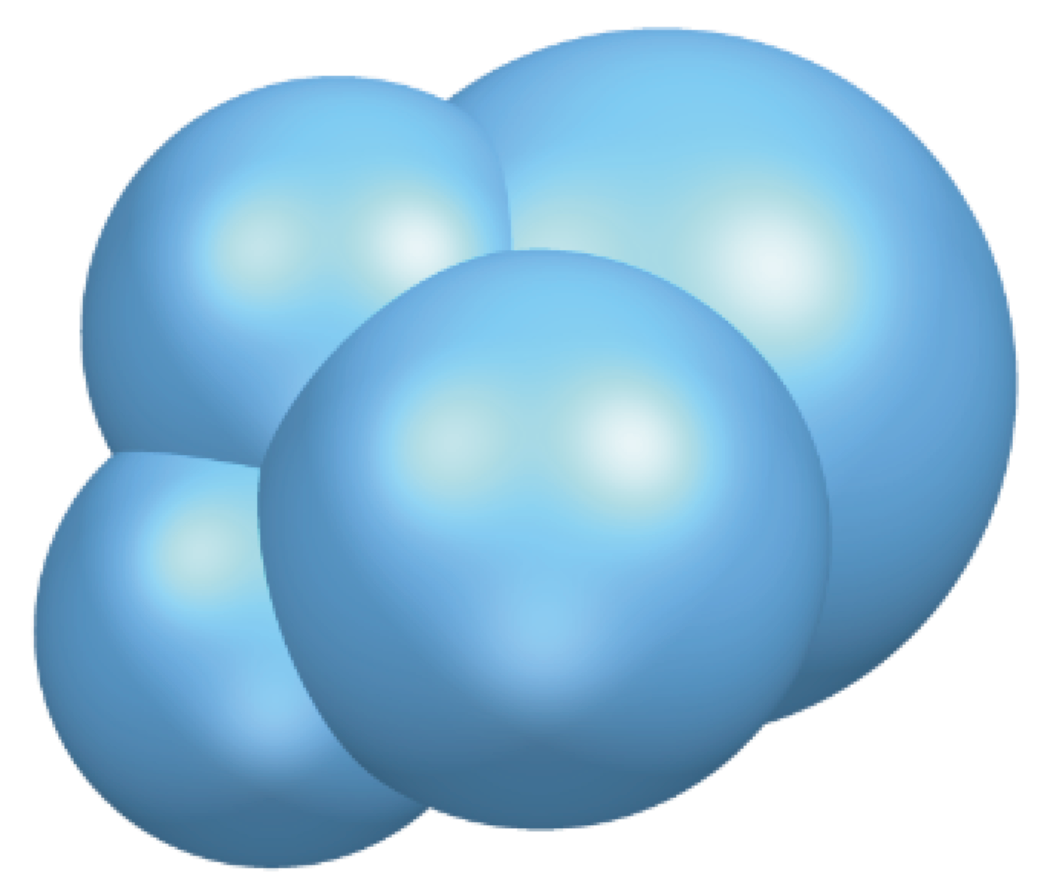} 
\caption{A tetrahedral polypol of spheres.}
\label{fig:kuleflater}
\end{figure}

Inductively, starting with any simple polyhedral quadric polypol $P$ and three quadric boundary surfaces with a common vertex $v$, we may  as above find a triple of conics of the first kind in these three quadrics that lie in a unique quadric $Q_v$, such that the four form a tetrahedral quadric polypol.   Replacing the common vertex $v$ of the three quadrics, with the new quadric surface  $Q_v$ with three conic edge curves and their three new vertices of intersection, we get a simple polyhedral quadric polypol $P'$.
\end{remark}

Not all simple quadric polypols are polyhedral, as the following example demonstrates.

\begin{example}
\label{ex:3cycle}
Let $P$ be a polypol formed by three ellipsoids with three edge conics
$C_{1,2}\cup C_{1,3}\cup C_{2,3}$ and two vertices, the common intersection of the three conics. So the edge conics form a triple of the third kind. This is illustrated in Figure \ref{fig:morton}.
  
$P$ is a simple quadric polypol, but it has only three boundary surfaces so it is not polyhedral.
The residual point set $R(P)$ consists of the three residual conics, a triple of conics of the first kind that lie in a unique quadric surface.  This surface is an adjoint of the polypol $P$.
\end{example}

\begin{figure}[H]
\centering
\includegraphics[width=6.2cm,trim={0cm 0cm 0cm 7.2cm},clip]{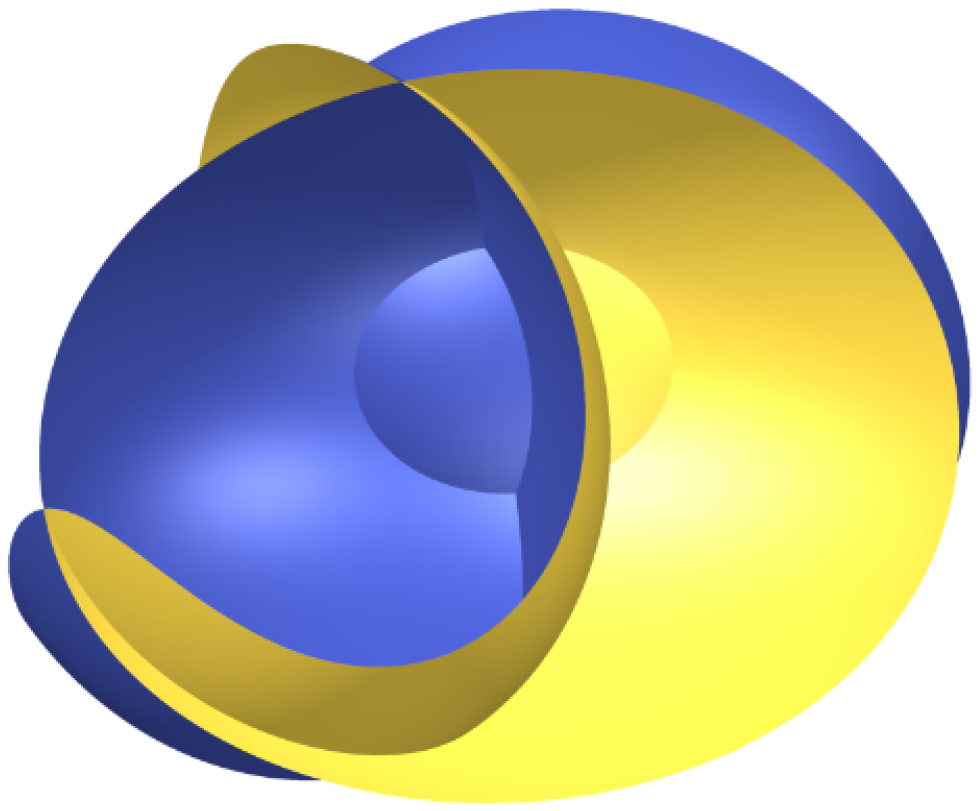}
\vskip -2cm
\caption{A cycle of three ellipsoids, whose quadric adjoint has no real point.}
\label{fig:morton}
\end{figure}
This example generalizes to a family of ``almost''  polypols  that we refer to as \emph{cycles of quadrics}.  They are certainly not polyhedral, but not quite polypols either: there are more than three edge curves through the vertices. Otherwise, they fit the definition of quadric polypols. 

\begin{definition} 
 Let $v_1, v_2\in \P^3$ be two points, and let  $S=S_1\cup \cdots \cup S_k,$ with $k\geq 3$, be a union of nonsingular quadric surfaces whose common intersection includes the two points $v_1,v_2$ when $k=3$, and equals the two points when $k>3$. Assume furthermore that there are $k$ \emph{edge  conics} $C_1,\dots,C_k$ passing through $v_1, v_2$ such that $C_i\subset S_i\cap S_{i+1}$ for  $i=1,\dots,k-1$ and $C_k \subset S_k\cap S_1$.
 Moreover, we assume that the singularities of $S$ are double points along transverse intersections of pairs of components $S_i$ and $S_j$, ordinary triple points between triples of components, plus the two ordinary $k$-tuple points $v_1,v_2$.  Then we call  $S$ a  {\em cycle of quadrics} and the two points $v_1,v_2$ its \emph{poles}.
\end{definition}

 For a cycle of quadrics $S$, we denote by
 $R(S)$ the singular locus of $S$ that is residual to the set of edge conics, i.e., $R(S)$ consists of the union of pairwise intersection curves:  
 $$R(S)=\bigcup_{i \neq j} \bigl((S_i\cap S_j) \setminus (C_1\cup\cdots\cup C_k) \bigr).$$

 \begin{theorem} Let $S=S_1\cup \cdots \cup S_k$ be a cycle of quadrics with $k\geq 3$.  There is a unique surface $A_S\subset \PP^3$ of degree $2k-4$ that contains the residual singularities $R(S)$ of $S$ and none of the edge conics $C_1, \dots,C_k$. 
 We call $A_S$ the \emph{adjoint surface} of $S$. When $k>3$, the surface $A_S$ passes through the poles of the cycle.
 \end{theorem}
 \begin{proof}
We set $R_k=B_1\cup \cdots \cup B_k\cup\bigcup_{ij}B_{i,j}$, where $B_i$ is residual to $C_i$ in $S_{i}\cap S_{i+1}$ and $B_{i,j}=S_i\cap S_j, j\neq i+1$,
and give a proof by induction. 
In addition to the claim of the theorem, we show that $h^1({\mathcal I}_{R_k}(2k-4))=0$.  
We start with  $k=3$, which is \Cref{ex:3cycle}.  
We give a proof here that is more amenable to induction, so we first
consider $R_3=B_1\cup B_2\cup B_3$. Let $S_3=\{q_3=0\}$ and consider the exact sequence of ideal sheaves
\[0\longrightarrow {\mathcal I}_{R'_3}\xrightarrow{\;\cdot q_3\;} {\mathcal I}_{R_3}(2)\xrightarrow{\;|_{S_3}\;} {\mathcal I}_{T_3,S_3}(2)\longrightarrow 0,\]
where  
${\mathcal I}_{T_3,S_3}$ is the ideal of $T_3=B_2\cup B_3$ on $S_3$.   Then $R'_3$ is the component $B_1$ of $R_3$ that is not in $S_3$.  
So the first sheaf in the sequence is the ideal sheaf of a conic in $\PP^3$, and therefore $h^0({\mathcal I}_{R'_3})=h^1({\mathcal I}_{R'_3})=0$. Clearly $T_3$ is a curve of bidegree $(2,2)$ on $S_3$, so $h^0({\mathcal I}_{T_3,S_3}(2))=1$ and $h^1({\mathcal I}_{T_3,S_3}(2))=0$. Therefore, also $h^0({\mathcal I}_{R_3}(2))=1$ and
$h^1({\mathcal I}_{R_3}(2))=0$.  

For $k>3$, we let $S_k=\{q_k=0\}$ and consider the exact sequence of ideal sheaves
\begin{equation}\label{eq:exactsequence1}
    0\to {\mathcal I}_{R'_k}(2k-6)\xrightarrow{\;\cdot q_k\;}  {\mathcal I}_{R_k}(2k-4)\xrightarrow{\;|_{S_k}\;} {\mathcal I}_{T_k,S_k}(2k-4)\to 0,
\end{equation}
where $T_k=R_k\cap S_k=B_{k-1}\cup B_k \cup \bigcup_{i=2}^{k-2} B_{i,k}$.  
Then $T_k$ is a curve of bidegree $(2k-4,2k-4)$ on $S_k$, so ${\mathcal I}_{T_k,S_k}(2k-4)={\mathcal O}_{S_k}$. Hence,  $h^0({\mathcal I}_{T_k,S_k}(2k-4))=h^0(S_k,{\mathcal O}_{S_k})=1$ and 
$h^1({\mathcal I}_{T_k,S_k}(2k-4))=h^1(S_k,{\mathcal O}_{S_k})=0$.
To complete the proof, we show below that  $h^0({\mathcal I}_{R'_k}(2k-6))=h^1({\mathcal I}_{R'_k}(2k-6))=0$.

The curve $R'_k$ is the part of $R_k$ that is not contained in $S_k$.  
Assuming for a moment that $$B_{k-1,1}=S_{k-1}\cap S_1=C'\cup B',$$ where $C'$ is a conic that contains $v_1$ and $v_2$ and $B'$ is a different plane conic, then $S_1,\dots,$ $S_{k-1}$ is a cycle of quadrics with edge conics $C_1,\dots,C_{k-2},C'$.
Thus, $$R'_k=R_{k-1}\cup C'=\overline{(R_{k-1}\setminus B')}\cup B_{k-1,1}.$$  We use the first equality in our induction argument below, and from the latter equality we see that 
the argument does not depend on the momentary assumption.

Consider the exact sequences of sheaves
  $$0\to {\mathcal I}_{R_{k-1}}(2k-6)\to {\mathcal O}_{\PP^3}(2k-6)\to {\mathcal O}_{R_{k-1}}(2k-6)\to 0,$$
  and 
  $$0\to {\mathcal I}_{R_{k-1}\cup C'}(2k-6)\to {\mathcal O}_{\PP^3}(2k-6)\to {\mathcal O}_{R_{k-1}\cup C'}(2k-6)\to 0.$$
  
  By induction,
  $h^0({\mathcal I}_{R_{k-1}}(2k-6))=1$ and
  $h^1({\mathcal I}_{R_{k-1}}(2k-6))=0$, so 
  \begin{align}
      \label{eq:residualVsP3}
     h^0({\mathcal O}_{R_{k-1}}(2k-6))=h^0({\mathcal O}_{\PP^3}(2k-6))-1.
  \end{align}
  Furthermore, the adjoint surface defined by the unique section in $H^0({\mathcal I}_{R_{k-1}}(2k-6))$ does not contain $C'$,
  so  $h^0({\mathcal I}_{R_{k-1}\cup C'}(2k-6))=0$.  Moreover, the intersection $R_{k-1}\cap S_{k-1}$ is a curve of bidegree $(2k-6,2k-6)$ on $S_{k-1}$, so $R_{k-1}\cap C'$ is a divisor of degree $4k-12$ on $C'$, and hence
  there is an equivalence of divisors $(R_{k-1}\cap C')\cong (2k-6)H\cap C'$ on $C'$, where $H$ is the class of a plane in $\PP^3$. 
Now consider the exact sequence of sheaves:
  $$0\to {\mathcal O}_{R_{k-1\cup C'}}(2k-6)\to {\mathcal O}_{R_{k-1}}(2k-6)\oplus {\mathcal O}_{C'}(2k-6)\to {\mathcal O}_{R_{k-1}\cap C'}(2k-6)\to 0.$$
  Since $(R_{k-1}\cap C')\cong (2k-6)H\cap C'$, the restriction $$H^0({\mathcal O}_{C'}(2k-6))\to H^0({\mathcal O}_{R_{k-1}\cap C'}(2k-6))$$ is surjective.
  Therefore, 
  \begin{align*}h^0({\mathcal O}_{R_{k-1}\cup C'}(2k-6))&=h^0({\mathcal O}_{R_{k-1}}(2k-6))+h^0({\mathcal O}_{C'}(2k-6))- (4k-12)\\
  &=h^0({\mathcal O}_{R_{k-1}}(2k-6))+1=h^0({\mathcal O}_{\PP^3}(2k-6)),
  \end{align*}
 where the latter equality was shown in \eqref{eq:residualVsP3}.
  This implies that $h^1({\mathcal I}_{R_{k-1}\cup C'}(2k-6))=0$.  We finally conclude from the cohomology of the sequence (\ref{eq:exactsequence1}) that 
  $h^0({\mathcal I}_{R_{k}}(2k-4))=1$ and $h^1({\mathcal I}_{R_{k}}(2k-4))=0$.
  
  Now we show that the adjoint surface $A_S$ does not contain the edge conics $C_i$.
  The adjoint $A_S$ intersects each quadric $S_i$ in a curve $A_i=A_S\cap S_i$ of bidegree $(2k-4,2k-4)$.   It contains every curve of intersection $S_i\cap S_j$ except for the two conics  $C_{i-1}$ and $C_i$ in the intersection with $S_{i-1}$ and $S_{i+1}$, respectively.  The union of these curves has 
bidegree $(2k-4,2k-4)$ on $S_i$, so we get an equality
$$ \textstyle A_i= ( \bigcup_{i,j}S_i\cap S_j) \setminus ( C_{i-1}\cup C_i ).$$

Finally, when $k>3$, there are $\binom{k}{2}-k=k(k-3)/2$  pairs of  the boundary quadrics that do not share an edge conic, so their intersection passes through the poles. 
Therefore also the adjoint surface $A_S$ passes through the poles.
 \end{proof}

Although the existence of a unique adjoint depends only on the complex boundary surface, the motivation both with the polypols and the positive geometries lies in the application to a real boundary.
Therefore, we provide examples of cycles of quadrics, where all boundary quadrics are real with a real nonsingular point.
\begin{example}
Let $v_1,v_2$ be real points, and assume $C_1,\ldots,C_k$ are real nonsingular ellipses through $v_1,v_2$.  Furthermore, assume that the $C_i$ lie in distinct planes and that the tangent lines at a vertex $v_i$ of any three of them span $\R^3$. Let $S_1$ be a general real quadric through $C_1\cup C_k$ and let $S_i$ be a general real quadric through $C_{i-1}\cup C_{i}$ for $i=2,\ldots,k$.  Then the surface $S=S_1\cup \cdots \cup S_k$ forms a real cycle of quadrics. 
\end{example}

\subsection{The adjoint of a quadric polyhedral polypol}\label{subsection:quadricpolyhedral}

Here we will prove Theorem \ref{quadjoint}.  We conjecture that the theorem holds for any simple quadric polypol.  In fact, the following lemma applies to any simple quadric polypol.
The idea of the lemma is that the existence of a curve $A_i$ that would be the intersection of the adjoint $A_P$ with the boundary surface $S_i$ follows from the definitions.
 
\begin{lemma}\label{lem:adjointonquadric}
Let $P$ be a simple quadric polypol with boundary surfaces $S_1,...,S_k$ and edge curves $C_{i,j}\subset S_i\cap S_j$, and let $C_i=C_{i,j_1}\cup \ldots \cup C_{i,j_{d_i}}$ be the union of edge curves on $S_i$. 
Then there is a unique curve $A_i$ of bidegree $(2k-4,2k-4)$ on $S_i$ with the following components:
\begin{enumerate}
    \item $(S_i\cap S_j)\setminus C_{i,j}$, whenever $C_{i,j}$ is an edge curve, 
    \item $S_i\cap S_j$ (for $i \neq j$), whenever this intersection does not contain an edge curve, 
    \item a residual curve $A_{i,r}$ of bidegree $(d_i-2,d_i-2)$ that intersects $C_i$ exactly in the singularities of $C_i$ that are not vertices of $P$.
\end{enumerate}
Moreover, if $Z_i$ is the set of singular points on $C_i$ that are not vertices of $P$, then we have that $h^1(S_i,{\mathcal I}_{Z_i}(d_i-2))=0$.

In  particular, any adjoint surface $A_P$ to $P$ contains no edge curve or vertex on $P$.
\end{lemma}

\begin{proof}
The restriction $A_i=A_P\cap S_i$ is a curve on $S_i$ of bidegree $(2k-4,2k-4)$ on $S_i$  that contains the intersections $S_i\cap S_j, i\not= j$, except for the edge curves on $S_i$.  Assume there are $d_i$ edge curves $C_{i,j}$ on $S_i$.  
Then  $A_i$ contains a component of $S_i\cap S_j$  for each $j$.
We denote by $A_{i,e}$ the union in $A_i$ of the $d_i$ conic sections $(S_i\cap S_j)\setminus C_{i,j}$ where $C_{i,j}$ is an edge curve. These conic sections all have bidegree $(1,1)$, so $A_{j,e}$ has bidegree $(d_i,d_i)$ on $S_i$  The curve $A_i$ also contains the intersections $S_i\cap S_j$ that do not contain an edge curve. These curves have bidegree $(2,2)$. We denote the union of them by 
$A_{i,s}$.  So $A_{i}=A_{i,e}\cup A_{i,s}\cup A_{i,r},$ where $A_{i,r}$ is a curve that contains no component in the singular locus of the boundary surface.
Since $A_{i,e}$ consists of $d_i$ conic sections, the curve $A_{i,s}$ has degree $4(k-1-d_i)$ and bidegree $(2(k-1-d_i),2(k-1-d_i))$.  So the curve $A_{i,r}$ has degree $2(2k-4)-2d_i-4(k-1-d_i)=2d_i-4$ and bidegree $(d_i-2,d_i-2)$ on $S_i$.   
The uniqueness of $A_i$ is now equivalent to the uniqueness of the curve $A_{i,r}$ with the given properties.

Similar to the plane curve case, see Proposition \ref{prop:uniqueAdjoint}, the curve $A_{i,r}$ is an adjoint to $C_i$ on $S_i$.  Let $\pi:\tilde S_i\to S_i$ be the blowup of $S_i$ in the singularities of $C_i$ that are not vertices.  The strict transform $\tilde C_i$  of $C_i$ on $\tilde S_i$ forms a cycle of rational curves.  It has arithmetic genus one, so by adjunction $\mathcal{O}_{\tilde C_i}(\tilde C_i+K_{\tilde S_i})$ has a unique section with no zeros.  Since $h^0(\tilde S_i,K_{\tilde S_i})=h^1(\tilde S_i, K_{\tilde S_i})=0$, it follows from cohomology of the exact sequence of sheaves 
$$0\to K_{\tilde S_i}\to \mathcal{O}_{\tilde S_i}(\tilde C_i+K_{\tilde S_i})\to \mathcal{O}_{\tilde C_i}(\tilde C_i+K_{\tilde S_i})\to 0 $$
that $$h^0(\tilde S_i, \mathcal{O}_{\tilde S_i}(\tilde C_i+K_{\tilde S_i}))=1\;\;{\rm and}\;\; h^1(\tilde S_i, \mathcal{O}_{\tilde S_i}(\tilde C_i+K_{\tilde S_i}))=0,$$ and that the unique section in $H^0(\tilde S_i,\mathcal{O}_{\tilde S_i}(\tilde C_i+K_{\tilde S_i})$ defines a curve $\tilde A_{i,r}$ that has no zeros on $\tilde C_i$.  Now, $K_{S_i}$ is a divisor of bidegree $(-2,-2)$ on $S_i$, so the image $\pi(\tilde A_{i,r})$ of $\tilde A_{i,r}$ on $S_i$ is a curve of bidegree $(d_i-2,d_i-2)$ that intersects $C_i$ only in the points blown up. 
The cohomology  $h^1(S_i,{\mathcal I}_{Z_i}(d_i-2))=0$ since it is equivalent to $h^1(\tilde S_i, \mathcal{O}_{\tilde S_i}(\tilde C_i+K_{\tilde S_i}))=0$.

Finally, we assume for contradiction that an adjoint surface $A_P$ to $P$ contains an edge curve $C_{i,j}$.  Then $A_i=S_i\cap A_P$ contains this curve $C_{i,j}$ as a component.
Let $A_{i,r}\subset A_i$ be the curve residual to the first two kinds of components in the intersections $S_i\cap S_j$ as in the lemma. Then $A_{i,r}$ is a curve of bidegree $(d_i-2,d_i-2)$ that contains $C_{i,j}$ and passes through all  singularities of $C_i$ that are not vertices.   
 Let $C_{i,l}$ be an edge curve on $S_i$ with a common vertex with $C_{i,j}$.
Then there are two vertices and  $2d_i-4$ other singularities of $C_i$ that lie on $C_{i,l}$.  The curve $A_{i,r}$ passes through one of the two vertices and therefore has at least $2d_i-3$ common points with $C_{i,l}$, while the intersection number $A_{i,r}\cdot C_{i,l}= 2d_i-4$.
Therefore, also $C_{i,l}$ is a component of $A_{i,r}$.  
Repeating this argument along the cycle of edge curves on $S_i$, we deduce 
$C_i\subset A_{i,r}$, which is absurd since the bidegree of $A_{i,r}$ is $(d_i-2,d_i-2)$, while it is $(d_i,d_i)$ for $C_i$.
\end{proof}
\begin{proof}[Proof of Theorem \ref{quadjoint} ]   
Let us start with the quadric tetrahedron. Let  $S=S_1\cup S_2\cup S_3\cup S_4$ be such that $S_i\cap S_j=C_{i,j}\cup B_{i,j}$ -- two nonsingular conic sections for each $i,j$ -- and such that any three $S_i$ intersect transversally in $8$ points. 
Let $v_{i,j,k}=C_{i,j}\cap C_{i,k}\cap C_{j,k}$.  So there are six conic curves $C_{ij}$ and four vertices $v_{i,j,k}$. 
The residual locus $R(P)$ of the quadric tetrahedron consists of conic curves $B_{i,j}$ and some singular points on the edge curves.
Similar to the edge curves $C_{i,j}$, there are six conic curves $B_{i,j}$.  There are $32$ triple points in $S$, eight for each triple of components.  Of these $24$ lie on the $B_{i,j}$, while the vertices $v_{i,j,k}$ and four more, $p_1,\dots,p_4$, lie only on the curves $C_{i,j}$, i.e., outside the $B_{i,j}$. 

We now argue that there is at least one quartic surface passing through the six conics $B_{i,j}$ and the four triple points $p_1,\dots,p_4$.  Notice that two conics $B_{i,j}$ and $B_{m,n}$ intersect in two points if they share an index, and do not intersect at all if they do not.  Let $l_{i,j}$ be the linear form defining the plane $P_{i,j}$ of $B_{i,j}$ and consider the exact sequence of ideal sheaves \small
$$0\to {\mathcal I}_{B_{1,2}\cup B_{1,3}\cup B_{1,4}}(2)\xrightarrow{\cdot l_{2,3}} {\mathcal I}_{B_{1,2}\cup B_{1,3}\cup B_{1,4}\cup B_{2,3}}(3)\xrightarrow{|_{P_{2,3}}} {\mathcal I}_{P_{2,3}\cap(B_{1,2}\cup B_{1,3}\cup B_{1,4}\cup B_{2,3})}(3)\to 0.$$ \normalsize
 Then $B_{1,2}\cup B_{1,3}\cup B_{1,4}$ is a curve of bidegree $(3,3)$ in $S_1$, so $h^0({\mathcal I}_{B_{1,2}\cup B_{1,3}\cup B_{1,4}}(2))=1$ and $h^1({\mathcal I}_{B_{1,2}\cup B_{1,3}\cup B_{1,4}}(2))=0$. The intersection $P_{2,3}\cap(B_{1,2}\cup B_{1,3}\cup B_{1,4}\cup B_{2,3})$ is the union of the conic $B_{2,3}$ and the two points of intersection $P_{2,3}\cap B_{1,4}$, so 
 $$h^0({\mathcal I}_{P_{2,3}\cap(B_{1,2}\cup B_{1,3}\cup B_{1,4}\cup B_{2,3})}(3))=1\;{\rm and}\; h^1({\mathcal I}_{P_{2,3}\cap(B_{1,2}\cup B_{1,3}\cup B_{1,4}\cup B_{2,3})}(3))=0.$$  Therefore, $$h^0({\mathcal I}_{B_{1,2}\cup B_{1,3}\cup B_{1,4}\cup B_{2,3}}(3))=2\;{\rm and}\; h^1({\mathcal I}_{B_{1,2}\cup B_{1,3}\cup B_{1,4}\cup B_{2,3}}(3))=0.$$  Similarly, by the exact sequence of ideal sheaves   
  $$0\to {\mathcal I}_{B_{1,2}\cup B_{1,3}\cup B_{1,4}\cup B_{2,3}}(3)\xrightarrow{\;\cdot l_{2,4}\;} {\mathcal I}_{B'}(4)\xrightarrow{\;|_{P_{2,4}}\;} {\mathcal I}_{P_{2,4}\cap B'}(4)\to 0,$$
  where $B'=B_{1,2}\cup B_{1,3}\cup B_{1,4}\cup B_{2,3}\cup B_{2,4}$, 
   we conclude that $$h^0({\mathcal I}_{B'}(4))=6\;{\rm and}\; h^1({\mathcal I}_{B'}(4))=0.$$  
   Let $B=\bigcup_{1\leq i<j\leq 4}B_{i,j}=B'\cup B_{3,4}.$
   The conic  $B_{3,4}$ intersects  the $5$ conics in $B'$ in altogether eight points, so $h^0({\mathcal I}_B(4))\geq 5$ and $h^1({\mathcal I}_B(4))=h^0({\mathcal I}_B(4))-5$.  The points $p_1,\dots,p_4$ impose at most independent conditions on these sections, so there is at least one quartic surface containing the six conics $B_{i,j}$ and the four triple points $p_1,\dots,p_4$, i.e., we have $h^0({\mathcal I}_{R(P)}(4))\geq 1$.
   Now, 
   $R(P)=B\cup \{p_1,..,p_4\}$, so $\chi({\mathcal I}_{R(P)}(4))=
   \chi({\mathcal I}_{B}(4))-4$ while $h^i({\mathcal I}_{R(P)}(4))=
   h^i({\mathcal I}_{B}(4))$ for $i>1$.
   Therefore, $h^1({\mathcal I}_{R(P)}(4))=h^0({\mathcal I}_{R(P)}(4))-1$.
   
 To show that $h^0({\mathcal I}_{R(P)}(4))= 1$, we assume for contradiction that there is a pencil of quartic surfaces containing the $B_{i,j}$ and the points $p_1,\ldots,p_4$. Let $A$ be one of them.  On each surface $S_i$, the intersection $A\cap S_i$ is a $(4,4)$-curve that contains the three $(1,1)$-curves $B_{i,j}$ and the three triple points on $S_i$ that are not on the $B_{i,j}$.  As the three latter points cannot be collinear, the intersection  $A\cap S_i$ is independent of the choice of $A$. 
Therefore, in the pencil of surfaces $A$, there is one that contains the surface $S_i$.  But then it contains, on each of the other surfaces $S_j$, the four $(1,1)$-curves $B_{j,k}$ (for $k\neq j$) and $C_{ij}$ in addition to a triple point of intersection on the three surfaces $S_k$ (for $k\neq i$), which is impossible.  
So the surface $A$ is unique, i.e.  $h^0({\mathcal I}_{R(P)}(4))=1$ and $h^1({\mathcal I}_{R(P)}(4))=0$.  
Notice that the curve $A_i=A\cap S_i$ intersects the $C_{i,j}$ only in triple points different from the vertices, as in Lemma \ref{lem:adjointonquadric}.
 
In the general case, we argue  by induction on the number of components $k$, and want to prove that $h^0({\mathcal I}_{R(P)}(2k-4))=1$ and $h^1({\mathcal I}_{R(P)}(2k-4))=0$.   
 We assume that $S_k$ contains three edge curves and consider removing the surface $S_k$  from the polypol $P$. 
  
Removing the facet in the $k$-th hyperplane  of the convex polyhedron $H$ that corresponds to the polyhedral polypol $P$, the remaining facets naturally extend to facets of a polyhedron $H'$.  Removing the surface $S_k$, the remaining surfaces define a  polyhedral polypol $Q$ corresponding to $H'$:  

The three vertices in $S_k$ lie on three edges that are not in $S_k$.  These three edge curves intersect in one (or two) points.  Either one may be chosen as vertex in $Q$, we call it $v_Q$. 
Thus, the edge curves of $Q$ are the edge curves in $P$ that do not lie in $S_k$, 
and the vertices of $Q$ are vertices of $P$ that are not in $S_k$ and the new vertex $v_Q$.
 
 Consider now the following exact sequence of sheaves,
 where $S_k = \{ q_k=0 \}$:
  $$0\to {\mathcal I}_{R(Q)\cup {v_Q}}(2k-6) \xrightarrow{\;\cdot q_k\;} {\mathcal I}_{R(P)}(2k-4) \xrightarrow{\;|_{S_k}\;} {\mathcal I}_{R(P)\cap S_k}(2k-4)\to 0.$$
Using the notation from \Cref{lem:adjointonquadric}, the intersection  $R(P)\cap S_k$ is the union of a curve of bidegree 
$(2k-2-d_k,2k-2-d_k)$ and the set $Z_k$ of singularities on $C_k$ that are not vertices of $P$. So ${\mathcal I}_{R(P)\cap S_k}(2k-4)={\mathcal I}_{Z,S_k}(d_k-2)$, and hence, by Lemma \ref{lem:adjointonquadric}, 
$$h^0({\mathcal I}_{R(P)\cap S_k}(2k-4))=1\;{\rm and}\;h^1({\mathcal I}_{R(P)\cap S_k}(2k-4))= h^1({\mathcal I}_{Z_k,S_k}(d_k-2))=0.$$
By the induction hypothesis, we assume $$h^0({\mathcal I}_{R(Q)}(2k-6))=1 \;\textrm{ and }\;  h^1({\mathcal I}_{R(Q)}(2k-6))=0.$$ 
 
Then, as in Lemma \ref{lem:adjointonquadric}, the adjoint surface $A_Q$ does not contain any vertex, including $v_Q$. 
So only the zero section in ${\mathcal I}_{R(Q)}(2k-6)$ vanishes at $v_Q$. 
It follows that $$h^0({\mathcal I}_{R(Q)\cup v_Q}(2k-6))=h^1({\mathcal I}_{R(Q)\cup v_Q}(2k-6))=0,$$  and therefore  

    $$h^0({\mathcal I}_{R(P)}(2k-4))=1\; {\rm and}\; h^1({\mathcal I}_{R(P)}(2k-4))=0.$$ 
   The adjoint surface $A_P$ is defined by the unique section of ${\mathcal I}_{R(P)}(2k-4)$.  
  \end{proof}

\section{Outlook}
In the previous sections we have studied polypols and positive geometries from the point of view of algebraic geometry. These efforts have led to several conclusive results, but they also give rise to a series of natural follow-up questions. With the hope of inspiring future research directions, we present some of these questions below. 

\paragraph{Wachspress's conjecture.}
The conjecture by E. Wachspress that the adjoint curve of a regular rational polypol in the plane does not pass through the interior of the polypol (\Cref{conj:Wachs}) is still widely open. We attempted to prove the first non-trivial case for regular polycons bounded by three ellipses, but -- as explained in \Cref{sec:appendix} -- a formal proof is still missing for 11 such polycons.

\paragraph{Hyperbolicity.}
It was shown in Section \ref{Outside Adjoints} that in the case of plane convex polygons the adjoint curve is always hyperbolic, while for polytopes in higher dimensions as well as polycons formed by three ellipses it can  both be hyperbolic and non-hyperbolic. A natural question is to find some sufficient conditions on the polypols/polypoldrons which guarantee the hyperbolicity of the related adjoint hypersurfaces.

\paragraph{Singular adjoints.}
A more specific (complex) problem related to the previous question is as follows. Consider the (closure of the) space  of pairs consisting of a triple of generic conics and a $9$-tuple of their $12$ points of pairwise intersection where we have arbitrarily removed one point from each pairwise intersection (containing $4$ points). For a general such $9$-tuple of points, the unique cubic curve passing through them is nonsingular. 
An interesting question is to find a description of $9$-tuples coming from triples of conics for which the respective adjoint curve is singular, i.e., to describe the discriminant in this space.
One can easily observe that this discriminant has several components:
The cubic adjoint curve is singular, for instance, if either all three conics pass through the same point or if two of the conics are tangent to each other and this point is included in the $9$-tuple of residual points.

\paragraph{Adjoint maps.}
In Theorem \ref{th: finite adjoint}, we have identified all types of planar polypols with a finite adjoint map. However, finding the degrees of these maps is a nice computational challenge, as well as an interesting theoretical question. In particular, can one provide a theoretical argument for Conjecture \ref{conj:heptagon-adjoints}? How many of the 864 heptagons in this conjecture can be real? How many can be convex?

By Proposition \ref{lem:(1,3,1,3)}, the adjoint map $\alpha_{1,3,1,3}$ is not dominant. A natural question to ask is which quintic curves are the adjoint of a $(1,3,1,3)$-polypol?
 
\paragraph{Limits of canonical forms.}
This topic is mentioned in Ch.~10 of \cite{arkani2017positive}. We state it here in the simplest possible form. Consider a convex real-algebraic lamina $\Upsilon\subset \R^2$, i.e., a convex domain whose boundary is an oval of a real algebraic curve. Consider a sequence $\{P_n\}$ of convex inscribed polygons which exhaust $\Upsilon$ when $n\to \infty.$ Let $\Omega_n:=\Omega(P_n)$ be the canonical form of $P_n$. 
What is the limit 
$\lim_{n\to\infty} \Omega_n$? In particular, does it exist and is it independent of the sequence $\{P_n\}$? What is its description in terms of $\Upsilon$?

Observe that the limit $\lim_{n\to\infty} \Omega_n$ (if it exists) will typically be non-rational. Furthermore, 
since $\Omega_n$ can be interpreted as the scaled moment-generated function of the dual polygon $P_n^*$ \cite[Section 7.4.1]{arkani2017positive}, one can hope that 
$\lim_{n\to\infty} \Omega_n$ exists and coincides with the scaled moment-generating function for the convex domain $\Upsilon^*$ dual to $\Upsilon$.

\paragraph{Pushforward conjecture.} While the pushforward relation phrased as a heuristic (Heuristic \ref{heur:pushfwd}) has been settled in the case of the toric moment map (\cite[Thm.~7.12]{arkani2017positive}) and in the one-dimensional case (Proposition \ref{prop:dim1}), a general argument is still missing. A sketch of such an argument, based on the commutation of push-forward and taking residues, is given in \cite{arkani2017positive}. However, the presented results are non-conclusive. Proposition \ref{prop:dim1} suggests that requiring $\phi$ to induce a morphism of positive geometries (Definition \ref{def:morphism}) is too restrictive. It is an important remaining challenge to identify the necessary assumptions on $\phi$ for the push-forward relation to hold and to give a rigorous proof. 

\paragraph{Higher dimensional polypols.} Our discussion of polypols of dimension at least three is limited to polypols in $\R^3$ with quadric surfaces as boundary components.
Clearly, generalizations to rational boundary components of higher degree and to higher dimensions would be interesting; in particular, to find  the polypols with a unique adjoint. 

Fitting these polypols into the theory of positive geometries requires a discussion of canonical forms.   A central question is which real $n$-dimensional polypols allow a canonical form, and how positive geometries may be obtained from unions and differences of real polypols as in \Cref{ssec:additivity}. The natural candidate for a canonical form is a rational $n$-form with poles along the boundary hypersurface and zeros along an adjoint hypersurface to that hypersurface. 
These questions are, as far as we know, open already for real quadric polypols.

\small
\section*{Acknowledgements}
The authors want to thank C.~Eur, T.~Lam, D. Pasechnik, C. Riener, and B. Schröter for numerous discussions of the topic and their interest in our project. We are sincerely grateful to G.~M.~Polotovskiy at the Nizhny Novgorod Campus of the Higher School of Economics  for his help with Section~\ref{sec:Pol} which is essentially just a polished translation of his preliminary note on non-realizability of certain configurations of three ovals sent to B.~Shapiro.
We thank the referee for suggestions that improved various aspects of this article.
M.-{\c S}. Sorea is thankful to Antonio Lerario for the supportive working environment at SISSA (Trieste, Italy) during her postdoc.
Finally, we acknowledge that the software packages GeoGebra, Inkscape, LibreOffice Draw, Microsoft Paint, Plots.jl, Contour.jl and Surfer have been indispensable for the presentation of our results. 

K.~Kohn and F.~Rydell were partially supported by the Knut and Alice Wallenberg Foundation within their WASP (Wallenberg AI, Autonomous Systems and Software Program) AI/Math initiative.
M.-{\c S}. Sorea was supported by the project ``Mathematical Methods and Models for
Biomedical Applications'' financed by National Recovery and Resilience Plan PNRR-III-C9-2022-I8.

\normalsize

\bibliographystyle{alpha}
\bibliography{lit}

\section*{Authors' addresses}

\small

\noindent Kathl\'en Kohn, KTH Royal Institute of Technology,
\hfill  {\tt kathlen@kth.se}

\noindent Ragni Piene, University of Oslo,
\hfill  {\tt ragnip@math.uio.no}

\noindent Kristian Ranestad, University of Oslo,
\hfill  {\tt ranestad@math.uio.no}

\noindent Felix Rydell, KTH Royal Institute of Technology,
\hfill  {\tt felixry@kth.se}

\noindent Boris Shapiro, Stockholm University,
\hfill {\tt shapiro@math.su.se}

\noindent Rainer Sinn,  Universit\"at Leipzig,
\hfill  {\tt rainer.sinn@uni-leipzig.de}

\noindent Miruna-Stefana Sorea, Lucian Blaga University of Sibiu,
\hfill  {\tt mirunastefana.sorea@ulbsibiu.ro}

\noindent Simon Telen,  MPI MiS Leipzig,
\hfill  {\tt simon.telen@mis.mpg.de}

\normalsize

\newpage
\appendix
\section{Wachspress's conjecture for three ellipses}\label{sec:appendix}
In what follows, we first create the catalog of the $44$ admissible configurations of three ellipses described in Theorem \ref{thm:wachspressellipses}; see Sections \ref{sec:makeCatalog}--\ref{sec:catalog}.
Next, we provide an argument that proves Wachspress's  conjecture for $28$ of these configurations (and all polycons existing in these configurations) by showing that the adjoint curve has to be hyperbolic with  the oval lying strictly outside of the polycon; see \Cref{prop:polyconHyperbolic}.
This leaves us with $16$ problematic configurations. Finally, we provide a more intricate argument for $5$ of these in \Cref{prop:problematicPolycons}, and compute the adjoint of example instances for the problematic polycons in the remaining $11$ configurations; see Figure \ref{fig:problematicAdjoint}.

\subsection{Creating the catalog}
\label{sec:makeCatalog}

Our first aim is to classify all topologically distinct configurations of three ellipses in $\R^2$ that intersect transversally such that all three of them do not intersect at the same real point and, additionally, each pair intersects at least twice in $\R^2$. By topologically equivalent configurations we mean configurations which can be obtained from one another by a diffeomorphism of $\R^2$. 

We distinguish such configurations of three ellipses by their \emph{intersection type}:  (222)~-- all three pairs intersect exactly twice in $\R^2$, (224) -- precisely one pair of two ellipses intersects in four real points, (244) -- precisely one pair of two ellipses intersects only twice in $\R^2$, and (444) -- all pairs of ellipses intersect  in four real points. 

\begin{remark}
In real algebraic geometry, the latter case is usually referred to   as the \emph{$M$-case}, see \cite{MR1384316}. 
That article contains, in particular, statistical information about all possible topological configurations of  three real nonsingular conics transversally  intersecting  in $\P^2(\R)$ with intersection type (444), i.e., in the $M$-case. According to the row 7 of Table 1 of that paper, there exist 105 such configurations. 
Notice that there are more  projective non-equivalent configurations  of three conics than affine configurations of three ellipses.   
In particular, a configuration of two ellipses intersecting each other in four real points is projectively non-equivalent to a configuration consisting of an ellipse and a hyperbola intersecting each other in four real points.
\end{remark} 

Below we provide a method/algorithm, consisting of  four  steps,  for finding all possible non-equivalent configurations of three ellipses in $\R^2$ as described in \Cref{thm:wachspressellipses}. We note that this method turned out to be  sufficient for our purposes, but it might need some extra steps to solve a similar problem for more than three ellipses. 
We present the final outcome of our method in \Cref{sec:catalog}.

\begin{enumerate}
\item[Step 0.] \textbf{Subdivision according to intersection types} 

Fix an intersection type: (222), (224), (244), or (444).
\item[Step 1.] \textbf{Obtaining a preliminary excessive catalog}

During this step, consisting of three substeps, we create all possible topological configurations  of two ellipses and an oval (that is not necessarily convex) of the intersection type chosen in Step 0.
By an oval we mean a simple closed curve in $\R^2$. (Notice that during this step we enumerate configurations of two ellipses and an oval, identifying those which can be obtained from one another by a continuous deformations and  global symmetries. As a result  we enumerate the latter configurations up to global diffeomorphisms of $\R^2$.) 
    \item[Step 1.A] Draw two ellipses, one vertical  and one horizontal, that intersect each other in a way consistent with the chosen intersection type.
    The horizontal  ellipse cuts  the vertical one  either into two or four arcs.
    
    Observe that the intersection type determines whether the third oval can intersect the vertical ellipse in two or four real points. We list all topologically distinct cases of how this oval can intersect the vertical ellipse locally, i.e., we choose two resp. four short segments that meet the vertical ellipse in all topologically distinct ways (up to symmetries).

     To illustrate Step 1.A, let us provide more details in the (244)-case, see  Figure~\ref{fig:3conics}. We start by drawing the vertical and the horizontal ellipses intersecting each other  in four real points. Then we find all ways the third oval can intersect the vertical ellipse locally in two points. In other words, we are drawing two short segments of the oval where it intersects the vertical ellipse. Up to symmetry, we get five different cases of possible local intersections  shown in  Figure~\ref{fig:3conics}.
     \begin{figure}[htb]
    \centering
    \includegraphics[scale=0.6]{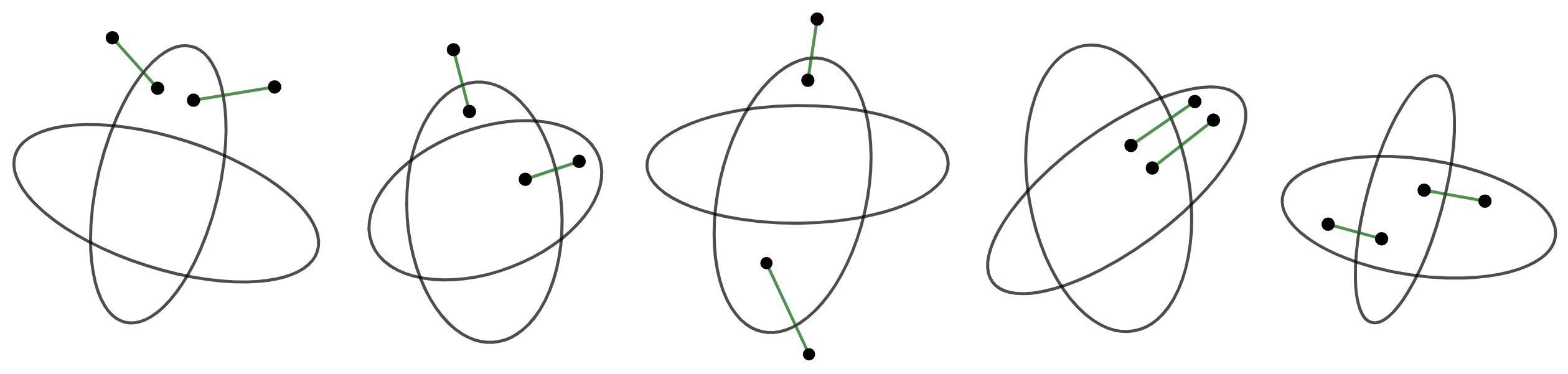}
    \caption{Five possible topologically distinct cases of how the third oval can intersect the vertical ellipse locally; these local intersections are shown by two green  segments.}
    \label{fig:3conics}
\end{figure}

     \begin{figure}[htb]
    \centering
    \includegraphics[scale=1.25]{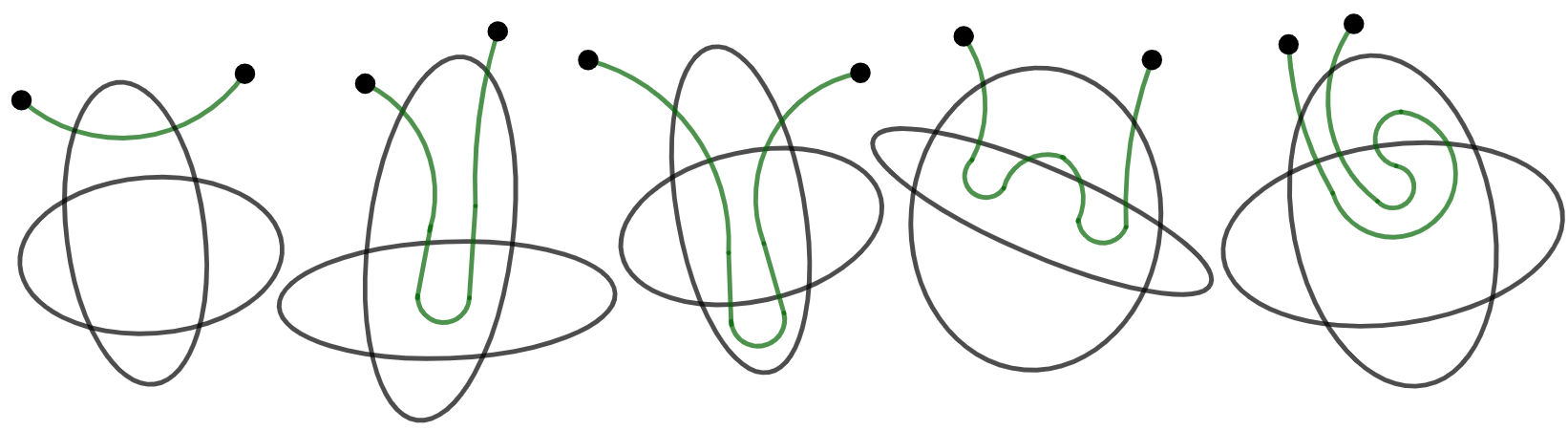}
    \caption{The five subcases of the leftmost case in Figure \ref{fig:3conics}.}
    \label{fig:3conicsSubc}
\end{figure}
    \item[Step 1.B] Subdivide each case obtained in Step 1.A  further by connecting  the short segments from Step 1.A  in all admissible ways (consistent with the chosen intersection type) inside the union of the two ellipses. This determines how the third oval intersects  the horizontal ellipse in the interior of the vertical ellipse. 
        For the leftmost case in  Figure~\ref{fig:3conics}, all such connections are shown  in  Figure~\ref{fig:3conicsSubc}.
    \item[Step 1.C] Finally, for each subcase obtained in Step 1.B, complete the curve in all possible  ways to get  topologically distinct ovals (recall that an oval here is a simple closed curve) that intersect the two ellipses we started with according to the chosen intersection type. 
   For the leftmost subcase in Figure \ref{fig:3conicsSubc}, all its possible completions with intersection type (244) are shown in Figure \ref{fig:3conics_bis}.
    \begin{figure}[htb]
    \centering
    \includegraphics[scale=0.3]{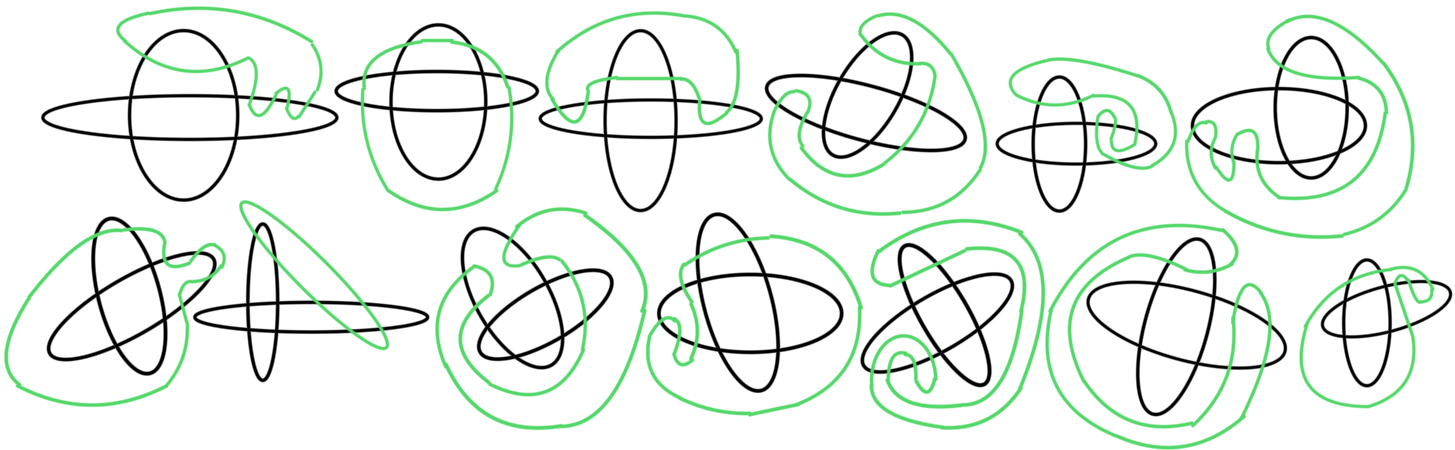}
    \caption{Up to simple symmetries such as reflection, precisely these thirteen ovals are obtained in Step 1.C from the leftmost subcase in Figure \ref{fig:3conicsSubc}. We number them from left to right as 1--6 in the first row and 7--13 in the second row.} \label{fig:3conics_bis}
\end{figure}
    \item[Step 2.] \textbf{Reduction}
    
    In this step, we determine which configurations found in Step 1  cannot be realized by three convex ovals, using the following two arguments\footnote{In our setting, all configurations that cannot be excluded in this way are realizable as convex configurations. We are not aware of a similar criterion for the case of more than three ovals and/or higher degree curves.}. 
    
    First, the intersection of two ellipse interiors must be convex. For instance, this excludes configurations 1, 3, 4, 6, 9, 10 and 13 from Figure \ref{fig:3conics_bis}.
    
    Second, a line intersects a convex oval in at most two points.
    This excludes configurations 5, 11 and 12 in Figure \ref{fig:3conics_bis}. To see this in configuration 12, consider the 4-sided convex intersection of the two ellipses and the line $L$ that passes through its lower-right and upper-left vertex. It is clear that $L$ must intersect the green oval in at least four points. In configuration 5, we reach the same conclusion if we consider the line spanned by the two upper left points of the intersection between the horizontal ellipse and the green oval.
    
    It not hard to see that all three out of the eleven configurations shown in Figure \ref{fig:3conics_bis} that are left can be drawn with convex ovals. These configurations, 2, 7, and 8, are realized with ellipses in Figures \ref{fig:244_321} (middle and right) and \ref{fig:244_433}.

    \item[Step 3.] \textbf{Identification}
    
    During this step, we decide which configurations of three ovals found in Step 2 are topologically equivalent and which are distinct.  
    To this end, we define the \textit{outer-arc type} of a configuration  as follows. Consider the complement of the interiors of the two ellipses and the oval. Count the boundary arcs of this complement and how many of them belong to each of the three curves. 
    Finally, order these three numbers decreasingly.
    Such a triple is called the \emph{outer-arc type} of a configuration; see Figures \ref{fig:222}--\ref{fig:444}. 
    We observe that two  configurations with distinct outer-arc types cannot be topologically equivalent.
    
    It is left to decide which configurations of the same outer-arc type are topologically equivalent.
 
    If the numbers of polycons inside two configurations differ, then the configurations are different. 
    Similarly, we can count regions with more than three sides to distinguish between distinct configurations.
    Finally, we identify that two configurations are equivalent by considering all permutations of the three ovals in one of the configurations.  
    \item[Step 4.] \textbf{Realization} 
    
    Represent the topological configurations of three ovals remaining after Step 3 by three ellipses (e.g., by using  \texttt{Geogebra} \cite{hohenwarter2002geogebra}) or show their non-realizability by ellipses using a method suggested by S.~Orevkov in \cite{MR1679799} that we describe in \Cref{sec:Pol}.
\end{enumerate}

\subsection{Non-realizability of certain configurations of ovals by ellipses}\label{sec:Pol}

In Step 4 of the method described above, we were able to realize all configurations with three ellipses, except the 5 configurations of intersection type $(444)$ in Figure \ref{Fig:exc}.

\begin{figure}[htb]
     \centering
         \begin{subfigure}[b]{0.19\textwidth}
         \centering
         \includegraphics[width=\textwidth]{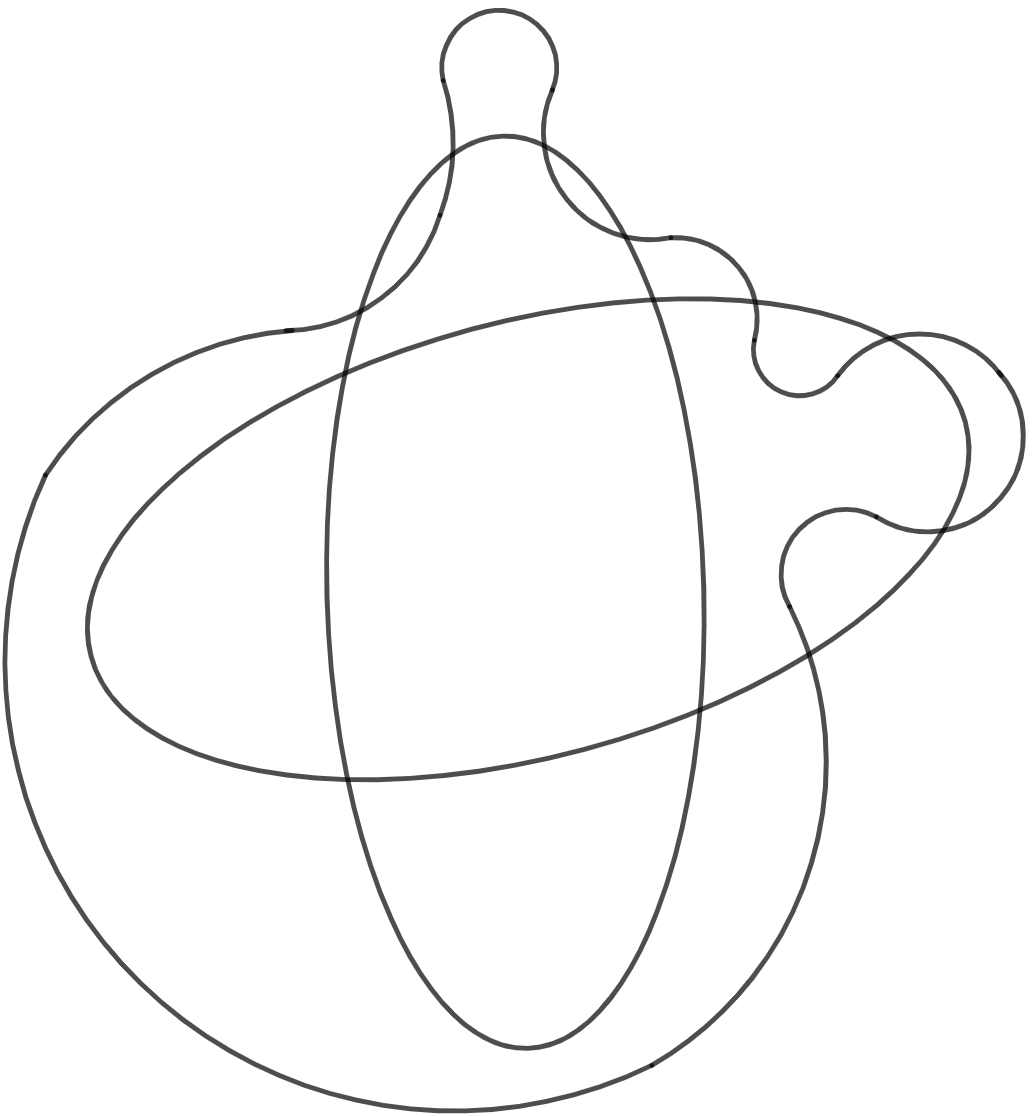}
         \caption{$544445112222$}
         \label{fig:exc-1}
     \end{subfigure}
     \hfill
         \begin{subfigure}[b]{0.20\textwidth}
         \centering
         \includegraphics[width=\textwidth]{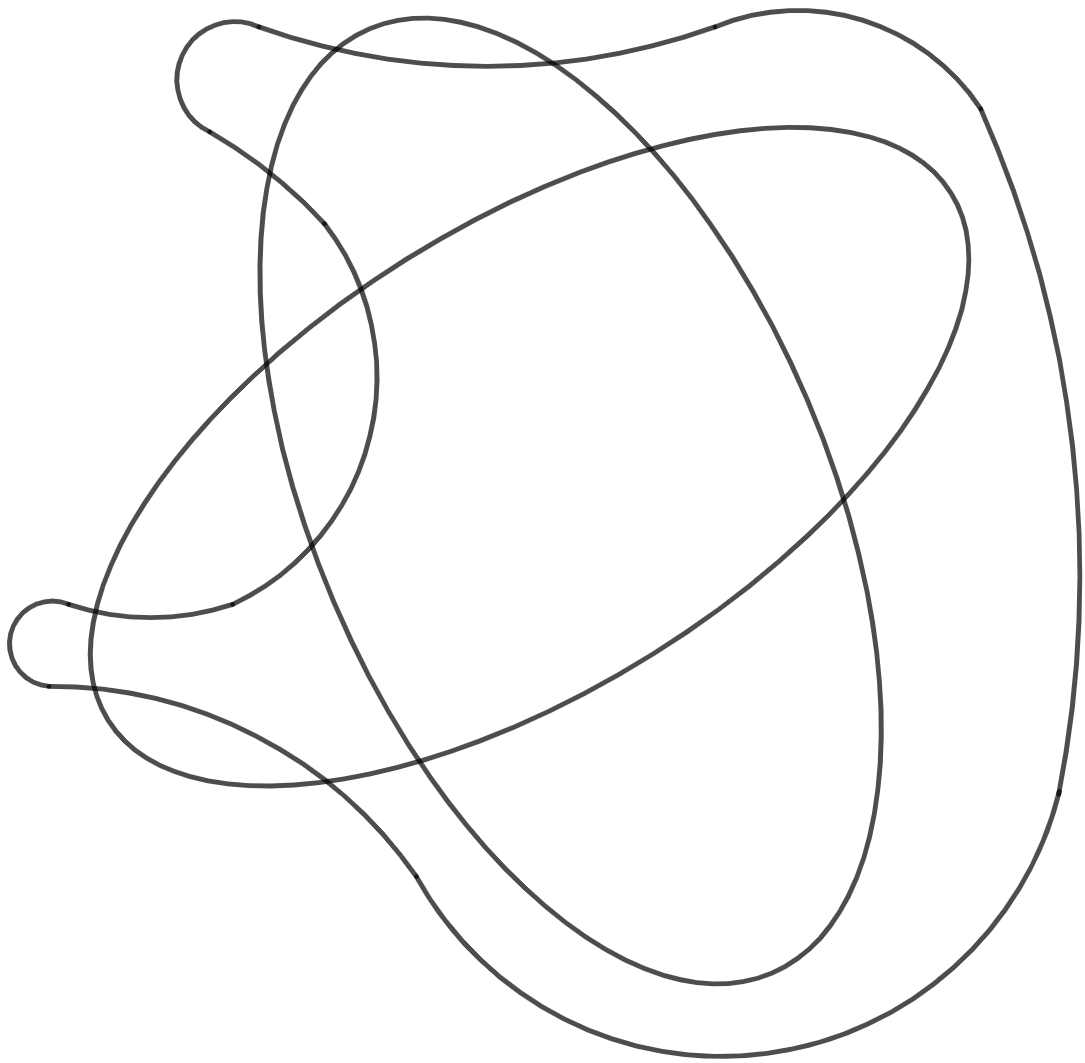}
         \caption{$4411522212125$}
         \label{fig:exc-2}
     \end{subfigure}
     \hfill
         \begin{subfigure}[b]{0.19\textwidth}
         \centering
         \includegraphics[width=\textwidth]{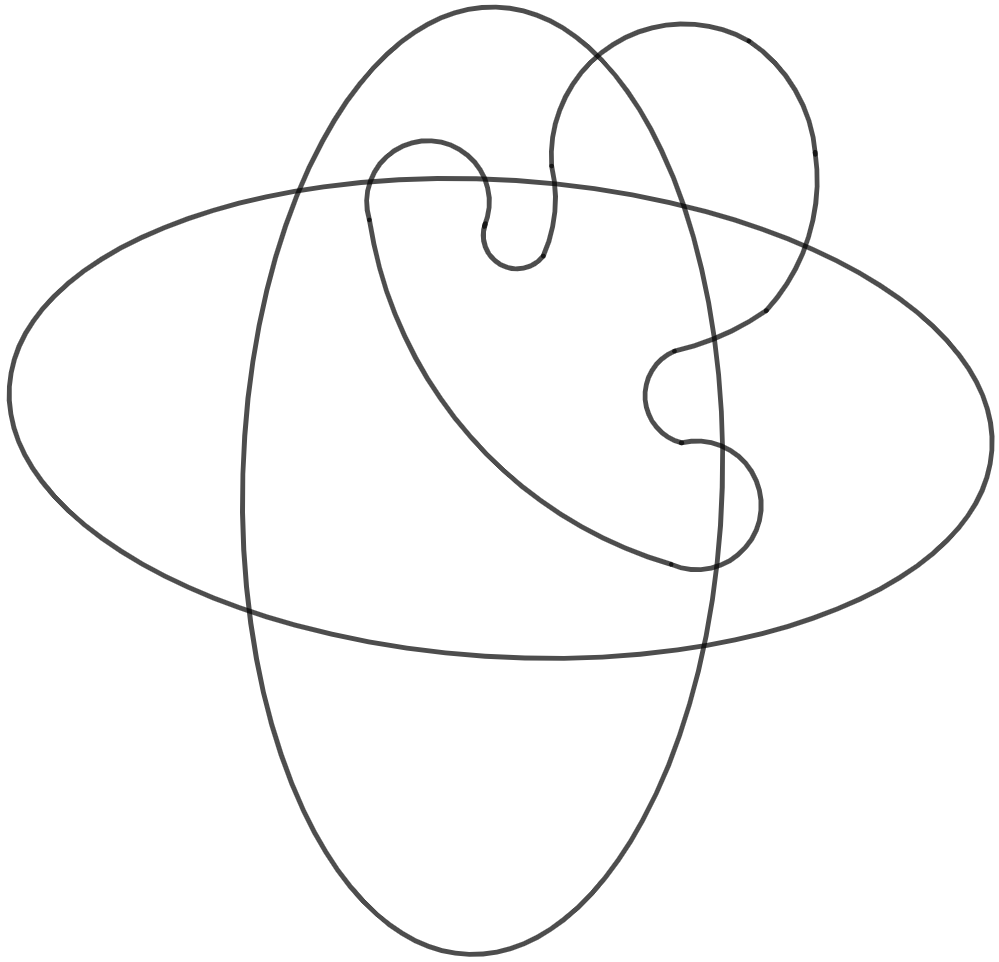}
         \caption{$554211121214$}
         \label{fig:exc-3}
     \end{subfigure}
      \hfill
    \begin{subfigure}[b]{0.19\textwidth}
         \centering
         \includegraphics[width=\textwidth]{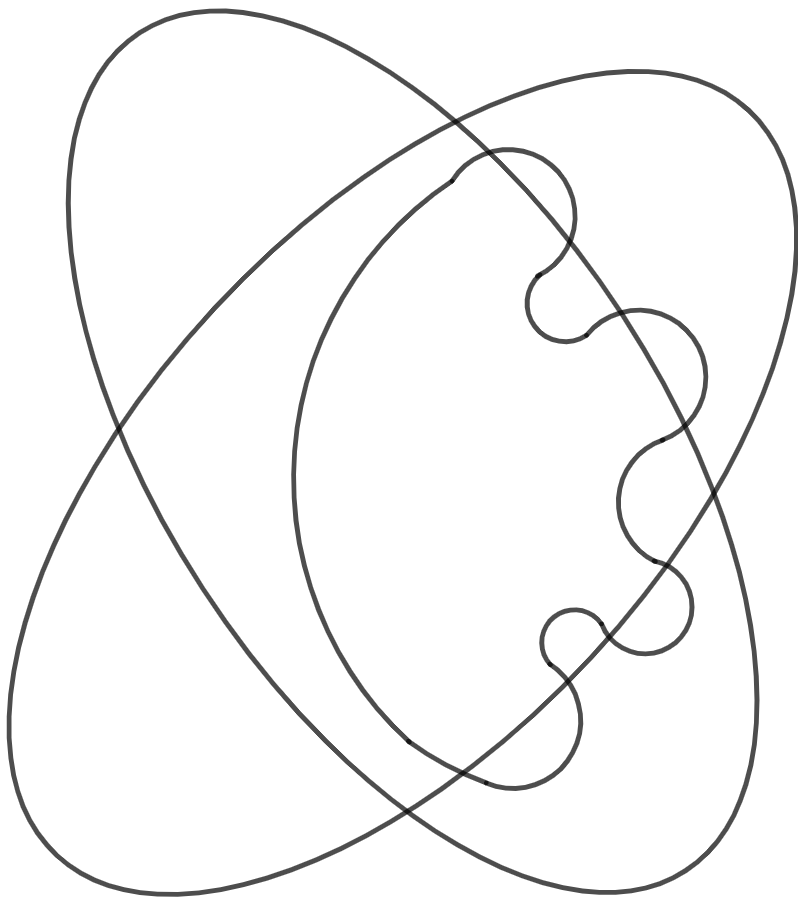}
         \caption{$554555522114$}
         \label{fig:exc-4}
     \end{subfigure}
      \hfill
         \begin{subfigure}[b]{0.19\textwidth}
         \centering
         \includegraphics[width=\textwidth]{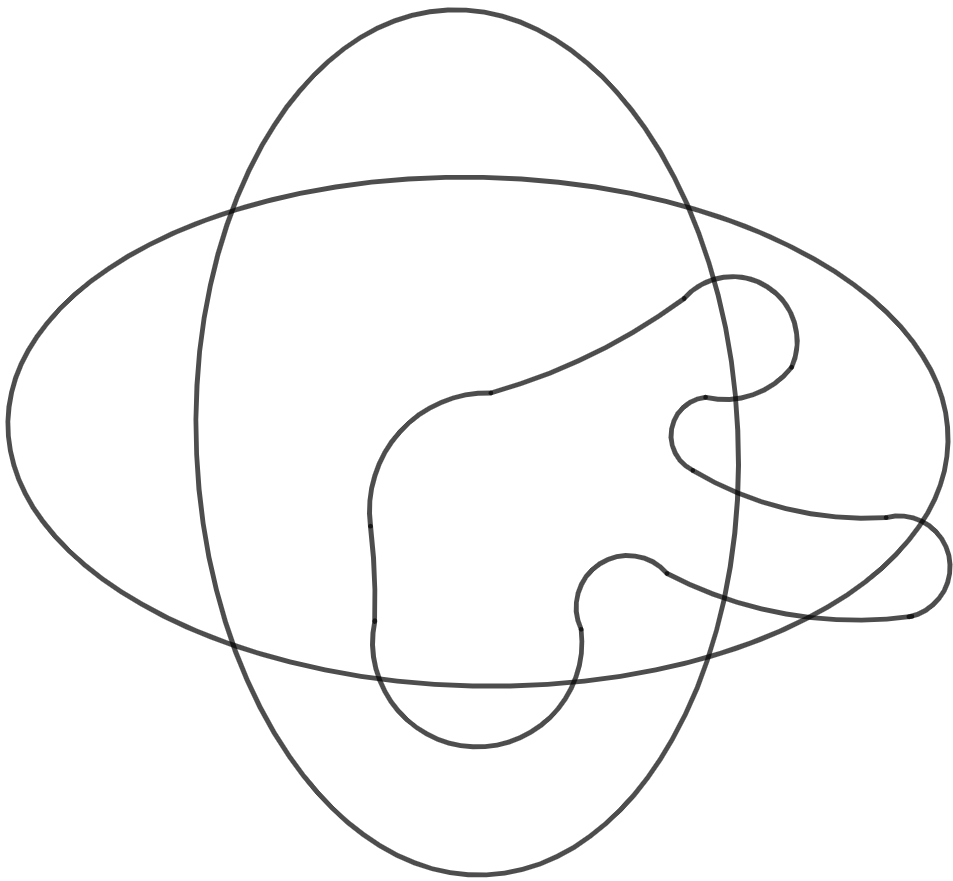}
         \caption{$544545542211$}
         \label{fig:exc-5}
     \end{subfigure}
        \caption{Configurations non-realizable by three ellipses. The subcaptions show their $\times$-codes. For instance, $544445112222$ is the shorthand for $\times_5\times_4\times_4\times_4\times_4\times_5\times_1\times_1\times_2\times_2\times_2\times_2$. }
        \label{Fig:exc}
\end{figure}

Further results of this subsection are due to G.~M.~Polotovskiy. 

\begin{proposition} \label{prop:1} 
None of the five configurations of three ovals shown in Figure~\ref{Fig:exc} can be realized as a union of three real conics. 
\end{proposition}

\begin{proof}  The argument uses the method based on the theory of braids and links suggested by S.~Orevkov  in \cite{MR1679799}. This method has become standard in the problems of topology of reducible real algebraic curves, see e.g.,  \cite{MR1935547}, and \cite{MR1958018}; below we  only present its short account sufficient for the basic understanding of the argument. 

Let $C_m$ be a real projective algebraic plane curve of degree $m$  (i.e., a curve defined by a real homogeneous degree-$m$ polynomial in three variables), all singularities of which are nodes. We write  $C_m(\R)$ for  its real part.

Assume that there exists a point  $p\in \PP^2(\R)\setminus  C_m(\R)$ such that the pencil $L_p$  of lines through  $p$ is \textit{maximal}, which means the following: 
\begin{enumerate}
    \item[a)] $L_p$ contains a line $l_0$ that intersects the curve 
$C_m(\R)$ in $m$  distinct real points. We call $l_0$ a \textit{maximal line}.
\item[b)] Every line  $l\in L_p$ intersects $C_m(\R)$ in at least $m - 2$ distinct real points.
\item[c)] Each line of the pencil has no more than one real point of double intersection with $C_m(\R)$. Each such critical line is either tangent to 
 $C_m(\R)$ or intersects  $C_m(\R)$  at a (cru)node, i.e., a real node where two real branches intersect each other.
\end{enumerate}

For each configuration in Figure \ref{Fig:exc} (assuming it could be realized by three conics), a maximal pencil would obviously exist: the point $p$ can be chosen in the intersection of the interior of the three ovals (e.g., as in Figure \ref{Fig4} left). 
With such a choice  of $p$ every line in the pencil  $L_p$ intersects the curve in six real points (counting multiplicities).  We observe that condition c)  can always be achieved by a small perturbation of the point $p$.

Let us now choose affine coordinates $(x,y)$ in $\R^2$ in such a way that the line $l_0$  becomes the line at infinity (which implies that the point $p$ is also located at infinity) and such that the pencil $L_p$ will become the pencil of parallel lines   $\{l_t\}$  where $l_t$ is the line given by the equation $x=t$; see Figure~\ref{Fig2} left, where $\{l_t\}$ is shown as parallel lines.

\begin{figure}[htb]
    \centering
	\includegraphics[width=\textwidth]{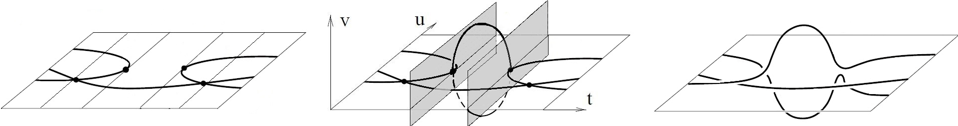}
	\caption{Constructing a link using a maximal pencil. }
	\label{Fig2}
\end{figure}

\begin{figure}[htb]
	\centering
	\includegraphics[height=4cm]{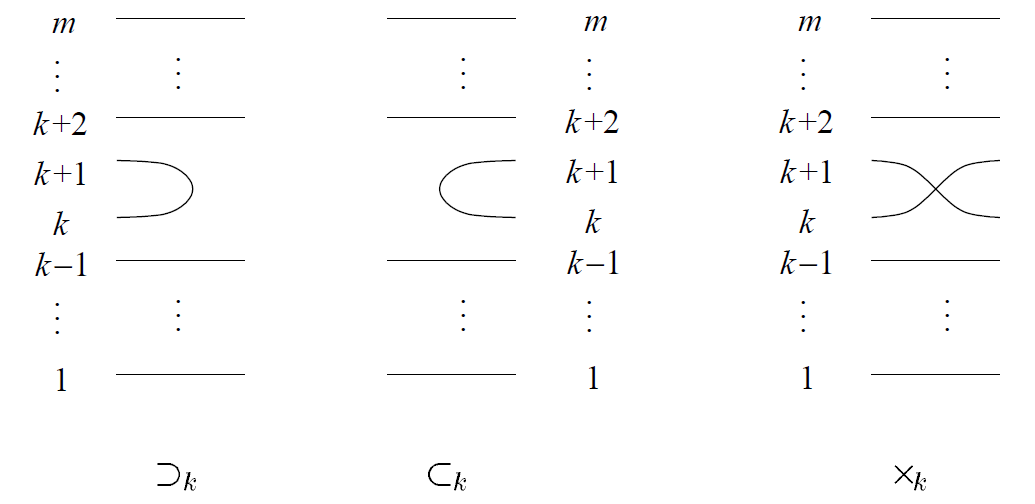}
		\caption{Possible symbols of the $\times$-code.}
		\label{Fig3}
	\end{figure}

Let $\{l_{t_1}, \dots , l_{t_s}\}$ be the set of critical lines (i.e., lines passing through crunodes  or tangent to $C_m(\R)$), ordered according to the increase of parameter values ${t_i}$.
The scheme  of location of the curve $C_m(\R)$ with respect to the pencil  $L_p$  is coded by the word $u_1\cdots u_s$ where the letter $u_i$ characterizes 
the local behavior  of $C_m(\R)$ near $l_{t_i}$ and attains one of the three possible values: $\supset\!_k$, $\subset\!_k$, $\times\!_k$ ($k\in \{1,\dots,m-1\}$)
as explained in Figure~\ref{Fig3}. 

In what follows, we call the coding word the \emph{$\times$-code}.  
In the configurations in Figure \ref{Fig:exc} (assuming they could be realized by three conics), the symbols     $\subset\!_k, \supset\!_k$
are unnecessary since  every line from the pencil $L_p$ intersects each of the conics transversally. Thus, their $\times$-codes  contain only the symbols $\times\!_k$ for $k\in \{1,\dots,5\}$. 

Next we include $C_m(\R)$ into a bigger  one-dimensional singular curve $M\subset \PP^2$ defined as  $M:=C_m(\C) \cap
L_p(\C)$,    
where  $L_p(\C)$ is the complexification of the real pencil $L_p$, i.e., we substitute each  line in $L_p$ by its complexification in $\PP^2$. 
Observe that $L_p(\C)$ is a  $3$-dimensional subset of $\PP^2$ and every complex line from $L_p(\C)$  intersects $C_m(\C)$ in finitely many points. 
The union of these points (for all values of the real parameter in $L_p$ running over $\PP^1(\R)$) forms the  curve $M\subset \PP^2$ which obviously includes 
$C_m(\R)$.

The curve $M$ is homeomorphic to a collection of circles, some of which are pairwise glued together at the nodes of $C_m(\R)$  and at the tangency points of the pencil $L_p$ with this curve (see Figure~\ref{Fig2} center\footnote{The figure is schematic -- the ``imaginary axis'' $v$ is $2$-dimensional.}).

Removing all the gluing points in a standard way (see Figure~\ref{Fig2} right),  we obtain the link $K(C_m,p)\subset L_p(\C)$. (Geometric details of this resolution are a bit lengthy and can be found in  \cite{MR1679799}, pp. 13--14.)
Let $b(C_m,p)$ denote a braid with $m$ strands whose closure coincides with $K(C_m,p)$. 
For what follows  it is important to observe\footnote{In principle,  Orevkov's method is applicable even in the case when condition  b) does not hold, but in such situation it is very difficult to present/check all possible occurring links in the complex domain.} that condition b)  guarantees that the braid $b(C_m,p)$ is uniquely determined (up to conjugation in the group $B_m$ of braids with $m$ strands) by the relative position of the curve $C_m(\R)$ and the pencil $L_p$ 
  in 
$\PP^2(\R)$.  
Recall that the group 
$B_m$ has the following standard (co)representation in terms of the generators $\sigma_k$:
$$
\langle\sigma_1,\dots,\sigma_{m-1}\,|\,\text{
	$\sigma_i\sigma_j = \sigma_j\sigma_i$ if
	$|i-j|>1$,
	$\sigma_i\sigma_j\sigma_i=\sigma_j\sigma_i\sigma_j$
	if $|i-j|=1$}\rangle.
$$
If the initial curve is algebraic, then the obtained braid $b(C_m,p)$ must be \textit{quasipositive} \cite{MR683760}, i.e., it has 
to admit a presentation in the form
$\prod_{j=1}^k\omega_j\sigma_{i_j}\omega_j^{-1}$, where 
$\omega_j$  are some words in the alphabet $\{\sigma_1, \dots , \sigma_{m-1}, \sigma_1^{-1}, \dots ,
\sigma_{m-1}^{-1}\}$. 

As a necessary condition of quasipositivity, S.~Orevkov \cite{MR1679799} suggested to use  the following (for notation see ibid.): 

\smallskip{}
\noindent
{\bf Murasugi--Tristram inequality}. 
\textit{If $b=\prod \sigma_i^{k_i}$ is a quasipositive braid with $m$ strands, then  its closure satisfies the inequality $$|\sigma(b)|+m-e(b)-n(b)\leq0,$$
	where $\sigma(b)$ and $n(b)$ are the signature and the defect of the closure of the braid $b$, and $e(b)=\sum k_i$ is the algebraic degree of the braid $b$}. 
	
\smallskip{}	
	Recall that the signature and the defect of a link are defined as the signature and defect of its Seifert matrix (quadratic form)  which, by definition, is the intersection matrix of the cycles on the Seifert surface of the link, see details in e.g., \cite{MR1679799}, \S~2.5--2.6.

\begin{figure}[htb]
    \centering
    \includegraphics[scale=0.3]{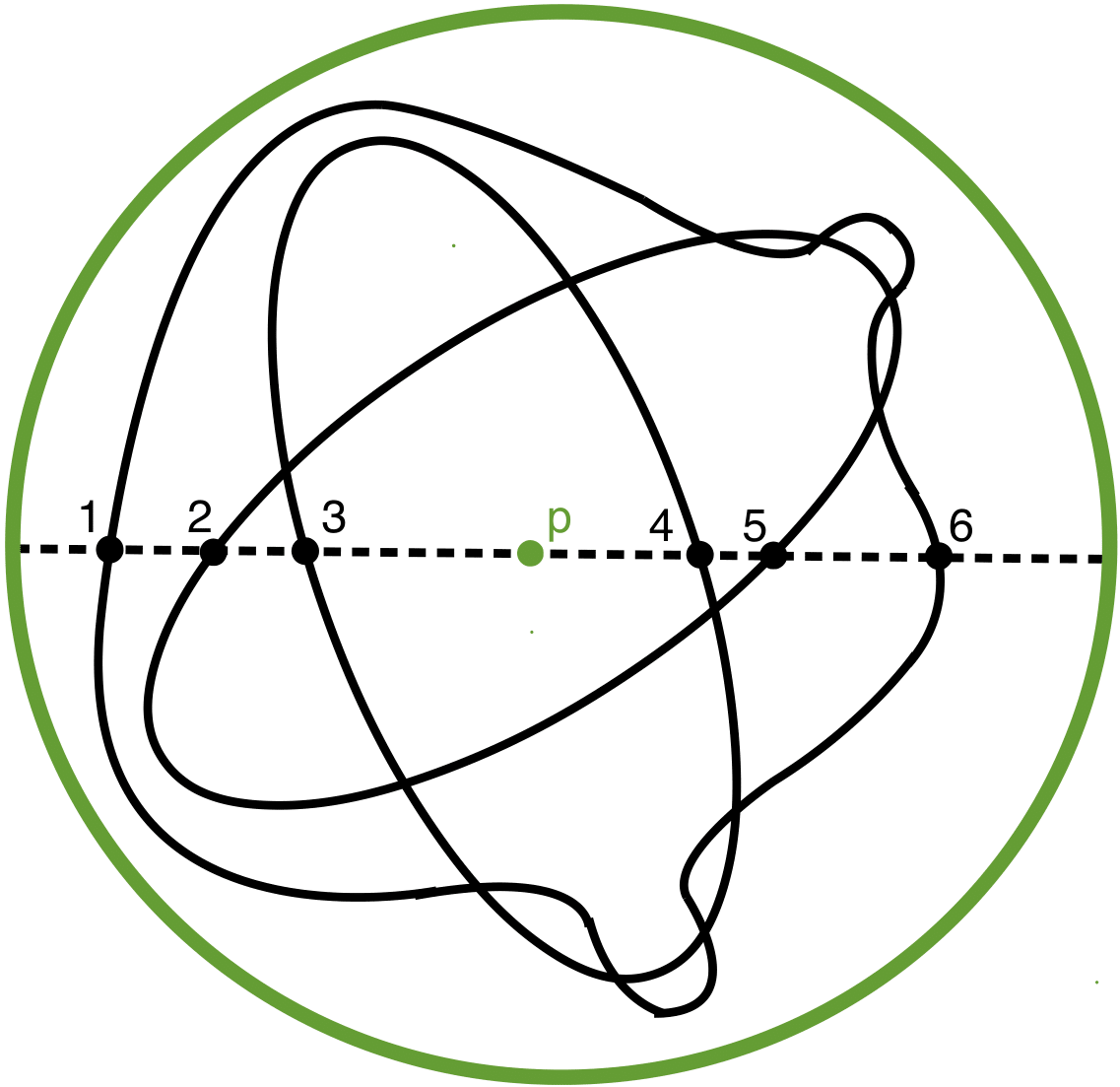}
    \hspace{2cm}
    \includegraphics[scale=0.3]{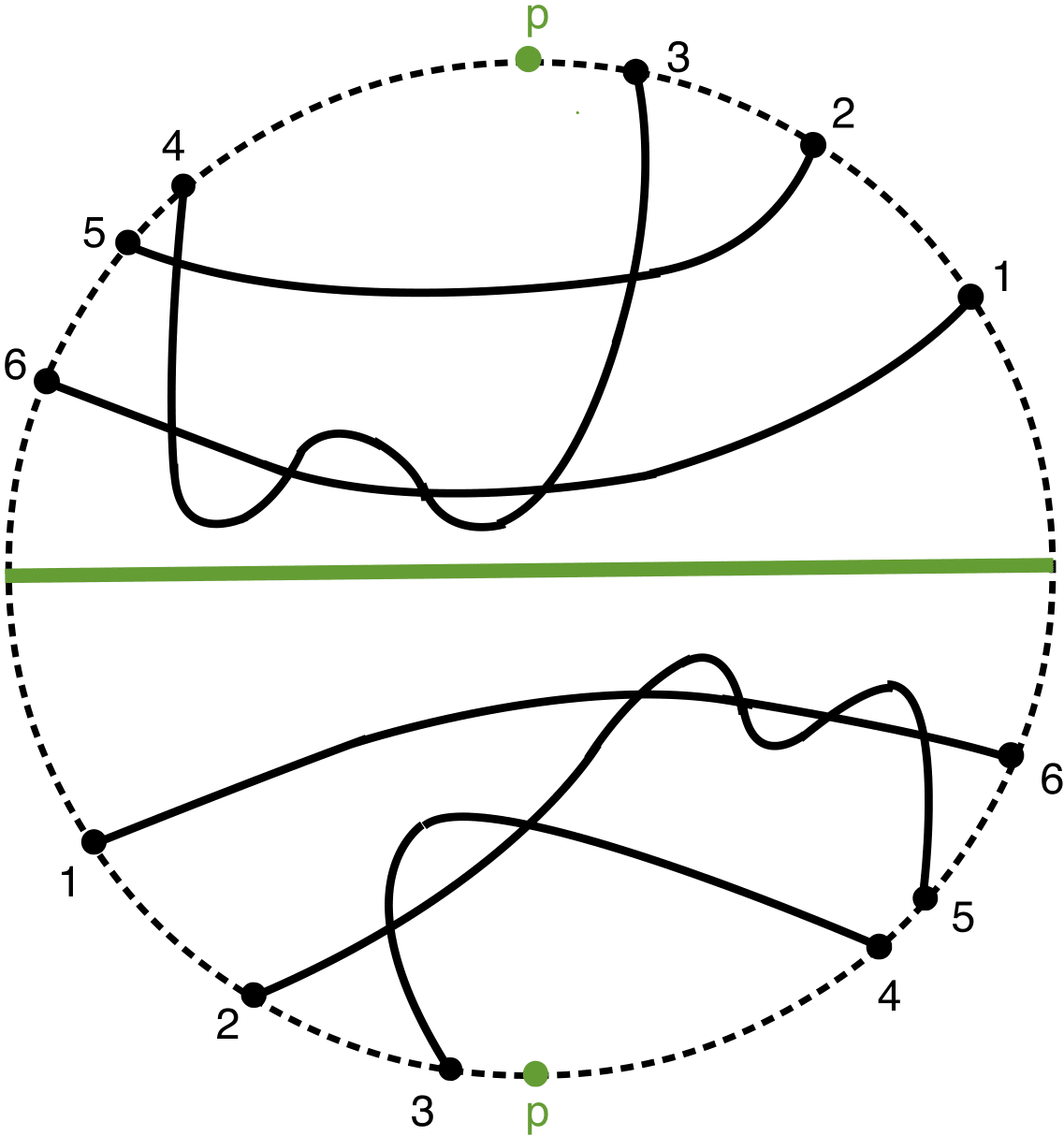}
    \caption{The leftmost configuration from Figure \ref{Fig:exc} in $\PP^2(\R)$ with a line (dashed) that intersects the three ovals in six real points.}
    \label{Fig4}
\end{figure}

\begin{figure}[htb]
    \centering
    \includegraphics[scale = 0.45]{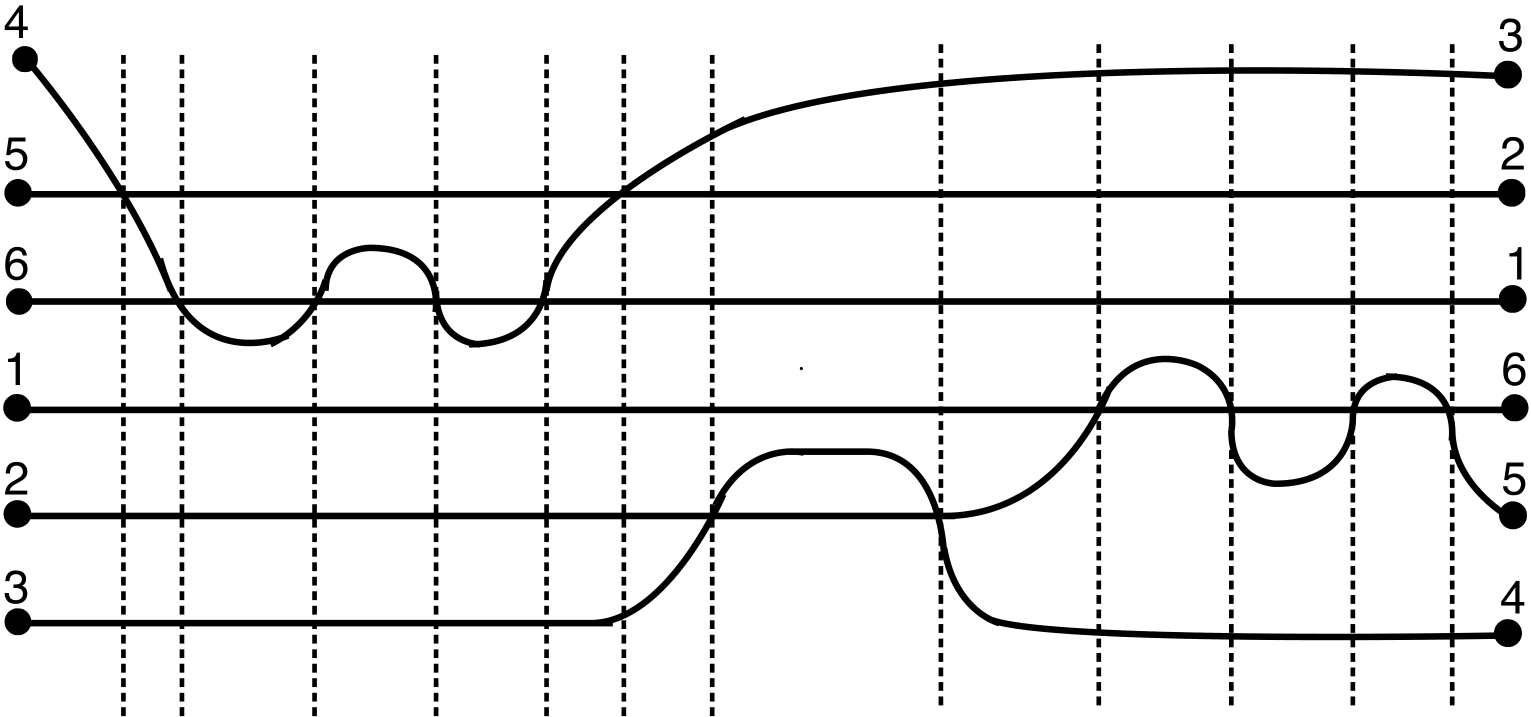}
    \caption{Affine picture of \Cref{Fig4} (right) such that all lines through $p$ are vertical.}
    \label{Fig5}
\end{figure}

For the leftmost configuration in Figure~\ref{Fig:exc}, we choose the point $p$ and the line $l_0$ as shown in Figure~\ref{Fig4} left, where the curve is shown in the real projective plane modelled on the unit disk with pairwise identified opposite  points of its boundary circle. For the convenience of producing the  $\times$-code, let us choose the system of coordinates in a different way -- namely, such that 
$l_0$ becomes the boundary circle as in Figure~\ref{Fig4} right.  Now using the affine plane 
(i.e., sending the line $l_0$ to infinity), we obtain Figure~\ref{Fig5} from which it is easy to obtain the  $\times$-code: 
$\times\!_5\times\!_4\times\!_4\times\!_4\times\!_4\times\!_5\times\!_1\times\!_1\times\!_2\times\!_2\times\!_2\times\!_2$.
Similarly, we obtain the $\times$-codes of the remaining four configurations in \Cref{Fig:exc}, presented in the subcaptions. 

The computer program written in the early 2000's by   M.~Gushchin,
allows us to calculate the left-hand side of the Mirasugi--Tristram inequality using the $\times$-code.\footnote
{S.~Orevkov has written his own code  which is partially published in \cite{MR1935547}. In all previously tested cases   both programs give the same outcome.}
For all the five $\times$-codes obtained above, the computer results for the left-hand side are equal to $2$, which means that neither of the configurations in Figure~\ref{Fig:exc} can be realized as a union of three conics. 
\end{proof}

\subsection{The catalog}\label{sec:catalog}

We present all 44 configurations of three ellipses described in \Cref{thm:wachspressellipses} in Figures~\ref{fig:222}--\ref{fig:444}. We found some configurations difficult to draw without having some arcs of the conics to be very close to each other.
We provide topological sketches of these configurations in \Cref{fig:sketches}. In the next subsection, we prove Wachspress's  conjecture for 28 of the 44 configurations.
The remaining 16 problematic configurations are framed in Figures~\ref{fig:222}--\ref{fig:444}.
We will see that each of these has in fact only one problematic polycon up to symmetry, shaded in orange in the figures.

\begin{figure}[htb]
     \centering
         \begin{subfigure}[b]{0.24\textwidth}
         \centering
         \includegraphics[width=\textwidth]{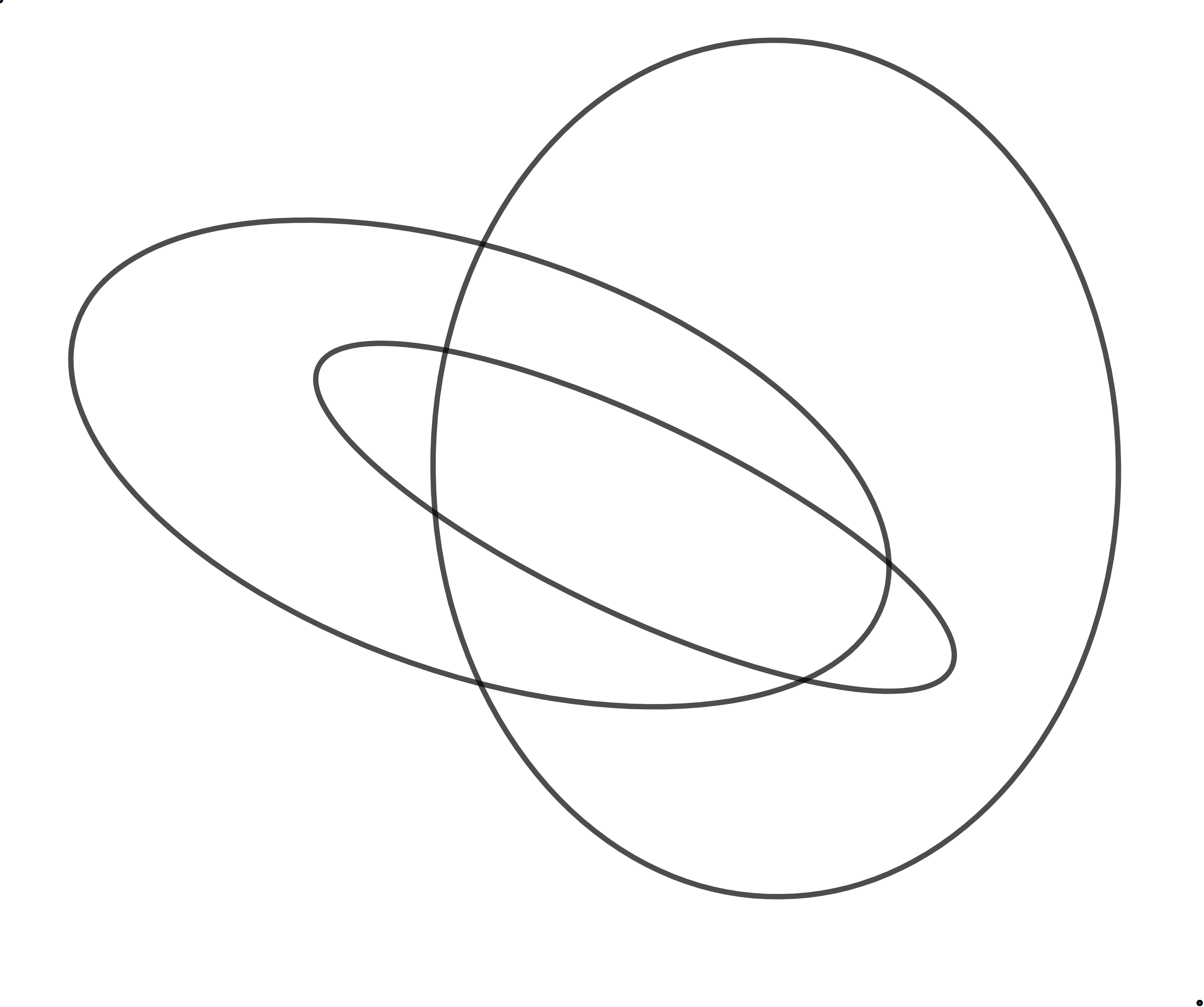}
         \caption{110}
         \label{fig:222_110}
     \end{subfigure}
     \hfill
         \begin{subfigure}[b]{0.24\textwidth}
         \centering
         \frame{\includegraphics[height=2.5cm, width=3cm]{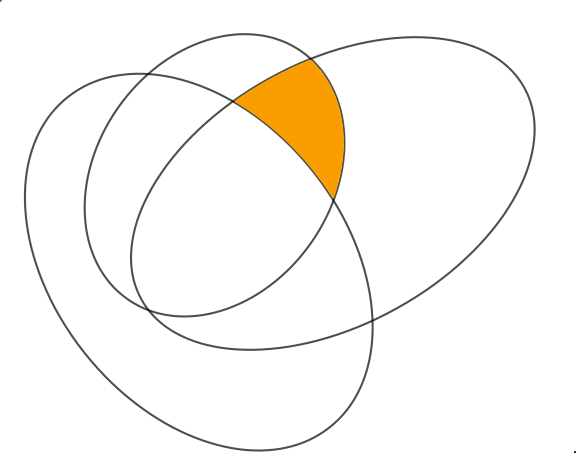}}
         \caption{111}
         \label{fig:222_111}
     \end{subfigure}
     \hfill
         \begin{subfigure}[b]{0.24\textwidth}
         \centering
         \frame{\includegraphics[height=2.5cm,width=3cm]{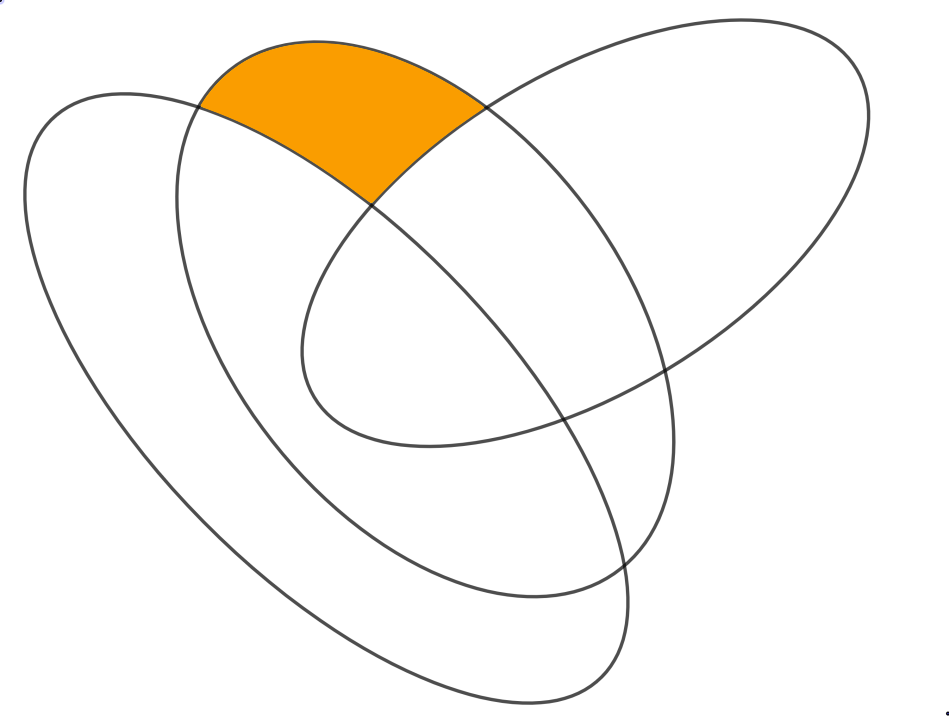}}
         \caption{211}
         \label{fig:222_211}
     \end{subfigure}
    \begin{subfigure}[b]{0.24\textwidth}
         \centering
         \includegraphics[width=\textwidth]{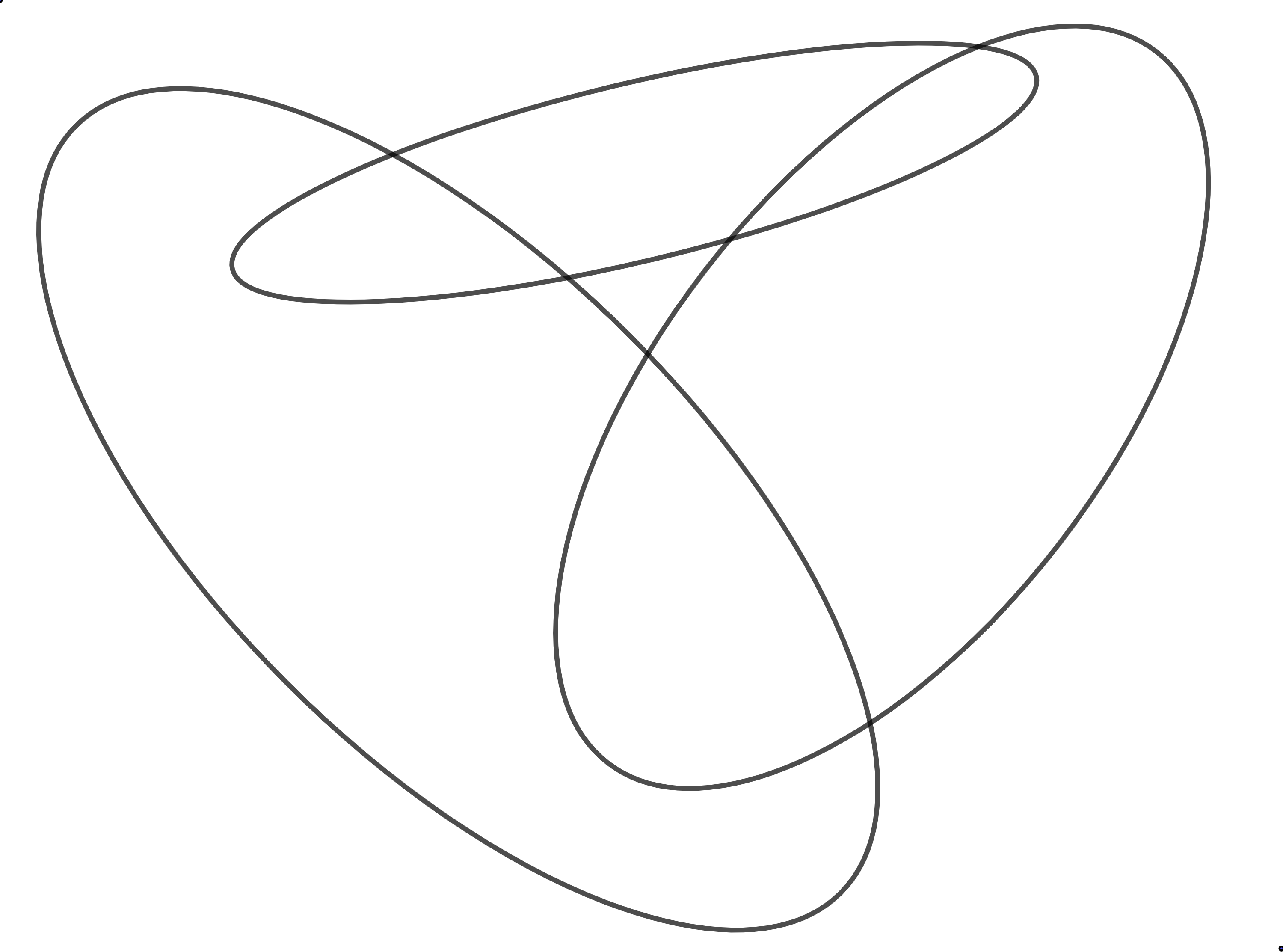}
         \caption{222}
         \label{fig:222_222}
     \end{subfigure}
        \caption{Intersection type (222). The subcaptions show the outer arc type.
        Problematic configurations are framed: The unique (up to symmetry) problematic polycon is orange.}
        \label{fig:222}
\end{figure}

\begin{figure}[htb]
     \centering
     \begin{subfigure}[b]{0.24\textwidth}
         \centering
         \includegraphics[width=\textwidth]{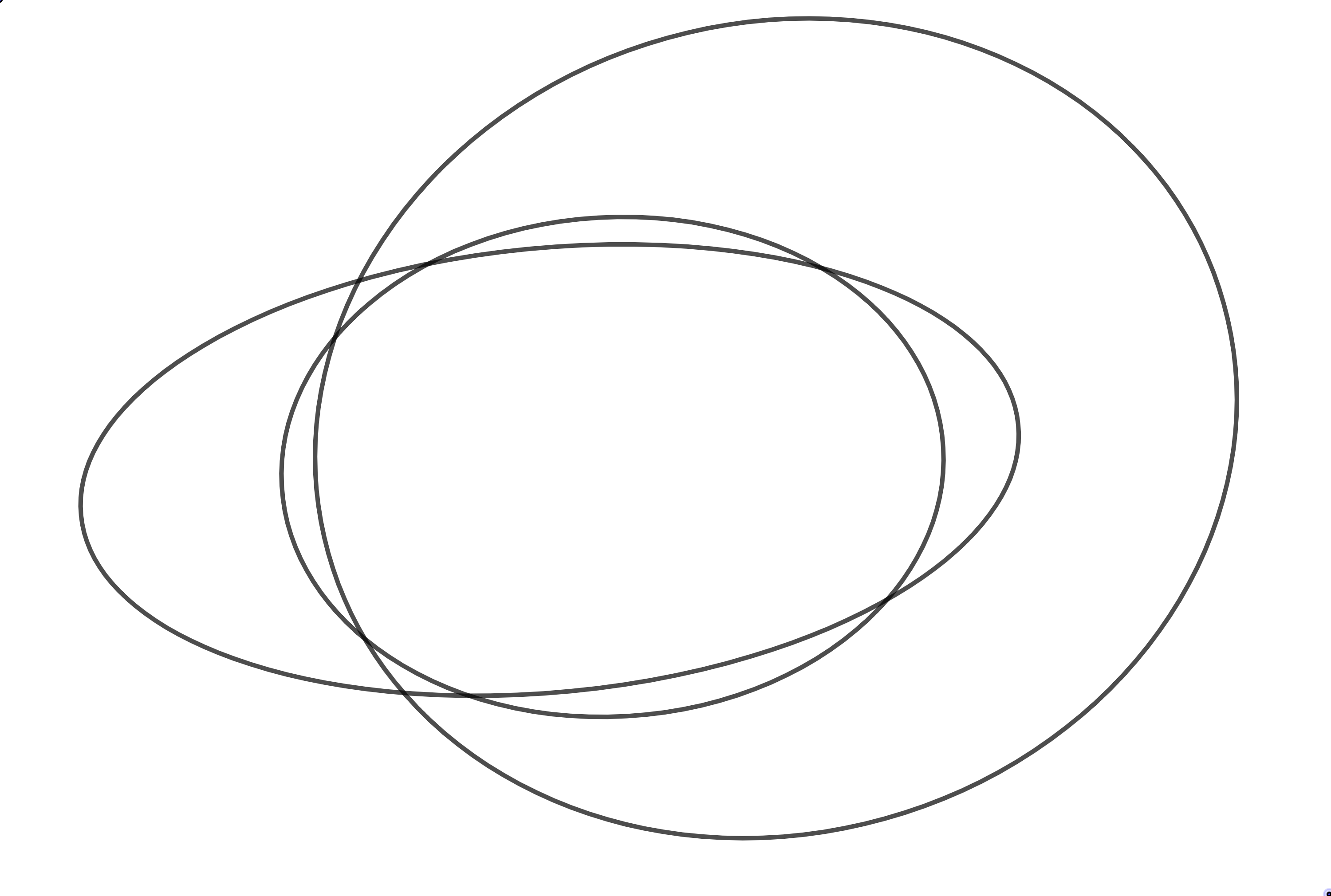}
         \caption{110}
         \label{fig:224_110}
     \end{subfigure}
         \begin{subfigure}[b]{0.24\textwidth}
         \centering
         \frame{\includegraphics[height=2.5cm,width=3cm]{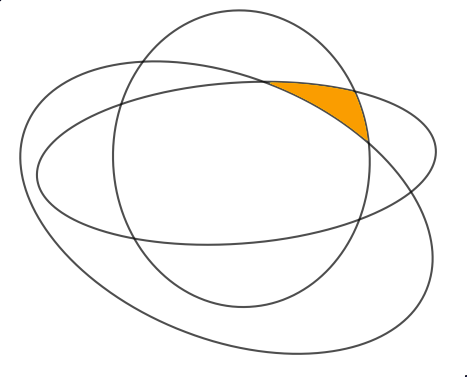}}
         \caption{111}
         \label{fig:224_111}
     \end{subfigure}
     \hfill
      \begin{subfigure}[b]{0.48\textwidth}
         \centering
         \frame{\includegraphics[height=2.5cm,width=3cm]{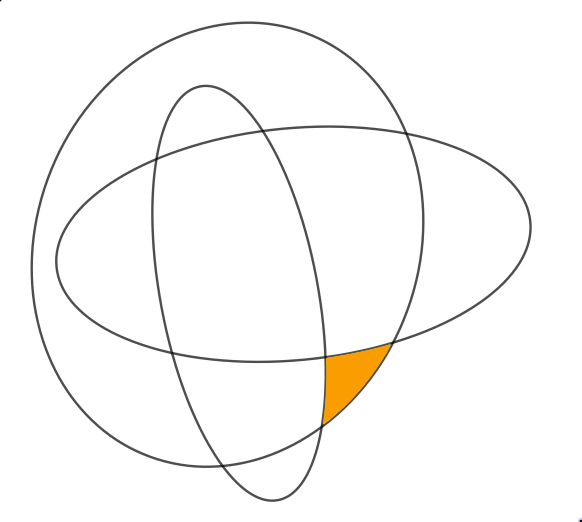}}
         \includegraphics[width=0.49\textwidth,width=3cm]{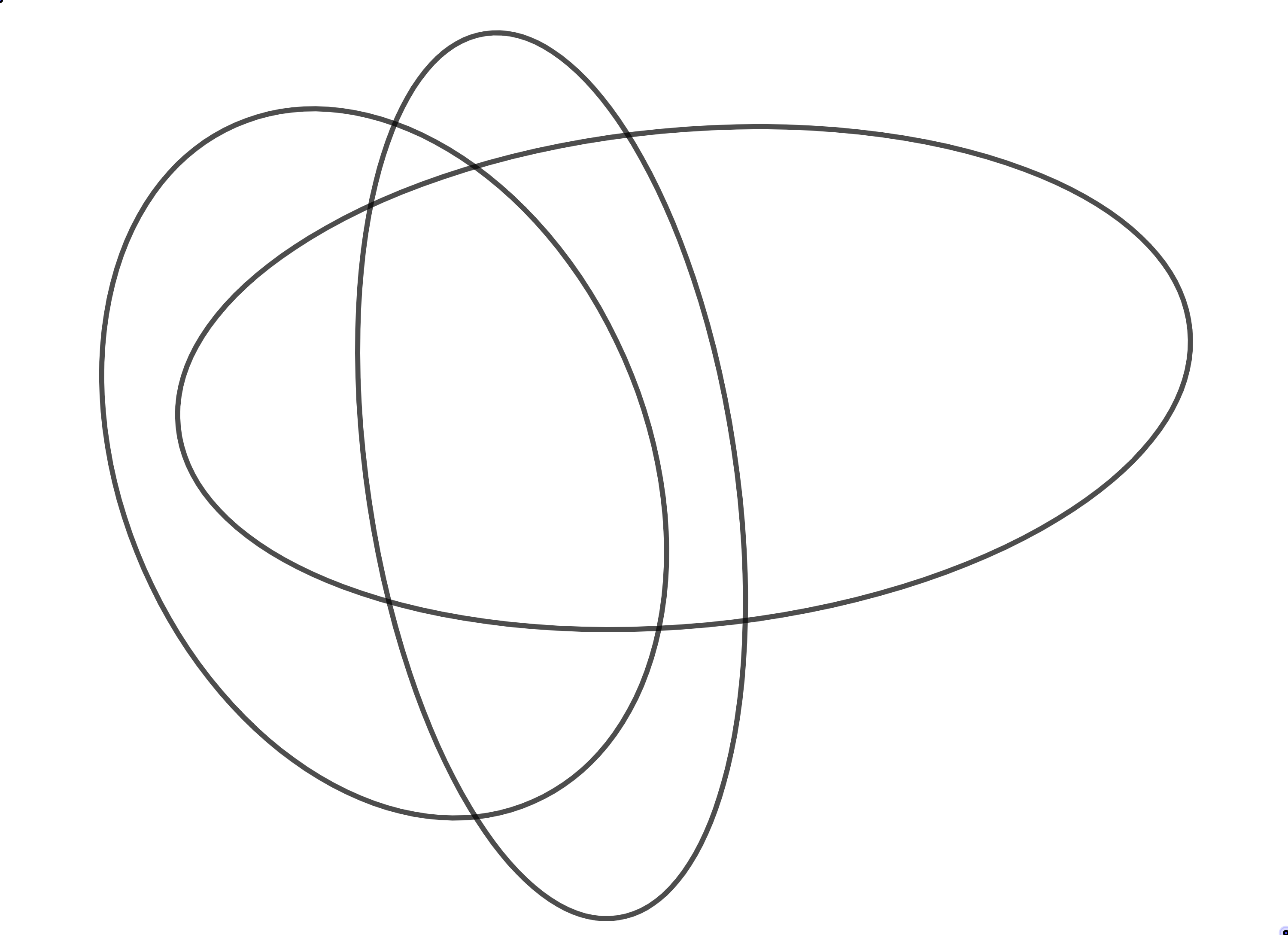}
         \caption{211}
         \label{fig:224_211}
     \end{subfigure}  
         \begin{subfigure}[b]{0.24\textwidth}
         \centering
         \includegraphics[width=\textwidth]{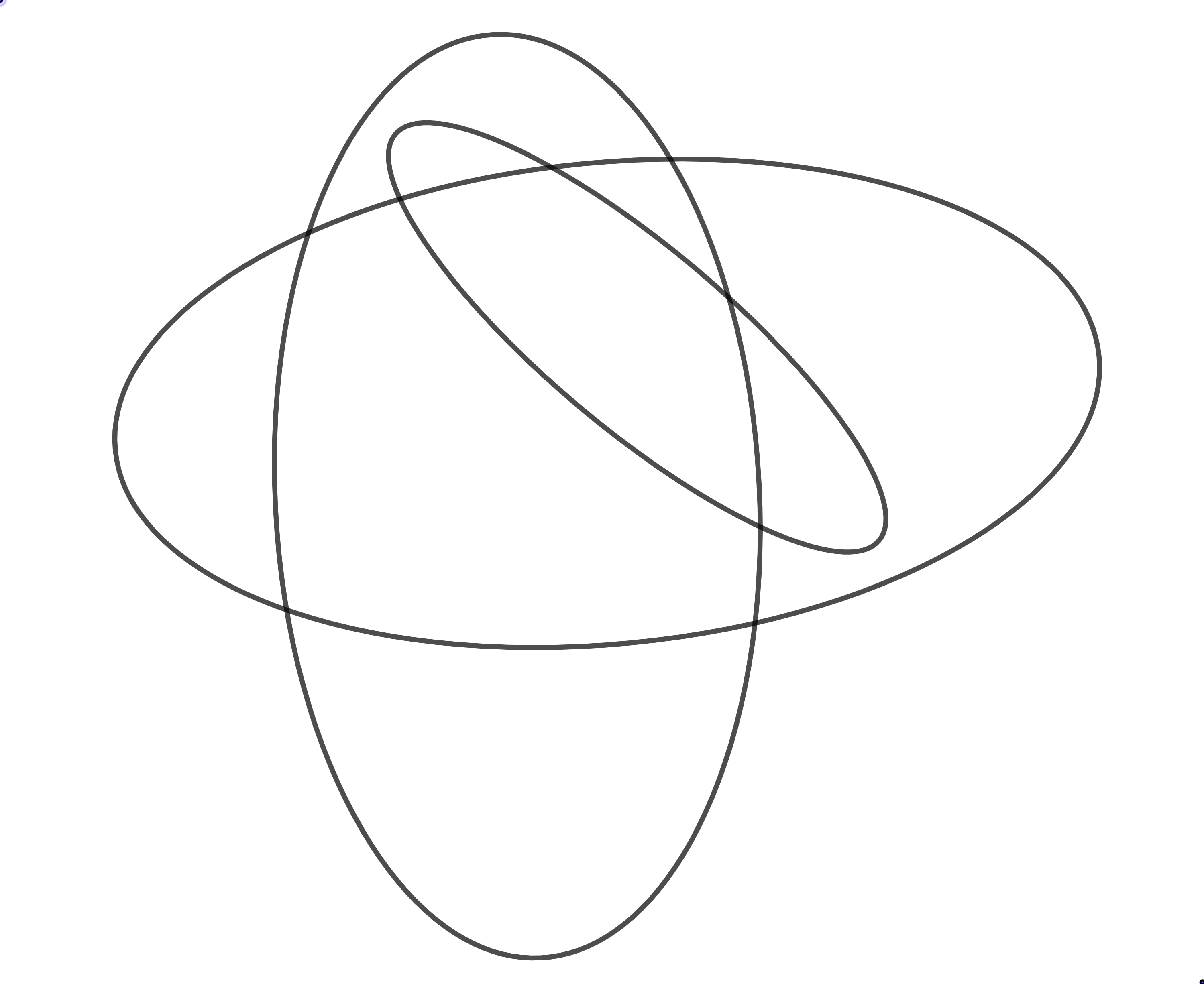}
         \caption{220}
         \label{fig:224_220}
     \end{subfigure}
     \hfill
         \begin{subfigure}[b]{0.24\textwidth}
         \centering
         \frame{\includegraphics[height=2.5cm,width=3cm]{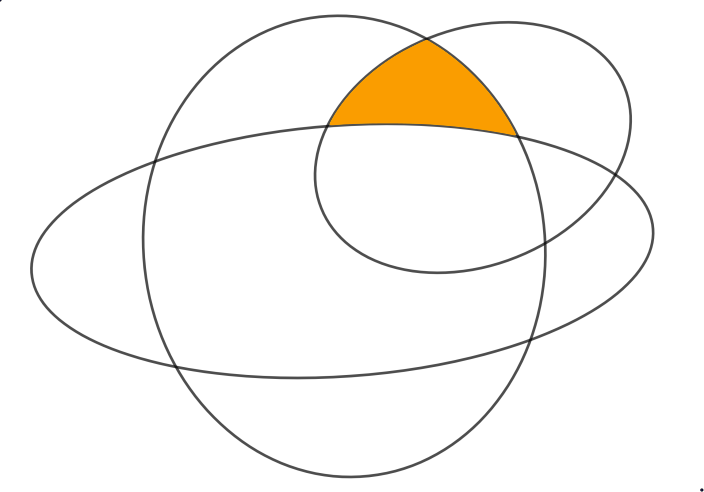}}
         \caption{221}
         \label{fig:224_221}
     \end{subfigure}
     \hfill
         \begin{subfigure}[b]{0.24\textwidth}
         \centering
         \frame{\includegraphics[height=2.5cm,width=3cm]{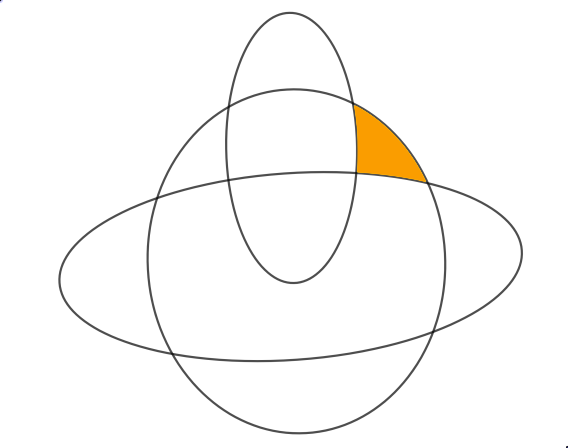}}
         \caption{321}
         \label{fig:224_321}
     \end{subfigure}
    \begin{subfigure}[b]{0.24\textwidth}
         \centering
         \includegraphics[width=\textwidth]{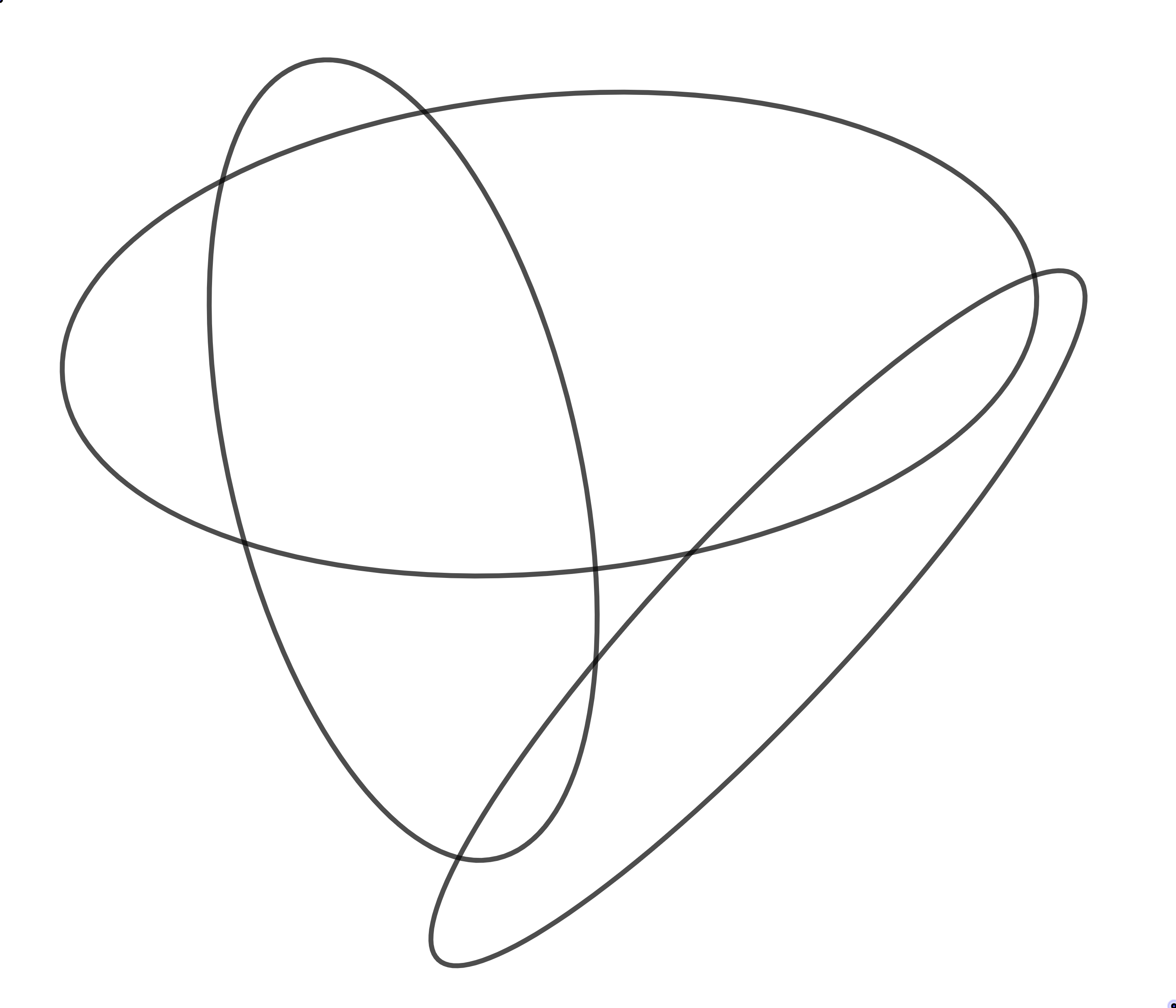}
         \caption{332}
         \label{fig:224_332}
     \end{subfigure}
        \caption{Intersection type (224). The subcaptions show the outer arc type.   Problematic configurations are framed: The unique (up to symmetry) problematic polycon is orange.}
        \label{fig:224}
\end{figure}

\begin{figure}[htb]
     \centering
     \begin{subfigure}[b]{0.19\textwidth}
         \centering
         \includegraphics[width=\textwidth,height=2cm]{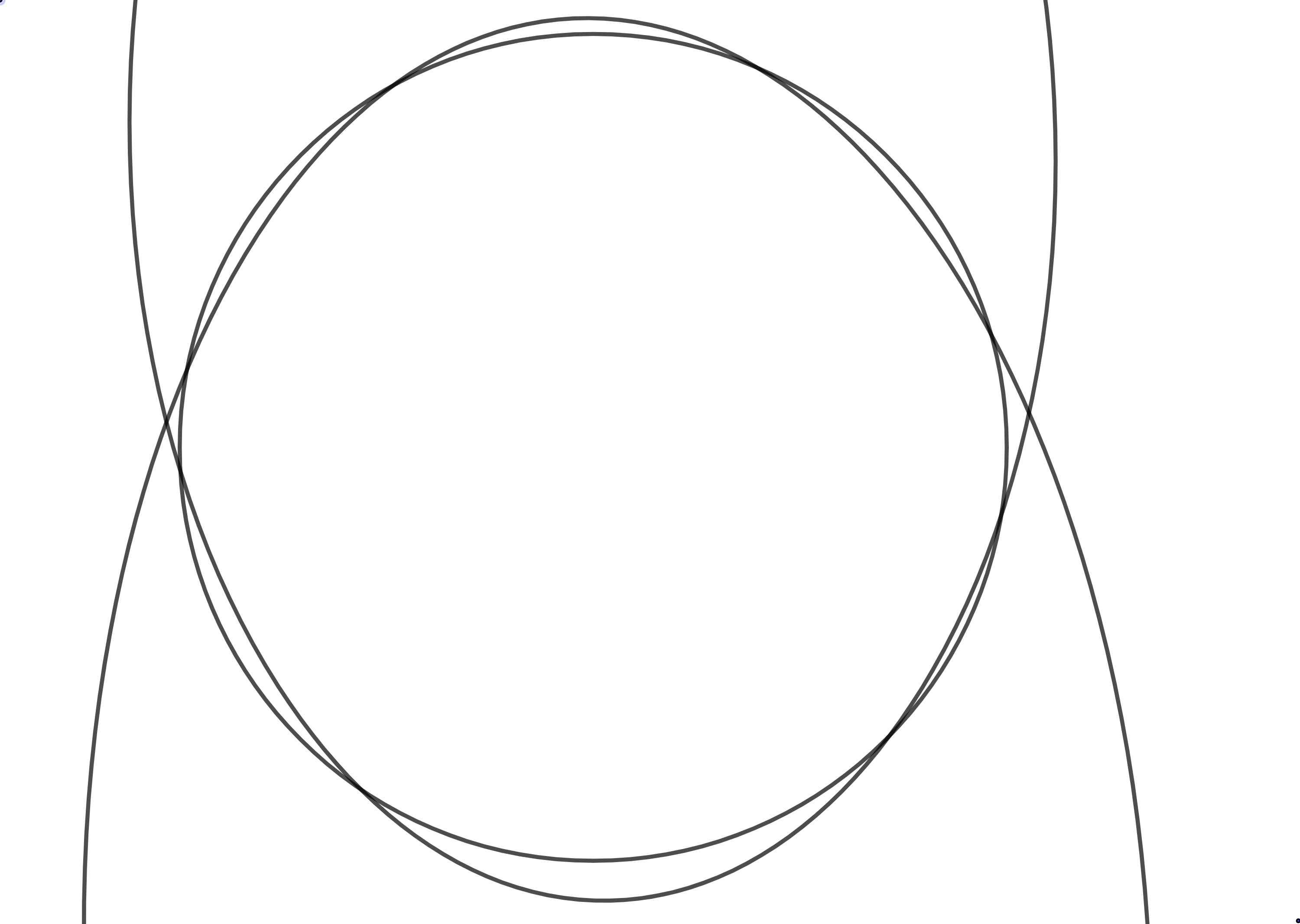}
         \caption{110}
         \label{fig:244_110}
     \end{subfigure}
     \hfill
         \begin{subfigure}[b]{0.19\textwidth}
         \centering
         \includegraphics[width=\textwidth]{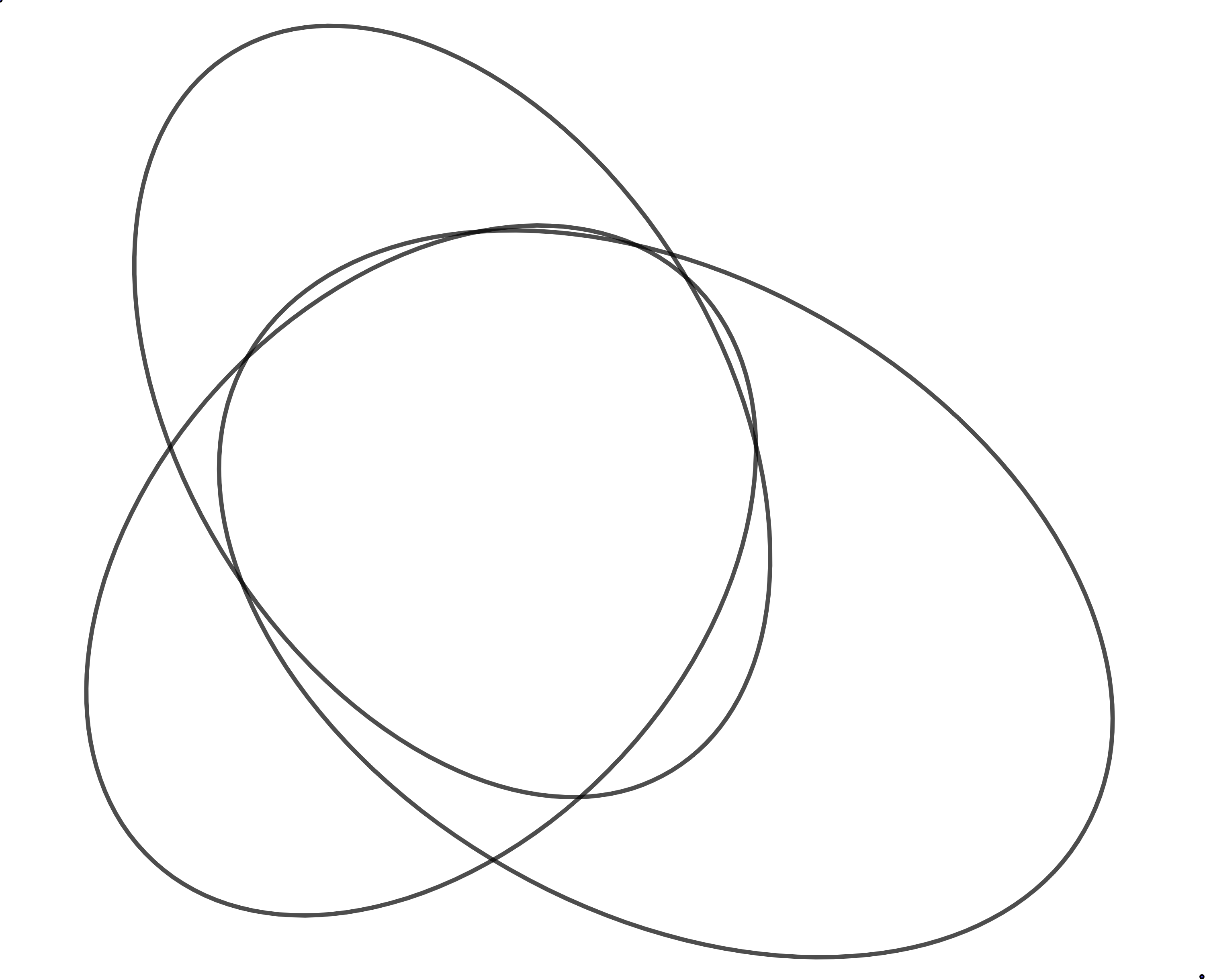}
         \caption{111}
         \label{fig:244_111}
     \end{subfigure}
     \hfill
         \begin{subfigure}[b]{0.19\textwidth}
         \centering
         \includegraphics[width=\textwidth,height=2cm]{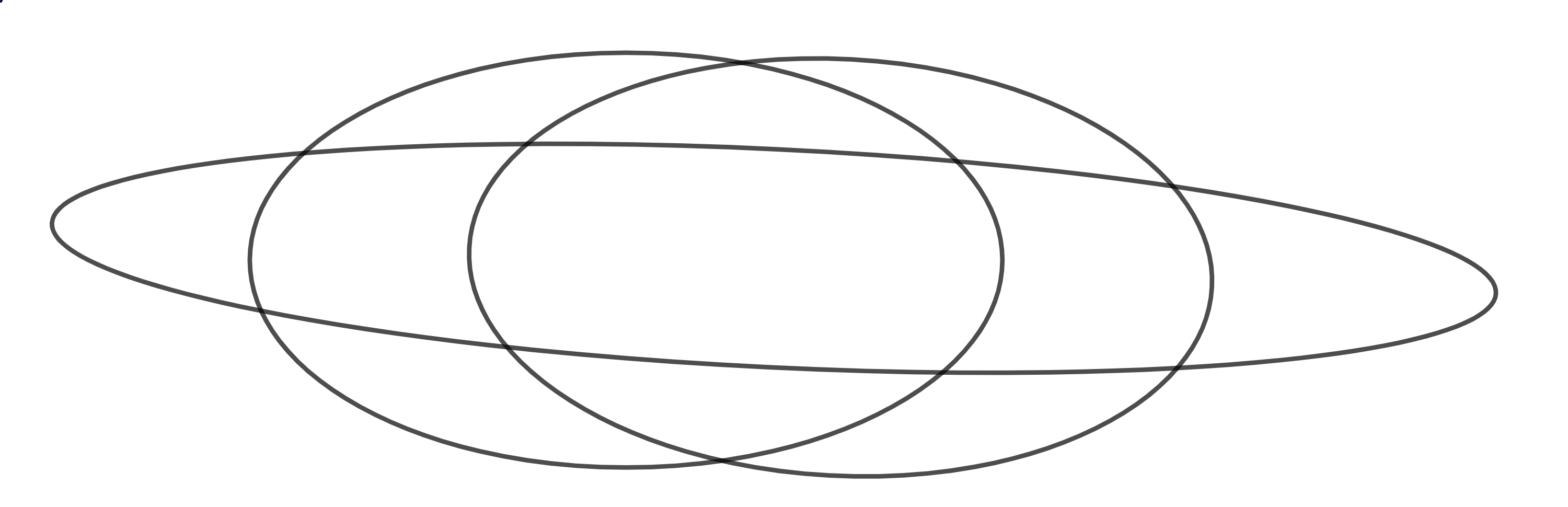}
         \caption{222}
         \label{fig:244_222}
     \end{subfigure}
     \hfill
         \begin{subfigure}[b]{0.19\textwidth}
         \centering
         \frame{\includegraphics[height=2cm,width=2.6cm]{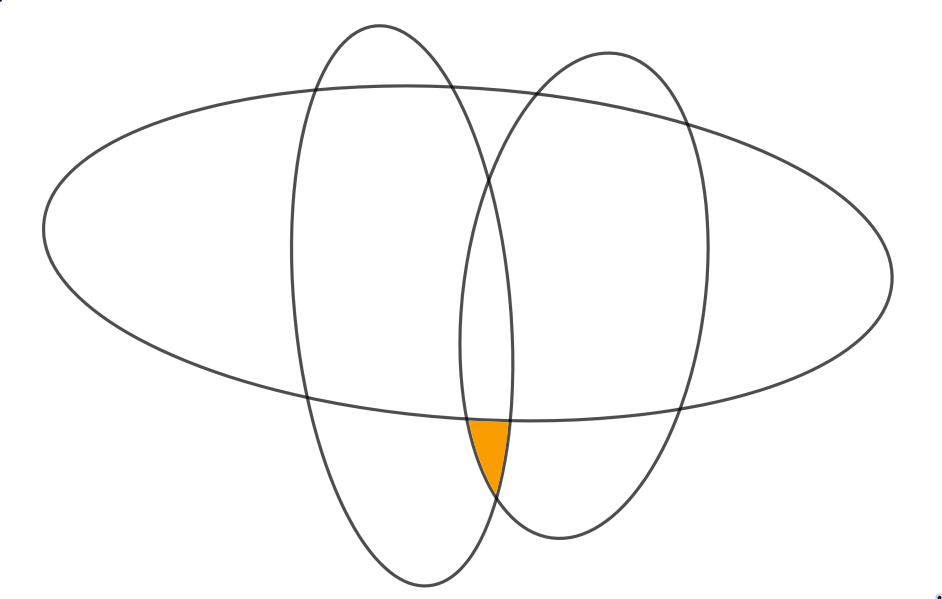}}
         \caption{322}
         \label{fig:244_322}
     \end{subfigure}
     \hfill
         \begin{subfigure}[b]{0.19\textwidth}
         \centering
         \frame{\includegraphics[height=2cm,width=2.6cm]{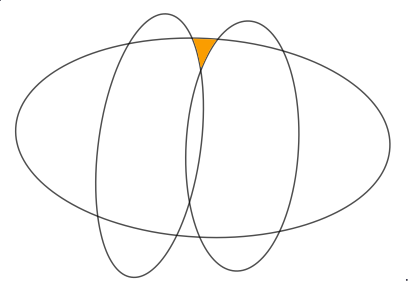}}
         \caption{422}
         \label{fig:244_422}
     \end{subfigure}
    \begin{subfigure}[b]{0.19\textwidth}
         \centering
         \includegraphics[width=\textwidth]{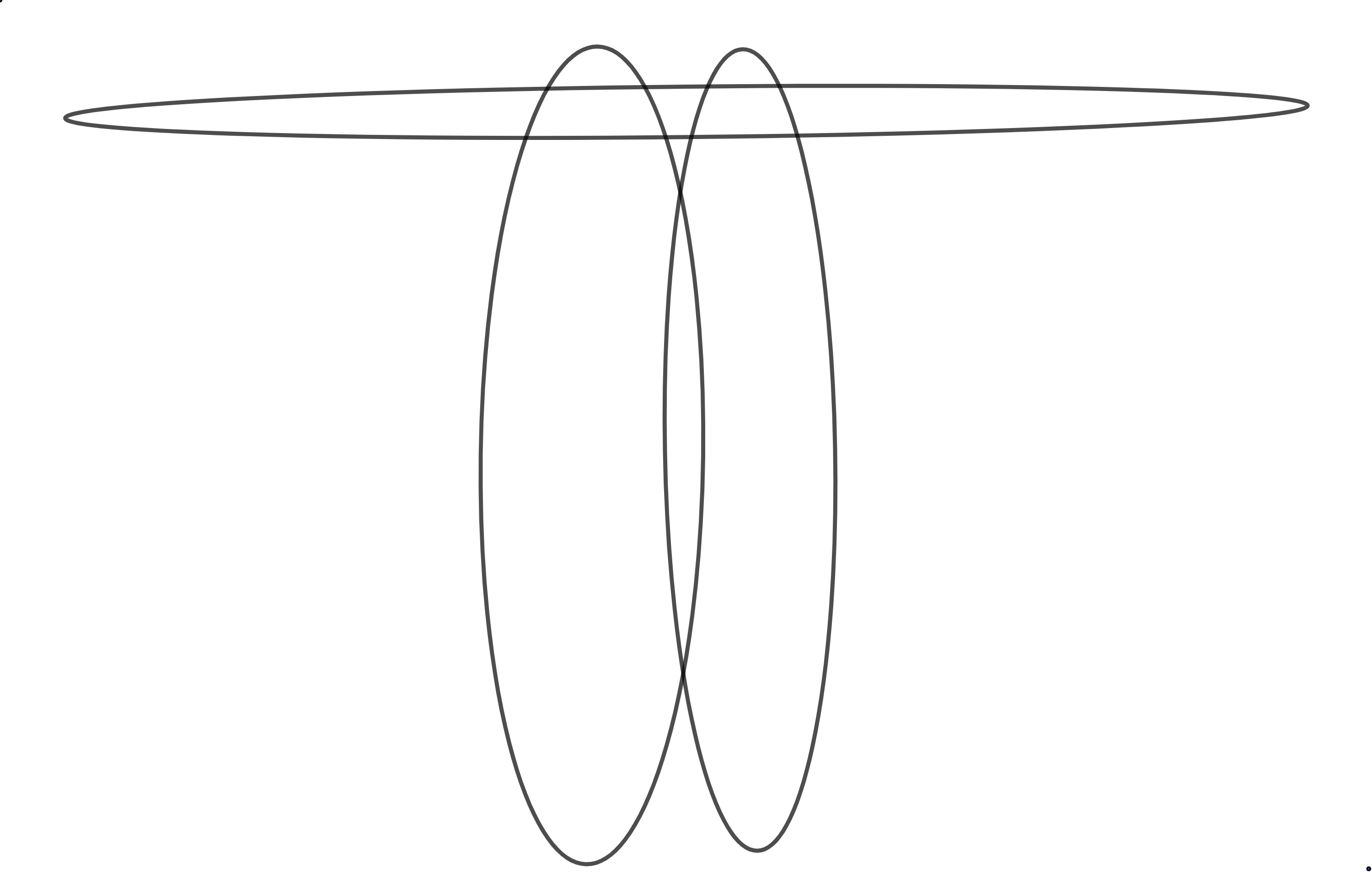}
         \caption{433}
         \label{fig:244_433}
     \end{subfigure}
     \hfill
     \begin{subfigure}[b]{0.38\textwidth}
         \centering
         \includegraphics[width=0.49\textwidth]{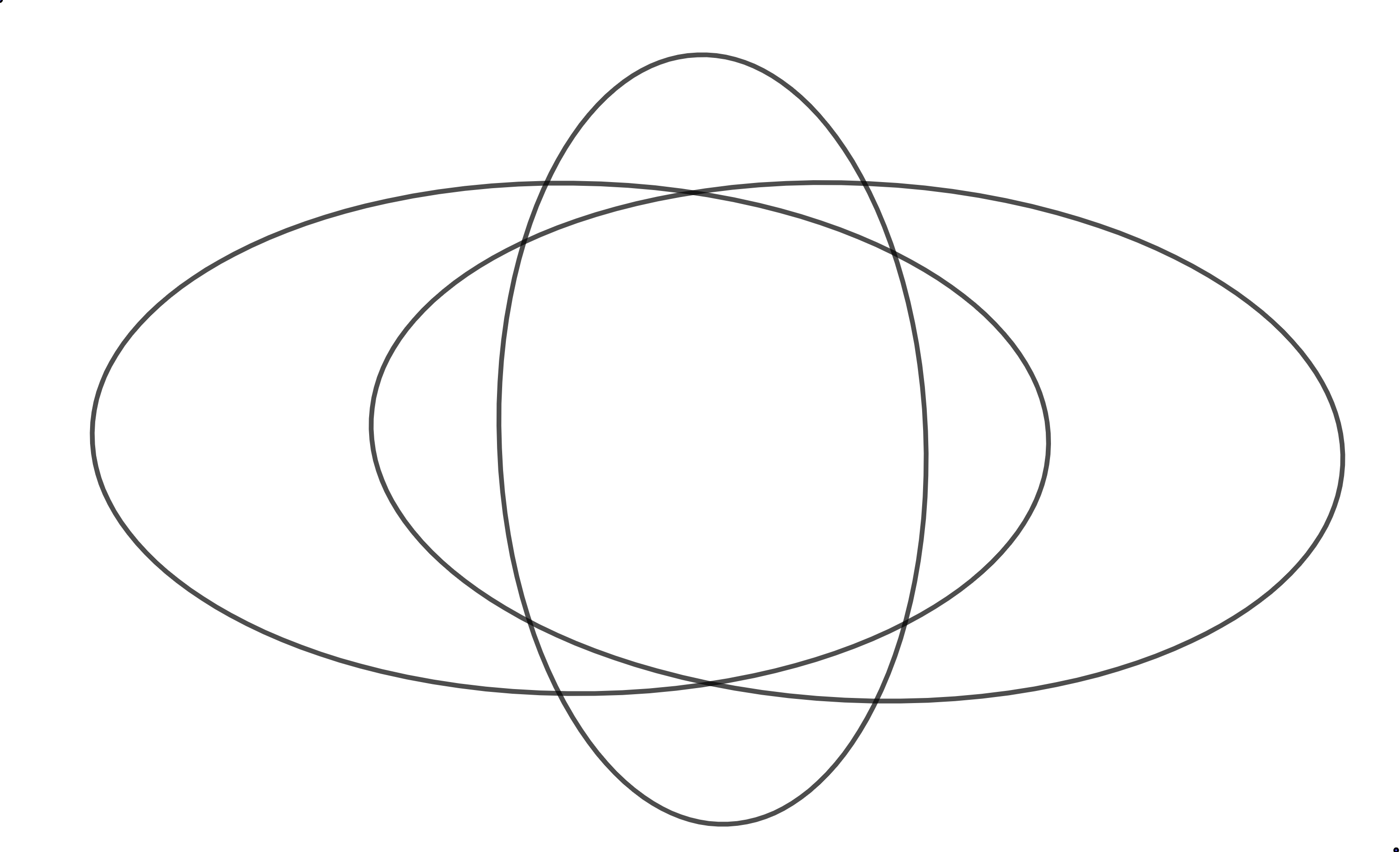}
         \includegraphics[width=0.49\textwidth]{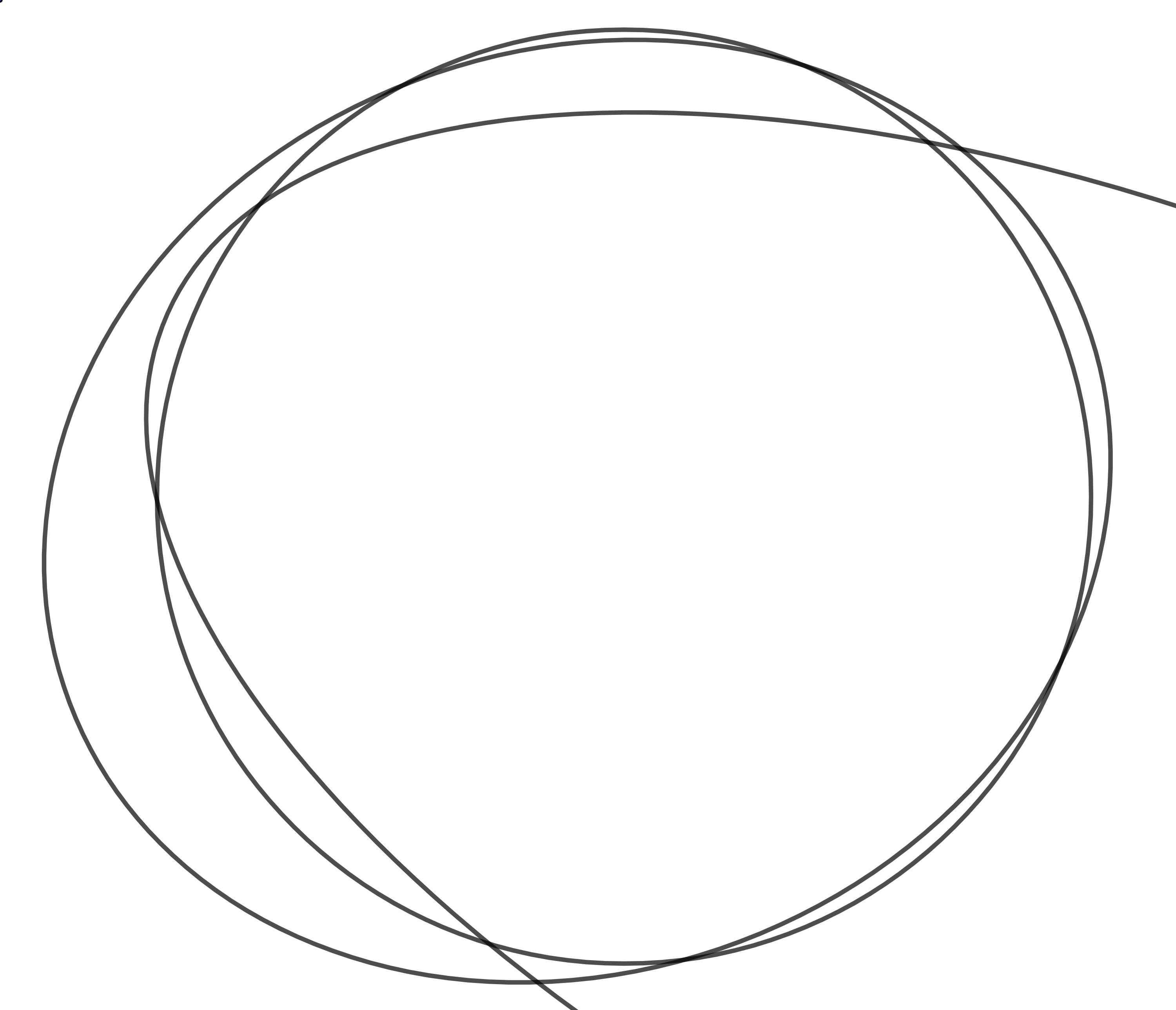}
         \caption{211}
         \label{fig:244_211}
     \end{subfigure}     \hfill
     \begin{subfigure}[b]{0.38\textwidth}
         \centering
         \includegraphics[width=0.48\textwidth]{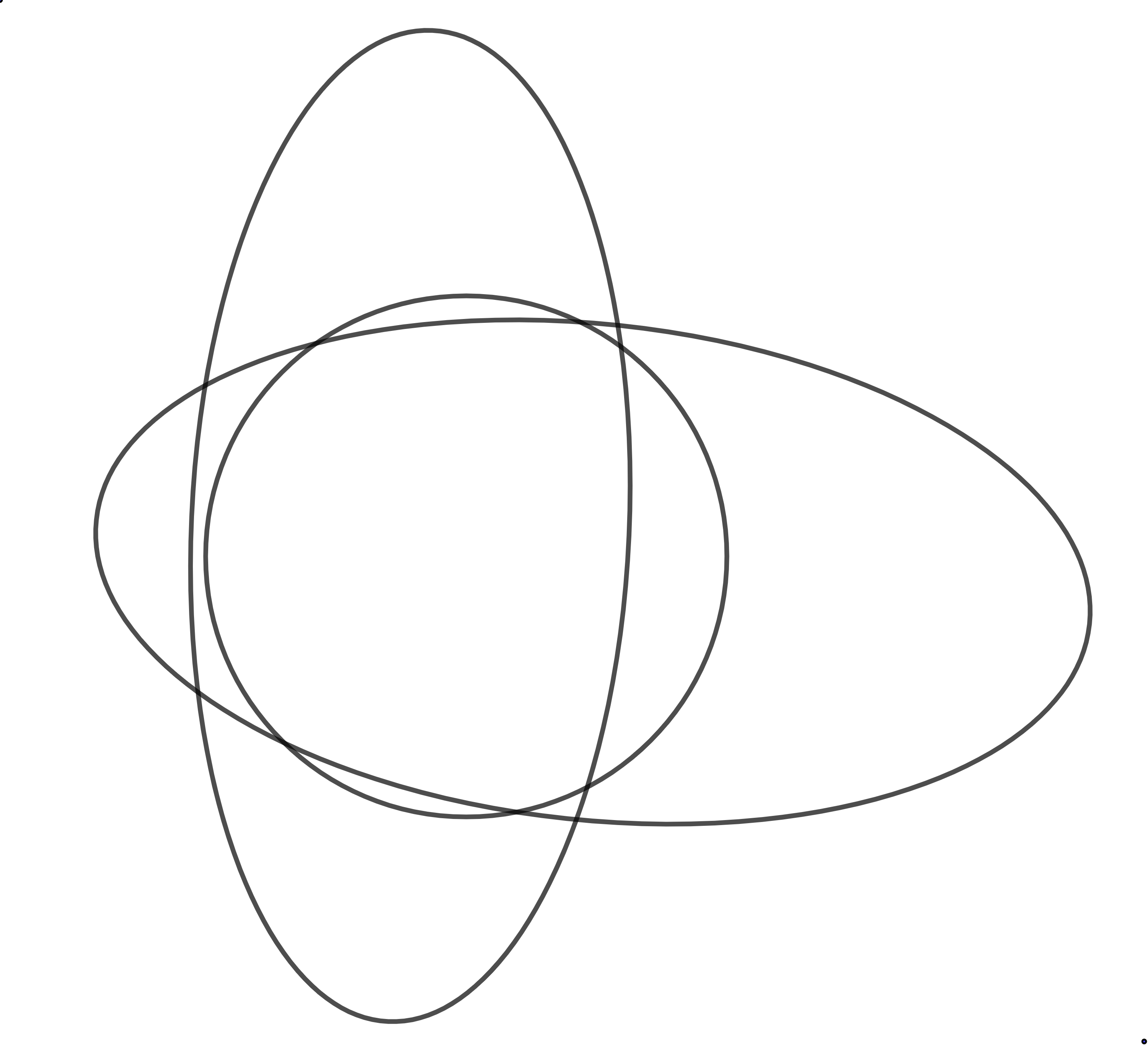}
         \includegraphics[height=2cm]{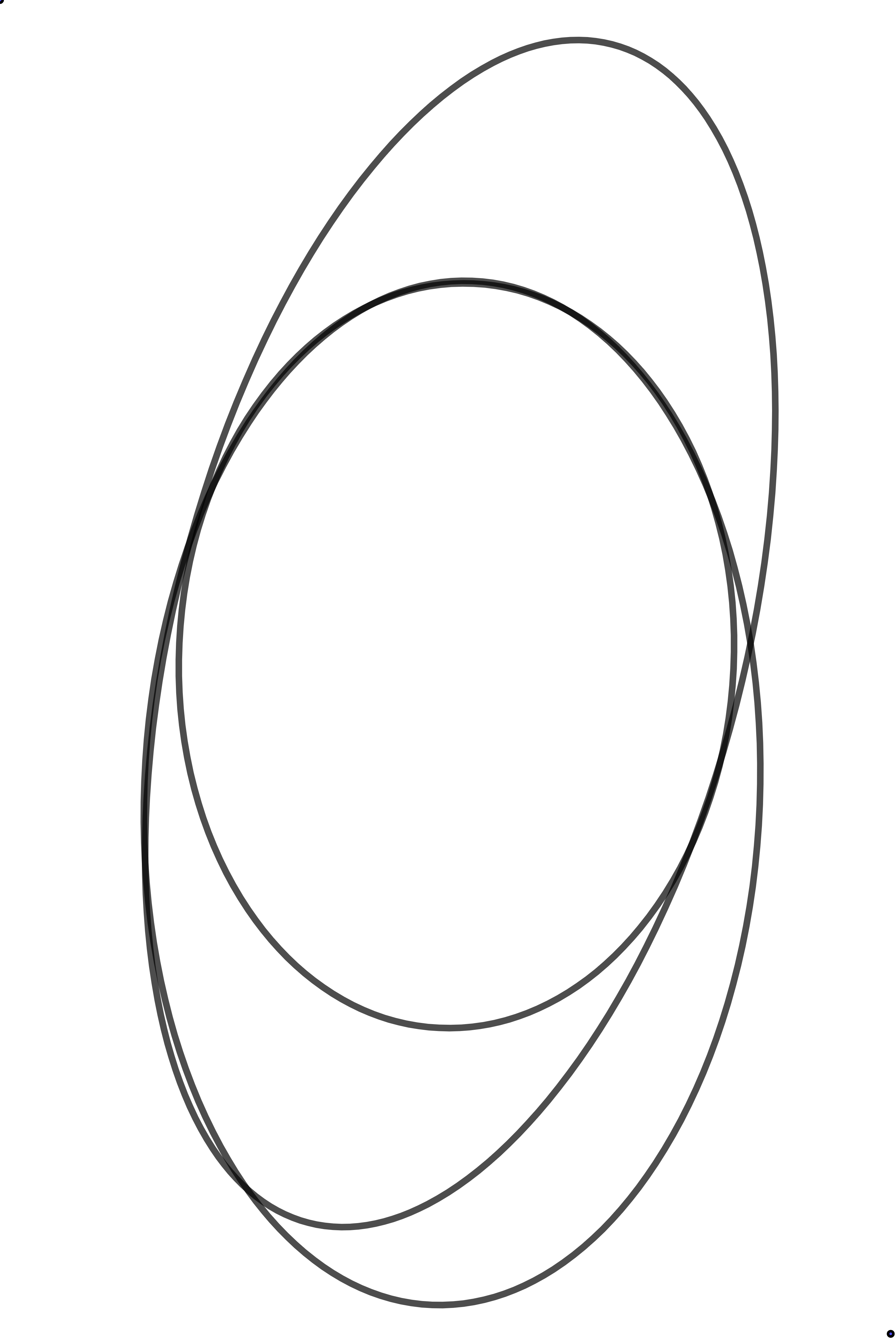}
         \caption{220. Right sketched in Fig. \ref{fig:244_220_top}.}
         \label{fig:244_220}
     \end{subfigure}     
     \begin{subfigure}[b]{0.38\textwidth}
         \centering
         \frame{\includegraphics[height=2cm,width=2cm]{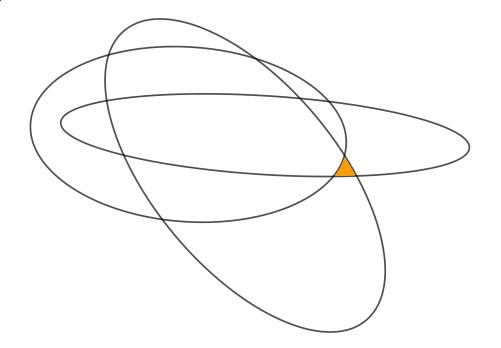}}
         \frame{\includegraphics[height=2cm,width=2cm]{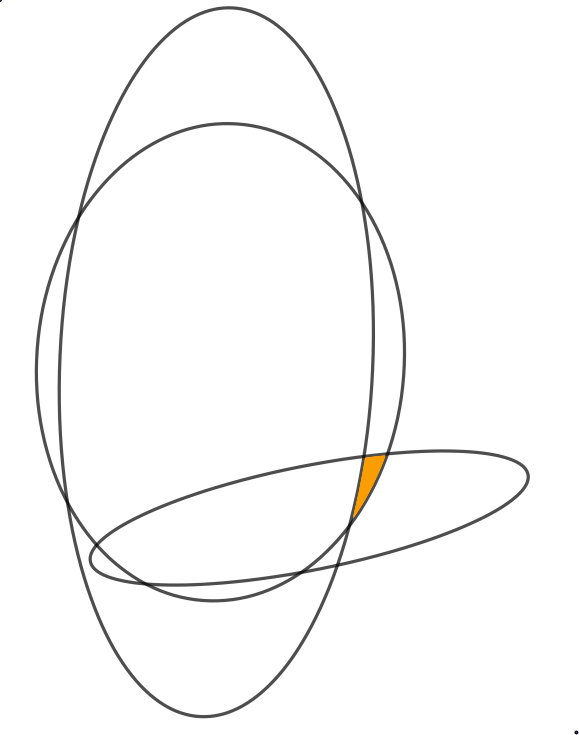}}
         \caption{221}
         \label{fig:244_221}
     \end{subfigure}  \hfill
     \begin{subfigure}[b]{0.57\textwidth}
         \centering
         \includegraphics[width=0.32\textwidth]{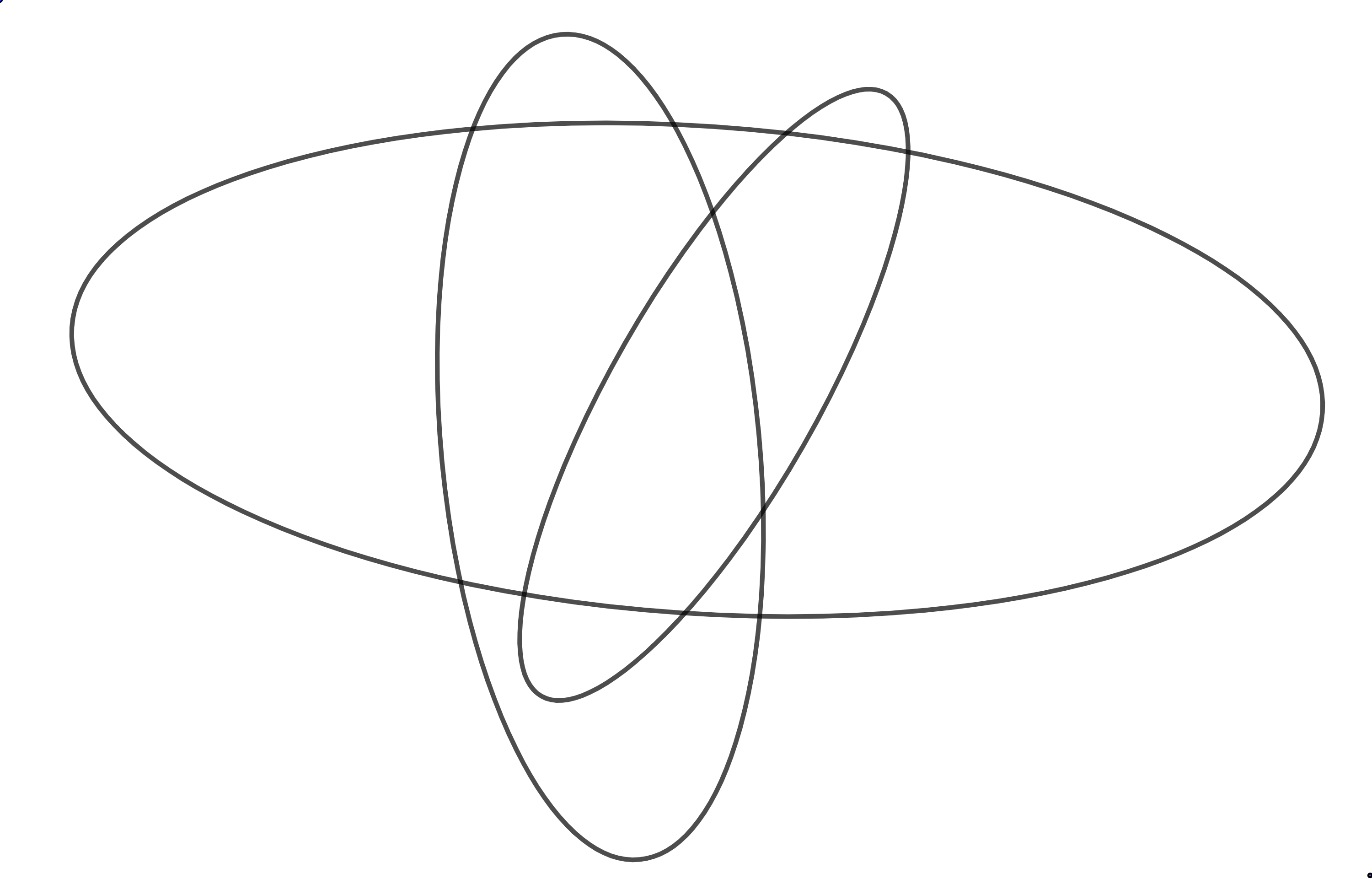} 
         \frame{\includegraphics[height=2cm,width=2cm]{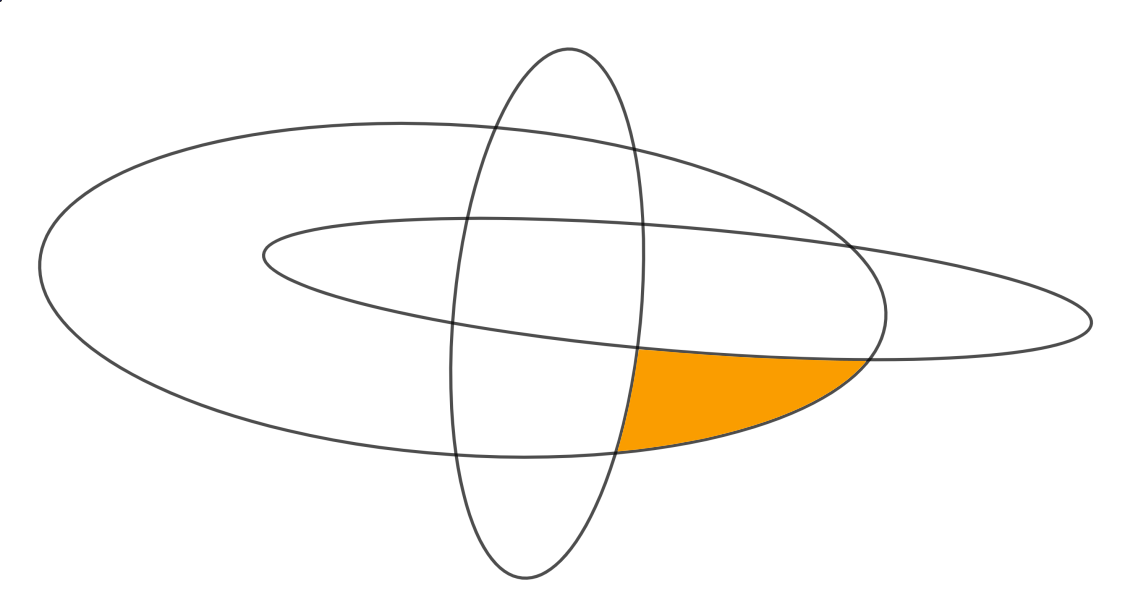}}
         \frame{\includegraphics[height=2cm,width=2cm]{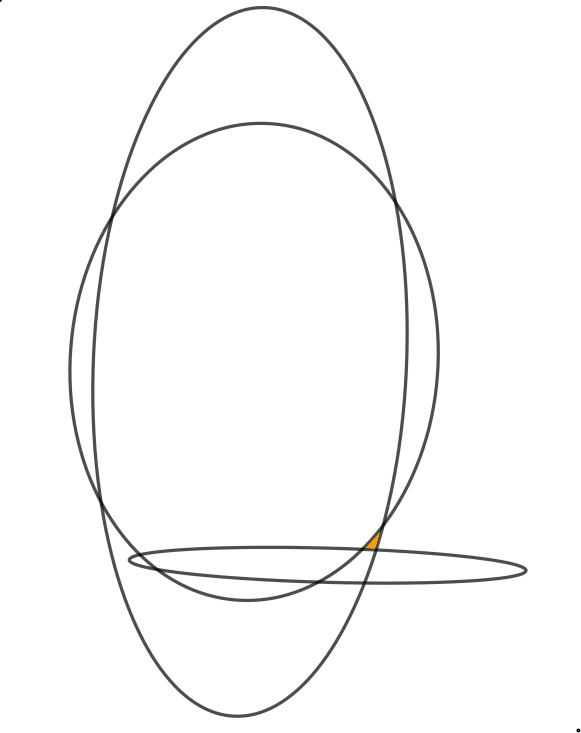}}
         \caption{321}
         \label{fig:244_321}
     \end{subfigure}
        \caption{Intersection type (244). The subcaptions show the outer arc type. 
        Problematic configurations are framed: The unique (up to symmetry) problematic polycon is orange.}
        \label{fig:244}
\end{figure}

\begin{figure}[htb]
     \centering
     \begin{subfigure}[b]{0.22\textwidth}
         \centering
         \includegraphics[height=2cm]{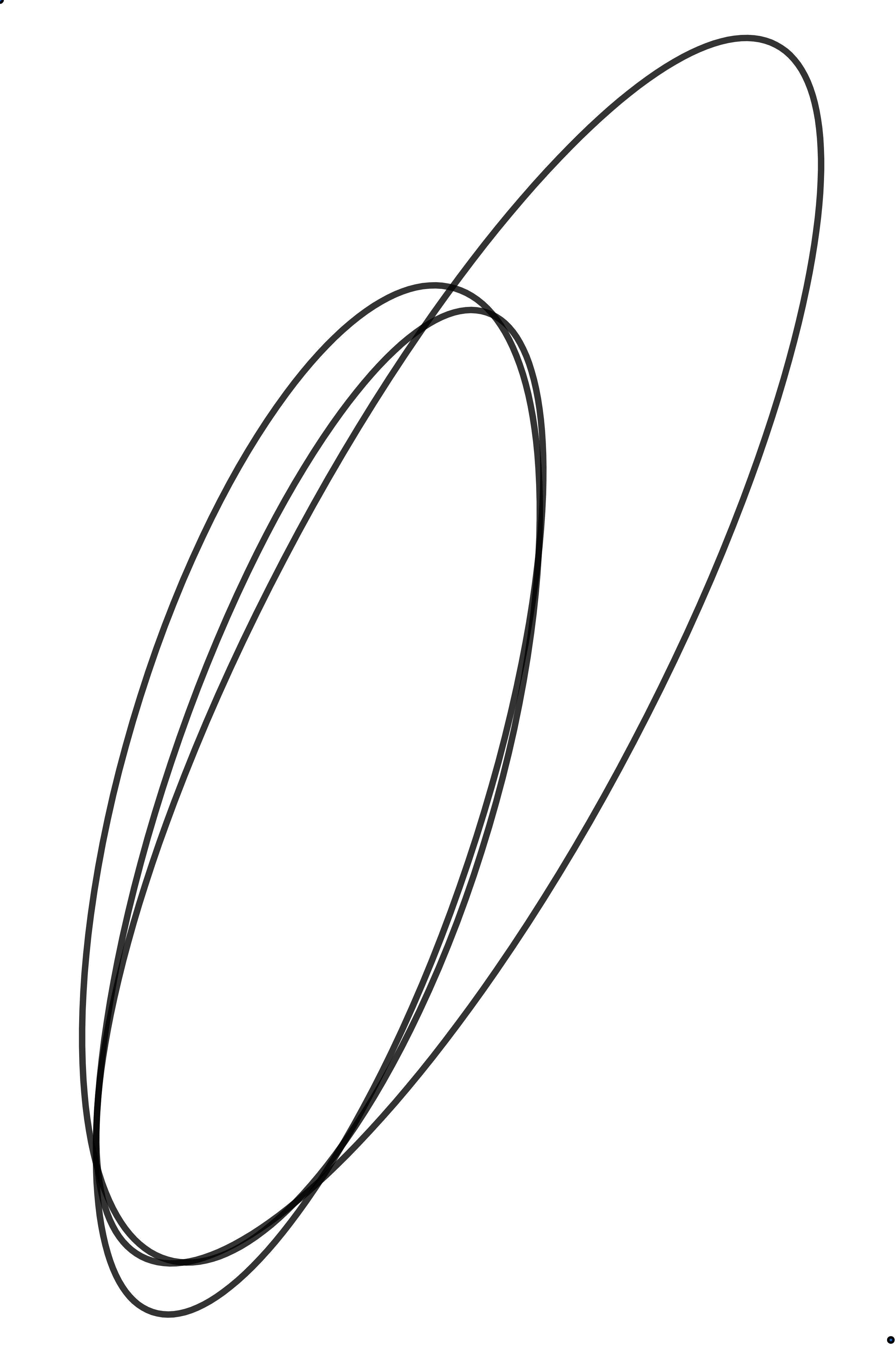}
         \caption{111.  See  Fig.\ref{fig:444_111_top}.}
         \label{fig:444_111}
     \end{subfigure}
     \hfill
     \begin{subfigure}[b]{0.19\textwidth}
         \centering
         \includegraphics[width=\textwidth]{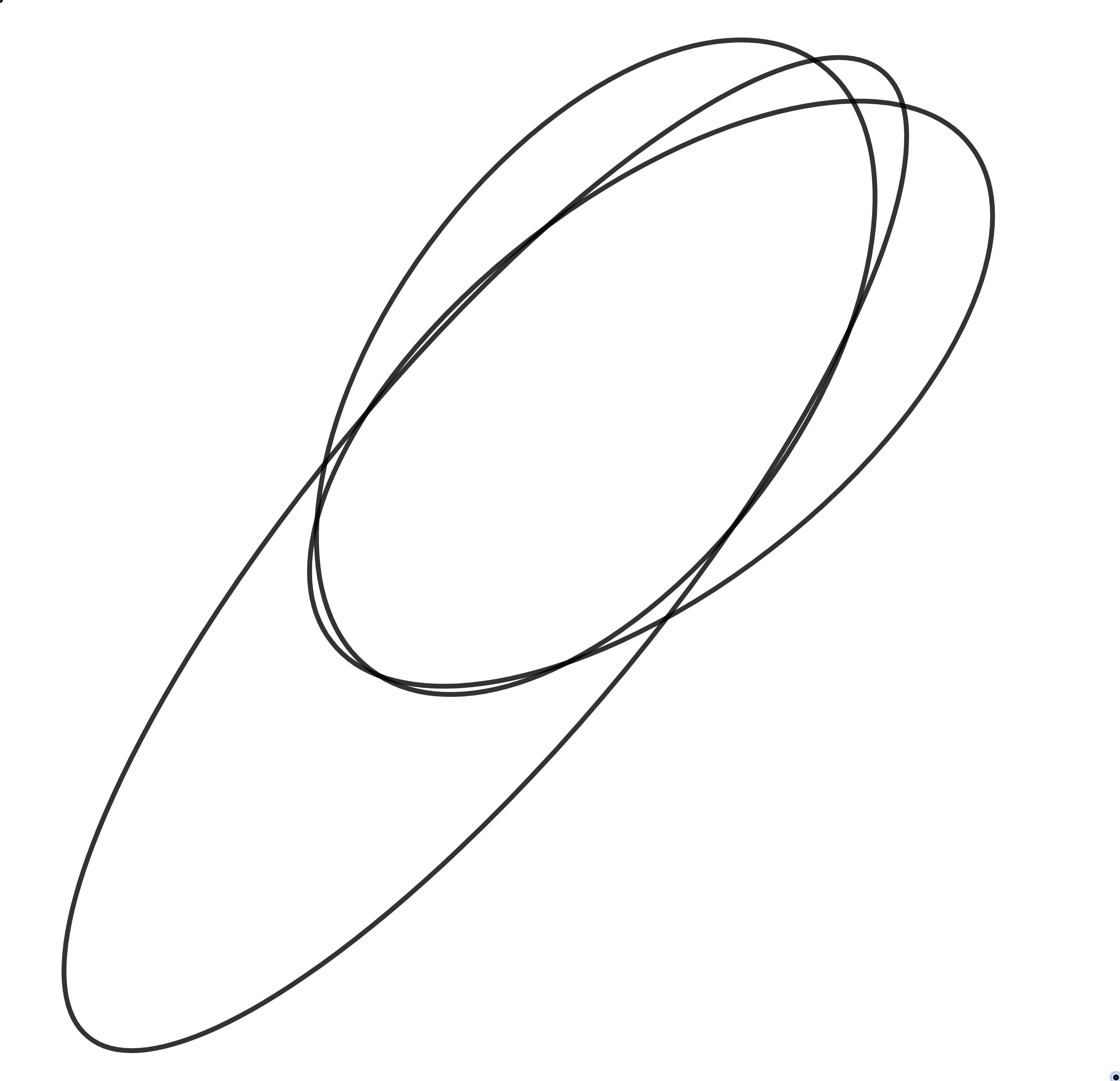}
         \caption{211}
         \label{fig:444_211}
     \end{subfigure}
     \hfill
     \begin{subfigure}[b]{0.18\textwidth}
         \centering
         \includegraphics[width=\textwidth]{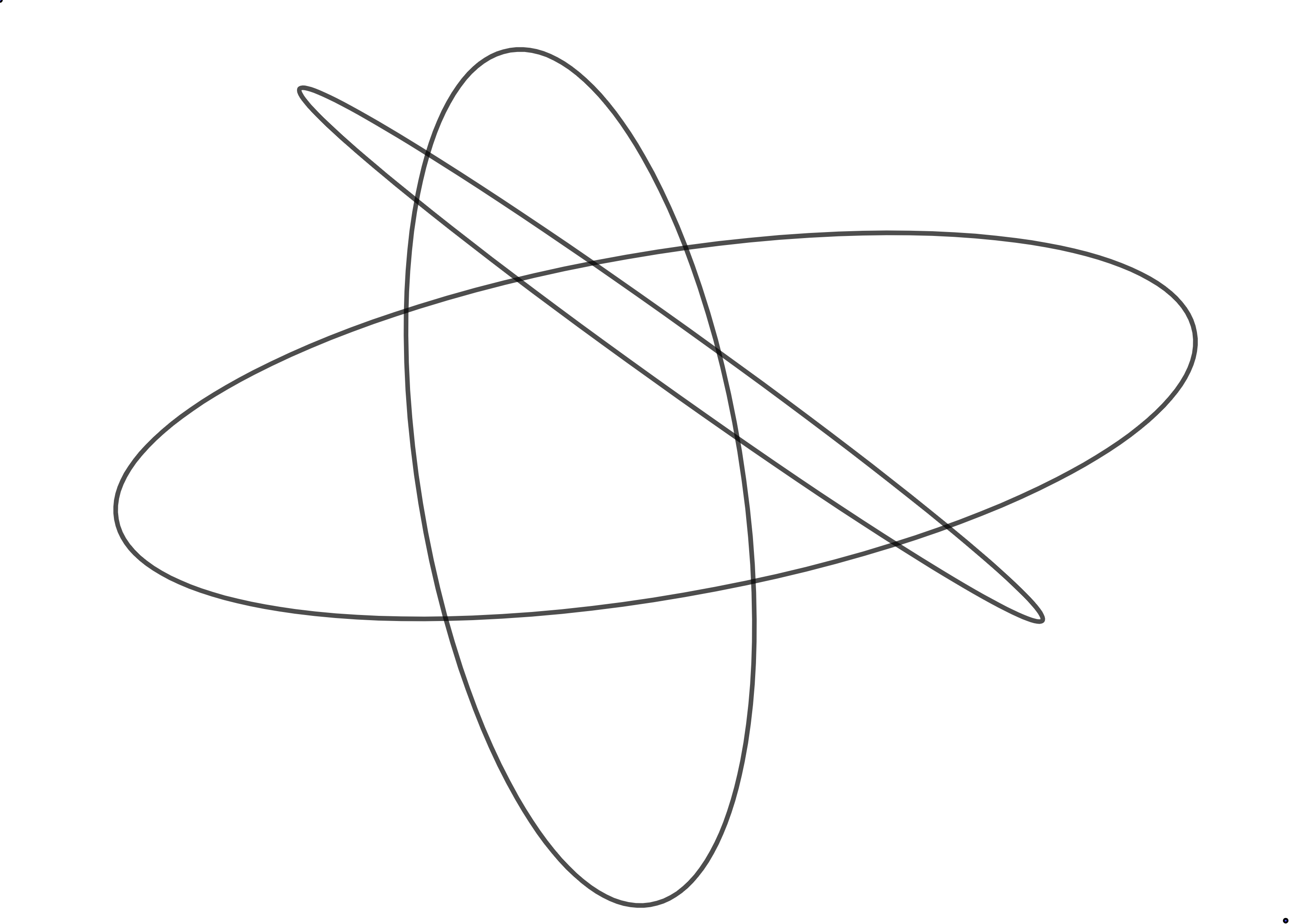}
         \caption{332}
         \label{fig:444_332}
     \end{subfigure}
     \begin{subfigure}[b]{0.18\textwidth}
         \centering
         \includegraphics[width=\textwidth]{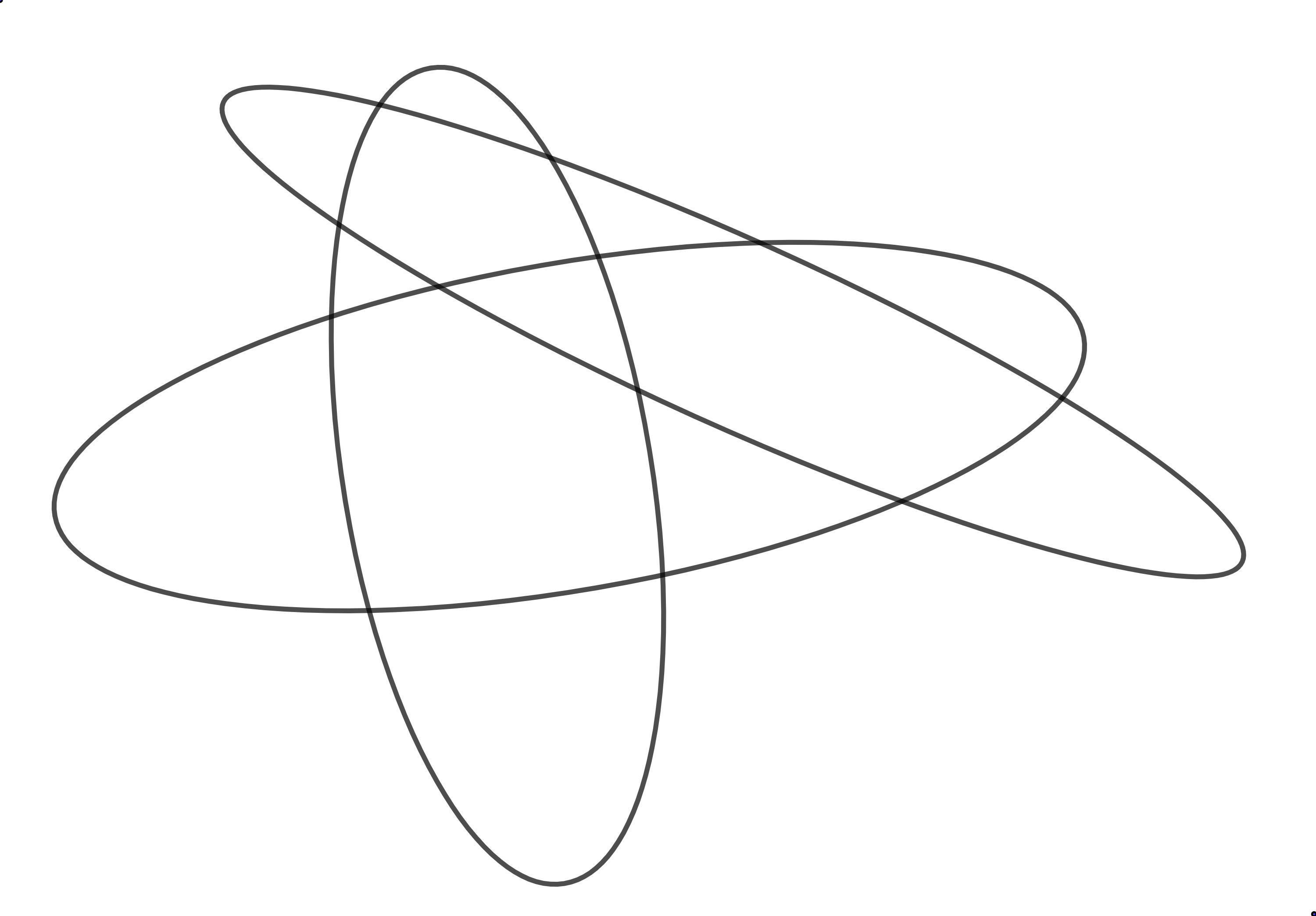}
         \caption{333}
         \label{fig:444_333}
     \end{subfigure}
     \begin{subfigure}[b]{0.18\textwidth}
         \centering
         \includegraphics[width=\textwidth]{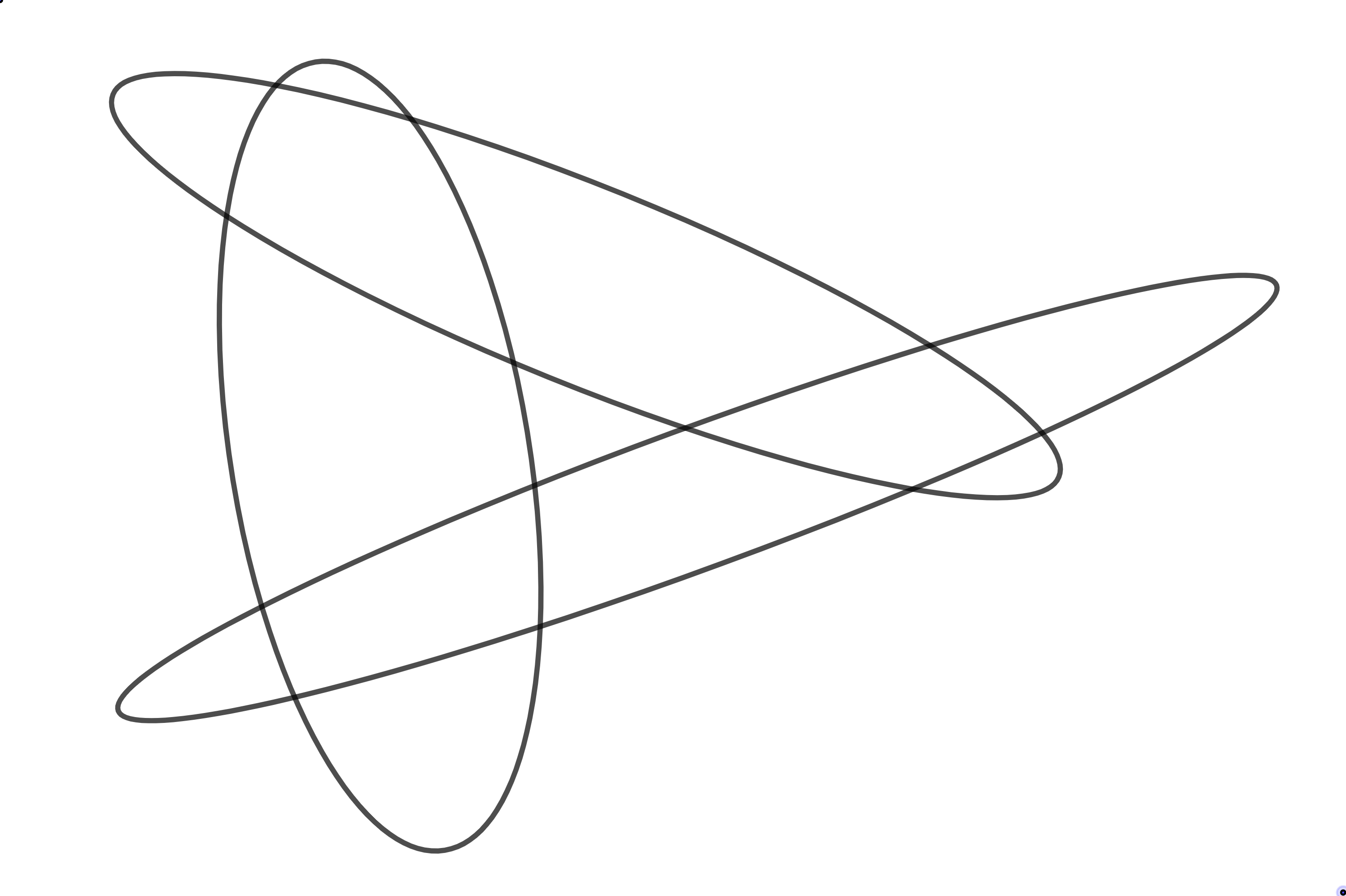}
         \caption{444}
         \label{fig:444_444}
     \end{subfigure}
     \begin{subfigure}[b]{0.48\textwidth}
         \centering
         \includegraphics[width=0.49\textwidth]{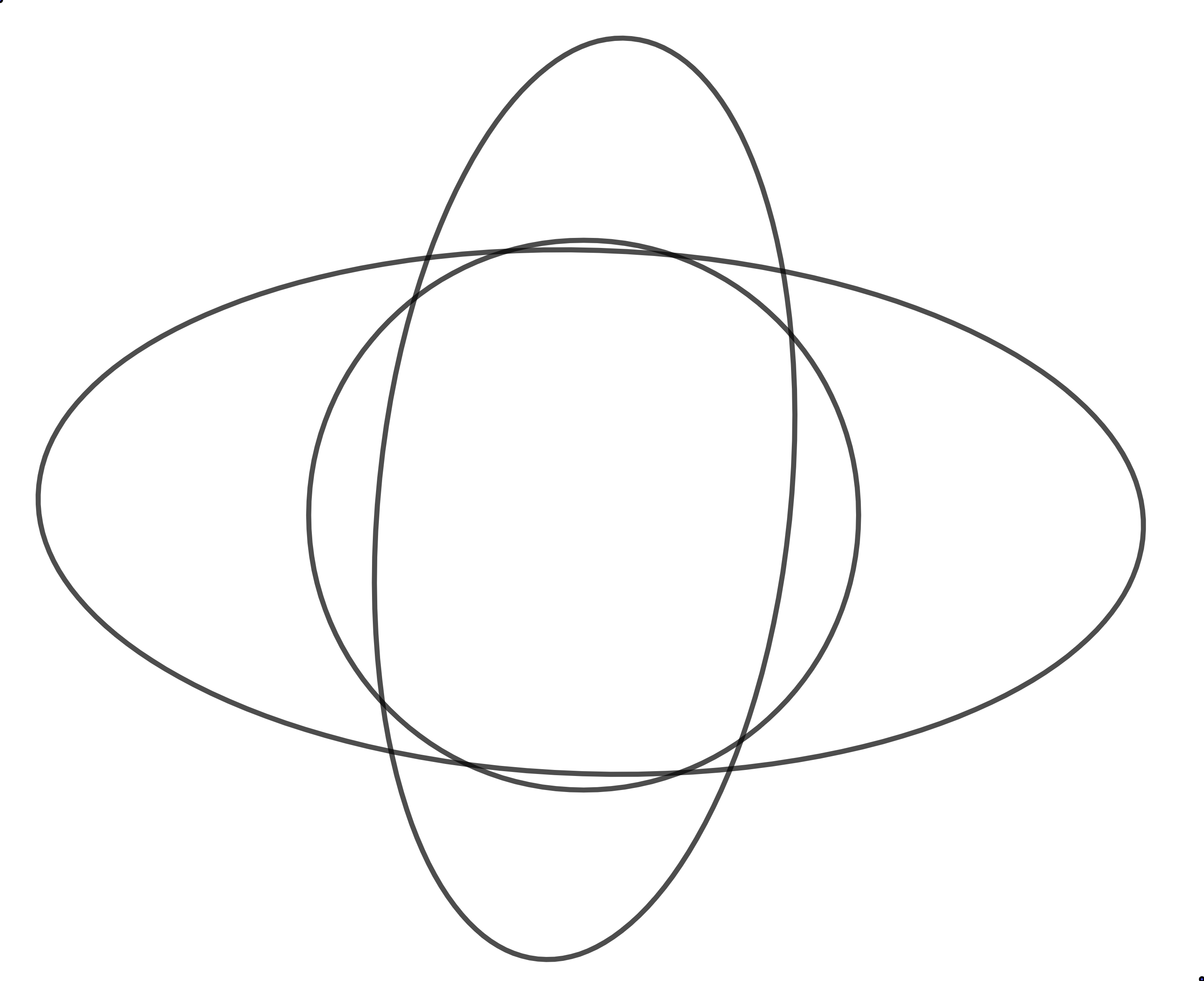}
         \includegraphics[height=2cm]{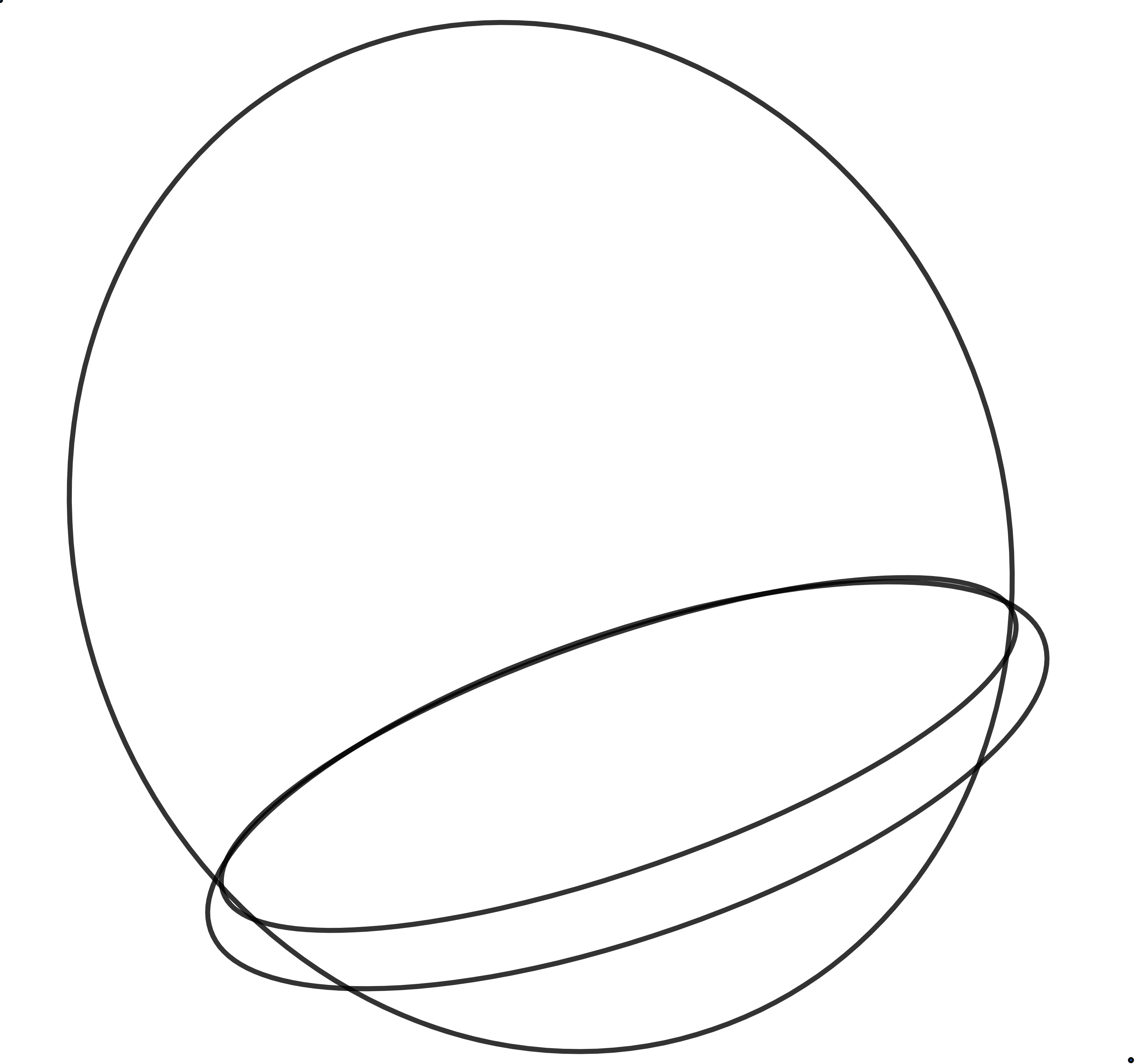}
         \caption{220. Right sketched in Fig. \ref{fig:444_220_top}.}
         \label{fig:444_220}
     \end{subfigure}\hfill
     \begin{subfigure}[b]{0.48\textwidth}
         \centering
         \includegraphics[width=0.49\textwidth]{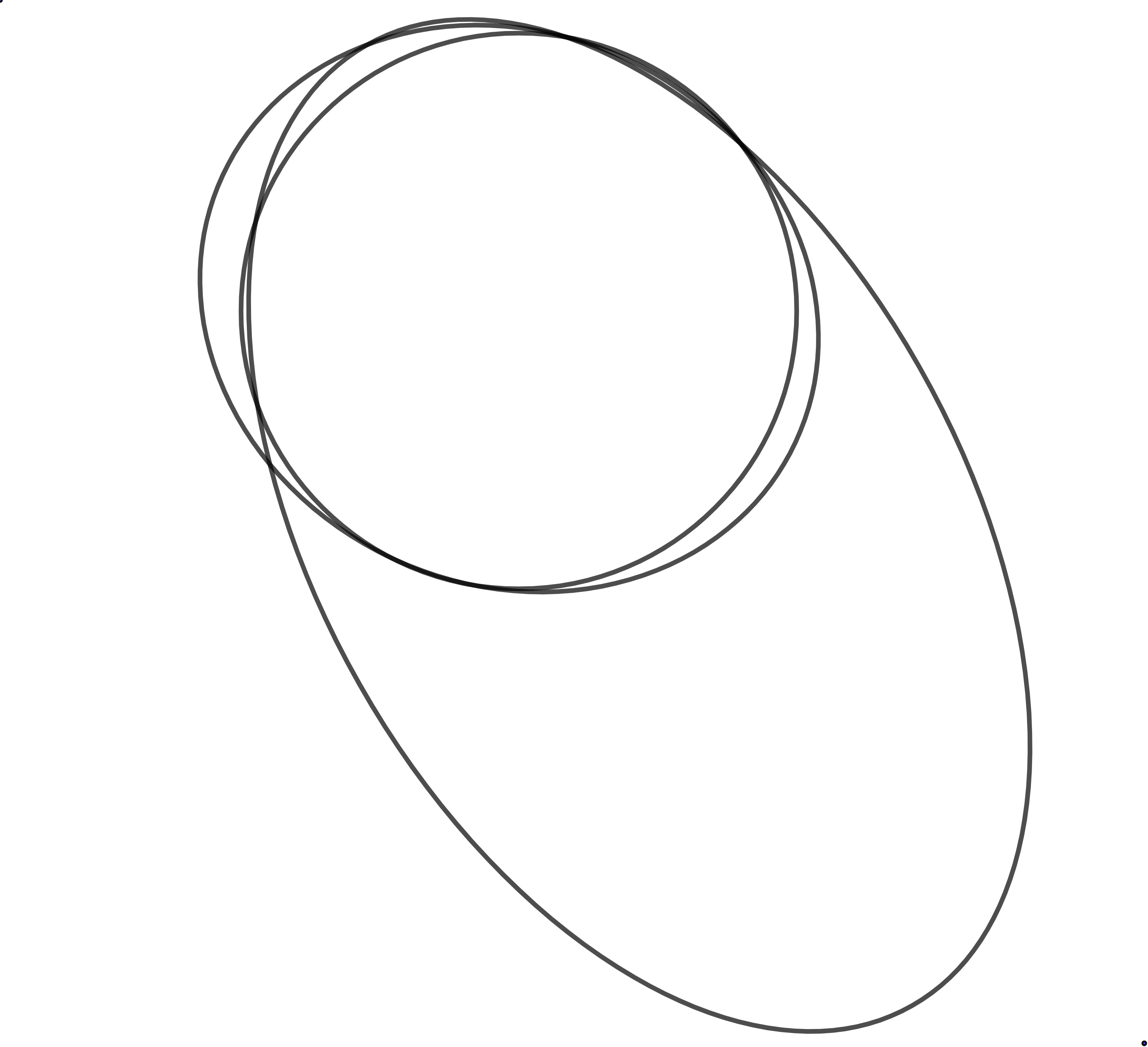}
         \includegraphics[width=0.49\textwidth]{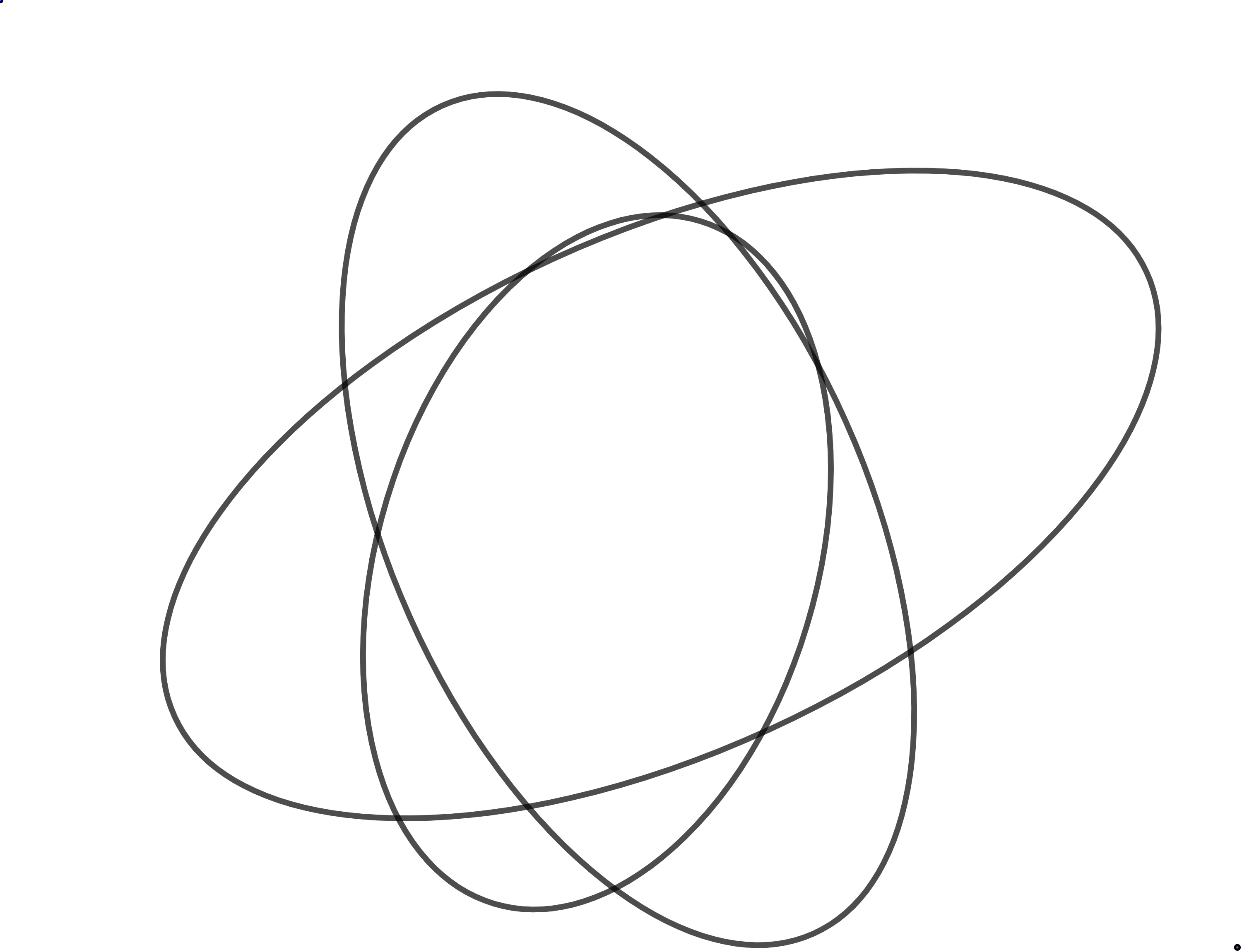}
         \caption{221. Left sketched in Fig. \ref{fig:444_111_top}.}
         \label{fig:444_221}
     \end{subfigure}
     \begin{subfigure}[b]{0.48\textwidth}
         \centering
         \includegraphics[height=2cm]{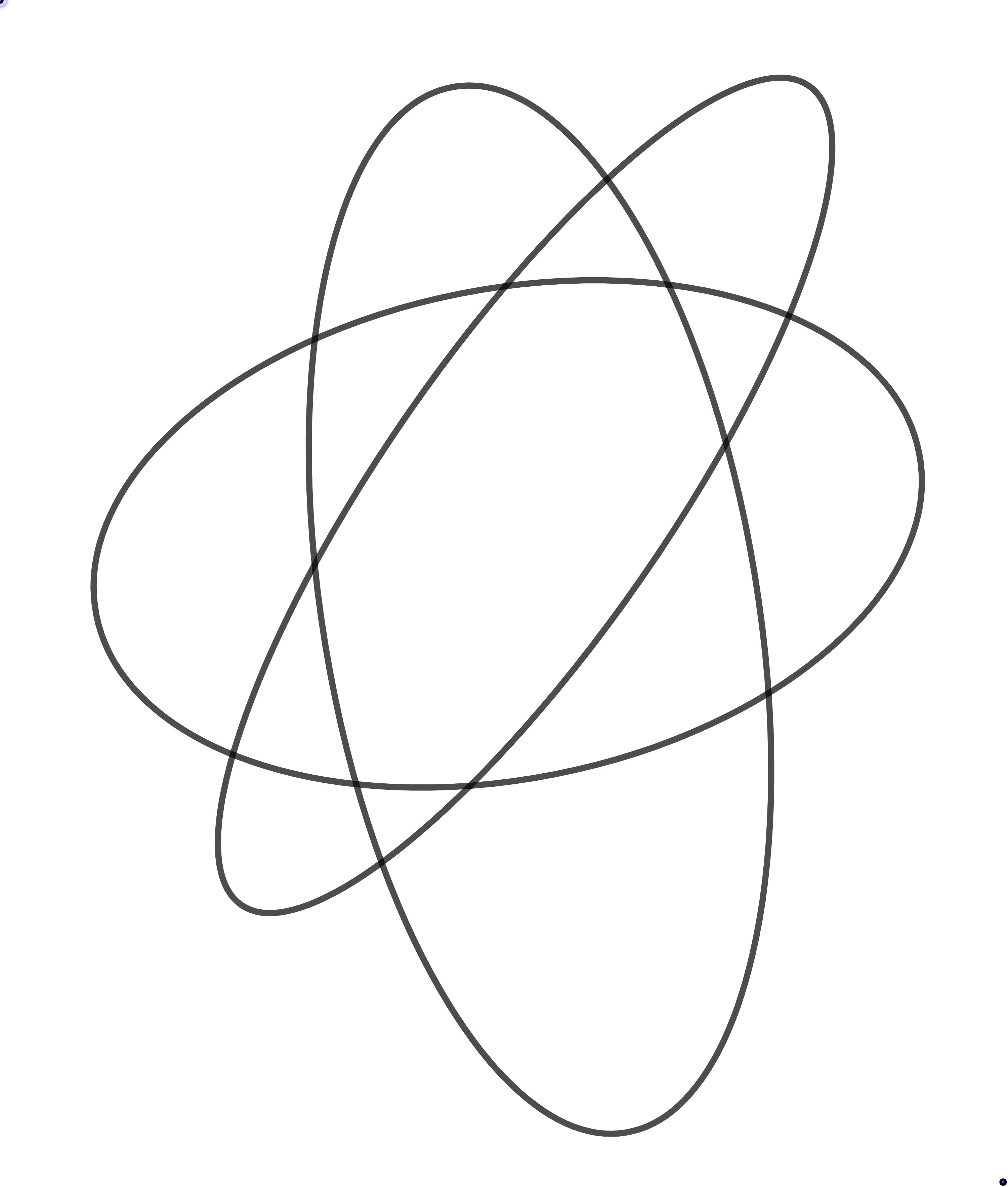}
         \includegraphics[width=0.49\textwidth]{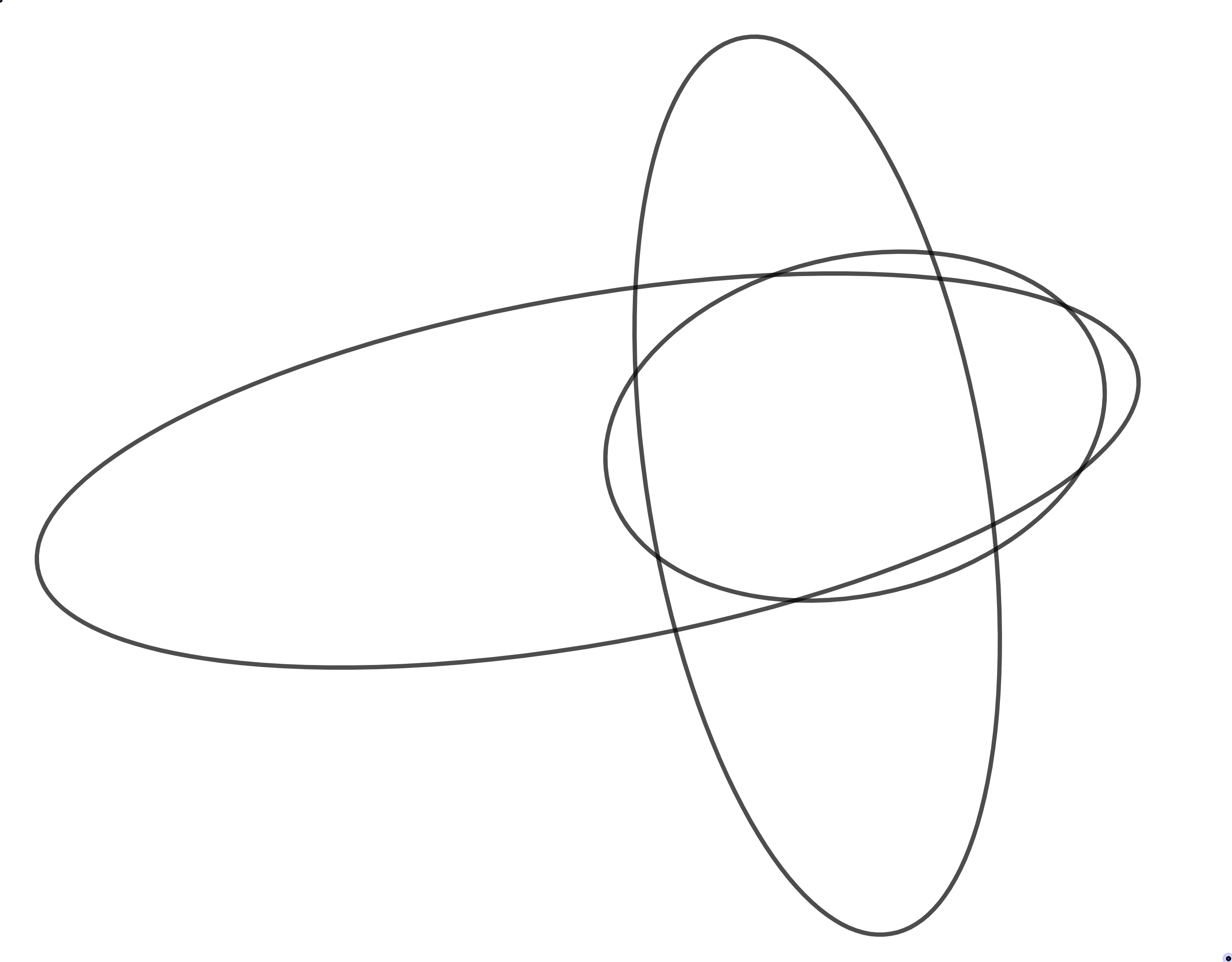}
         \caption{222}
         \label{fig:444_222}
     \end{subfigure}\hfill
     \begin{subfigure}[b]{0.48\textwidth}
         \centering
         \includegraphics[width=0.49\textwidth]{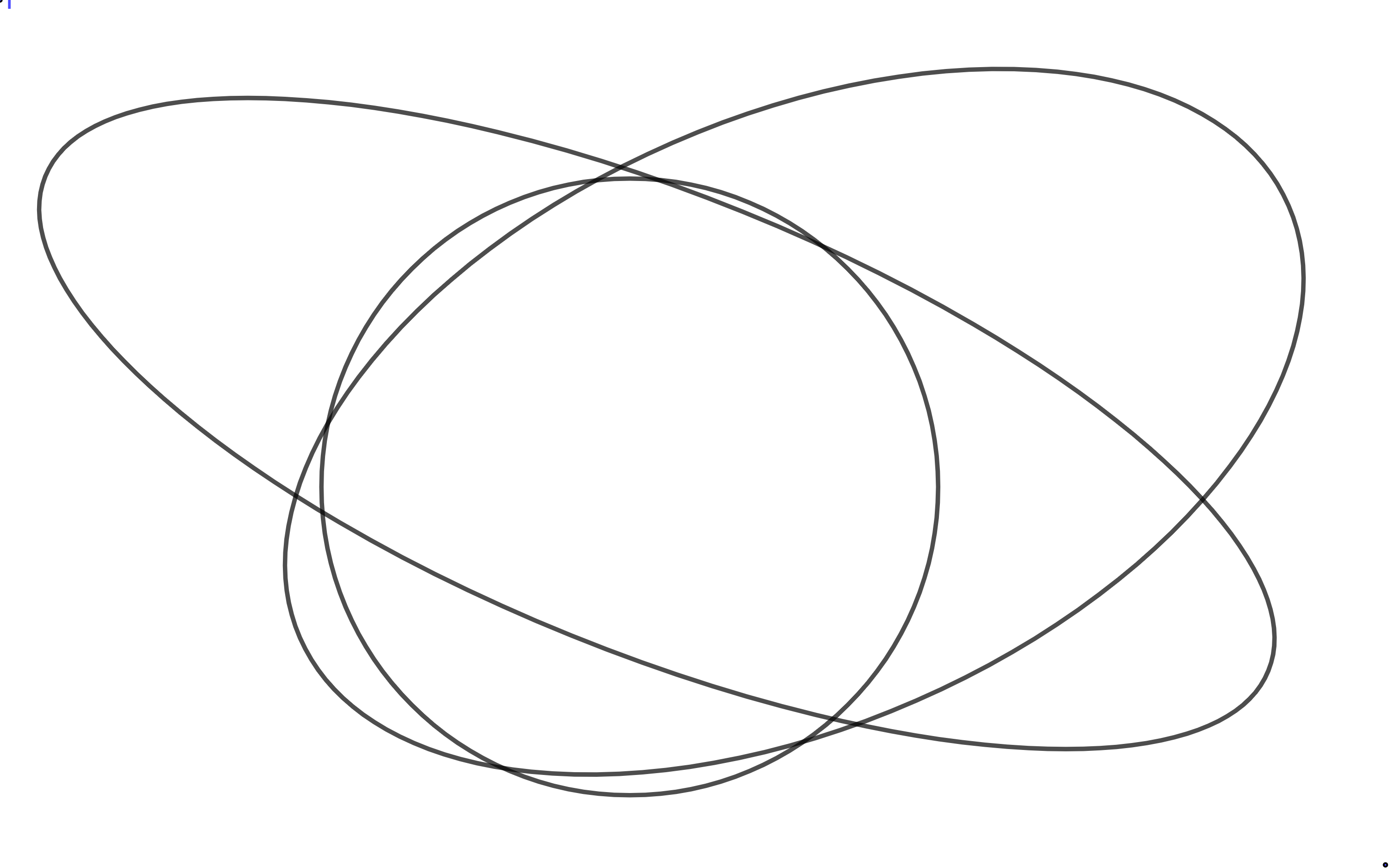}
         \includegraphics[height=2cm]{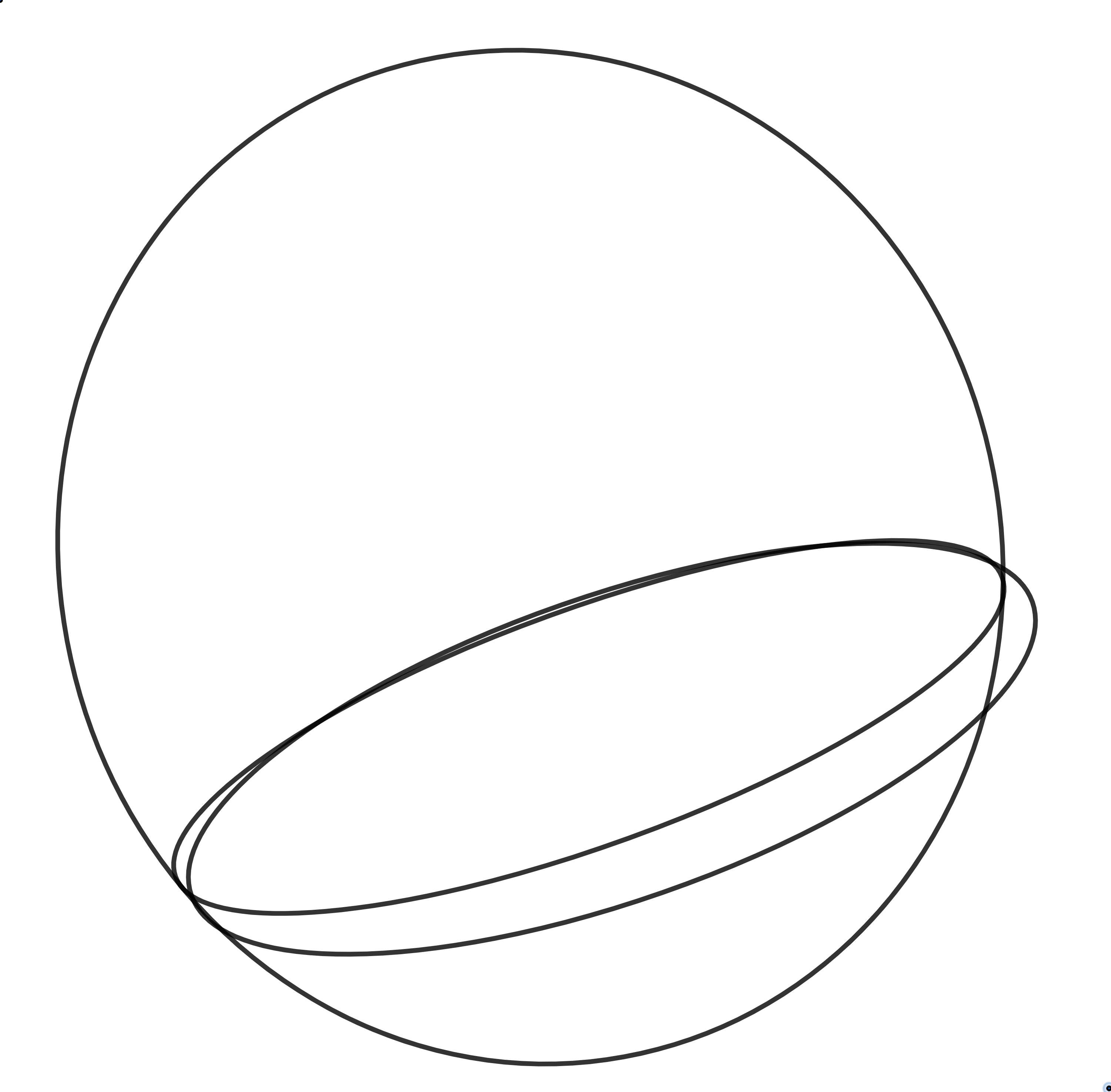}
         \caption{321. Right sketched in Fig. \ref{fig:444_111_top}.}
         \label{fig:444_321}
     \end{subfigure}
     \begin{subfigure}[b]{0.48\textwidth}
         \centering
         \frame{\includegraphics[height=2cm]{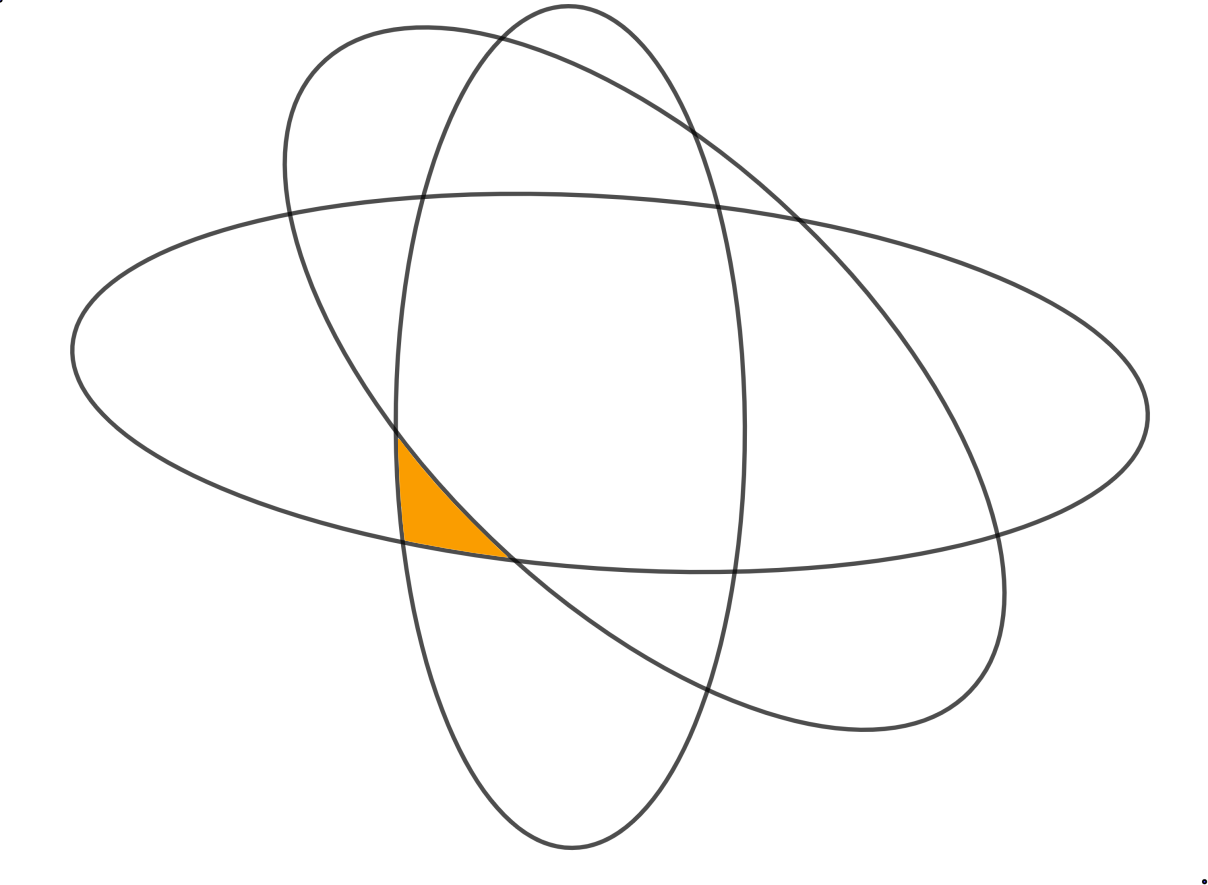}}
         \frame{\includegraphics[height=2cm]{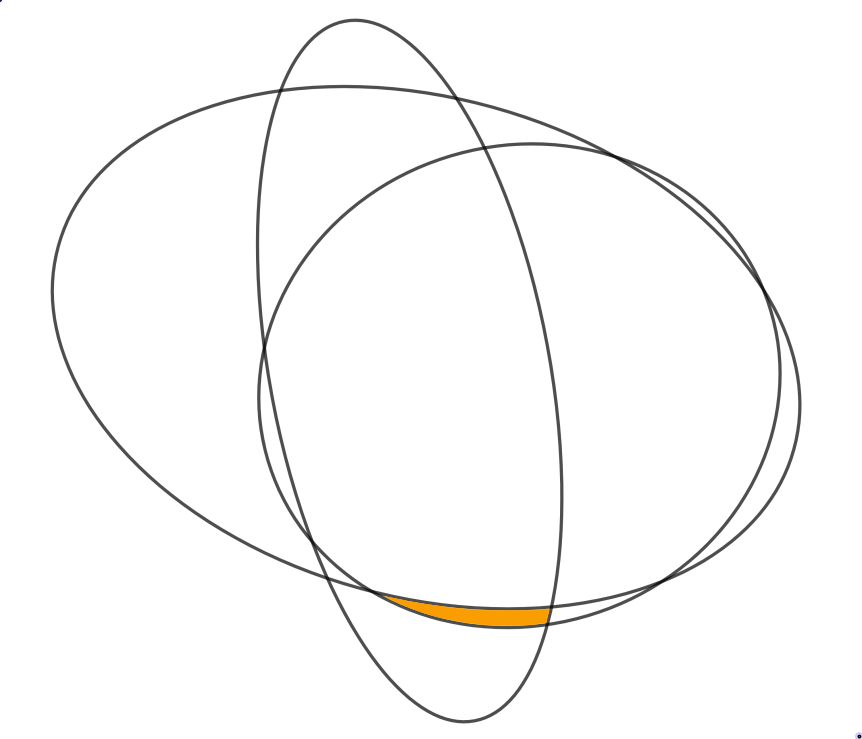}}
         \caption{322}
         \label{fig:444_322}
     \end{subfigure}\hfill
     \begin{subfigure}[b]{0.48\textwidth}
         \centering
         \frame{\includegraphics[height=2cm]{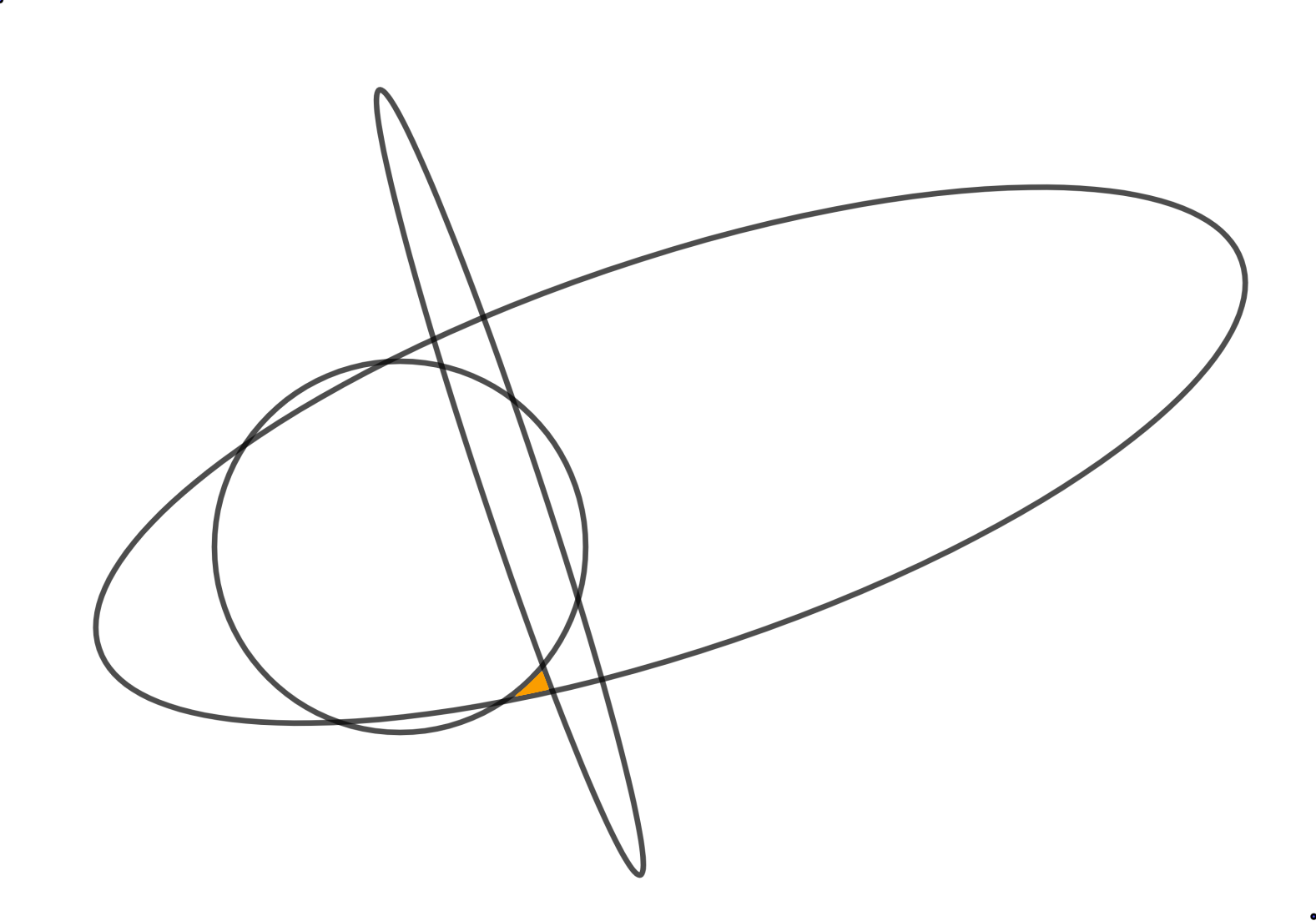}}
         \frame{\includegraphics[height=2cm]{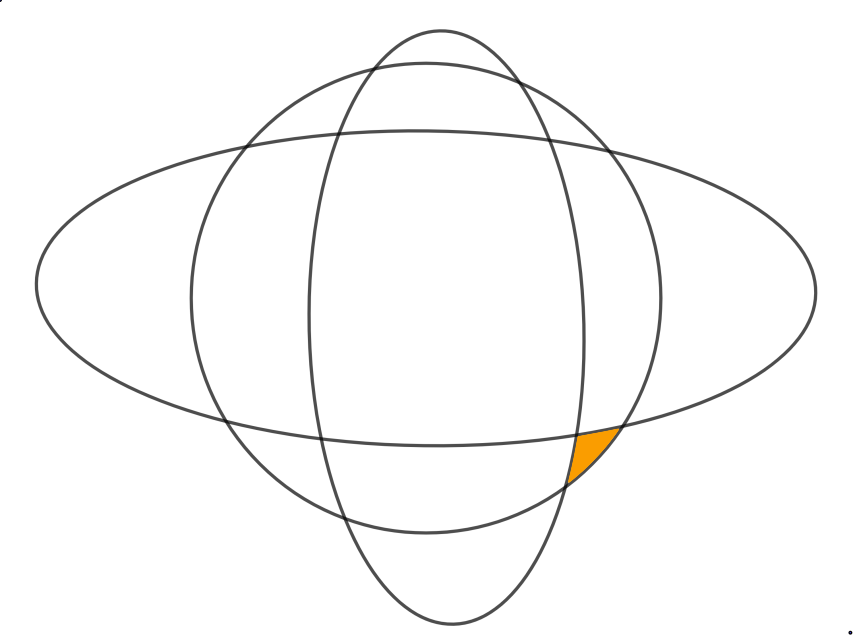}}
         \caption{422}
         \label{fig:444_422}
     \end{subfigure}
        \caption{Intersection type (444). The subcaptions show the outer arc type.
        Problematic configurations are framed: The unique (up to symmetry) problematic polycon is orange.}
        \label{fig:444}
\end{figure}

\begin{figure}[htb]
     \centering
         \begin{subfigure}[b]{0.19\textwidth}
         \centering
         \includegraphics[width=2cm,height=2.8cm]{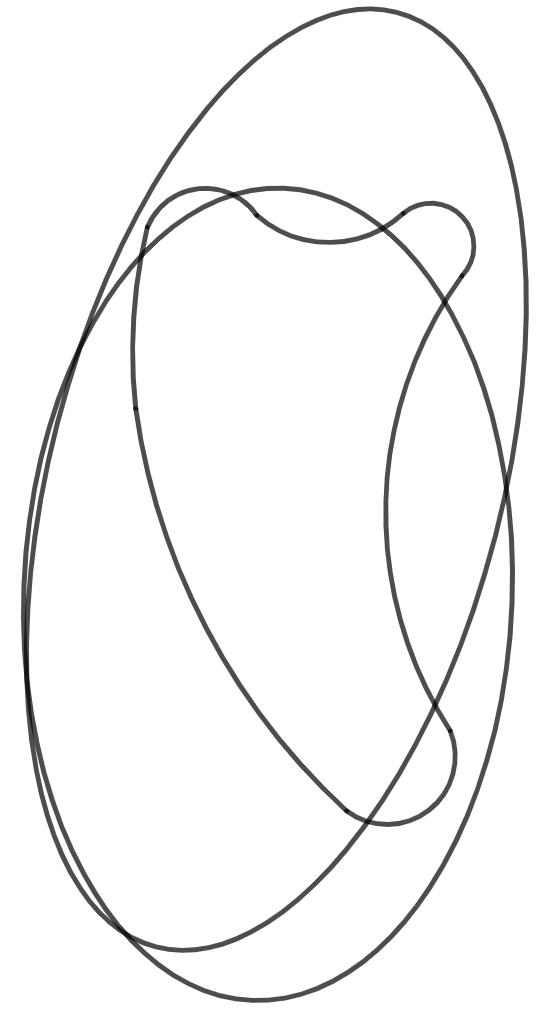}
         \caption{Fig. \!\ref{fig:244_220}, right.}
         \label{fig:244_220_top}
     \end{subfigure} \hfill
     \begin{subfigure}[b]{0.19\textwidth}
         \centering
         \includegraphics[width=2cm,height=2.8cm]{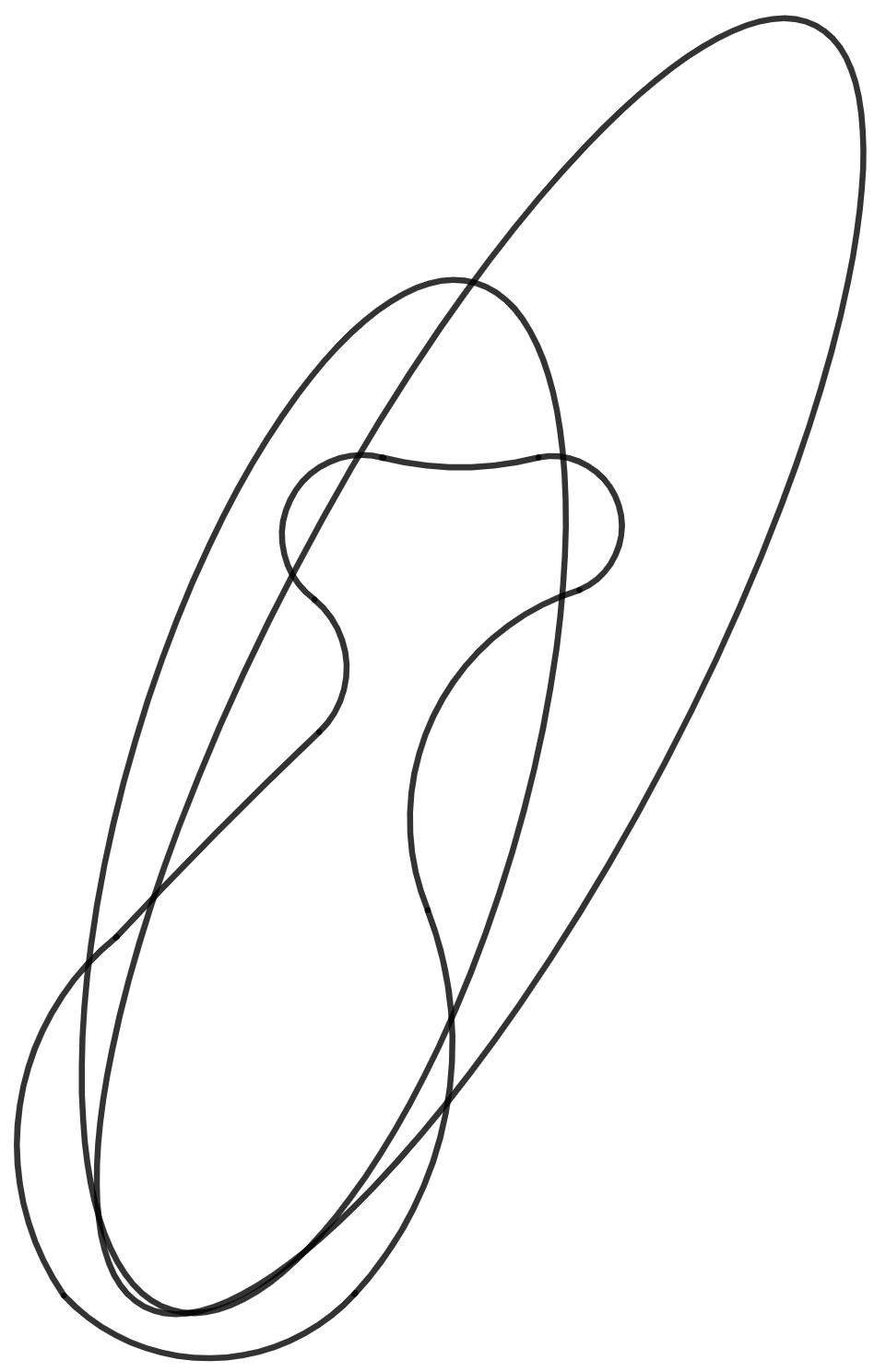}
         \caption{\Cref{fig:444_111}.}
         \label{fig:444_111_top}
     \end{subfigure}
     \hfill
         \begin{subfigure}[b]{0.19\textwidth}
         \centering
         \includegraphics[width=\textwidth]{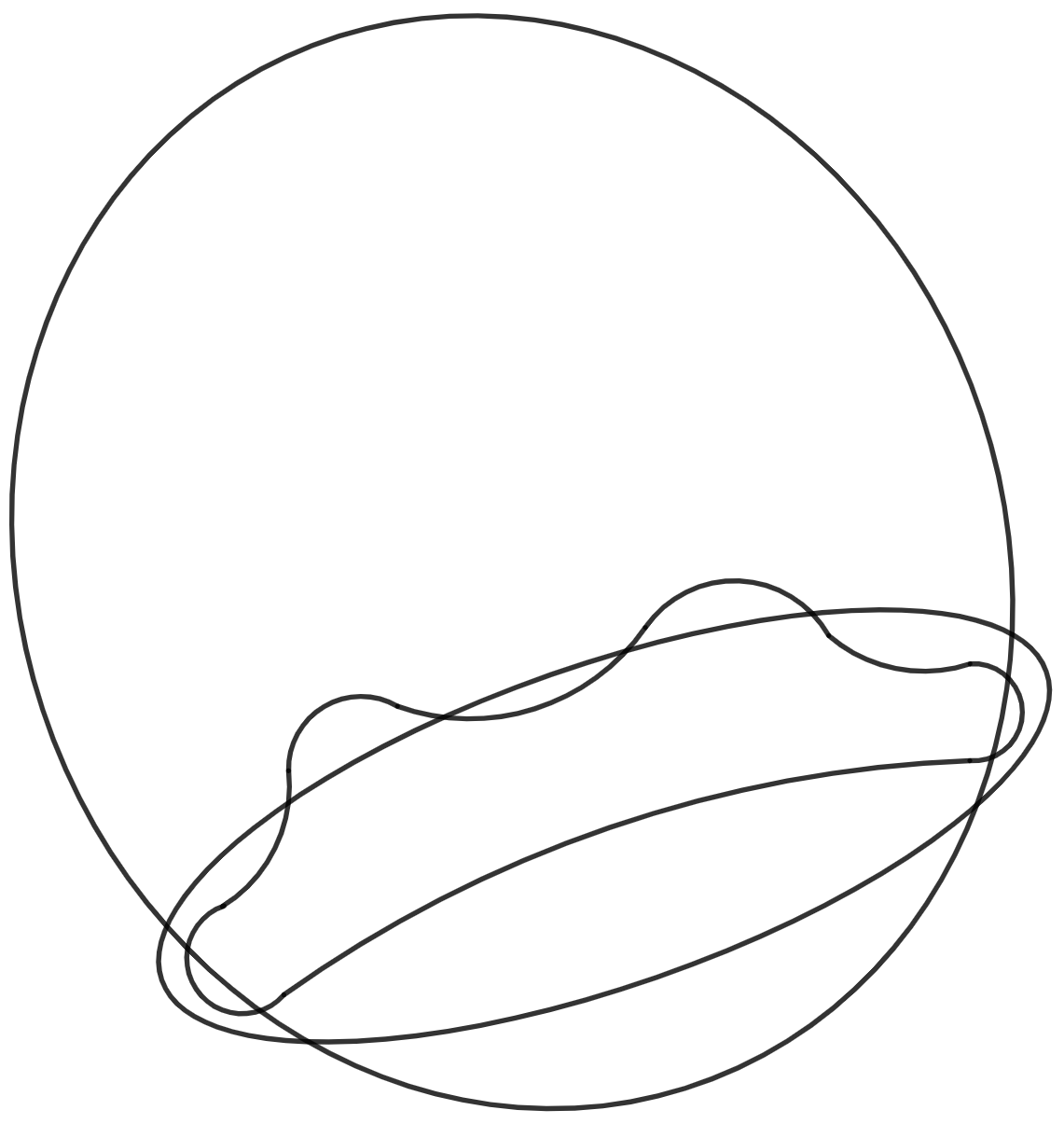}
         \caption{Fig. \ref{fig:444_220}, right.}
         \label{fig:444_220_top}
     \end{subfigure}
     \hfill
         \begin{subfigure}[b]{0.19\textwidth}
         \centering
         \includegraphics[width=\textwidth]{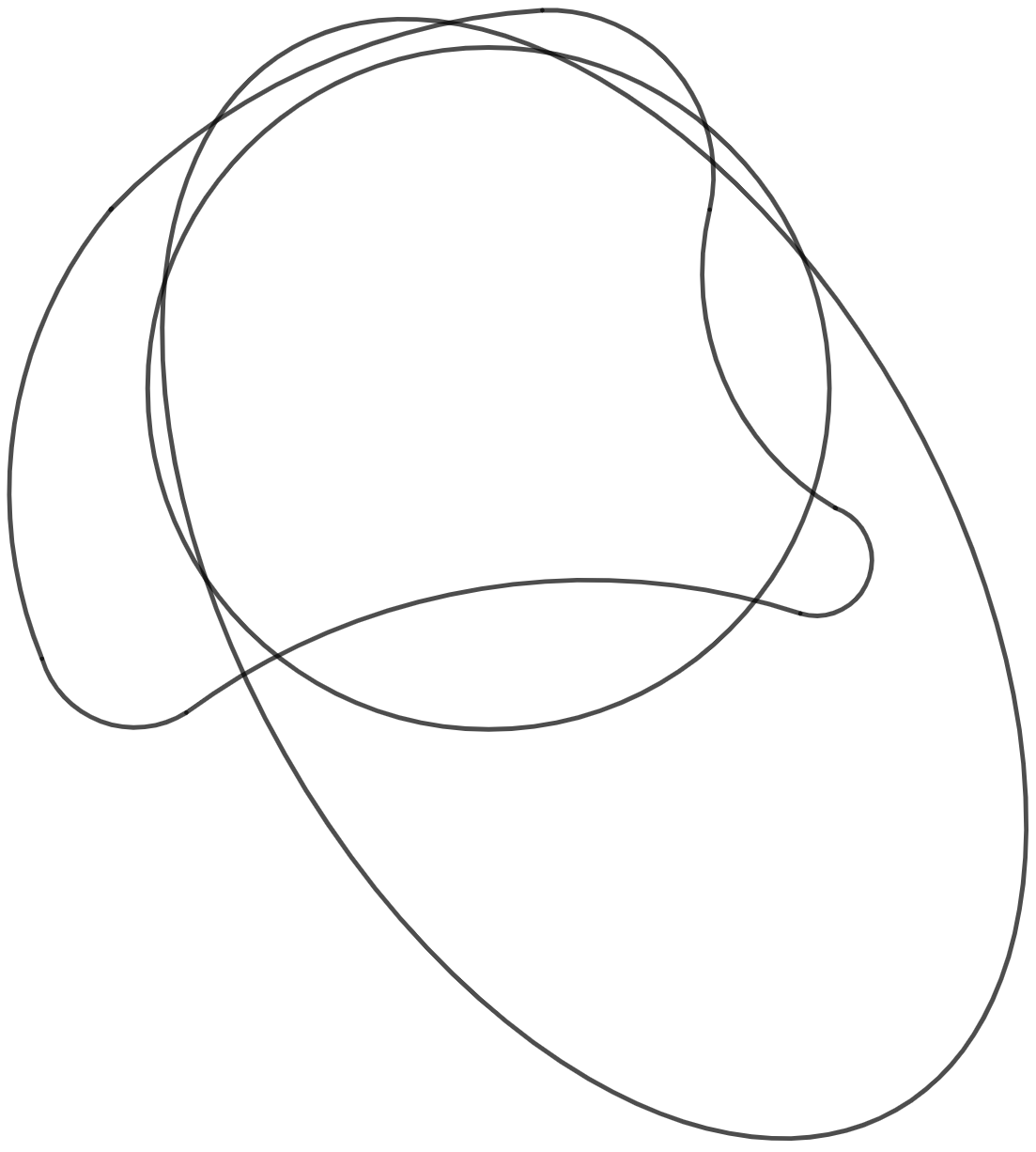}
         \caption{Fig. \ref{fig:444_221}, left.}
         \label{fig:444_221_top}
     \end{subfigure}
    \begin{subfigure}[b]{0.19\textwidth}
         \centering
         \includegraphics[width=\textwidth]{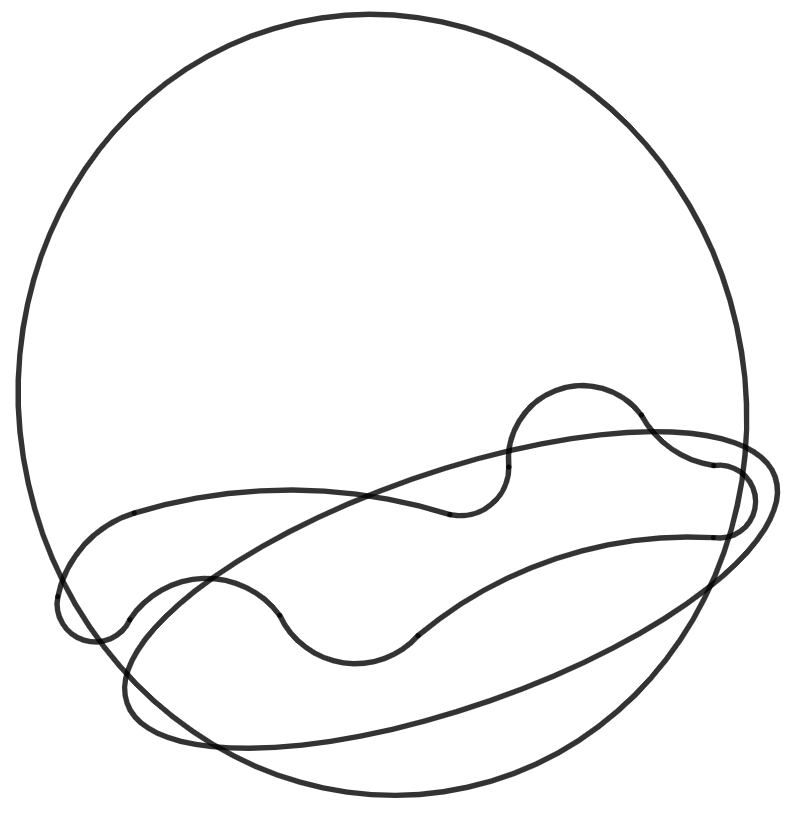}
         \caption{Fig. \ref{fig:444_321}, right.}
         \label{fig:444_321_top}
     \end{subfigure}
        \caption{Topological sketches of some configurations of three ellipses.}
        \label{fig:sketches}
\end{figure}

\subsection{Proving hyperbolicity}

The adjoint curve of a polycon defined by three conics is a cubic curve. So there are only two possible real geometries if the adjoint curve is nonsingular:
\begin{enumerate}
    \item The real points are connected, 
    in which case this connected component is a pseudoline and its complement in $\P^2(\R)$ is connected.
    \item The real part of the adjoint has two connected components, in which case it is a hyperbolic cubic, i.e., it has one connected component that bounds a simply connected region, called the oval, and the other connected component is a pseudoline. 
\end{enumerate}
Since the adjoint curve does not intersect the Euclidean boundary of a regular polycon (cf.~\Cref{lemma:transversalResidual}), the pseudoline component cannot pass through the interior of the polycon.
The essential question is therefore if the adjoint curve has an oval and, in case that it does, if we can determine its location.
In most cases, we can show by an analysis of the signs of the adjoint curve -- similar to the case of polygons -- that the intersection pattern forces the adjoint curve to be hyperbolic and the oval to be outside of  the polycon, which proves \Cref{thm:wachspressellipses} in those cases. 

\begin{proposition}
\label{prop:polyconHyperbolic}
In the 28 configurations of three ellipses depicted without frames in Figures~\ref{fig:222}--\ref{fig:444}, the real part of the adjoint curve of any regular polycon $P$ in the configuration is hyperbolic and does not intersect the interior of $P_{\ge 0}$.
\end{proposition}

\begin{proof}  
Each pair of ellipses intersects in four complex points, three of which are residual points. 
The adjoint curve intersects each one of the ellipses only in the six residual points on it by \Cref{lemma:transversalResidual}. 
Since the adjoint is a cubic curve, each one of the six points is a simple root, and we can determine the sign of the adjoint along each conic. 
We illustrate this in \Cref{fig:ellipseadjointexm}, where the signs are marked as red and blue: 
The sign changes at every residual point, so each of them is adjacent to two red arcs and two blue arcs. 
By \Cref{lemma:transversalResidual}, the adjoint has to pass through every residual point, intersecting each of the two conics transversely  and separating the blue from the red arcs. That is to say that every sufficiently small real circle around a residual point intersects the adjoint in two real points, and the two red branches of the conics intersect the circle in one of the intervals created by these intersection points and the blue branches in the other interval.

This local information is enough to determine the location of the oval of the adjoint in \Cref{fig:ellipseadjointexm}:
The blue triangle on the right side of the polycon in the picture is surrounded by an oval of the adjoint because the triangle lies in a simply connected region of the complement of the red arcs of the conics and the adjoint cannot cross the boundary of this region. We leave it as an exercise for the reader that the same argument applies to all regular polycons in the non-framed configurations in Figures~\ref{fig:222}--\ref{fig:444}.
\end{proof}

\begin{figure}[htb]
    \centering
    \includegraphics[width=0.6\textwidth,height=5cm]{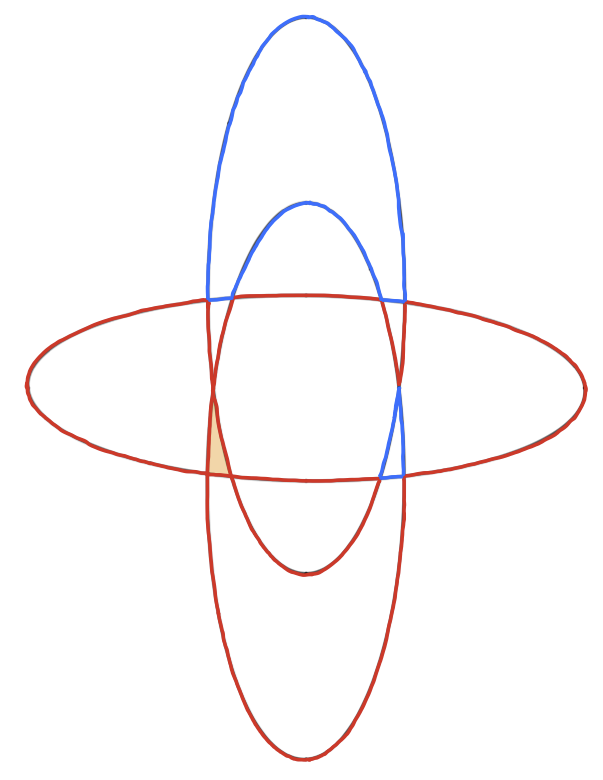}
    \caption{A polycon (orange) bounded by three ellipses with intersection type 244 and outer arc type 211; cf. Figure \ref{fig:244_211}, left. Red and blue distinguish the sign of the adjoint. }
    \label{fig:ellipseadjointexm}
\end{figure}

\subsection{Problematic configurations}

For the 16 framed configurations of three ellipses in Figures \ref{fig:222}--\ref{fig:444}, 
the argument in \Cref{prop:polyconHyperbolic} does not apply to \emph{all} regular polycons in the configuration. 
However, we invite the reader to check that the argument does in fact apply to all but \emph{one} problematic polycon, up to symmetry. The problematic polycon is shaded orange in the figures.
For instance, the configuration in \Cref{fig:224_221} has six regular polycons: Exactly two of those are problematic (i.e., the argument in \Cref{prop:polyconHyperbolic} does not apply), but they are the same up to the symmetry in the configuration.

We now prove Wachspress's  conjecture for five of the problematic polycons.

\begin{proposition} \label{prop:problematicPolycons}
In the 5 problematic configurations of three ellipses in Figures~\ref{fig:244_322} and \ref{fig:444}, the real part of the adjoint curve of any regular polycon $P$ in the configuration does not intersect the interior of $P_{\ge 0}$.
\end{proposition}

To prove this assertion, we need the following key lemma.

\begin{lemma}\label{le:twointersections} Let $P$ be a regular polycon in $\mathbb{R}^2$ defined by three ellipses. If for every point $p$ in the interior of $P_{\ge 0}$ there is a line passing through $p$ that meets the adjoint curve $A_P$ outside of $P_{\ge0}$ at least twice, then $A_P$ does not intersect $P_{\ge 0}$.
\end{lemma}
\begin{proof} This is a count of intersection points because the adjoint curve does not intersect the boundary of $P_{\ge 0}$ by \Cref{lemma:transversalResidual}. So a connected component of $A_P(\R)$ inside $P_{\ge0}$ would have to be an oval (possibly singular). In the case of a nonsingular oval, we choose a point $p$ in its interior, and see that the line from the statement now intersects $A_P$ in at least two additional points. This is not possible since the adjoint is of degree 3. In the singular case, we choose $p$ to be the singular point of $A_P$ inside $P_{\ge 0}$. Then the line from the statement passes through $p$ with multiplicity two and we arrive at the same contradiction.   
\end{proof}

The five problematic polycons described in \Cref{prop:problematicPolycons} are shown in Figure~\ref{fig:problematicFive}, together with the sign of the adjoint along each ellipse.
If the real part of the cubic adjoint would have an oval or a singularity outside the polycon, then there could be no  oval or isolated node inside the polycon.
Thus, no real connected component of the adjoint could be strictly contained inside the polycon and Lemma~\ref{lemma:transversalResidual} would imply
Wachspress's conjecture. 
Hence, in the following we assume that all real residual points lie on one connected nonsingular pseudoline of $A_P$. 
Knowing how the adjoint passes through the residual points (separating red and blue arcs), we can connect the adjoint in some regions but not in others. 
This local topological behavior of the adjoint is sketched in Figure~\ref{fig:problematicFive}.
For each region bounded by two chains of red and blue arcs for which we have not yet drawn how the adjoint passes through the interior,  there are two nonsingular possibilities how to connect the adjoint, either ``along the red sides''  (see \Cref{fig:connectAlongRedSides}) or ``along the blue sides'' (as in \Cref{fig:tentaclesClosestPolyconConnected}).

\begin{figure}[htb]
    \centering
    \begin{subfigure}[b]{0.19\textwidth}
         \centering
         \includegraphics[width=\textwidth]{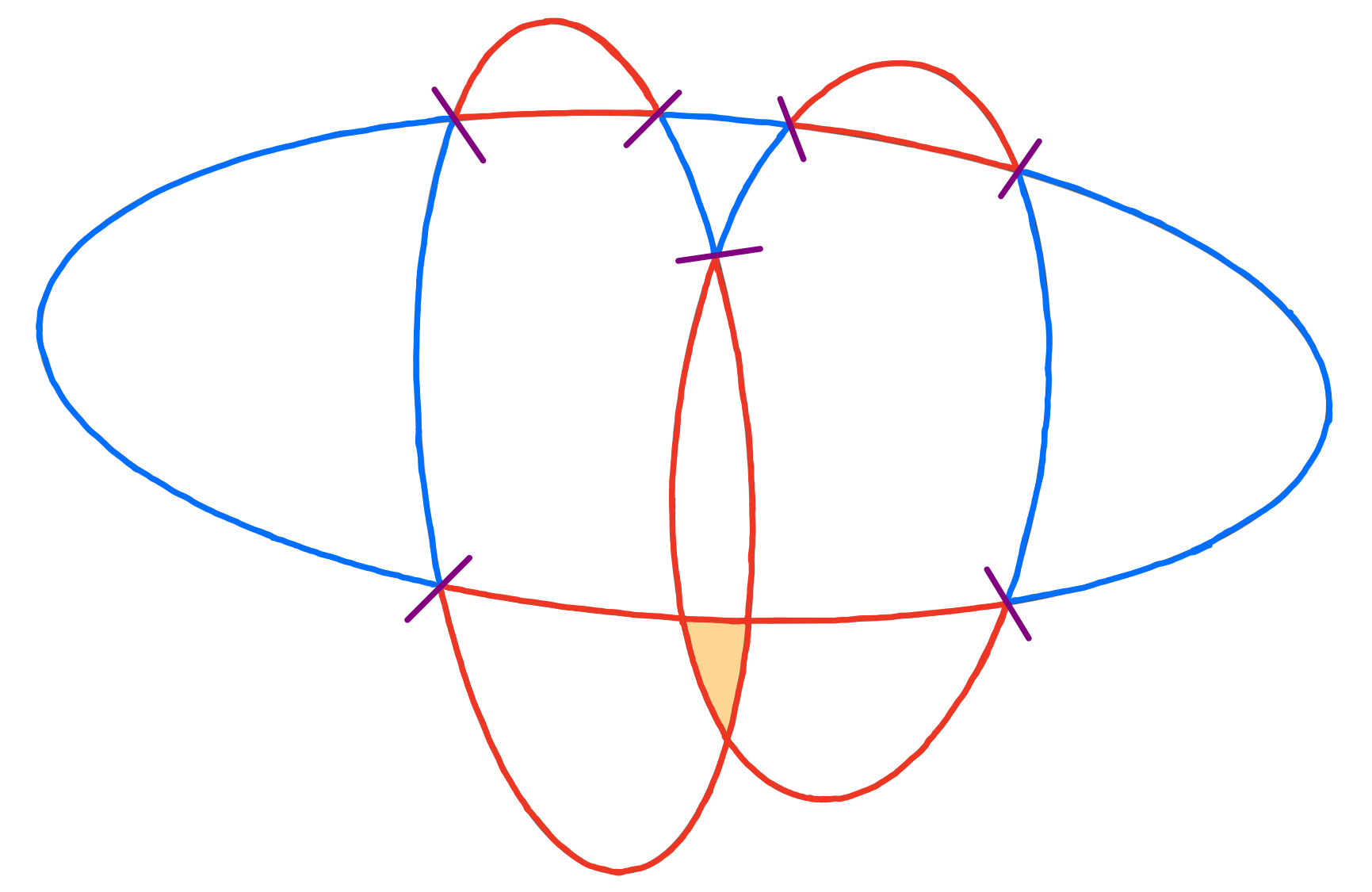}
         \caption{Fig. \ref{fig:244_322}.}
         \label{fig:problem_244_322}
     \end{subfigure}
         \begin{subfigure}[b]{0.19\textwidth}
         \centering
         \includegraphics[width=\textwidth]{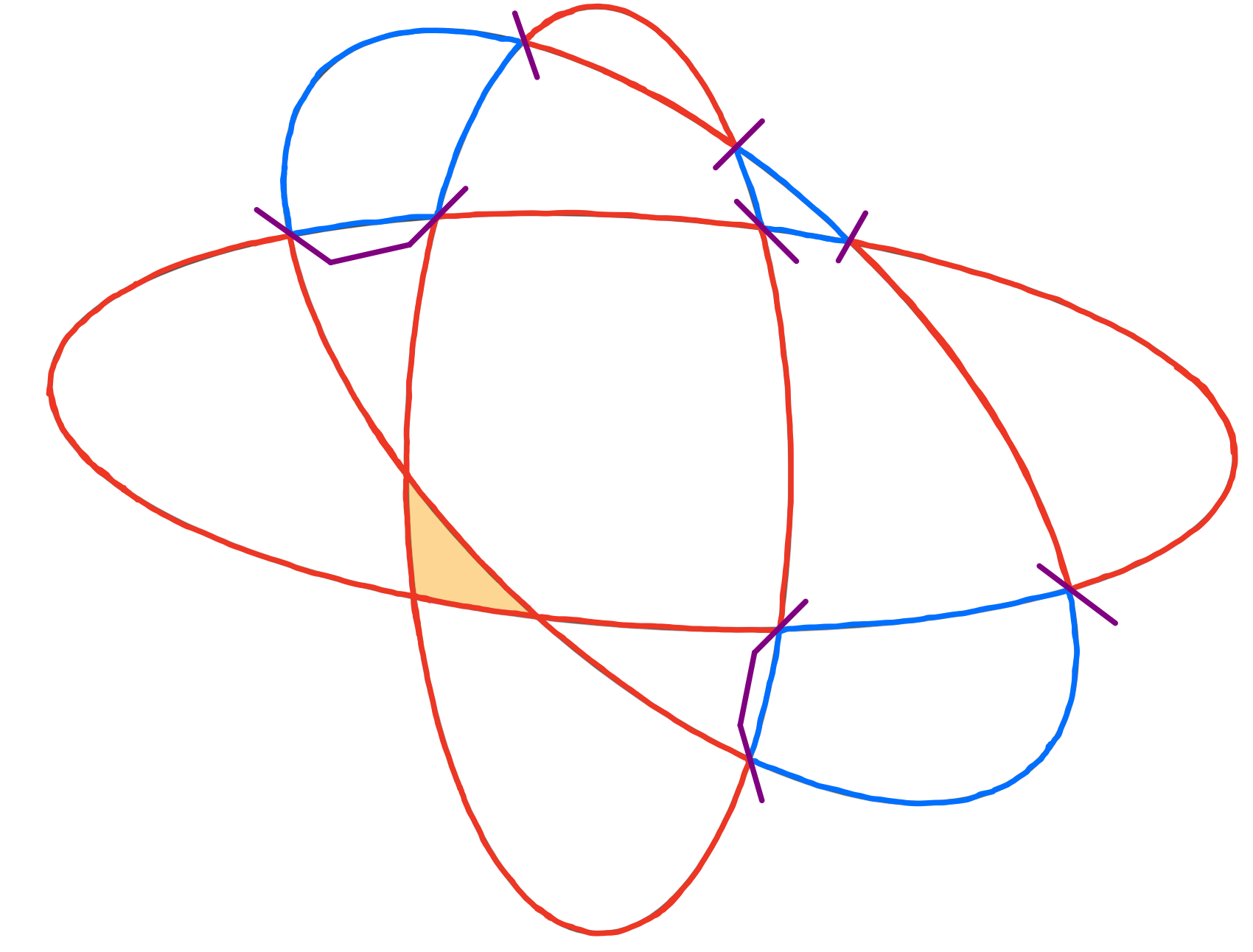}
         \caption{Fig. \ref{fig:444_322}, left.}
         \label{fig:problem_444_322_left}
     \end{subfigure}
     \begin{subfigure}[b]{0.19\textwidth}
         \centering
         \includegraphics[width=\textwidth]{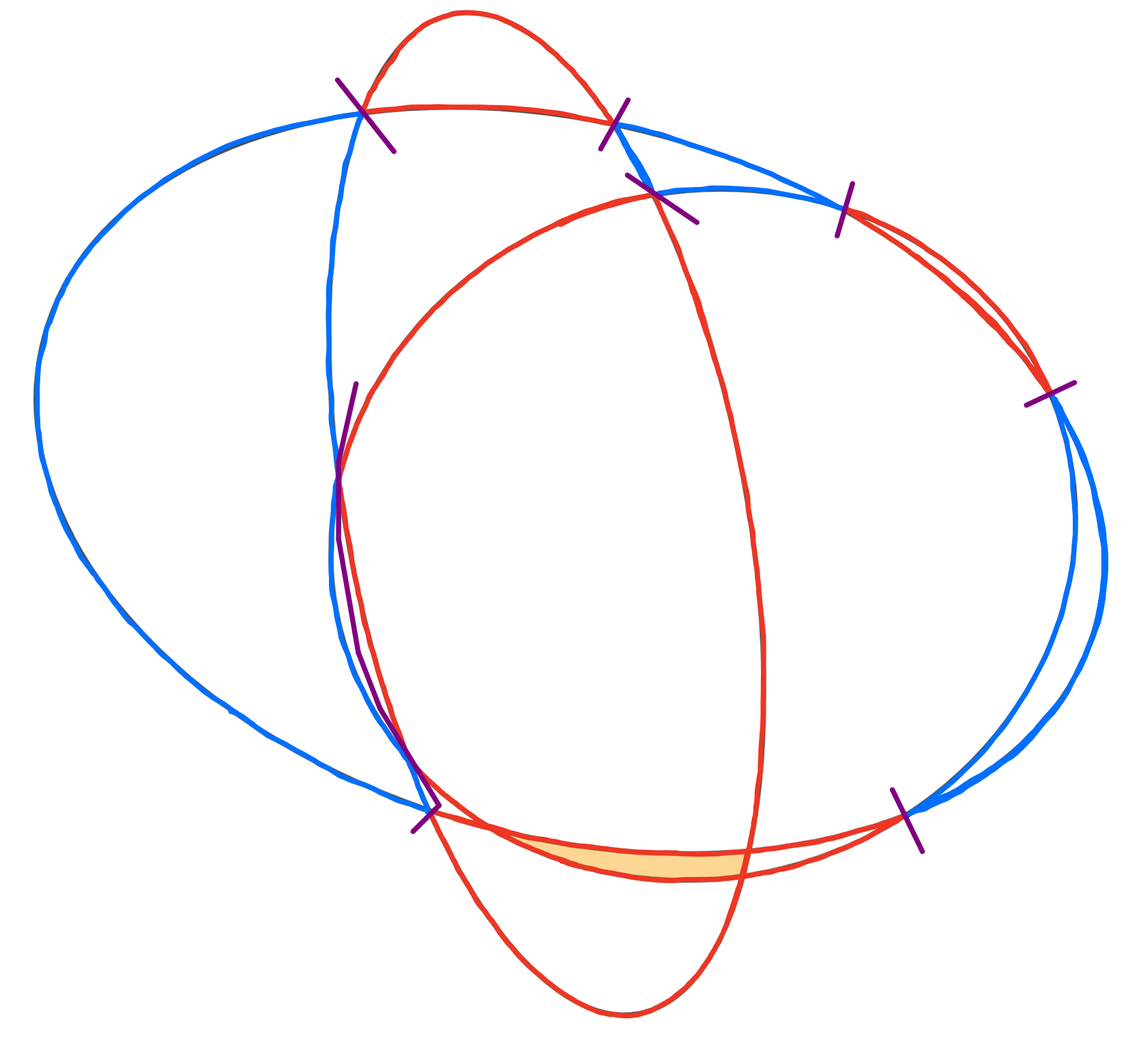}
         \caption{Fig. \ref{fig:444_322}, right.}
         \label{fig:problem_444_322_right}
     \end{subfigure}
              \begin{subfigure}[b]{0.19\textwidth}
         \centering
         \includegraphics[width=\textwidth,height=2.5cm]{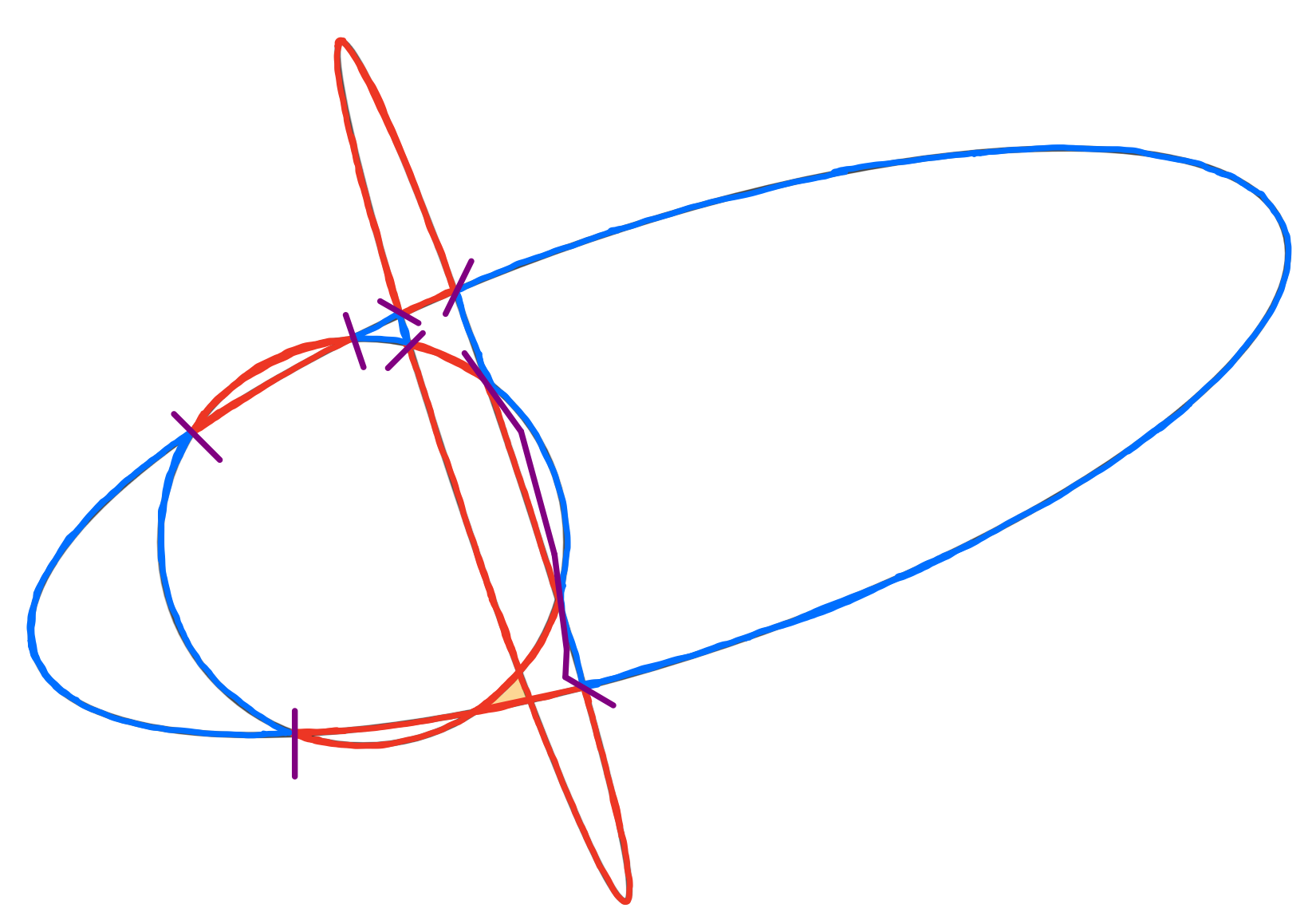}
         \caption{Fig. \ref{fig:444_422}, left.}
         \label{fig:problem_444_422_left}
     \end{subfigure}
     \begin{subfigure}[b]{0.20\textwidth}
         \centering
         \includegraphics[width=\textwidth]{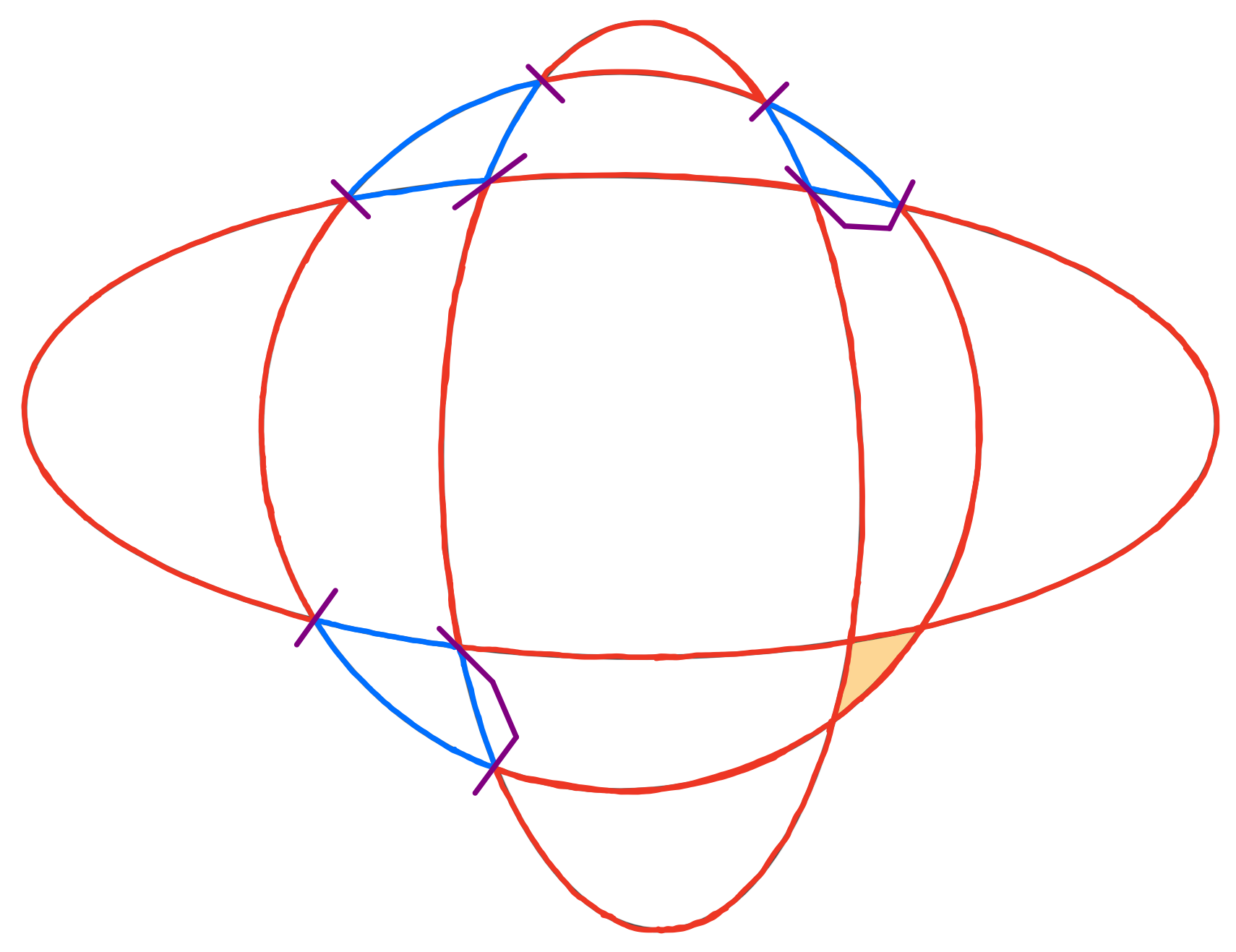}
         \caption{Fig. \ref{fig:444_422}, right.}
         \label{fig:problem_444_422_right}
     \end{subfigure}
    \caption{The five problematic polycons from \Cref{prop:problematicPolycons}. Orange, red and blue are as in Figure \ref{fig:ellipseadjointexm}. The purple curve segments show the topological behaviour of the adjoint.}
    \label{fig:problematicFive}
\end{figure}

\begin{lemma} \label{lem:connectAlongRedSegments} If two adjoint curve segments in \Cref{fig:problematicFive} connect ``along the red sides'', the adjoint curve does not intersect $P_{\ge 0}$.
\end{lemma}

\begin{figure}[htb]
    \centering
    \begin{subfigure}[b]{0.49\textwidth}
         \centering
         \includegraphics[width=\textwidth]{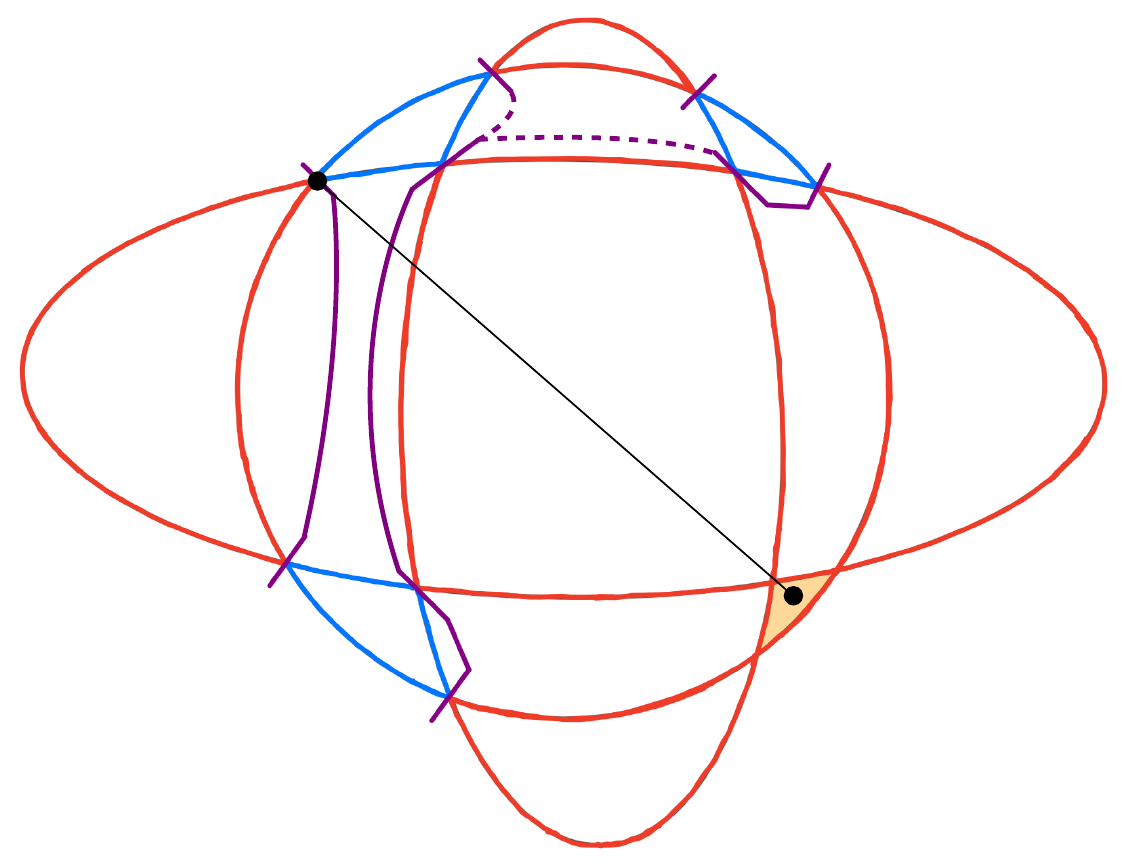}
         \caption{The adjoint curve segments connect along the red sides in the left 4-sided region.}
         \label{fig:connectAlongRedSides}
    \end{subfigure} \hfill
        \begin{subfigure}[b]{0.49\textwidth}
         \centering
         \includegraphics[width=\textwidth]{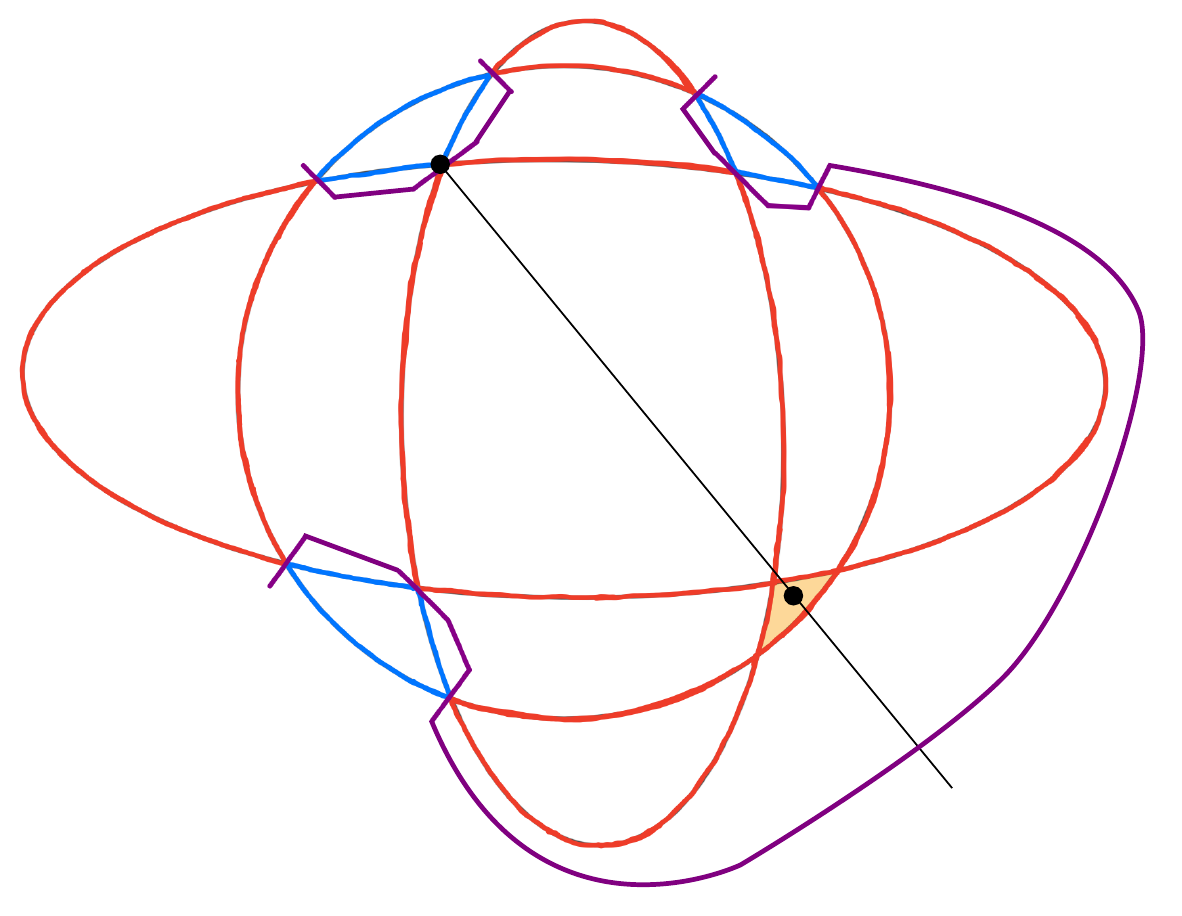}
         \caption{The two tentacles closest to the polycon are directly connected.}
         \label{fig:tentaclesClosestPolyconConnected}
    \end{subfigure}
    \caption{Example illustrations  for the problematic polycon in Figure \ref{fig:problem_444_422_right}. }
    \label{fig:problematicExampleIllustrations}
\end{figure}

\begin{proof} We first observe that if one of the 3 ellipses satisfies that 1) the polycon is in its interior and 2) the pseudoline of the adjoint separates the ellipse into disjoint regions such that one residual point $p'$ on the ellipse lies in a different region than the polycon, then we are done by Lemma \ref{le:twointersections}. This is because for any point $p$ in the polycon, the line segment from $p$ to $p'$ is contained inside the ellipse and must intersect the adjoint at least twice: once at $p'$ and once where the pseudoline separates the ellipse.

We show now for the problematic polycon in Figure \ref{fig:problem_444_422_right} that one of the three ellipses satisfies 1) and 2) above. The other cases are similar. We consider the unique ellipse $C$ that has the polycon in its interior. If the adjoint curve connects along the red sides, say in the left 4-sided region in \Cref{fig:connectAlongRedSides}, then there are two possibilities for how to connect the pseudoline in the upper 4-sided region (marked with the dashed curve segments.) Either way, the interior of the ellipse $C$ is separated into disjoint regions by the pseudoline as in 2).

\end{proof}

For each of the polycons in Figure \ref{fig:problematicFive}, there are 6 branches of the pseudoline of the adjoint that leave the configuration of three ellipses. We call these \textit{tentacles}. We are now ready to prove Proposition \ref{prop:problematicPolycons} and thereby finalize the proof of Theorem \ref{thm:wachspressellipses}. 

\begin{proof}[Proof of Proposition \ref{prop:problematicPolycons}] 
As  previously discussed, we may assume that the adjoint curve segments in Figure \ref{fig:problematicFive} form one connected nonsingular pseudoline.
Due to Lemma \ref{lem:connectAlongRedSegments}, we also assume that the adjoint curve segments connect along the blue sides inside the configuration of the three ellipses. 
Hence, to conclude the topological picture of the pseudoline, it is left to distinguish how the six tentacles connect outside of the configuration and which tentacles go to infinity.  The pseudoline intersects the line at infinity in one or three points, so its intersection with the affine chart (that is the complement of the line at infinity) has one or three components. This means that the number of connected components of the pseudoline in the affine chart that intersect the triple of ellipses is one, two or three. In each case we now argue how the tentacles may connect to form these components.
Since the adjoint curve segments in Figure \ref{fig:problematicFive} form one connected nonsingular pseudoline and cannot intersect the ellipses in any other points, only neighbouring tentacles can be directly connected with each other (i.e., without meeting the line at infinity first). 
In particular, the two tentacles that are closest to the polycon are either directly connected or both meet the line at infinity. They cannot connect to their other direct neighbor because this would create an oval. This leaves us with the following three cases.

\textbf{Case 1: The two tentacles nearest the polycon are directly connected.}
Let $p$ be any point in the interior of the problematic polycon. In Figures \ref{fig:problem_444_422_left} and \ref{fig:problem_444_422_right}, the polycon is outside two of the three ellipses, denoted by $C_1,C_2$, and inside the third ellipse $C_3$. 
Let $p'$ be any of the three residual points on the intersection of $C_1$ with $C_2$. We consider the line spanned by $p$ and $p'$ (depicted in Figure \ref{fig:tentaclesClosestPolyconConnected}) and split it into three line segments: the first goes from the point at infinity to $p'$, the second from $p'$ to $p$, and the third from $p$ to infinity. 
Since $p'$ lies on the boundary of the ellipse $C_1$ (resp. $C_2$) and $p$ lies outside of the ellipse, the line segment from $p$ to infinity intersects neither $C_1$ nor $C_2$. 
Hence, it has to leave the polycon via its side on the ellipse $C_3$ and then intersect the adjoint curve segment  that connects the two tentacles closest to the polycon. 
Thus, the line spanned by $p$ and $p'$ meets the adjoint outside of the polycon in two points and Wachspress's  conjecture holds for the polycons in Figures \ref{fig:problem_444_422_left} and \ref{fig:problem_444_422_right} by \Cref{lemma:transversalResidual}.

For the other problematic polycons in Figure \ref{fig:problematicFive}, we change the argument slightly as follows. 
In these configurations, the polycon is inside two of the ellipses, now denoted $C_1,C_2$, and outside the third ellipse $C_3$. 
Let $p'$ be a real residual point on the intersection of $C_1$ with $C_2$. 
We consider the three segments of the line passing through $p$ and $p'$ as above. 
Since this time both $p$ and $p'$ are inside the ellipse $C_1$ (resp. $C_2$), the line segment from $p$ to infinity must intersect the boundary of $C_1$ (resp. $C_2$) exactly once. 
We see from Figures \ref{fig:problem_244_322}--\ref{fig:problem_444_322_right} that the residual point $p'$ is inside the ellipse $C_3$. 
Since $p$ lies outside $C_3$, the line segment from $p$ to infinity does not meet $C_3$.
Hence, this line segment leaves the configuration of three ellipses via one of the two red arcs that lie between the two residual points on the boundary of the configuration that are closest to the polycon. 
This shows that the line segment from $p$ to infinity intersects the adjoint curve segment that connects the two tentacles closest to the polycon and, as above,  we see from \Cref{lemma:transversalResidual} that Wachspress's  conjecture holds for the problematic polycons in Figures \ref{fig:problem_244_322}--\ref{fig:problem_444_322_right}.

\textbf{Case 2: All six tentacles meet the line at infinity before connecting to any other tentacle.}
In this case, the pseudoline has three branches as illustrated in Figure \ref{fig:6tentaclesMeetInfinity_example}. We assume for contradiction that there is a point $p$ inside the polycon such that every line through $p$ meets the pseudoline in exactly one point. The complement of the three (affine) lines $L_1, L_2, L_3$ that are spanned by $p$ and one of the three points on the pseudoline at infinity in the affine chart $\R^2$ consists of six ``pizza slice'' regions; see Figure \ref{fig:6tentaclesMeetInfinity_sketch} for a sketch. The three branches of the pseudoline have to be contained in every other ``pizza slice''. 
For each of the five problematic polycons in \Cref{fig:problematicFive}, one of its three ellipses, denoted by $C$, satisfies the following two conditions: 1) The polycon is outside of the ellipse, and 2) every blue region has the ellipse on its boundary.
Since the pseudoline segments are connected along the blue sides inside the configuration of the three ellipses, each of the three branches of the pseudoline goes around one of the three blue regions. Thus, the ellipse $C$ has to pass through every ``pizza slice'' containing a branch of the pseudoline.
If we focus for instance on the top right ``pizza slice'' sketched in Figure \ref{fig:6tentaclesMeetInfinity_sketch}, we may assume by symmetry that the ellipse $C$ enters the ``pizza slice'' by crossing $L_1$. 
Since the point $p$ lies outside of $C$, the second intersection point of $L_1$ with the ellipse $C$ also has to lie on the boundary of that same ``pizza slice''.
This implies for the left ``pizza slice'' containing a branch of the pseudoline that $C$ has to meet its boundary half-line on $L_2$ twice.
Similarly, we get two intersection points of the ellipse $C$ at the bottom half of $L_3$.
All in all, we obtain the situation sketched in Figure \ref{fig:6tentaclesMeetInfinity_sketch}, which is a contradiction due to the convexity of the ellipse $C$ because it does not contain $p$.
Hence, by Lemma \ref{le:twointersections}, Wachspress's  conjecture holds in this case.
\begin{figure}[htb]
    \centering
    \begin{subfigure}[b]{0.49\textwidth}
         \centering
         \includegraphics[width=\textwidth]{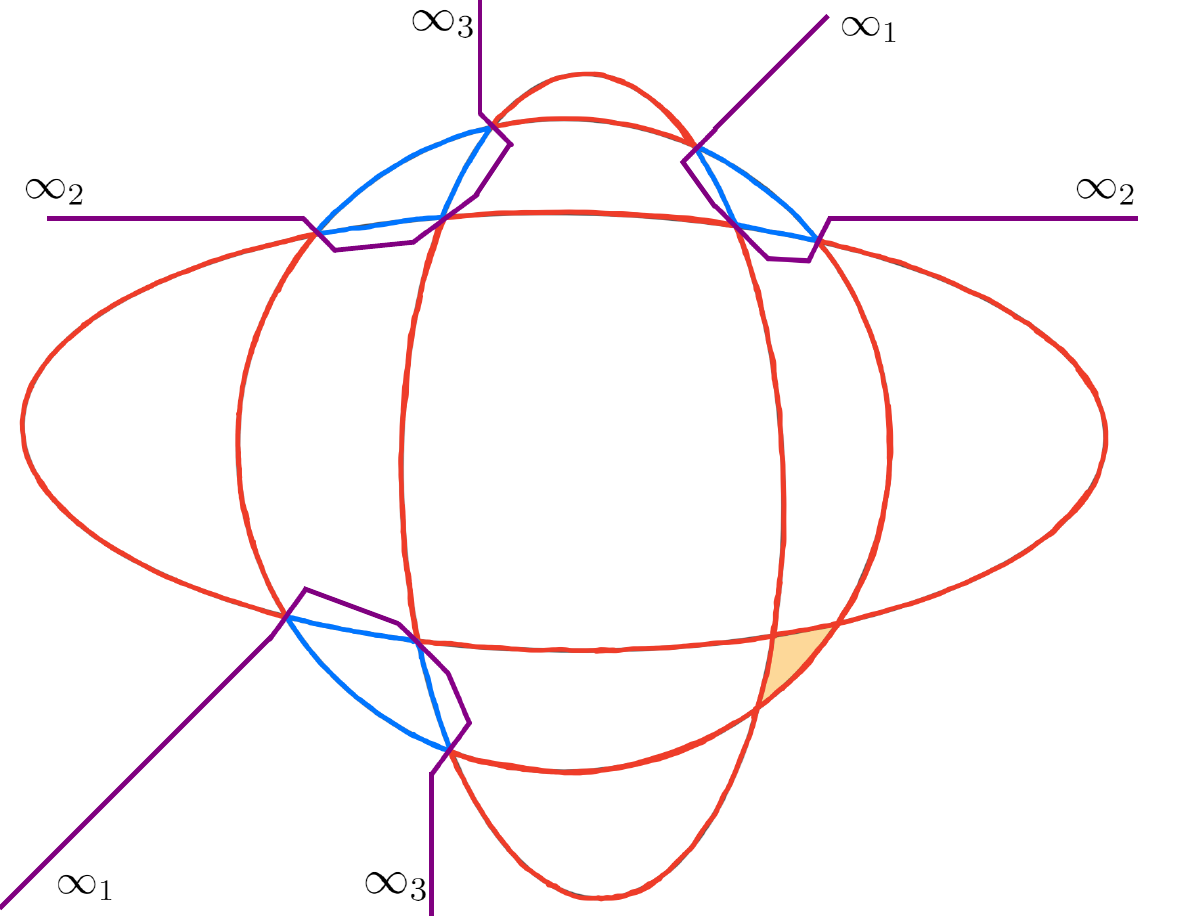}
         \caption{Example illustration for the problematic polycon in Figure \ref{fig:problem_444_422_right}.}
         \label{fig:6tentaclesMeetInfinity_example}
    \end{subfigure} \hfill
        \begin{subfigure}[b]{0.49\textwidth}
         \centering
         \includegraphics[width=0.75\textwidth]{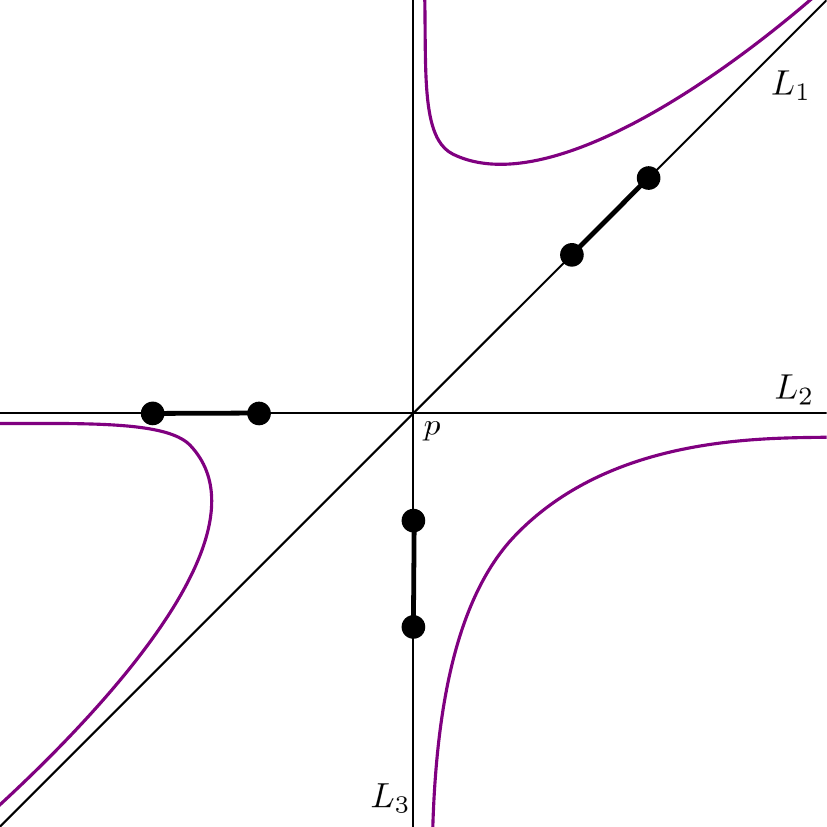}
         \caption{Sketch of the contradiction in Case 2 in the proof of \Cref{prop:problematicPolycons}. The ellipse $C$ would intersect the lines $L_1,L_2,L_3$ at the bold points.}
         \label{fig:6tentaclesMeetInfinity_sketch}
    \end{subfigure}
    \caption{All six tentacles meet the line at infinity.}
    \label{fig:6tentaclesMeetInfinity}
\end{figure}

\textbf{Case 3: The two tentacles nearest the polycon meet the line at infinity, and two other tentacles are directly connected.}
Recall that only neighboring tentacles can be directly connected and that their connection should not create an oval.
For any of the five problematic polycons in \Cref{fig:problematicFive}, this leaves only two pairs of tentacles that can be directly connected.
In the following, we distinguish three more subcases, illustrated in Figure~\ref{fig:tentaclesClosestPolyconInfinity}.
We stress that the following arguments apply to all problematic polycons in \Cref{fig:problematicFive} although Figure~\ref{fig:tentaclesClosestPolyconInfinity} only shows the  polycon from \Cref{fig:problem_444_422_right}.
This is because their essential topological properties are the same: 
Each of the five problematic polycons induces three blue regions and six tentacles of the pseudoline of the adjoint such that every tentacle leaves the configuration of the three ellipses at one of the blue regions.
Moreover, inside the configuration, since the pseudoline segments are connected along the blue sides, the pseudoline encloses each of the three blue regions.
The three subcases are:

\begin{enumerate}
    \item[a)] The remaining pair of tentacles is also directly connected.
    
    This is illustrated in Figure \ref{fig:case3a}. We show that this is impossible by counting real inflection points of the adjoint curve in this case. A real plane cubic curve has exactly three real inflection points (\cite[p.~44]{fischer}). To give a lower bound, we travel along the purple adjoint curve in Figure \ref{fig:case3a} from the bottom left to the top right branch. To enclose a blue region, the curve must bend to the left. To enclose a red region (bounded by two intervals on the ellipses) between two of the blue regions, the curve must bend to the right. In total, we transition four times between these cases and so we must have at least four (real) inflection points in this picture. This is impossible for a plane cubic.

    \begin{figure}[htb]
        \centering
            \begin{subfigure}[b]{0.32\textwidth}
         \centering
         \includegraphics[width=\textwidth]{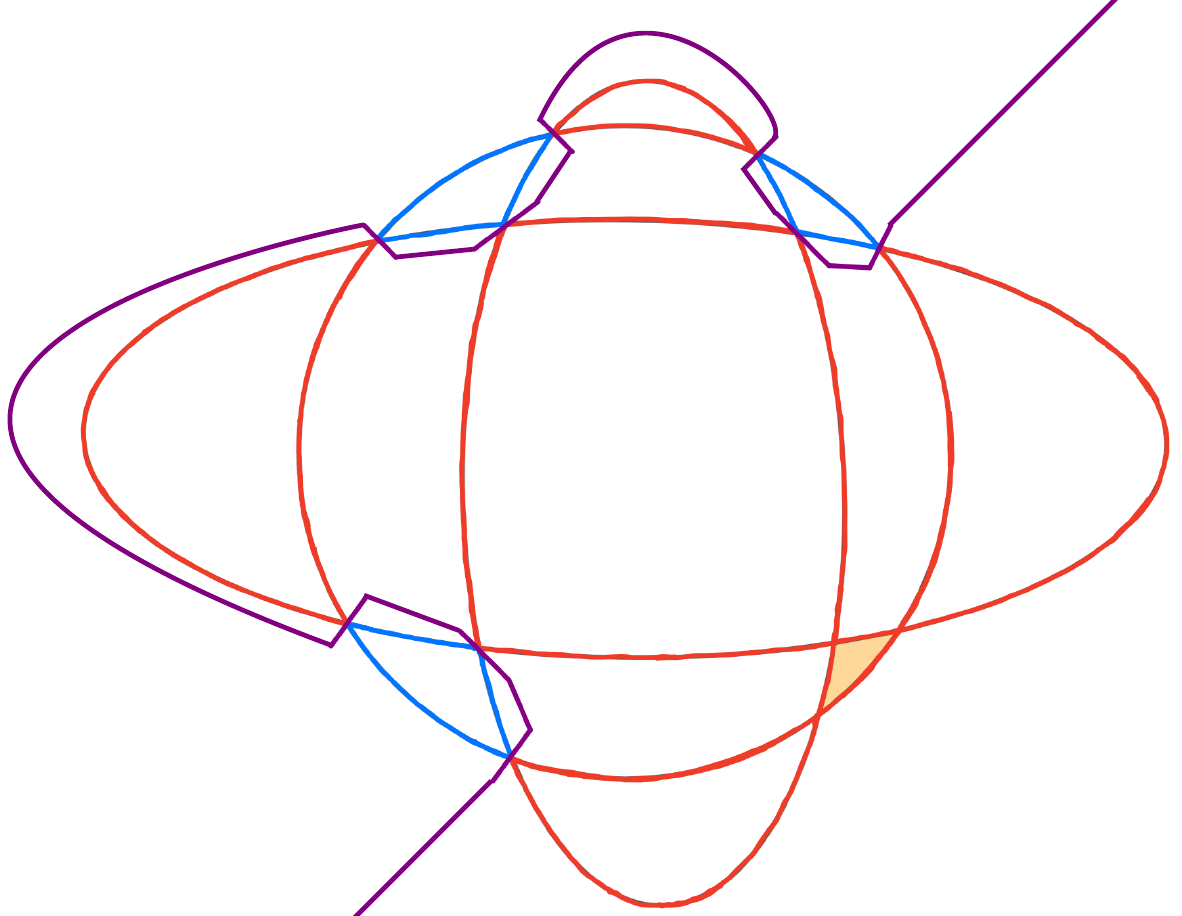}
         \caption{Case 3a).}
         \label{fig:case3a}
    \end{subfigure} \hfill
        \begin{subfigure}[b]{0.32\textwidth}
         \centering
         \includegraphics[width=\textwidth]{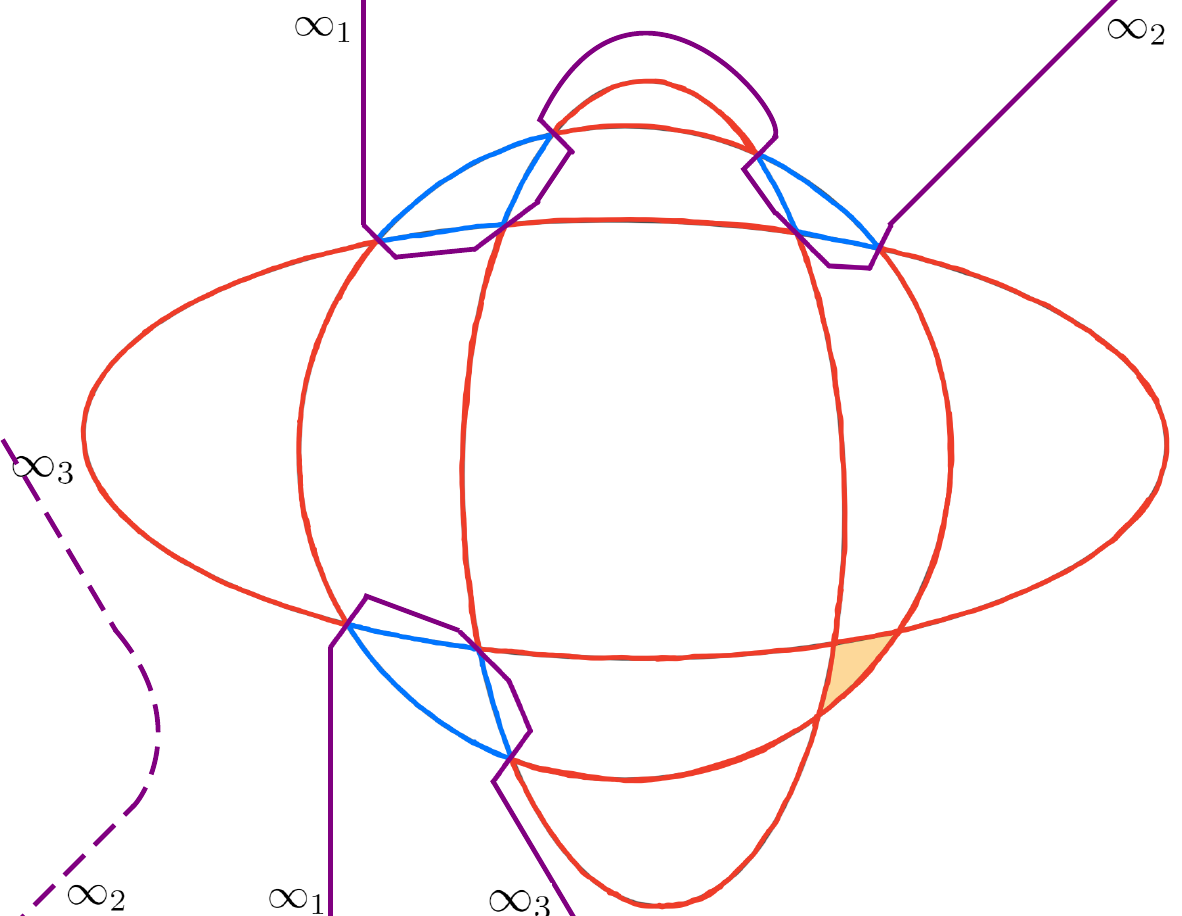}
         \caption{Case 3b) (i).}
         \label{fig:case3bi}
    \end{subfigure}
\hfill
        \begin{subfigure}[b]{0.32\textwidth}
         \centering
         \includegraphics[width=\textwidth]{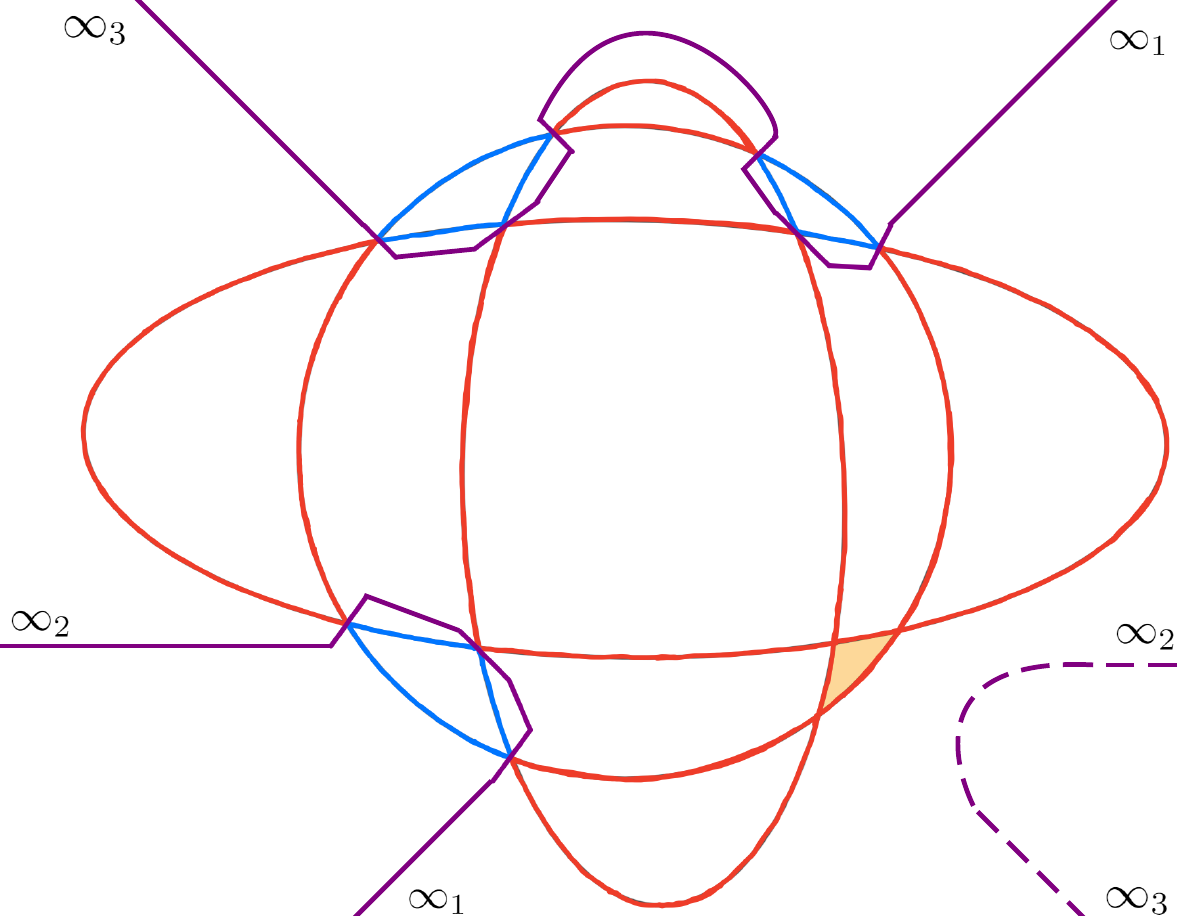}
         \caption{Case 3b) (ii).}
         \label{fig:case3bii}
    \end{subfigure}
        \caption{Example illustrations for the proof of \Cref{prop:problematicPolycons} for the problematic polycon in Figure \ref{fig:problem_444_422_right}.}
        \label{fig:tentaclesClosestPolyconInfinity}
    \end{figure}

    \item[b)] The remaining two tentacles meet the line at infinity.
    \begin{enumerate}
        \item[(i)] The two tentacles closest to the polycon meet the line at infinity in distinct points.
        
        In this case, the line at infinity meets the adjoint curve in three real points.
        Since only four of the six tentacles of the pseudoline segments intersect the line at infinity, there has to be another branch of the adjoint curve with two points at infinity. That branch has to be located as depicted in Figure \ref{fig:case3bi} (dashed). Indeed, due to the nonsingularity of the pseudoline, the branch can only be placed in between two of the four tentacles with points at infinity, and the placement of the branch between any two other tentacles than as shown in the figure would either create an oval or contradict the assumption in Case (i). Now we can apply the same argument as in Case 1, since Figure \ref{fig:case3bi} can be obtained from Figure \ref{fig:tentaclesClosestPolyconConnected} by moving the line at infinity such that it severs the direct connection between the two tentacles closest to the polycon.

        \item[(ii)] The two tentacles closest to the polycon meet the line at infinity in the same point.
        
        As in the previous case, there has to be another branch of the adjoint curve that meets the line at infinity twice. This time the branch has to be located as shown in Figure \ref{fig:case3bii}. Now the same argument (counting inflection points) as in Case 3a) applies, since Figure \ref{fig:case3bii} can be obtained from \Cref{fig:case3a}  by moving the line at infinity. \qedhere
    \end{enumerate}
    \end{enumerate}
    \end{proof}

    \subsection{Examples of adjoint curves for problematic configurations}

Figure \ref{fig:problematicAdjoint} shows the adjoint curve for one instance of each problematic polycon in Figures \ref{fig:222}--\ref{fig:444}.
We see that Wachspress's  conjecture holds in every instance, i.e., the adjoint curve does not intersect the interior of the polycon.
We proved the conjecture for the last five polycons in Figure \ref{fig:problematicAdjoint}.
It it still an open problem to provide a formal proof for the first eleven polycons.

\begin{remark}
\label{rem:nonHyperbolic}
We stress that the adjoint curve of many problematic polycons is \emph{not} hyperbolic. For instance, this is the case for the polycon in \Cref{fig:problematicAdjoint}(b). To show that the adjoint curve is not hyperbolic, we compute critical points for a projection with center $e\in \P^2(\R)$. If the point is not in the interior of an oval of the real locus of the adjoint, such an oval produces at least two critical points for this projection. The critical points for the projection of the curve defined by an equation $\alpha_P$ away from $e$ are the intersection points with its polar curve defined by the directional derivative $D_e \alpha_P$. This is the curve in green in \Cref{fig:adjointNonHyp}. The picture shows that there are no critical points on an oval which therefore cannot exist.
\begin{figure}
    \centering
    \includegraphics[width=0.6\textwidth,height=5cm]{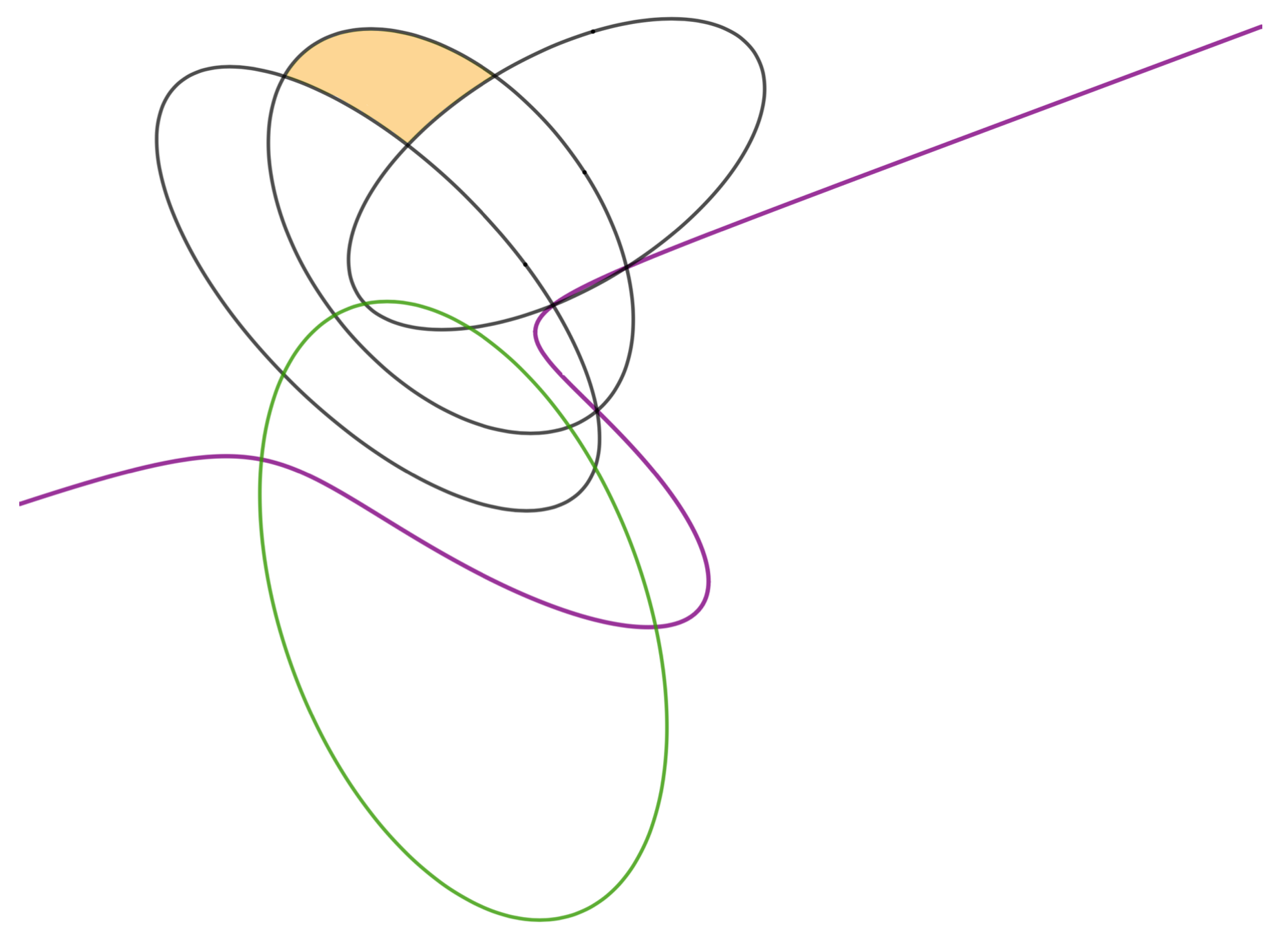}    
    \caption{Case (222) -- 211 with the adjoint in purple and its polar curve in green. The notation (222) -- 211 denotes intersection type (222) and outer-arc type 211.}
    \label{fig:adjointNonHyp}
\end{figure}

\end{remark}

Finally, we note that the adjoint curves of the five problematic polycons addressed in \Cref{prop:problematicPolycons} have different behaviors:
On the one hand, the adjoint curves in Figures \ref{fig:adjoint244_322}, \ref{fig:adjoint444_322} right, and \ref{fig:adjoint444_422} left are ``connected along the red sides'' (cf. \Cref{fig:problematicFive}) and so \Cref{lem:connectAlongRedSegments} implies Wachspress's  conjecture for these polycons.
On the other hand, the adjoint curves in Figures \ref{fig:adjoint444_322} left and \ref{fig:adjoint444_422} right are ``connected along the blue sides'' and all their six tentacles go to infinity, so Wachspress's  conjecture holds by Case 2 in the proof of \Cref{prop:problematicPolycons}.

   \begin{figure}[h!]
          \begin{subfigure}[b]{0.24\textwidth}
         \includegraphics[width=\textwidth]{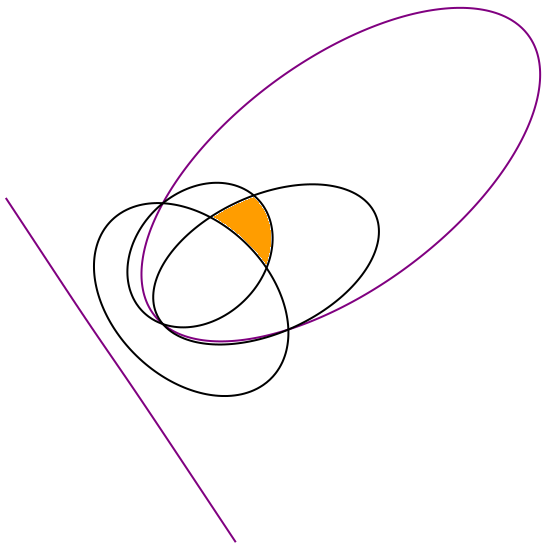}
         \caption{(222) -- 111}
         \label{fig:adjoint222_111}
    \end{subfigure}\hfill
     \begin{subfigure}[b]{0.24\textwidth}
         \includegraphics[width=\textwidth]{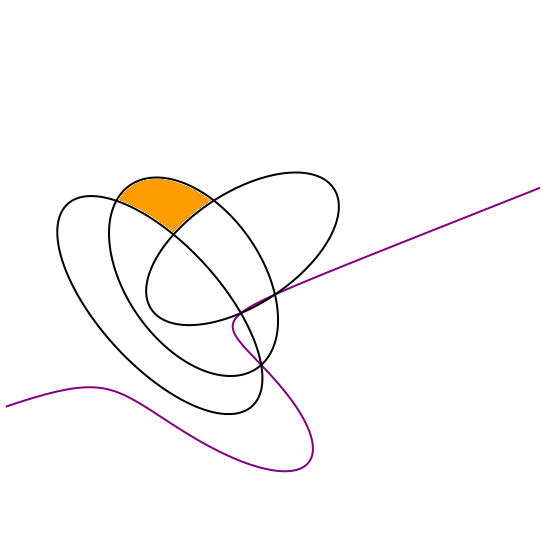}
         \caption{(222) -- 211}
         \label{fig:adjoint222_211}
    \end{subfigure}\hfill
     \begin{subfigure}[b]{0.24\textwidth}
         \includegraphics[width=\textwidth]{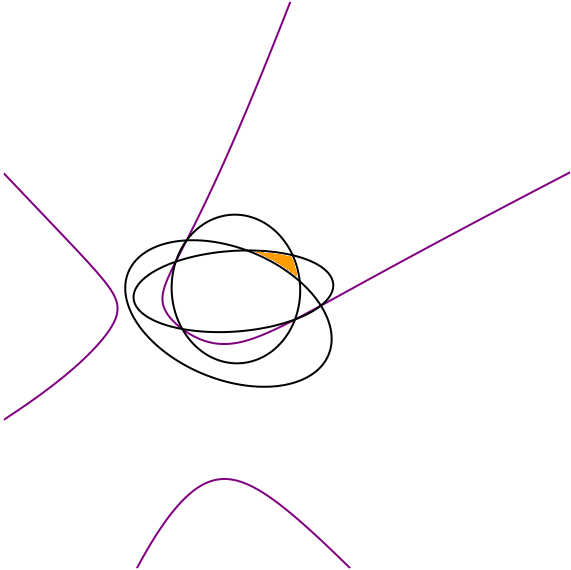}
         \caption{(224) -- 111}
         \label{fig:adjoint224_111}
    \end{subfigure}\hfill
     \begin{subfigure}[b]{0.24\textwidth}
         \includegraphics[width=\textwidth]{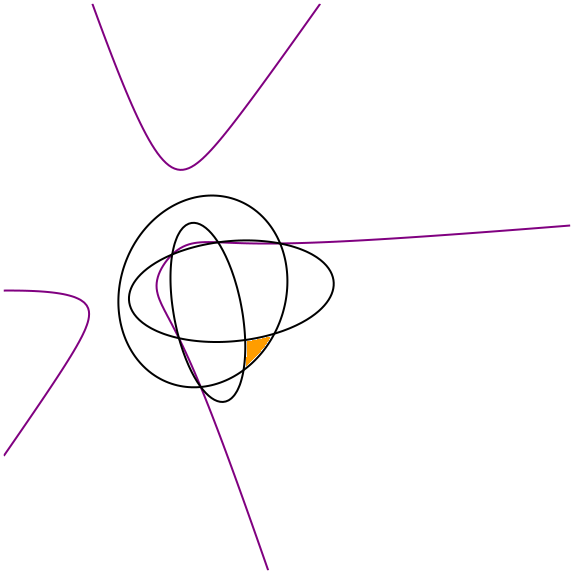}
         \caption{(224) -- 211}
         \label{fig:adjoint224_211}
    \end{subfigure}
             \begin{subfigure}[b]{0.24\textwidth}
         \includegraphics[width=\textwidth]{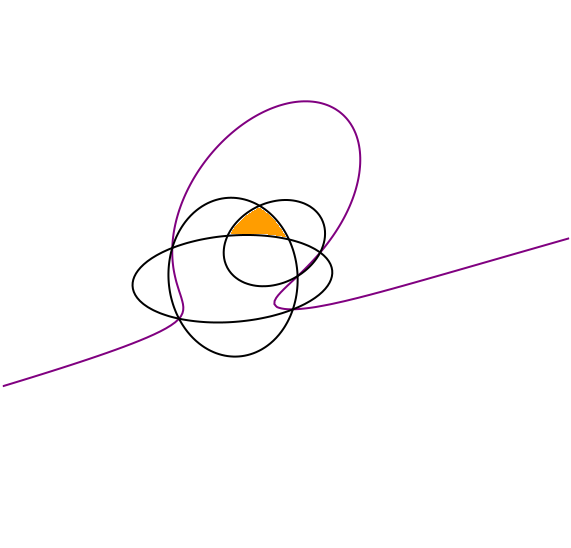}
         \caption{(224) -- 221}
         \label{fig:adjoint224_221}
    \end{subfigure}\hfill
     \begin{subfigure}[b]{0.24\textwidth}
         \includegraphics[width=\textwidth]{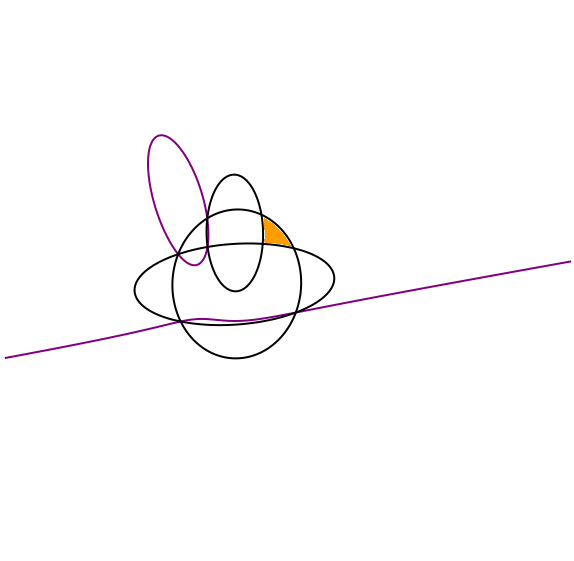}
         \caption{(224) -- 321}
         \label{fig:adjoint224_321}
   \end{subfigure}\hfill
     \begin{subfigure}[b]{0.49\textwidth}
         \includegraphics[width=0.49\textwidth]{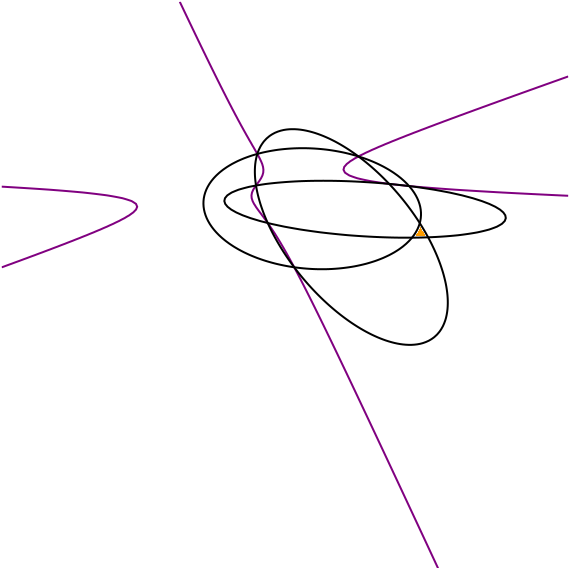}
         \includegraphics[width=0.49\textwidth]{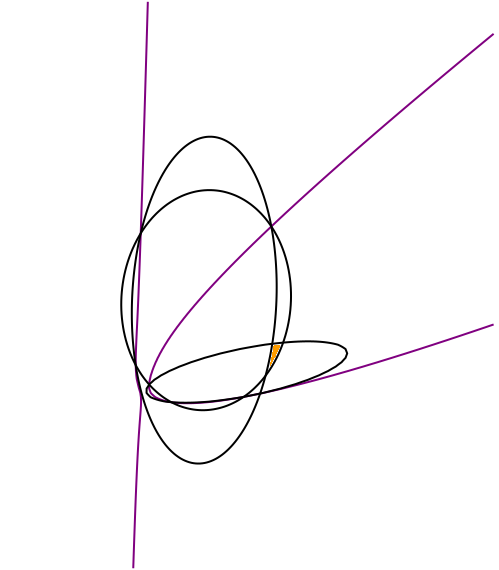}
         \caption{(244) -- 221}
         \label{fig:adjoint244_221}
    \end{subfigure}
     \begin{subfigure}[b]{0.49\textwidth}
         \includegraphics[width=0.49\textwidth]{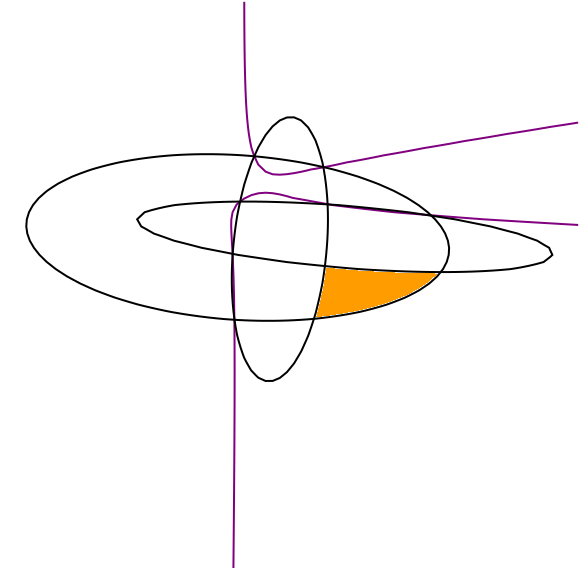}
         \includegraphics[width=0.49\textwidth]{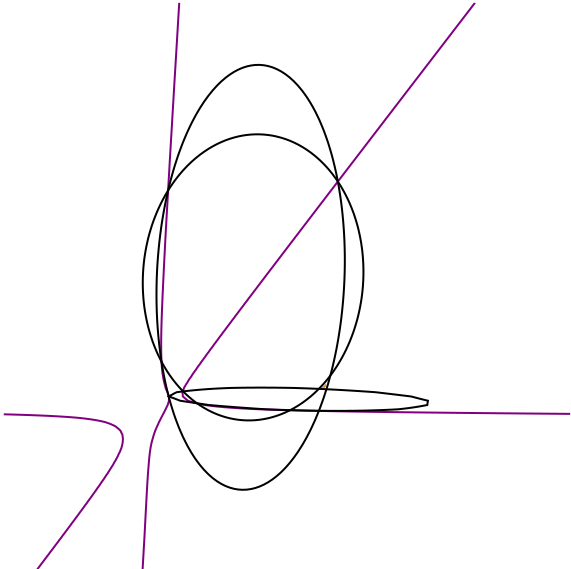}
         \caption{(244) -- 321}
         \label{fig:adjoint244_321}
   \end{subfigure}\hfill
     \begin{subfigure}[b]{0.24\textwidth}
         \includegraphics[width=\textwidth]{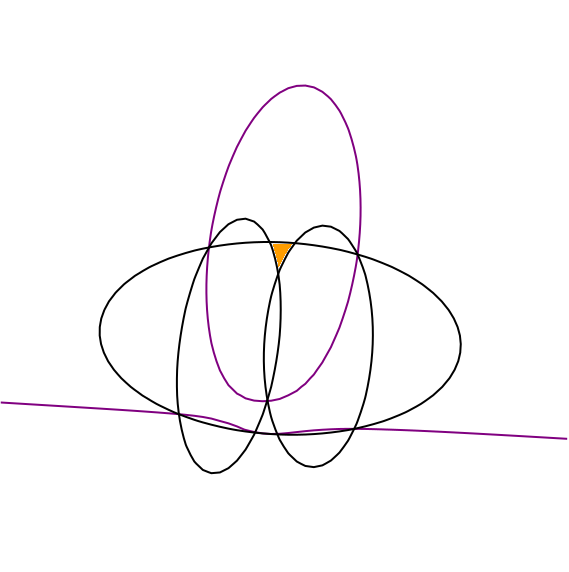}
         \caption{(244) -- 422}
         \label{fig:adjoint244_422}
    \end{subfigure}\hfill
     \begin{subfigure}[b]{0.24\textwidth}
         \includegraphics[width=\textwidth]{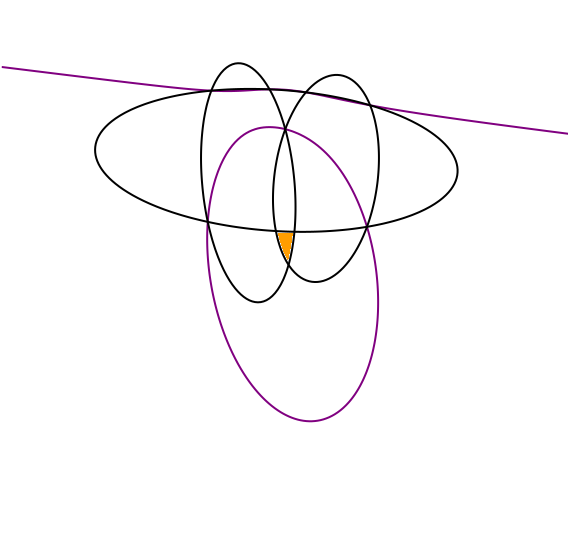}
         \caption{(244) -- 322}
         \label{fig:adjoint244_322}
    \end{subfigure}
     \begin{subfigure}[b]{0.49\textwidth}
         \includegraphics[width=0.49\textwidth]{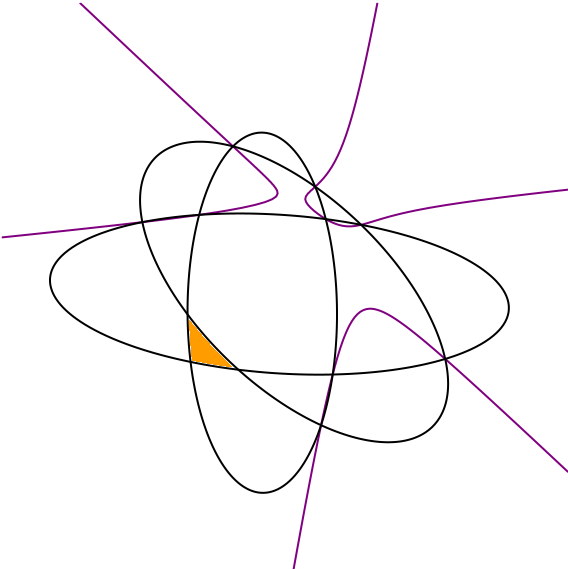}
         \includegraphics[width=0.49\textwidth]{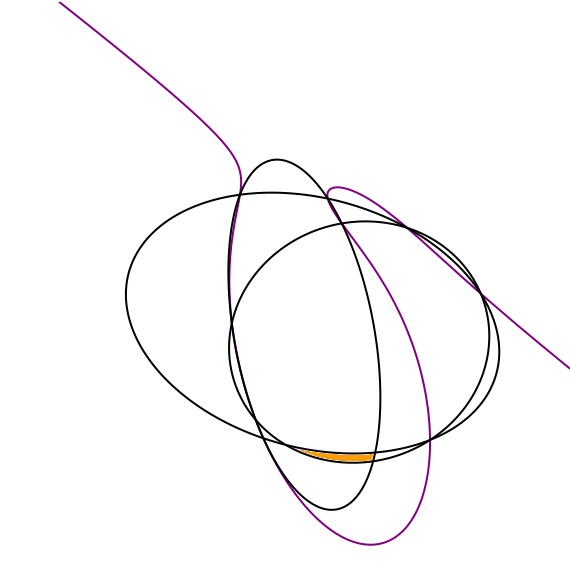}
         \caption{(444) -- 322}
         \label{fig:adjoint444_322}
   \end{subfigure}\hfill
     \begin{subfigure}[b]{0.49\textwidth}
         \includegraphics[width=0.49\textwidth]{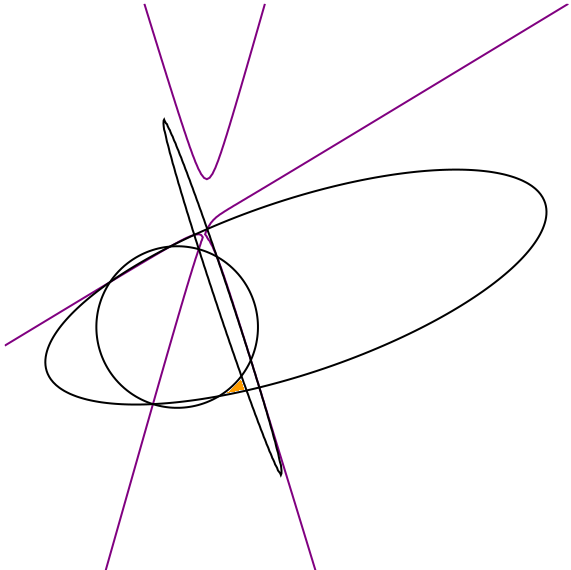}
         \includegraphics[width=0.49\textwidth]{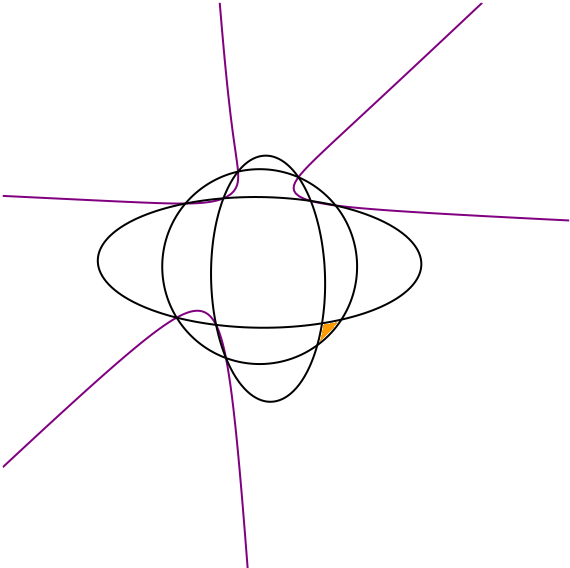}
         \caption{(444) -- 422}
         \label{fig:adjoint444_422}
    \end{subfigure}
        \caption{Problematic polycons and their adjoint curves, using the same notation as in Figure \ref{fig:adjointNonHyp}.} 
        \label{fig:problematicAdjoint}
    \end{figure}

\end{document}